\documentclass[12pt, a4paper]{amsart}
\usepackage{graphics,color,hyperref,verbatim}
\usepackage{amsmath}
\usepackage{amsfonts}
\usepackage{amssymb}
\usepackage{amsthm}
\usepackage{graphicx}
\usepackage{psfrag}
\usepackage{marvosym}
\usepackage{mathrsfs}
\usepackage{enumitem}

\usepackage[a4paper, includeheadfoot, margin=2 cm, top=0.7 cm, footskip=0.7 cm]{geometry} 

\usepackage{multicol} 
\usepackage[usenames,dvipsnames]{xcolor} 
\usepackage{tikz, tikz-3dplot, pgfplots}
\usepackage{tikz-cd}
\usetikzlibrary[positioning,patterns] 
\usetikzlibrary{matrix,arrows,decorations.pathmorphing}

\definecolor{light-gray}{gray}{0.7}

\renewcommand\theenumi{\arabic{enumi}}

\usepackage{ytableau}

\input xy
\xyoption{all}

\usepackage{calligra}
\usepackage{mathrsfs}
\DeclareMathAlphabet{\mathcalligra}{T1}{calligra}{m}{n}
\DeclareFontShape{T1}{calligra}{m}{n}{<->s*[1.5]callig15}{}

\newtheorem{theorem}{Theorem}[section]

\newtheorem{lemma}[theorem]{Lemma}
\newtheorem{proposition}[theorem]{Proposition}

\newtheorem{corollary}[theorem]{Corollary}

\theoremstyle{definition}
\newtheorem{definition}[theorem]{Definition}
\newtheorem{construction}[theorem]{Construction}
\newtheorem{example}[theorem]{Example}

\newtheorem{remark}[theorem]{Remark}
\newtheorem{theorem-definition}[theorem]{Theorem-Definition}
\newtheorem{lemma-definition}[theorem]{Lemma-Definition}

\newtheorem{variant}[theorem]{Variant}
\newtheorem{notation}[theorem]{Notation}
\newtheorem{warning}[theorem]{Warning}

\numberwithin{equation}{section}

\makeatletter
\renewcommand\part{%
   \if@noskipsec \leavevmode \fi
   \par
   \addvspace{4ex}%
   \@afterindentfalse
   \secdef\@part\@spart}

\def\@part[#1]#2{%
    \ifnum \c@secnumdepth >\m@ne
      \refstepcounter{part}%
      \addcontentsline{toc}{part}{Part \thepart.\hspace{1em}#1}%
    \else
      \addcontentsline{toc}{part}{#1}%
    \fi
    {\parindent \z@ \raggedright
     \interlinepenalty \@M
     \normalfont
     \ifnum \c@secnumdepth >\m@ne
     \centering 
       \large\bfseries \partname\nobreakspace\thepart
       \nobreak. 
     \fi
     \large \bfseries { #2}%
     \par}%
    \nobreak
    \vskip 3ex
    \@afterheading}
\def\@spart#1{%
    {\parindent \z@ \raggedright
     \interlinepenalty \@M
     \normalfont
     \huge \bfseries #1\par}%
     \nobreak
     \vskip 3ex
     \@afterheading}
\makeatother

\renewcommand{\thepart}{\Roman{part}}

\renewcommand{\AA} {\mathbb{A}}

\newcommand{\CC} {\mathbb{C}}
\newcommand{\DD} {\mathbb{D}}
\newcommand{\EE} {\mathbb{E}}

\newcommand{\GG} {\mathbb{G}}

\newcommand{\LL} {\mathbb{L}}

\newcommand{\NN} {\mathbb{N}}

\newcommand{\PP} {\mathbb{P}}

\newcommand{\RR} {\mathbb{R}}

\newcommand{\VV} {\mathbb{V}}

\newcommand{\ZZ} {\mathbb{Z}}

\newcommand {\shB} {\mathcal{B}}
\newcommand {\shC} {\mathcal{C}}
\newcommand {\shD} {\mathcal{D}}
\newcommand {\shE} {\mathcal{E}}

\newcommand {\shH} {\mathcal{H}}

\newcommand {\shM} {\mathcal{M}}
\newcommand {\shN} {\mathcal{N}}

\newcommand {\shS} {\mathcal{S}}

\newcommand {\shP} {\mathcal{P}}

\newcommand {\sA} {\mathscr{A}}
\newcommand {\sB} {\mathscr{B}}
\newcommand {\sC} {\mathscr{C}}

\newcommand {\sE} {\mathscr{E}}
\newcommand {\sF} {\mathscr{F}}
\newcommand {\sG} {\mathscr{G}}
\newcommand {\sH} {\mathscr{H}}
\newcommand {\sI} {\mathscr{I}}

\newcommand {\sK} {\mathscr{K}}
\newcommand {\sL} {\mathscr{L}}
\newcommand {\sM} {\mathscr{M}}
\newcommand {\sN} {\mathscr{N}}
\newcommand {\sO} {\mathscr{O}}

\newcommand {\sR} {\mathscr{R}}

\newcommand {\sU} {\mathscr{U}}
\newcommand {\sV} {\mathscr{V}}
\newcommand {\sW} {\mathscr{W}}

\newcommand {\foA} {\mathfrak{A}}

\newcommand {\foC} {\mathfrak{C}}

\newcommand {\foM} {\mathfrak{M}}

\newcommand {\foX} {\mathfrak{X}}

\newcommand {\foZ} {\mathfrak{Z}}

\newcommand {\bS} {\mathbf{S}}

\newcommand{\blank}{\underline{\hphantom{A}}}

\newcommand {\codim} {\operatorname{codim}}
\newcommand {\Coker} {\operatorname{Coker}}

\newcommand {\D} {\operatorname{D}}

\newcommand {\Ext} {\operatorname{Ext}}
\newcommand{\sExt}{\mathscr{E} \kern -1pt xt}

\renewcommand {\H} {\operatorname{H}}
\newcommand {\Hom} {\operatorname{Hom}}
\newcommand {\sHom}{\mathscr{H}\kern-5pt\mathcalligra{om}}

\newcommand {\id} {\operatorname{id}}

\renewcommand {\Im} {\operatorname{Im}}

\newcommand {\Pic} {\operatorname{Pic}}
\newcommand {\Proj} {\operatorname{Proj}}

\newcommand {\pr} {\mathrm{pr}}

\newcommand {\rank} {\operatorname{rank}}
\newcommand {\rk} {\operatorname{rk}}
\newcommand {\red} {{\operatorname{red}}}

\newcommand {\Spec} {\operatorname{Spec}}

\newcommand {\Sym} {\operatorname{Sym}}

\newcommand {\Tor} {\operatorname{Tor}}
\newcommand{\sTor}{\mathscr{T} \kern -3pt or}

\newcommand {\Vect} {\operatorname{Vect}}

\newcommand {\Bl} {\operatorname{Bl}}

\newcommand{\Hilb}{\mathrm{Hilb}}

\newcommand{\Set}{\operatorname{Set}} 
\newcommand{\op}{\mathrm{op}}

\newcommand{\Perf}{\mathrm{Perf}}

\newcommand{\Dqc}{\mathrm{D}_{\mathrm{qc}}}

\newcommand{\bDelta}{\operatorname{{\bf \Delta}}}

\newcommand{\Nhc}{\operatorname{N}_{\bullet}^{\rm hc}}

\newcommand{\Sp}{\mathrm{Sp}}
\newcommand{\cn}{\mathrm{cn}}

\newcommand{\St}{\mathrm{St}}

\newcommand{\Cat}{\shC\mathrm{at}}
\newcommand{\Catinf}{\Cat_{\infty}}
\newcommand{\hatCat}{\widehat{\Cat}_{\infty}}

\renewcommand{\Pr}{\shP \mathrm{r}}

\newcommand{\Ind}{\operatorname{Ind}}
\newcommand{\Fun}{\mathrm{Fun}}
\newcommand{\FunL}{\mathrm{Fun^L}}

\newcommand{\Shv}{\shS\mathrm{hv}}

\newcommand{\Map}{\mathrm{Map}}
\newcommand{\Maps}{\mathrm{Map}}
\newcommand{\Mapsp}{\underline{\Map}}

\newcommand{\CAlg}{\operatorname{CAlg}}
\newcommand{\CAlgDelta}{\operatorname{CAlg}^\Delta}

\newcommand{\Mod}{\operatorname{Mod}}
\newcommand{\Modcn}{\operatorname{Mod}^{\rm cn}}

\newcommand {\perf}{\mathrm{perf}}
\newcommand {\proj} {\mathrm{proj}}

\newcommand{\Ch}{\operatorname{Ch}}

\newcommand{\SCRMod}{\operatorname{SCRMod}}
\newcommand{\SCRModcn}{\operatorname{SCRMod}^{\rm cn}}

\newcommand{\cl}{\mathrm{cl}}

\newcommand{\fib}{\operatorname{fib}}
\newcommand{\cofib}{\operatorname{cofib}}

\newcommand{\QCoh}{\operatorname{QCoh}}

\newcommand{\QCAlgDelta}{\operatorname{QCAlg}^\Delta}

\newcommand{\depth}{\operatorname{depth}}

\newcommand {\T} {\operatorname{T}}

\newcommand {\bwedge} {\mathbf{\Lambda}}

\title[]{Derived Projectivizations of complexes}
\author[Q.Y.\ JIANG]{Qingyuan Jiang}

\address{School of Mathematics, University of Edinburgh, James Clerk Maxwell Building, Peter Guthrie Tait Road, Edinburgh EH9 3FD, United Kingdom.}
\email{qingyuan.jiang@ed.ac.uk}

\begin{document}

\begin{abstract} In this paper, we study the counterpart of Grothendieck’s projectivization construction in the context of derived algebraic geometry. Our main results are as follows: First, we define the derived projectivization of a connective complex, study its fundamental properties such as finiteness properties and functorial behaviors, and provide explicit descriptions of their relative cotangent complexes. We then focus on the derived projectivizations of complexes of perfect-amplitude contained in $[0, 1]$. In this case, we prove a generalized Serre’s theorem, a derived version of Beilinson’s relations, and establish semiorthogonal decompositions for their derived categories. Finally, we show that many moduli problems fit into the framework of derived projectivizations, such as moduli spaces that arise in Hecke correspondences. We apply our results to these situations. 
\end{abstract}

\maketitle
\tableofcontents
 
 \section{Introduction}
 
Grothendieck’s construction of projectivizations plays a vital role in classical algebraic geometry. For example, projectivizations and their properties are crucial for defining projective schemes and morphisms and studying their properties. In this paper, we study the counterpart of Grothendieck’s construction in the framework of derived algebraic geometry. 
Our main results are as follows:

\begin{enumerate}[leftmargin=*]
	 \item For a derived scheme $X$ (or more generally, a prestack $X$) and a connective quasi-coherent sheaf $\sE$ on $X$:
	 \begin{enumerate}
		\item we will define a (relative) derived scheme $\PP(\sE)$ over $X$, called the {\em derived projectivization} of $\sE$ over $X$ (\S \ref{sec:def:dproj:qcoh}), whose underlying classical scheme agrees with Grothendieck’s construction for the classical truncation (Proposition \ref{prop:proj-classical}). 
		\item We study the fundamental properties of the derived projectivizations, including the finiteness properties (Proposition \ref{prop:proj:finite}) and their functorial behaviors under base change, tensoring with line bundles (Proposition \ref{prop:proj-4,5}), and quotient maps (Corollary \ref{cor:proj:closedimmersion}, Proposition \ref{prop:proj:PB}). We also explicitly describe their relative cotangent complexes by proving a generalized Euler sequence (Theorem \ref{thm:proj:Euler}).
\end{enumerate}
	 
	\item  In the case where $\sE$ has perfect-amplitude contained in $[0,1]$, the derived projectivization $\PP(\sE)$ enjoys special pleasant features. In this case:
		\begin{enumerate}
			\item We will prove a generalized Serre’s theorem for $\PP(\sE)$ (Theorem \ref{thm:Serre:O(d)}).
			\item We show that $\PP(\sE)$ has a natural pair of dual relative exceptional sequences which satisfies a derived version of Beilinson’s relations (Proposition \ref{prop:PG:dualexc}).
			\item We establish a semiorthogonal decomposition of the $\infty$-category $\QCoh(\PP(\sE))$ of quasi-coherent sheaves on $\PP(\sE)$ (Theorem \ref{thm:structural}).
		\end{enumerate}
		
	\item We apply the our results to various situations:
		\begin{enumerate}
			\item (Classical situations) For example, we could apply these results and obtain:
			\begin{itemize}[leftmargin=*]
				\item Semiorthogonal decompositions of the derived categories of schemes of the form $X \coprod\nolimits_{D} \PP_D^1$, where $D \subset X$ is any effective Cartier divisor, or the reducible schemes appearing as the central fiber of deformation to normal cones; see \S \ref{sec:intro:reducible}. 
				\item Descriptions of derived categories under stabilization maps of nodal curves (\S \ref{sec:intro:curves}).
				\item Semiorthogonal decompositions of derived categories of certain threefolds with rational double points (\S \ref{sec:intro:threefolds}).
				\item Examples of derived equivalences of classical schemes induced by non-trivially derived schemes (\S \ref{sec:intro:bir}).
		\end{itemize} 

			\item (Moduli problems) Many moduli spaces have natural derived enhancements which fit into the framework of derived projectivizations, such as moduli spaces that arise in Hecke correspondences and moduli of extensions. 
			
	\begin{itemize}[leftmargin=*]
		\item We study the situations of Hecke correspondences for schemes, and apply our construction and results to this situation (\S \ref{sec:Hecke}).
		\item  In \cite{J22b}, we apply the results of this paper to study the derived Abel maps for integral curves to their compactified Jacobians, and provide semiorthogonal decompositions of the derived Hilbert schemes of points on integral curves.
\end{itemize}
			
		\end{enumerate}
\end{enumerate}

\begin{remark}[Conventions] This paper works in the framework of derived algebraic geometry and uses {\em Lurie's terminology} (\S \ref{sec:DAG}, \S \ref{sec:Notations}). Most results of this paper require no conditions on the base  $X$: it could be arbitrarily stacky or derived (or singular). However, since the projectivization is a ``relative phenomenon", the readers might feel free to assume $X$ to be a nice classical (underived) scheme and still get the essence of most of the results in this introduction. If one wishes to assume that $X$ is a classical Noetherian scheme, beware the following {\em distinctions of conventions} between derived and classical algebraic geometry:

\begin{itemize}[leftmargin=*]
	\item A {\em quasi-coherent sheaf} $\sE$ on $X$ in the derived context (as in this paper) corresponds to an {\em unbounded complex} $\sE_* = [ \,\cdots \to \sE_1 \to \sE_0 \to \sE_{-1} \to \cdots \,]$ of quasi-coherent sheaves on $X$ in the classical sense. The {\em $i$th homotopy group} $\pi_i(\sE)$ corresponds to the sheaf (co)homology $\sH_i(\sE_*) = \sH^{-i}(\sE_*)$. We say $\sE$ is {\em connective} if satisfies $\pi_i(\sE)=\sH_i(\sE_*) \simeq 0$ for all $i < 0$, and we say $\sE$ is {\em discrete} if $\pi_i(\sE) =\sH_i(\sE_*) \simeq 0$ for all $i \ne 0$. In particular, {\em discrete} quasi-coherent sheaves on $X$ correspond to the usual quasi-coherent sheaves on $X$ in the classical sense.
	\item The collection of all quasi-coherent sheaves on $X$ is organized into a stable $\infty$-category $\QCoh(X)$, whose homotopy category is the usual {derived category} $\Dqc(X)$ of quasi-coherent complexes on $X$. The subcategory of discrete quasi-coherent sheaves, denoted by $\QCoh(X)^\heartsuit$, is equivalent to the (nerve of the) usual abelian category ${\rm qcoh}(X)$ in the classical setting. 
	\item $\Perf(X)$ denotes the stable $\infty$-category of perfect objects of $\QCoh(X)$, whose homotopy category is the full triangulated subcategory $\D^{\perf}(X)$ of $\Dqc(X)$ spanned by perfect complexes.
	\item A quasi-coherent sheaf $\sE \in \QCoh(X)$ is said to have {\em perfect-amplitude (contained) in $[0,1]$} if it is represented by a perfect complex $\sE_*$ on $X$ of homological Tor-amplitude in $[0,1]$ (or equivalently, of cohomological Tor-amplitude in $[-1,0]$) in the classical sense.
\end{itemize}
\end{remark}

\subsection{Derived projectivizations}

Let $X$ be a classical scheme, and let $\sE$ be a discrete quasi-coherent sheaf on $X$. Then Grothendieck’s projectivization of $\sE$ over $X$, which we shall denote as $\PP_{\cl}(\sE) = \Proj \Sym_{\cl}^* \sE$, is the scheme over $X$ which represents the following functor: for any morphism of schemes $\eta \colon T \to X$, the set of $T$-points $\PP_{\cl}(\sE)(\eta)$ over $\eta$ is the set of equivalence classes of rank-one locally free quotients of $\sE$. 

Let $X$ be a derived scheme (or more generally, a prestack), and let $\sE$ be a connective quasi-coherent sheaf on $X$. We define the {\em derived projectivization $\PP(\sE)$ of $\sE$ over $X$} via a similar manner as above, that is, it is the following functor:
\begin{itemize}
	\item For every map $\eta \colon T \to X$ from a {\em derived} scheme $T$, the space of $T$-points $\PP(\sE)(\eta)$ over $\eta$ is the space of maps of quasi-coherent sheaves $\eta^* \sE \to \sL$ which induce a surjection $\pi_0(\eta^* \sE) \twoheadrightarrow \pi_0(\sL)$
	, where $\sL$ is a line bundle on $T$.
\end{itemize}	
(See \S \ref{sec:def:dproj:qcoh} for more details.) More informally, we could think of $\PP(\sE)$ as the {\em derived} moduli space which parametrizes rank-one locally free quotients of $\sE$. 

The derived projectivization theory extends its classical counterpart in two ways:
\begin{enumerate}[label=(\alph*), leftmargin=*]
	\item We allow $X$ to be any derived scheme (or more generally, any prestack).
	\item We allow $\sE$ to be any connective quasi-coherent sheaf, that is, it needs not be discrete. 
\end{enumerate}

We show that the derived projectivization $\PP(\sE)$ enjoys the following features:

\begin{enumerate}[leftmargin=*]
	\item[$(0)$] ({\bf Underlying classical schemes}) The underlying classical scheme of $\PP(\sE)$ is canonically equivalent to Grothendieck's projectivization construction applied to its classical truncation, that is, there is a canonical equivalence $(\PP_X(\sE))_{\cl} \simeq \PP_{X_{\cl}, \cl}(\pi_0(\sE))$ (Proposition \ref{prop:proj-classical}). 
	
	For example, suppose $X$ is a classical scheme and $\sE$ is represented by a complex $\sE_*$ of discrete quasi-coherent sheaves on $X$. In this case, the derived projectivization $\PP(\sE)$ has the same underlying classical scheme as Grothendieck's construction $\PP_{\cl}(\sH_0(\sE_*))$, but usually with a {non-trivial derived} structure induced from the ``hidden higher  information" of $\sE$.
	
\end{enumerate}
	\begin{warning} Let $X$ be classical and let $\sE$ be represented by a complex $\sE_*$. The {``hidden higher information"} of $\sE$ which contributes to the derived structure of $\PP(\sE)$ includes {not only} all the higher homologies $\{\sH_i(\sE_*)\}_{i >0}$ of $\sE_*$ {but also} the higher homologies $\{\sH_i(\Sym^d(\sE_*))\}_{i>0, d>0}$ of the {\em derived symmetric powers} of $\sE_*$. For example, even in the simple case where $\sE = \sO_{p}^{\oplus 2}$ on $\AA^1$, where $p \in \AA^1$ is the origin, the derived projectivization $\PP(\sO_{p}^{\oplus 2})$ is {not} isomorphic to Grothendieck's projectivization $\PP_{\cl}(\sO_{p}^{\oplus 2}) \simeq \PP^1$ but is equipped with a non-trivial derived structure (Example \ref{eg:A^1:PFcl}). If $X$ is an integral Cohen--Macaulay scheme and $\sE$ has homological dimension one, we will provide criterion for $\PP(\sE)$ to be classical; see Lemma \ref{lem:criterion:P:classical}.
	\end{warning}

\begin{enumerate}[leftmargin=*]	
	\item[$(1)$] ({\bf Representability}) The functor $\PP(\sE)$ is representable by a relative derived scheme over $X$ (Proposition \ref{prop:proj:represent}). 
	\item[$(2)$] ({\bf Functoriality}) The formation of derived projectivizations commutes with arbitrary base change and tensoring with line bundles (Proposition \ref{prop:proj-4,5}).
	\item[$(3)$] ({\bf Finiteness properties}) If $\sE$ is perfect to order $n$ (resp. almost perfect, perfect), then the projection $\pr \colon \PP(\sE) \to X$ is locally of finite presentation to order $n$ (resp. locally almost of finite presentation, locally of ﬁnite presentation); see Proposition \ref{prop:proj:finite}.
 	\item[$(4)$] ({\bf Closed immersions}) For any map $\varphi \colon \sF \to \sE$ of quasi-coherent sheaves on $X$ which is surjective on $\pi_0$, there is a canonical closed immersion $\iota \colon \PP(\sE) \hookrightarrow \PP(\sF)$ which is compatible with the formation of universal line bundles and tautological quotients. Moreover, we have an explicit description of the closed immersion $\iota$ and its relative cotangent complex in terms of the morphism $\varphi$ (Corollary \ref{cor:proj:closedimmersion} and Proposition \ref{prop:proj:PB}).
 	\item[$(5)$] ({\bf Euler fiber sequences}) The projection $\pr \colon \PP(\sE) \to X$ admits a connective relative cotangent complex $L_{\PP(\sE)/X}$ which fits into a canonical exact triangle
			$$L_{\PP(\sE)/X}  \otimes \sO(1) \to \pr^* \sE \xrightarrow{\rho} \sO(1)$$
		(here $\rho$ denotes the tautological morphism from $\pr^* \sE$ to $\sO(1)$); see Theorem \ref{thm:proj:Euler}. We shall refer to the above sequence as the {\em Euler fiber sequence}.
\end{enumerate}

\begin{remark}	The property $(5)$ seems to provide a perfect example where an object in derived algebraic geometry behaves {\em manifestly better} than its counterpart in classical algebraic geometry. The above Euler sequence provides an explicit description of the relative cotangent complexes for {\em all} derived projectivizations $\PP(\sE)$. On the other hand, as far as the author is aware, there seems to be no general explicit description of the relative cotangent complexes for classical projectivizations (except in the special cases where $\sE$ is a vector bundle or a ``generic" discrete sheaf of homological dimension one).
\end{remark}

The construction of derived projectivization is closely related to {\em affine cones} $\VV(\sE)$, which we will define and study in \S \ref{sec:affine-cone}. Our proof of the properties of derived projectivizations and affine cones depends heavily on the properties of the {\em derived symmetric algebras}, which we will explore in detail in \S \ref{sec:SymAlg} (the local situation) and \S \ref{sec:QStk} (the global situation).

\subsection{Derived projectivizations of sheaves of perfect-amplitude in $[0,1]$} The derived projectivization $\PP(\sE)$ enjoys further pleasant properties if $\sE$ has perfect-amplitude contained in $[0,1]$. For example, in this case, (as an immediate consequence of the above properties $(3)$ and $(5)$) the projection $\pr \colon \PP(\sE) \to X$ is a quasi-smooth proper morphism.

Our main results regarding derived projectivizations of quasi-coherent sheaves of perfect-amplitude in $[0,1]$ are as follows:

\begin{enumerate}[leftmargin=*]
	\item ({\bf Generalized Serre's theorem}) Let $X$ be a derived scheme (or more generally, a prestack), and let $\sE \in \QCoh(X)$ be of perfect-amplitude contained in $[0,1]$ and rank $\delta$. If $\sE$ has rank $\delta \ge 1$, then for all integers $d$, there are canonical equivalences
	$$
	 \pr_* (\sO(d))  \simeq 
	 	\begin{cases}
	\Sym_X^d \sE  & \qquad \text{if~}  d \ge 0; \\
	0 &  \qquad \text{if ~} -\delta+1 \le d \le -1; \\
	(\Sym_X^{-d-\delta} \sE \otimes_{\sO_X}  \det \sE)^{\vee}[1-\delta]	& \qquad \text{if~} d \le -\delta,
	\end{cases}
	$$
		where $\Sym_X^d \sE$ denotes the {derived} symmetric powers of $\sE$ over $X$ (see \S \ref{sec:non-abelian}, \S \ref{sec:sym:prestack}). If $\sE$ has rank $\delta \le 0$, then there is a canonical exact triangle
			$$\Sym_{X}^d (\sE) \to \pr_*(\sO(d)) \to (\Sym_X^{-d-\delta} \sE \otimes_{\sO_X} \det \sE)^{\vee} [1-\delta]$$
		which generalizes the classical Eagon--Northcott complex; see Theorem \ref{thm:Serre:O(d)}.
	\item ({\bf Beilinson's relations}) Let $X$ be a quasi-compact, quasi-separated derived scheme (or more generally, a perfect stack (Definition \ref{def:perfstacks})), and let $\sE \in \QCoh(X)$ be of perfect-amplitude in $[0,1]$ and rank $\delta \ge 1$. Let $\sO(1)$ denote the universal line bundle on $\PP(\sE)$ and let $\sO(i)$ denote $\sO(1)^{\otimes i}$, where $i \ge 0$ is an integer. Then the sequence of line bundles
		$$\sO, \sO(1), \cdots, \sO(\delta-1)$$
 forms a relative exceptional sequence (\S \ref{sec:exc:mut}) of $\QCoh(\PP(\sE))$ over $X$, whose left dual relative exceptional sequence is given by
		$$\Sym_{\PP(\sE)}^{\delta-1}(L_{\PP(\sE)/X}(1)[1]) , \cdots, L_{\PP(\sE)/X}(1)[1], \sO,$$
	where $L_{\PP(\sE)/X}$ is the relative cotangent complex for the projection $\pr \colon \PP(\sE) \to X$, $L_{\PP(\sE)/X}(1)$ denotes $L_{\PP(\sE)/X} \otimes \sO(1)$, and $[1]$ is the degree shift; see Proposition \ref{prop:PG:dualexc}. 
	By virtue of Illusie's result (Proposition \ref{prop:symwedgegamma} \eqref{prop:symwedgegamma-4}), there are canonical equivalences
		$$\Sym_{\PP(\sE)}^{i}\big(L_{\PP(\sE)/X}(1)[1]\big) \simeq \big(\bigwedge\nolimits_{\PP(\sE)}^{i} L_{\PP(\sE)/X}\big) (i)[i],$$
	where $\Sym_{\PP(\sE)}^i$ and $\bigwedge_{\PP(\sE)}^i$ denote the $i$th {derived} symmetric powers and $i$th {derived} exterior powers over ${\PP(\sE)}$, respectively (see \S \ref{sec:non-abelian}, \S \ref{sec:sym:prestack}). Therefore we could equally write the above left dual relative exceptional sequence as 
		$$\big(\bigwedge\nolimits_{\PP(\sE)}^{\delta-1} L_{\PP(\sE)/X}\big) (\delta-1)[\delta-1], \cdots, L_{\PP(\sE)/X}(1)[1], \sO.$$
	Furthermore, for all integers $i, j \in [0, \delta-1]$, we have the following derived version of Beilinson's relations (\cite{Be}, where similar results for projective spaces are obtained):
	\begin{equation*}
	\left\{
	\begin{split}
		& \Mapsp_{\PP(\sE)/X} (\sO(i), \sO(j)) \simeq \Sym_X^{j-i} (\sE), \\ 
		&  \Mapsp_{\PP(\sE)/X}\Big(\sO(i),\Sym_{\PP(\sE)}^{j}\big(L_{\PP(\sE)/X}(1)[1]\big) \Big) \simeq \delta_{i,j} \cdot \sO_X, \\
		& \Mapsp_{\PP(\sE)/X} \Big(\big(\bigwedge\nolimits_{\PP(\sE)}^{i} L_{\PP(\sE)/X}\big) (i), \big(\bigwedge\nolimits_{\PP(\sE)}^{j} L_{\PP(\sE)/X}\big) (j)\Big) \simeq (\bigwedge\nolimits_X^{i-j} \sE)^\vee,
	\end{split}
	\right.
	\end{equation*}
	where $\Mapsp_{\PP(\sE)/X}(\sF, \sG) \in \QCoh(X)$ denotes the {mapping object} of a pair of objects $\sF \in \Perf(\PP(\sE))$ and $\sG \in \QCoh(\PP(\sE))$ (Construction \ref{constr:mappingobject}), and $\sH^\vee$ denotes the dual in $\QCoh(X)$ of an object $\sH \in \Perf(X)$;
	see Corollary \ref{cor:PVdot:relations}.
	\item ({\bf Semiorthogonal decompositions}) Let $X$ be a derived scheme (or more generally, a prestack) and $\sE$ a quasi-coherent sheaf on $X$ with perfect-amplitude in $[0,1]$ and rank $\delta \ge 0$. Then the ``shifted derived dual" $\sE^\vee [1]$ also has perfect-amplitude contained in $[0,1]$. We let $\widehat{Z}$ be the {\em universal incidence locus} of $\PP(\sE)$ and $\PP(\sE^\vee[1])$ over $X$, which is defined as the derived closed subscheme of the (derived) fiber product $\PP(\sE) \times_X \PP(\sE^\vee[1])$ which classifies paths between the composite map 
	$$\sO_{\PP(\sE)}(-1) \boxtimes \sO_{\PP(\sE^\vee[1])}(-1) [1] \to \sE^\vee \boxtimes \sE \xrightarrow{\rm ev} \sO_{\PP(\sE) \times_X \PP(\sE^\vee[1])}$$
	and the zero map; see Notation \ref{notation:hatZ}. Therefore, we have a commutative diagram
	\begin{equation*} 
	\begin{tikzcd}[row sep= 2.5 em, column sep = 4 em]
		\widehat{Z} \ar{rd}{\widehat{\pr}}\ar{r}{r_+}  \ar{d}[swap]{r_-} &  \PP(\sE) \ar{d}{\pr_+}\\
		\PP(\sE^\vee[1]) \ar{r}{\pr_-} & X.
	\end{tikzcd}
	\end{equation*}
	We will show that the functor induced by the above correspondence $\widehat{Z}$,
		$$\Phi = r_{+ \,*} \circ r_-^* \colon \QCoh(\PP(\sE^\vee[1])) \to \QCoh(\PP(\sE)),$$
	is fully faithful, and if $\delta \ge 1$, then for all integers $i \in \ZZ$, the functors 
		$$\pr_+^*(\blank) \otimes \sO(i) \colon \QCoh(X) \to \QCoh(\PP(\sE))$$
	are fully faithful. Furthermore, there are induced semiorthogonal decomposition
	 \begin{align*}
		\QCoh(\PP(\sE))& = \big\langle \Phi \big(\QCoh(\PP(\sE^\vee[1])) \big), ~ \pr_+^*(\QCoh(X)) \otimes \sO(1), \ldots, \pr_+^* (\QCoh(X)) \otimes \sO(\delta) \big\rangle, \\
	\Perf(\PP(\sE)) &= \big\langle \Phi \big(\Perf(\PP(\sE^\vee[1])) \big), ~ \pr_+^*(\Perf(X)) \otimes \sO(1), \ldots, \pr_+^* (\Perf(X)) \otimes \sO(\delta) \big\rangle;
	\end{align*}
	see Theorem \ref{thm:structural} for more details. 
\end{enumerate}

\begin{remark} Combining these semiorthogonal decompositions with the above result $(2)$, we could obtain a whole series of semiorthogonal decompositions of $\QCoh(\PP(\sE))$ (or $\Perf(\PP(\sE))$) by mutating the relative exceptional sequence $\sO(1), \sO(2), \ldots, \sO(\delta)$.
\end{remark}

\begin{remark} In the context of classical algebraic geometry, the above semiorthogonal decompositions are referred to as the {\em projectivization formula} in \cite[Theorem 3.4]{JL18}; see also \cite[Theorem 5.5]{Kuz07}, \cite[Theorem 4.6.11]{Tod3}, and \cite[Theorem 6.16]{J21}. 
\end{remark}

\subsection[Applications to Classical Situations]{Applications to classical situations} The results of the preceding sections have many interesting applications even in the classical situation. 
For this purpose, we provide in  \S \ref{sec:classical} criteria for the derived projectivizations $\PP(\sE)$ and $\PP(\sE^\vee[1])$, the universal incidence loci $\widehat{Z}$, and derived symmetric products $\Sym^d(\sE)$ and $\Sym^d(\sE^\vee[1])$ to be classical (Lemmata  \ref{lem:criterion:P:classical} and \ref{lem:criterion:Sym:classical}). For example, assume that $X$ is an integral Cohen--Macaulay scheme, and $\sE$ is cofiber of a generically injective map between vector bundles. If $\rank \sE>0$, then under the above assumption, there is a hierarchy of conditions on the codimensions of degeneracy loci of $(X, \sE)$,
	$$\eqref{eqn:condition:PF} \impliedby \eqref{eqn:condition:irr:PF} \impliedby \eqref{eqn:condition:PK} \impliedby \eqref{eqn:condition:irr:PK} \impliedby \eqref{eqn:condition:hatZ} \impliedby \eqref{eqn:condition:irr:hatZ},$$
which is in one-to-one correspondence to the following sequence of conditions:
	\begin{align*}
	&\text{$\PP(\sE)$ is classical} \impliedby \text{$\PP(\sE)$ is irreducible}  \impliedby \text{$\PP(\sE^\vee[1])$ is classical}  \\ 
	&\impliedby 
	\text{$\PP(\sE^\vee[1])$ is irreducible} 
	 \impliedby \text{$\widehat{Z}$ is classical}  \impliedby \text{$\widehat{Z}$ is  irreducible}.
	 \end{align*}
For example, $\eqref{eqn:condition:PF}$, $\eqref{eqn:condition:irr:PF}$, and $\eqref{eqn:condition:PK}$ are the conditions $\codim_X(X_i) \ge i$, $\codim_X(X_i) \ge i+1$, and $\codim_X(X_i) \ge \rank \sE + i$, for all degeneracy loci $X_i \subseteq X$, respectively. There is also a similar description if $\rank \sE =0$; see Lemma \ref{lem:criterion:P:classical} for more details. In particular, there are numerous counter-examples for each of the inverse of the above implications ``$\impliedby$", and the results of the preceding subsection have interesting consequences in these situations. 

We will only mention some of these applications in this introduction, and refer the readers to \S \ref{sec:classical} in the main text for more examples. 

\subsubsection{Derived categories of reducible schemes; see Examples \ref{eg:DinX} and \ref{eg:normalcones}} \label{sec:intro:reducible}

Let $X$ be any scheme, let $D \subseteq X$ be any effective Cartier divisor. Let $\sE = \sO_X \oplus \sO_D$. Then (locally) condition \eqref{eqn:condition:PF} is satisfied but \eqref{eqn:condition:irr:PF} is not. More concretely, $\PP(\sE) =  X \coprod\nolimits_{D} \PP_D^1$ is a classical reducible scheme, which is the union of $X$ and $\PP^1_D = D \times \PP^1$ along $D$; and $\PP(\sE^\vee[1])
= \VV_D(\sO_D(-D)[1])$ is the total space $
{\rm Tot}(\shN_{D/X}[-1])$ of ``$(-1)$-shifted normal bundle of $D$ in $X$" in the sense \cite{ACH, PTVV}, which has the same underlying classical scheme as $D$ but equipped with non-trivial derived structure everywhere. Then Theorem \ref{thm:structural} implies that there are semiorthogonal decompositions
	\begin{align*}
	\QCoh\big(X \coprod\nolimits_{D} \PP_D^1\big)  & = \big\langle \QCoh(X), ~\QCoh({\rm Tot}(\shN_{D/X}[-1]) ) \big \rangle, \\
	\Perf\big(X \coprod\nolimits_{D} \PP_D^1\big)  & = \big\langle \Perf(X), ~\Perf({\rm Tot}(\shN_{D/X}[-1])) \big\rangle.
	\end{align*}
See Examples \ref{eg:DinX} for more details. 

On the other hand, if we let $\sE' = \sO_X \oplus \sO_D(D)$, then $\PP(\sE') = X \coprod_{D} \PP_D(\shN \oplus \sO_D)$, where $\shN = \shN_{D/X}$ is the normal bundle, and $\PP(\sE'^\vee[1]) = D \times \Spec \ZZ[\varepsilon]$. Therefore we obtain 
	\begin{align*}
	\QCoh\big(X \coprod\nolimits_{D} \PP_D(\shN \oplus \sO_D)\big)  & = \big\langle \QCoh(X), ~\QCoh(D) \otimes \QCoh(\Spec \ZZ[\varepsilon])  \big \rangle, \\
	\Perf\big(X \coprod\nolimits_{D} \PP_D(\shN \oplus \sO_D)\big)  & = \big\langle \Perf(X), ~\Perf(D) \otimes \Perf(\Spec \ZZ[\varepsilon]) \big\rangle.
	\end{align*}
Here, the reducible scheme $X \coprod_{D} \PP_D(\shN \oplus \sO_D)$ is precisely the special fiber appearing in the construction of {\em deformation to normal cones} of \cite[\S 5.1]{Ful}; see Variant \ref{eg:normalcones}. 

We could regard the semiorthogonal decompositions of the above two examples as {\em blowup formulae}, where $\PP(\sE)$ and $\PP(\sE')$ are the {\em derived blowups}  in the sense of \cite{KR18, He21} of $X$ along the virtual codimension two closed derived subschemes ${\rm Tot}(\shN_{D/X}[-1]) \subseteq X$ and $D \times \Spec \ZZ[\varepsilon] \subseteq X$, respectively; see Remark \ref{rmk:egs:blowup}. 
	
\subsubsection{Contracting rational tails and bridges for curves; see Examples \ref{eg:rational_tail} and \ref{eg:rational_bridge}} \label{sec:intro:curves}

In the preceding example, if we let $X=C$ be a reduced curve over a field $\kappa$, let $p \in C(\kappa)$ be a non-singular $\kappa$-point, and let $\sE = \sO_C \oplus \sO_p$, then we obtain semiorthogonal decompositions
	\begin{align*}
	\QCoh\big(C \coprod\nolimits_{p} \PP^1\big)  & = \big\langle \QCoh(C), ~\QCoh(\Spec \kappa[\varepsilon] ) \big \rangle, \\
	\Perf\big(C \coprod\nolimits_{p} \PP^1\big)  & = \big\langle \Perf(C), ~\Perf(\Spec \kappa[\varepsilon] ) \big\rangle,
	\end{align*}
where $\kappa[\varepsilon] = \Sym^*_\kappa(\kappa[1])$ is the ring of {\em derived} dual numbers over $\kappa$ (where we might informally think of $\varepsilon$ as a generator in homological degree $1$). See \ref{eg:rational_tail} for more details. This result agrees with Kuznetsov and Shinder's in \cite[Example 5.2, Theorem 5.12]{Kuz21}.  

We might also repeat the above process by setting $C_0 = C$, $p_0 = p$, and letting $C_i = C_{i-1} \coprod_{p_{i-1}} \PP^1$, where $p_{i-1} \in C_{i-1}(\kappa)$ is any nonsingular $\kappa$-point. 
Then for any $n \ge 1$, we obtain semiorthogonal decompositions:
\begin{align*}
	\QCoh(C_n)  & = \big\langle \QCoh(C), ~\text{$n$-copies of $\QCoh(\Spec \kappa[\varepsilon])$} \big \rangle, \\
	\Perf(C_n)  & = \big\langle \Perf(C), ~\text{$n$-copies of $\Perf(\Spec \kappa[\varepsilon])$} \big\rangle.
	\end{align*}
The morphism spaces of the latter components depend (only) on the configurations of the trees of $\PP^1$'s on $C_n$. For example, 
in the case where $C_n$ is the union of $C$ and a chain $\Gamma_{n}$ of $\PP^1$'s (and $\kappa$ has characteristic zero), the endomorphism algebra (of a choice of tilting bundle on $\Gamma_n$) of the latter component is described explicitly in \cite[Theorem 2.1]{Bur}. The cases of where $C_n$ is union of $C$ and trees of $\PP^1$'s can be described similarly, in view of \cite[Theorem 4.9]{KPS}. 

On the other hand, let $C = \{xy = 0\} \subset \AA_\kappa^2$, where $\kappa$ is a field, and let $o=(0,0)$ denote the  node of $C$. 
 Then $\widetilde{C}:=\PP(\Omega_{C/\kappa}^1)$ is the total transform of $C$ along the usual blowup $\Bl_{o}(\AA^2) \to \AA^2$, and the map $\pi \colon \widetilde{C} \to C$ is the contraction of the (non-reduced) rational bridge $E=2\PP_\kappa^1$.
We obtain from Theorem \ref{thm:structural} semiorthogonal decompositions
	\begin{align*}
	\QCoh(\widetilde{C} )  & = \big\langle \pi^*\QCoh(C), ~\QCoh(\Spec (\kappa[\varepsilon] \otimes \kappa[\delta]) ) \big \rangle, \\
	\Perf(\widetilde{C} )  & = \big\langle \pi^*\Perf(C), ~\Perf(\Spec (\kappa[\varepsilon]\otimes \kappa[\delta]) ) \big\rangle;
	\end{align*}
where $\kappa[\delta] = \kappa[\delta]/(\delta^2)$ denote the ring of usual dual numbers; 
see Example \ref{eg:rational_bridge} for more details. 
In this example, one could replace $C$ by any nodal curve $C$. Notice that a recent result  \cite[Corollary 6.8]{KS22b} implies that there are no semiorthogonal decompositions for the {\em reduced} scheme $\widetilde{C}^{\rm red}$ when $C$ is projective. 

These results altogether provide a comprehensive picture of the behavior of derived categories of nodal curves under stabilization maps, since the stabilization maps of nodal curves are compositions of contracting rational tails and rational bridges (\cite[\href{https://stacks.math.columbia.edu/tag/0E7Q}{Tag 0E7Q}]{stacks-project}).

\subsubsection{Threefolds with rational double points; see Example \ref{eg:A^2:(x,y,0)}}  \label{sec:intro:threefolds}
Let $X=\AA^2 = \Spec \ZZ[x,y,z]$, and let $\sF$ denote the cokernel of the map $\sO_{\AA^2} \xrightarrow{(x,y,0)^T} \sO_{\AA^2}^{\oplus 3}.$ Then \eqref{eqn:condition:irr:PF} is satisfied but \eqref{eqn:condition:PK} is not. In this case, $\PP(\sE) = Q \subseteq \AA^2 \times \PP^2 = \AA^2 \times \Proj \ZZ[X,Y ,Z]$ is the quadric cone defined by the equation  $xX + y Y =0$, which is an irreducible threefold with a unique $A_1$-singularity at the cone point $o =((0,0), [0,0,1]) \in Q$. Theorem \ref{thm:structural} implies that there are semiorthogonal decompositions:
\begin{align*}
	\QCoh(Q)  & = \big\langle \QCoh(\Spec (\ZZ[\varepsilon])), ~~ \QCoh(\AA^2) \otimes \sO_Q(1), \QCoh(\AA^2) \otimes \sO_Q(2)  \big \rangle, \\
	\Perf(Q)  & = \big\langle \Perf(\Spec (\ZZ[\varepsilon])), ~ ~ \Perf(\AA^2) \otimes \sO_Q(1), \Perf(\AA^2) \otimes \sO_Q(2) \big\rangle,
	\end{align*}
where $\ZZ[\varepsilon] = \Sym_\ZZ^*(\ZZ[1])$ is the ring of {\em derived} dual numbers. We may replace $X=\AA^2$ by any reduced surface $S$ over a field $\kappa$, and $\sF$ by $\sO_S \oplus \sI_p$, where $p \in S(\kappa)$ is a non-singular rational point. The above semiorthogonal decompositions are locally compatible with Kawamata's results \cite[\S 7.1]{Kaw19} on quadric cones in $\PP^4$, the results of Xie \cite[Theorem 4.6]{Xie21} and Pavic and Shinder \cite[Theorem 3.6]{PS21} on nodal del Pezzo threefolds. We may also consider the global version of this example by letting $X$ be any scheme, and $\sF$ by $\sO_X \oplus \sI_Y$, where $Y \subseteq X$ is a regularly immersed closed subscheme of codimension two; see Example \ref{eg:YinX}.

\subsubsection{Derived equivalences of classical schemes induced by non-trivially derived schemes} \label{sec:intro:bir} If the condition \eqref{eqn:condition:irr:PK} is satisfied but \eqref{eqn:condition:hatZ} is not, then the schemes in the semiorthogonal decompositions are classical, but the fully faithful embedding is induced by a derived scheme. In the case where $\sE$ has rank zero, we could obtain various examples where $\PP(\sE)$ and $\PP(\sE^\vee[1])$ are birationally equivalent classical, irreducible threefolds (or fivefolds, or even higher dimensional varieties), but the universal incidence locus $\widehat{Z}$ which induces the (non-trivial) equivalences of $\infty$-categories 
	$$r_{+,*} \, r_-^* \colon \QCoh(\PP(\sE^\vee[1])) \simeq \QCoh(\PP(\sE)) \qquad r_{+,*} \, r_-^* \colon \Perf(\PP(\sE^\vee[1])) \simeq \Perf(\PP(\sE))$$
is {\em not} classical (where $\PP(\sE) \xleftarrow{r_+} \widehat{Z} \xrightarrow{r_-}  \PP(\sE^\vee[1])$ is the incidence diagram). See Example \ref{eg:A^4:birational}. 

We refer the readers to \S \ref{sec:classical} for more examples, where we only list lower dimensional cases in each situation, but all these results have direct analogues in higher dimensions.

\subsection{Applications to Hecke correspondences moduli} Many moduli problems which have the nature of parametrizing extensions naturally fit  into the framework of derived projectivizations. In \S \ref{sec:Hecke}, we will consider one such example, the Hecke correspondence moduli which parametrizes ``one-point modifications of complexes". The geometry of these moduli spaces are extensively studied in Negu{\c{t}}'s works \cite{NegW, NegHecke, NegShuffle} in the classical context.

Let $X \to S$ be a family of smooth, separated surfaces, and let $\foM^{\perf}_{[0,1]}$ denote the prestack of perfect complexes on $X$ relative to $S$ of Tor-amplitude contained in $[0,1]$. Let $\foZ^{\perf}_{[0,1]}$ denote the moduli of ``one-point modifications of objects of $\foM^{\perf}_{[0,1]}$", that is, $\foZ^{\perf}_{[0,1]}$ is the prestack which parametrizes exact triangles of the form
	$$\zeta = (\sF' \to \sF \to x_*\sL),$$
where $\sF', \sF \in \foM^{\perf}_{[0,1]}$, $\sL$ is a line bundle on $S$, and $x \colon S \to X$ is a section of the projection $X \to S$. Then there are natural forgetful maps
		$$
		\begin{tikzcd}[column sep = 3 em, row sep = 1 em]
				& \foZ^{\perf}_{[0,1]} \ar{ld}[swap]{p_+} \ar{d}{p_X} \ar{rd}{p_-} & \\
				\foM^{\perf}_{[0,1]} & X & \foM^{\perf}_{[0,1]},
		\end{tikzcd}
		$$		
		where $p_+$, $p_-$, and $p_X$ are the maps which carry the triangle $\zeta$ to $\sF' \in \foM^{\perf}_{[0,1]} $, $\sF \in \foM^{\perf}_{[0,1]} $, and $x\in X$, respectively. We show that the natural projections 
			$$\pi_\pm = p_{\pm} \times_S p_X \colon \foZ^{\perf}_{[0,1]} \to \foM^{\perf}_{[0,1]} \times_S X$$
	 are {\em derived projectivizations} of perfect-complexes of Tor-amplitude in $[0,1]$ (Proposition \ref{prop:pi_pm:foZ:proj}).
	 
\begin{remark} In fact, we will show that, for any family $X \to S$ of (smooth separated) derived schemes of any given dimension $d$, the Hecke correspondence moduli $\foZ$ fits into the framework of derived projectivizations  (Propositions \ref{prop:pi_-:foZ:proj} and \ref{prop:pi_pm:foZ:proj});  and the results are even true for any type of flag of ``one-point modifications" $\foZ_{\lambda}$ (Propositions \ref{prop:pi_-:foZ_lambda:proj} and \ref{prop:pi_pm:foZ_lambda:proj}). 
\end{remark}

In particular, we could apply all of the above-mentioned results (the generalized Serre's Theorem \ref{thm:Serre:O(d)},  Beilinson's relations Proposition \ref{prop:PG:dualexc}, Corollary \ref{cor:PVdot:relations}, and Semiorthogonal Decompositions Theorem \ref{thm:structural}) to this situation. For example, if we let   $(\foZ_2^\bullet)^{\perf}_{[0,1]}$ denote the moduli of ``two-step flags of one-point modifications with support on one point", that is, the prestack which parametrizes the pair of exact triangles 
	$$\sF_0  \xrightarrow{\phi_1} \sF_1 \to x_*(\sL_1) \qquad \sF_1  \xrightarrow{\phi_2}  \sF_2 \to x_*(\sL_2),$$
where $\sL_1, \sL_2$ are line bundles on $S$ and $x \colon S \to X$ is a section of $X \to S$. We let $(\widehat{\foZ}_2^\bullet)^{\perf}_{[0,1]}$ denote the universal incidence locus for $\pi_\pm$, that is, the locus on $(\foZ_2^\bullet)^{\perf}_{[0,1]}$ where the composition
	 $$\sL_{1}^\vee \boxtimes (\sL_{2} \otimes \sK_{X/S}) [1] \to x^*\sF_1 \otimes x^* ( \sF_1^\vee) \xrightarrow{\rm ev} \sO$$
is nullhomotopic. Therefore, we have a correspondence diagram 
	$$\foZ^{\perf}_{[0,1]}  \xleftarrow{r_+} (\widehat{\foZ}_2^\bullet)^{\perf}_{[0,1]} \xrightarrow{r_-} \foZ^{\perf}_{[0,1]}.$$
We let $\foZ^\perf_{[0,1], \rk=\delta}$ denote the subfunctor of $\foZ^\perf_{[0,1]}$ for which $\sF$ and $\sF'$ have rank $\delta$, and let $\foM^{\perf}_{[0,1], \rk=\delta}$ be defined similarly. Theorem \ref{thm:structural} implies that, if $\delta > 0$, the functor
	$$\Phi = r_{+\,*} \, r_-^* \colon \QCoh(\foZ^\perf_{[0,1], \rk=\delta}) \to \QCoh(\foZ^\perf_{[0,1], \rk=\delta}),$$
and the functors, for all $i \in \ZZ$,
				$$\Psi_i = \pr_{+}^*(\blank) \otimes \sL_{\rm univ}^{\otimes i}  \colon \QCoh(\foX) \to \QCoh(\foZ^\perf_{[0,1], \rk=\delta}) \quad \text{where} \quad \foX = \foM^{\perf}_{[0,1], \rk=\delta} \times_S X$$
 are fully faithful, and there are induced semiorthogonal decompositions
			\begin{align*}
			\QCoh(\foZ^\perf_{[0,1], \rk=\delta})  &= \big\langle \Phi( \QCoh(\foZ^\perf_{[0,1], \rk=\delta})), 
		~ \Psi_1( \QCoh(\foX) ), \cdots, \Psi_\delta (\QCoh(\foX))  \big\rangle,\\	
			\Perf(\foZ^\perf_{[0,1], \rk=\delta})  &= \big\langle \Phi( \Perf(\foZ^\perf_{[0,1], \rk=\delta})),  
			~ \Psi_1( \Perf(\foX) ), \cdots, \Psi_\delta (\Perf(\foX))  \big\rangle.	
			\end{align*}
Similarly for other cases; see Corollary \ref{cor:Hecke:surface} for more details.

\begin{remark} The general framework of \S \ref{sec:Hecke} not only unifies various situations such as the Hilbert schemes of points (Example \ref{eg:Hilbn}) and moduli of discrete Gieseker stable sheaves (Example \ref{eg:Gieseker}), but also allows us to consider moduli spaces under various types of stability conditions. For example, we expect the results of \S \ref{sec:Hecke} to be helpful for the study of the similar algebraic structures for the moduli spaces $\foM_{\sigma}(X)$ of Bridgeland (semi-)stable objects on a surface $X$ (or a family $X \to S$ of surfaces), where the surface $X$ (or the family $X \to S$ of surfaces) is equipped with (a family of) Bridgeland stability conditions $\sigma$ in the sense of 
 \cite{Bri07, Bri08, AB12, BM11, BM14} (the absolute cases) and \cite{BLM+} (the family case).
 \end{remark}

\subsection{The framework of derived algebraic geometry} \label{sec:DAG} We work in the context of derived algebraic geometry, as developed by Lurie \cite{DAG, HTT, DAGV, HA, SAG, kerodon}, To{\"e}n and Vezzosi \cite{HAGI, HAGII, ToenDAG} and many others \cite{BFN10, GRI, TVa07, PTVV}. 

There are various models for the framework of derived algebraic geometry. For definiteness and completeness, we choose to work within Lurie’s framework, and we will use Lurie’s works as the only references for all results of derived algebraic geometry in the proofs of this paper. However, we expect our arguments to work in the other models of derived algebraic geometry. 

From a {\em technical} perspective, the framework of derived algebraic geometry plays an essential role in this paper from at least the following aspects:

\begin{itemize}
	\item {\em Non-abelian derived functors.} Historically, symmetric algebras of complexes are studied either using concrete Koszul models \cite{Ei, ABW} or simplicial resolutions \cite{DP, Quillen}. However, for this paper's goal, either approach has its advantages and certain drawbacks:
	\begin{itemize}
		\item The definition using Koszul models is practical for computations but is a priori only defined for complexes of vector bundles (and might not be our desired definition if the characteristic is not zero). For example, even the fact that the symmetric algebras depend only on the quasi-isomorphism classes of complexes is a theorem rather than an immediate consequence of this definition. 
		\item On the other hand, the approach using simplicial resolutions is general and functorial, but might not be effective for concrete computations. 
		\item Illusie's results in \cite{Ill} bridge the above two approaches, but the ``bridging process" is slightly complicated and lacks sufficient functoriality for this paper's goal.
	\end{itemize}
Lurie's theory of non-abelian derived functors \cite[\S 5.5.8]{HTT} using the framework of $\infty$-categories in a way unifies these two perspectives, and enjoys the {\em best features} of both approaches. See \S \ref{sec:dSym*}, \S \ref{sec:univ:fib:seq}. 
	\item {\em Functorial exact triangles (fiber sequences).} In the classical context of triangulated categories, the construction of cones is known to be non-functorial. Therefore, it is hard to show that two non-trivial exact triangles obtained from different ways are equivalent. The framework of derived algebraic geometry (see, for example, \cite{HTT, HA, SAG}) not only solves the problem of non-functorial triangles but also explains why certain triangles should be equivalent, by ``remembering their origins".  This perspective plays a vital role in many proofs of this paper; see, for example, the proofs of generalized Serre's theorem (Theorem \ref{thm:Serre:O(d)}) and Beilinson's relations (Proposition \ref{prop:PG:dualexc} and Corollary \ref{cor:PVdot:relations}).	
	\item {\em Translating problems of algebraic geometry to problems of algebraic topology}. The framework of derived algebraic geometry (for example, the theory of general derived higher moduli (pre)stacks, as developed by Lurie in \cite{DAG, SAG}, and To{\"e}n and Vezzosi in \cite{HAGI, HAGII}) allows us to reduce certain problems of algebraic geometry to that of algebraic topology, where the $\infty$-amount of homological information that is hard to manipulate in algebraic geometry is effectively encoded by topological spaces that are easier to handle in algebraic topology. This perspective plays a vital role in proving the Euler fiber sequence (Theorem \ref{thm:proj:Euler}).
	\item {\em Phenomena only visible in the derived context.} Let $V$ be a finite-dimensional vector space (or a vector bundle over a base scheme). The classical {\em universal incidence quadric} $\shH \subseteq \PP(V) \times \PP(V^\vee)$, defined as the zero locus of the composite map
		$$\sO(-1,-1) \to V^\vee \otimes V \xrightarrow{\rm ev} \sO,$$
	plays a crucial role in classical and homological projective duality \cite{Kuz07}. Suppose that $V = [V_1 \to V_0]$ is a complex of vector spaces (or vector bundles), and $V^\vee$ is the dual of $V$. Then the role of $\shH$ is now played by the {\em universal incidence locus} $\widehat{Z} \subseteq \PP(V) \times \PP(V^\vee[1]) $ (Notation \ref{notation:hatZ}), defined as the zero locus of the composite map
		$$\sO(-1,-1)[1] \to V^\vee \otimes V \xrightarrow{\rm ev} \sO.$$
	However, since the zero loci of non-zero degree shifts of line bundles only affect the derived structure and have no effects on the underlying classical schemes, the semiorthogonal decompositions of Theorem \ref{thm:structural} induced by $\widehat{Z}$ are best understood in the derived context, even in the cases where $\PP(V)$ and $\PP(V^\vee[1])$ are classical.
\end{itemize}

\subsection{Related works} 

The classical projectivization formula is studied in \cite[Theorem 5.5]{Kuz07}, \cite{Tod2, Tod3}, \cite[Theorem 3.4]{JL18}, \cite[Theorem 6.16]{J21}, and the Chow-theoretical counterpart is studied in \cite{J19}. We expect the current work to be helpful to extend the results of \cite{J19} to the context of bivariant and motivic theories for derived schemes.

The projectivization formula is also related to the theory of homological projective duality developed by Kuznetsov \cite{Kuz07}; see also \cite{JLX17, Pe19}.

Our construction of affine cones in \S \ref{sec:affine-cone} is closely related to the construction of Jiang and Thomas in \cite{JT17} in the context of perfect obstruction theory \cite{BF, LT} and $(-1)$-shifted symplectic structure \cite{PTVV}. It would be interesting to explore the further connections between the results of this paper and the framework of perfect obstruction theory.

In \cite{Kh21}, Khan defined and studied the derived symmetric algebras (\S \ref{sec:Sym*:prestack}) on derived Artin stacks and applied them to obtain an excess intersection formula for the algebraic $K$-theory.

In \cite{He21}, Hekking constructs projective spectra for all graded simplicial commutative algebras over a derived scheme. In the case where the graded algebra $\shB_*$ is given by a derived symmetric algebra $\Sym_X^*(\sE)$, our construction of $\PP(\sE)$ in \S \ref{sec:def:dproj:qcoh} is equivalent to his $\Proj(\shB_*)$.

In  \cite{Z20}, Zhao also established a version of the generalized Serre's theorem \ref{thm:Serre:O(d)}.

The derived projectivization is closely related to the theory of categorification of wall-crossing of Donalson--Thomas moduli spaces and the d-critical birational geometry, developed by Toda in \cite{Tod1, Tod2, Tod3, Tod4, Tod5, Tod6}.

The applications in \S \ref{sec:classical} are related to the study of derived categories of nodal curves in \cite{Bur, KPS} and nodal threefolds in \cite{Kuz21, Kaw19, Xie21,PS21}. 

The study of Hecke correspondence in \S \ref{sec:Hecke} is closely related to Negu{\c{t}}'s works \cite{NegW, NegHecke, NegShuffle}; see also \cite{KV19, PS19, PZ20, Z20, Z21}.


In \cite{J22a, J23}, we extend the constructions and results presented in this paper to the context of Grassmannians in characteristic zero. In this generalization, the derived symmetric and exterior products are replaced by more general derived Schur functors.

\subsection{Notations and conventions}  \label{sec:Notations} We made an effort to keep our notations and terminologies as close to Lurie's in \cite{DAG, HTT, HA, SAG, kerodon} as possible, so that readers of Lurie's works will not need to make a significant extra effort to adjust to those of this paper.

For the readers' convenience, here is a list of notations and terminologies that we will frequently use in this paper:
\begin{itemize}
	\item We let $\bDelta$ denote the {\em simplex category} (that is, the category of  nonempty finite linearly ordered sets, with morphisms given by nondecreasing maps). We let $\Set_\Delta = \Fun(\bDelta^\op, \Set)$ denote the simplicial category of simplicial sets. For each $n \geq 0$, we let $\Delta^n \in \Set_\Delta$ denote the simplicial set that represents the functor $\Hom_{\bDelta}(\blank, [n])$.
	\item We let $\shS$ denote the {\em $\infty$-category of spaces}, that is, the homotopy-coherent nerve of simplicial category of Kan complexes.
	\item We let $\Sp$ denote the {\em $\infty$-category of spectra} (\cite[Definition 1.4.3.1]{HA}).
	\item We let $\Catinf$ denote the $\infty$-category of small $\infty$-categories, and $\widehat{\Cat}_{\infty}$ the $\infty$-category of not necessarily small $\infty$-categories (\cite[Chapter 3]{HTT}). For an $\infty$-category $\shC$ and objects $C,D \in \shC$, $\Map_\shC(C,D) \in \shS$ denotes the {\em mapping space}. For a pair of $\infty$-categories $\shC$ and $\shD$, we let $\Fun(\shC,\shD)$ denote the $\infty$-category of functors from $\shC$ to $\shD$. 
	\item For an $\infty$-category $\shC$, we let $\shC^{\simeq}$ denote the largest Kan complex contained in $\shC$, that is, the $\infty$-category obtained from $\shC$ by discarding all non-invertible morphisms.
	\item If $\shC$ is a pointed $\infty$-category which admits cofibers (resp. fibers), we let $\Sigma = \Sigma_\shC$ (resp. $\Omega = \Omega_\shC$) denote the {\em suspension} (resp. {\em loop}) functor. If a pointed $\infty$-category $\shC$ admits both cofibers and fibers, we will define a translation functor $[i]$ for each $i \in \ZZ$ by the following rule: let $[i] = \Sigma^{\circ i}$ if $i>0$, $[i] = \Omega^{\circ -i}$ if $i<0$, and $[i] = \id_\shC$ if $i=0$.
	\item We let $\CAlgDelta$ denote the $\infty$-category of simplicial commutative rings (\cite[Definition 25.1.1.1]{SAG}), and we will refer to the elements of $\CAlgDelta$ as {\em simplicial commutative rings}
	\footnote{These are also referred to as {\em animated commutative rings}; see \cite[\S 5.1]{CS}, \cite[Appendix A]{BL22}.}.
	 The $\infty$-category $\CAlgDelta$ can be obtained from ordinary category ${\rm Poly}$ of finite type polynomial rings by freely adjoining sifted colimits (that is, $\CAlgDelta = \shP_\Sigma({\rm Poly})$), and we will review this construction in \S \ref{sec:non-abelian}. For a simplicial commutative ring $R \in \CAlgDelta$, we let $\CAlgDelta_R$ denote the coslice category $(\CAlgDelta)_{/R}$ (\cite[Notation 25.1.4.4]{SAG}).
	\item Let $\shC^\otimes \to {\rm Comm}^\otimes$ be a symmetric monoidal $\infty$-category (\cite[Definition 2.0.0.7]{HA}) with underlying $\infty$-category $\shC$. We let $\CAlg(\shC) = \mathrm{Alg}_{/\mathrm{Comm}}(\shC)$ denote {\em $\infty$-category of commutative algebra objects} of $\shC$ (\cite[Definition 2.1.3.1]{HA}). For $A \in \CAlg(\shC)$, we let $\Mod_A(\shC)$ denote the {\em $\infty$-category of $A$-module objects of $\shC$} (\cite[Definition 3.3.3.8]{HA}). 
	\item We let $\CAlg = \CAlg(\Sp)$ denote the $\infty$-category of $\EE_\infty$-rings (\cite[Definition 7.1.0.1]{HA}). There is a forgetful functor $\Theta \colon \CAlgDelta \to \CAlg$. For a simplicial commutative ring $A$, we will denote the image $\Theta(A)$ as $A^\circ$ and refer to it as the {\em underlying $\EE_\infty$-ring} of $A$ (\cite[Construction 25.1.2.1]{SAG}).
	\item For any $A \in \CAlgDelta$, we will let $\Mod_A$ denote the $\infty$-category $\Mod_{A^\circ}(\Sp)$, and refer to the elements of $\Mod_A$ as {\em $A$-modules}. For a simplicial commutative ring $A$ and an $A$-module $M$, we let $\pi_i(A) = \pi_i(A^\circ)$ (resp. $\pi_i(M)$) denotes the $i$th homotopy group of the underlying spectrum of $A$ (resp. $M$). The action map $R \otimes M \to M$ endows $\pi_*(M) = \bigoplus_{m \in \ZZ} \pi_m(M)$ with the structure of a graded $\pi_*(R) $-module. In particular, $\pi_0(M)$ is a module over the ordinary commutative ring $\pi_0(R)$. We will say that $M$ is {\em connective} if 
	$\pi_i(M) = 0$ for all $i<0$. We let $\Mod_A^\cn$ denote the full subcategory of $\Mod_A$ spanned by connective $A$-modules. See \cite[\S 7.1.1]{HA} for more details. We say a simplicial commutative ring $A$ (resp. an $A$-module $M$) is {\em discrete} if $\pi_i(A)=0$ for $i \ne 0$ (resp. $\pi_i(M)=0$ for $i \ne 0$). We let $\CAlg^\heartsuit$ and $\Mod_A^\heartsuit$ denote the full subcategory of $\CAlgDelta$ and $\Mod_A$ spanned by discrete objects, respectively. 
	\item By a {\em derived scheme} we mean a pair $X = (|X|, \sO_X)$ where $|X|$ is a topological space, and $\sO_X$ is an $\CAlgDelta$-valued sheaf on $|X|$, such that the following conditions are satisfied: (i) The underlying ringed space $X_\cl = (|X|, \pi_0 \sO_X)$ is a (classical) scheme, which is called the {\em underlying classical scheme of $X$}. (ii) Each of the sheaves $\pi_n(\sO_X)$ is a quasi-coherent $\pi_0(\sO_X)$-module on $X_\cl$. (iii) The structure sheaf $\sO_X$ is hypercomplete when regarded as an object of the $\infty$-topos $\Shv_\shS(X)$ in the sense of \cite[\S 6.5.2]{HTT}. 
	\item By a {\em prestack} we mean an element of the $\infty$-category $\Fun(\CAlgDelta, \shS)$, that is, a functor $X \colon \CAlgDelta \to \shS$. A map $f \colon X \to Y$ between prestacks is a natural transformation of the functors $X, Y \colon  \CAlgDelta \to \shS$. For a prestack $X$, we let $\QCoh(X)$ denote the $\infty$-category of quasi-coherent sheaves on $X$ (which we will review in  \S\ref{sec:QCoh}), and let $\Perf(X)$ denote the full subcategory of $\QCoh(X)$ spanned by perfect objects.
\end{itemize}

Here is a list of notations and terminologies that either we introduce in this paper or for which our usage is slightly different from (some of) Lurie's works:

\begin{itemize}
	\item We use $\Sym_A^n(M)$ to denote the {\em derived $n$th symmetric product} of a connective $A$-module $M$ over $A$ (\S \ref{sec:derived_sym}), and use $\Sym_X^n(\sM)$ to denote the global version, that is, the {\em derived $n$th symmetric power} of a connective quasi-coherent sheaf $\sM$ over a prestack $X$ (\S \ref{sec:sym:prestack}). This notation agrees with Lurie's in \cite{DAG}, but is {\em different} from his in \cite{HA, SAG}. \footnote{In \cite{HA, SAG}, Lurie uses $\Sym_A^n(M)$ to denote the $A$-module obtained by taking homotopy coinvariants for the action of the symmetric group $\Sigma_n$ on the $n$-fold relative tensor product $M \otimes_A \cdots \otimes_A M$, which is generally different from the derived $n$th symmetric power unless $A$ has  characteristic zero; We will not use this construction in this paper. Our notation $\Sym_A^n(M)$ corresponds to Lurie's notation ${\rm LSym}_A^n(M)$ in \cite[\S 25.2]{SAG}.}
	\item For an ordinary commutative ring $R$ and a discrete $R$-module $M$, $S_R^d(M)= \Sym_{\cl, R}^d(M)$ denote the {\em classical symmetric products}, reviewed in \S \ref{sec:classical_sym}. Let $\rho \colon M \to M'$ be a map of discrete $R$-modules, $\bS^d(R, \rho)$ denotes the Koszul complex defined in Construction \ref{constr:Koszul:S^n}.
	\item  $\QCAlgDelta(X)$ denotes the {\em $\infty$-category of quasi-coherent algebras} on a prestack $X$; see \S\ref{sec:QCAlgDelta}. For a morphism of prestacks $f \colon X \to Y$, we let $f_{\rm alg}^* \colon \QCAlgDelta(Y) \to \QCAlgDelta(X)$ denote the pullback functor, and $f_{*}^{\rm alg}$ the right adjoint of $f_{\rm alg}^*$ (if exists).
	\item  $\Spec_X$ denotes the {\em relative spectrum} functor over a prestack $X$; see Definition \ref{def:relative:spec}.
	\item $\VV(\sE)$ or $\VV_X(\sE)$ denotes the {\em affine cone} of a connective quasi-coherent sheaf $\sE$ over a prestack $X$ (\S \ref{sec:affine-cone}); $\PP(\sE)$ or $\PP_X(\sE)$ denotes the {\em derived projectivization} of a connective quasi-coherent sheaf $\sE$ over a prestack $X$ (\S \ref{sec:def:dproj:qcoh}).
	\item The Eagon--Northcott complexes $\mathbf{EN}_d(X, \sigma)$ are defined in Remark \ref{rmk:ENcomplexes}.
	\item For an perfect stack $X$ (Definition \ref{def:perfstacks}), a $X$-linear stable $\infty$-category $\shC$, $\Mapsp_{X}(C,D) = \Mapsp_{\shC/X}(C,D) \in \QCoh(X)$ denotes the mapping object defined in Construction \ref{constr:mappingobject}. For a relative exceptional object $E \in \shC$ over $X$ (Definition \ref{def:relexcseq}), $\LL_E$ and $\RR_E$ denote the left and right mutation functors, respectively, defined in Definition \ref{def:mutations}.
	\item The universal incidence locus $\widehat{Z}$ and the incidence diagram is defined in Notation \ref{notation:hatZ}.
	\item The notations $\foM$, $\foZ$, $\foZ_\lambda$, and their variants, denote the moduli functors that appear in Hecke correspondences, which we will define and only use in \S \ref{sec:Hecke}.
\end{itemize}

\subsection*{Acknowledgment}
The author thanks Arend Bayer for numerous helpful discussions throughout this project, 
Bertrand To{\"e}n for asking him the question of viewing the projectivization formula from the DAG perspective during his visit to Toulouse two years ago,
Fei Xie for helpful discussions on derived categories of nodal curves and nodal threefolds, and Dougal Davis for discussions on DAG, carefully reading the paper, and many helpful suggestions. 
The author is grateful to Richard Thomas, Yukinobu Toda, Sasha Kuznetsov, Andrei Negu{\c{t}}, Adeel Khan,  Yunfeng Jiang, Jeroen Hekking, and Yu Zhao for helpful communications and comments on an earlier draft of this paper.
This work is supported by the Engineering and Physical Sciences Research Council [EP/R034826/1].

 \section{Derived symmetric algebras}\label{sec:SymAlg}
 
This section reviews Lurie's construction of derived symmetric algebras and studies their basic properties. The constructions and results of this section play an essential role in our later study of affine cones and derived projectivizations in this paper. 

In \S \ref{sec:classical_sym}, we briefly review the classical constructions of the symmetric, exterior, and divided powers. \S \ref{sec:non-abelian} reviews Lurie's general theory of non-abelian derived functors, and \S \ref{sec:derived_sym} reviews  the constructions and properties of the {\em derived} symmetric, exterior, and divided powers. In \S \ref{sec:dSym*}, we define {\em derived} symmetric algebras and study their basic properties; see Propositions \ref{prop:Sym*} and \ref{prop:Sym*:cofib}. Finally, in \S \ref{sec:univ:fib:seq}, we apply Lurie's machinery \S \ref{sec:non-abelian} to construct two canonical series of functorial fiber sequences associated to symmetric powers. 

\begin{remark}
The material of this section has a direct generalization to a theory of derived Schur functors, which we will explore in a forthcoming paper.
\end{remark}

\subsection{Classical symmetric, exterior and divided power algebras}\label{sec:classical_sym}  
 For the readers' convenience, we will briefly review the classical definitions of symmetric, exterior and divided power algebras. In this subsection, we make the following assumption:
 
\begin{itemize}
	\item We let $R$ be a(n ordinary) commutative ring, and $M$ a discrete $R$-module. 
\end{itemize}
 
\begin{definition}[Classical tensor algebras; {see \cite[Chapter III, \S 5]{Bou}}] For a pair of discrete $R$-modules $M$ and $N$, we let $M \otimes_R^\cl N$ denote the classical tensor product $\Tor_0^R(M,N)$ over $R$. For an integer $n \ge 0$, we let $\T^n(M)$ denote the $n$-fold classical tensor product $M \otimes_R^\cl M \otimes_R^\cl \cdots \otimes _R^\cl M$. By convention, we set $\T^0(M) = R$, $\T^1(M) = M$. Then the {\em classical tensor algebra of $M$ over $R$}, denoted by $\T^*(M)$, or $\T_{R}^*(M)$, is the (ordinary) graded associative $R$-algebra $\T^*(M) = \bigoplus _{n \ge 0} \T^n(M)$, with the multiplication map 
	$$m_{p,q} \colon \T^p(M) \times \T^q(M) \to \T^{p+q}(M),$$
 where $p, q \ge 0$, given by the formula:
	$$m_{p,q} (x_1 \otimes \cdots \otimes x_p, x_{p+1} \otimes \cdots \otimes x_{p+q}) = x_1 \otimes \cdots \otimes x_p \otimes x_{p+1} \otimes \cdots \otimes x_{p+q}.$$
\end{definition}

 \begin{definition}[{Classical symmetric algebras; see \cite[Chapter III, \S 6]{Bou}}] \label{def:cl:sym} The {\em clasical symmetric algebra of $M$ over $R$}, denoted by $S^*(M)$, $S_{R}^*(M)$, or $\Sym_\cl^*(M)$, is the (ordinary) commutative $R$-algebra obtained as the quotient of the tensor algebra $\T_R(M)$ by the two-sided ideal $\foC = \foC_M$, where the ideal $\foC$ is generated by the elements of the form $x \otimes y - y \otimes x$ for $x, y \in M$. Since $\foC$ is a graded ideal generated by homogeneous elements of degree $2$, the algebra $S_{R}^*(M)$ is graded and equipped with a direct sum decomposition $S_R^*(M) = \bigoplus _{n \ge 0} S_R^n(M)$, where $S_R^n(M) = \T^n(M)/ (\foC \cap \T^n(M))$. We will refer to the $R$-module $S_R^n(M)$ as the {\em classical $n$th symmetric power of $M$ over $R$}. From $\foC^0 = \foC^1 = 0$, we obtain canonical identifications $S_R^0(M) = R$ and $S_R^1(M) = M$. The canonical map $M=S_R^1(M)  \hookrightarrow S_R^*(M)$ is referred to as the canonical injection. 
\end{definition}

\begin{definition}[{Classical exterior algebras, see \cite[Chapter III, \S 7]{Bou}}] \label{def:cl:wedge} The {\em exterior algebra of $M$ over $R$}, denoted by $\bigwedge_\cl^*(M)$, or $\bigwedge_{\cl, R}^*(M)$, is the ordinary $R$-algebra defined as the quotient of the tensor algebra $\T_R(M)$ by the two-sided ideal $\foA = \foA_M$, where $\foA$ is generated by elements of the form $x \otimes x$ for $x \in M$. Since $\foA$ is a graded ideal generated by homogeneous elements of degree $2$, the algebra $\bigwedge_\cl^*(M)$ is graded, and we have a direct sum decomposition $\bigwedge_\cl^*(M) = \bigoplus _{n \ge 0} \bigwedge_\cl^n(M)$, where $\bigwedge_\cl^n(M) = \T^n(M)/ ( \foA \cap \T^n(M))$. We will refer to $\bigwedge_\cl^n(M)$ as the {\em classical $n$th exterior power of $M$ over $R$}. From $\foA^0 = \foA^1 = 0$, we obtain canonical identifications $\bigwedge_\cl^0(M) = R$ and $\bigwedge_\cl^1(M) = M$. We will refer to the canonical map $M=\bigwedge_\cl^1(M)  \hookrightarrow \bigwedge_\cl^*(M)$ as the canonical injection. 
\end{definition}

\begin{definition}[{Classical divided power algebras, {\cite[Chapter III, \S 1, p. 248]{Ro}} }] \label{def:cl:Gamma}  The {\em classical divided power (or P.D., puissances divis{\'e}es) algebra of $M$ over $R$}, denoted by $\Gamma_\cl^*(M)$, $\Gamma_{\cl,R}^*(M)$, or $\DD_{\cl}^*(M)$, is the (ordinary) commutative $R$-algebra obtained as the quotient algebra of the free polynomial algebra $R[\{\gamma^n(x)\}_{n\ge 0, x \in M}]$ on symbols $\{\gamma^n(x)\}_{n \ge 0, x \in M}$ by the ideal generated by following relations:
	\begin{align*}
		& \gamma^{0}(x) = 1,  & \forall x \in M; \\
		& \gamma^n(r x) = r^n \gamma^n(x) & \forall r \in R, x \in M', n \ge 0; \\
		& \gamma^m(x) \cdot \gamma^n(x) = \binom{m+n}{m} \gamma^{m+n}(x) & \forall x \in M, n \ge 0;\\
		& \gamma^n(x+y) = \sum_{i=0}^n \gamma^{i}(x) \gamma^{n-i}(y), & \forall x ,y \in M, n \ge 0.
	\end{align*}
Since the ideal generated by above relations are homogeneous, the algebra $\Gamma_\cl^*(M)$ is graded, and we have a direct sum decomposition $\Gamma_\cl^*(M) = \bigoplus_{n \ge 0} \Gamma_\cl^n(M)$, where $\Gamma_\cl^0(M) = R$ and $\Gamma_\cl^1(M) = M$. We will refer to the $R$-module $\Gamma_\cl^n (M) = \Gamma_{\cl, R}^n(M)$ as the {\em classical $n$th  divided power of $M$ over $R$}. The assignment $x \mapsto \gamma^n(x)$ defines a map of $R$-modules $\gamma^n \colon M \to \Gamma_\cl^n(M)$. Notice that $\gamma^0\equiv 1 \colon M \to R$ is the constant map, and $\gamma^1 = \id \colon M \to M$ is the identity map on $M$. For an element $x \in M$, we will also use the notation $x^{(n)} : = \gamma^n(x) \in \Gamma_\cl^n(M)$; we will refer to $x^{(n)}= \gamma^n(x)$ as the {\em $n$th divided power} of $x$.
\end{definition}

We refer the readers to \cite[Chapter III, \S 6, \S 7]{Bou} and \cite{Ro} for the basic properties of the above classical constructions. Here, we will list some of them so that the reader can easily compare them with the corresponding properties of their derived counterparts: 

\begin{enumerate}[leftmargin=*]
		\item (Base change) Let $R \to R'$ be a homomorphism of (ordinary) commutative rings, and let $M$ be a discrete $R$-module. Then there are canonical isomorphisms
			\begin{align*}
			 R' \otimes_R^\cl S_R^* (M) & \simeq S_{R'}^* (R' \otimes_R^\cl M), \\
			 R' \otimes_R^\cl \bigwedge\nolimits_{\cl, R}^* (M) & \simeq  \bigwedge\nolimits_{\cl, R'}^* (R' \otimes_R^\cl M), \\
			 R' \otimes_R^\cl \Gamma_{\cl, R}^* (M) & \simeq \Gamma_{\cl, R'}^* (R' \otimes_R^\cl M)
		\end{align*}
		of graded $R'$-algebras; see \cite[Chapter III, \S 6.4, Proposition 7, \S 7.5, Proposition 8]{Bou} and \cite[Theorem III.3, p. 262]{Ro}.	
	\item (Freeness) If $M$ is a (locally) free $R$-module of rank $m$, then for any integer $n \ge 0$, $S^n(M)$,  $\bigwedge_\cl^n M$, and $\Gamma_\cl^n(M)$ are (locally) free $R$-modules of ranks $\binom{n+m-1}{n}$, ${m \choose n}$, and $ \binom{n+m-1}{n}$, respectively. In this case, we will drop the subscript $\cl$ and denote them by $S^n(M)$,  $\bigwedge^n M$, and $\Gamma^n(M)$; these notations will be justified in the next section.
	\item (Duality) If $M$ is a locally free $R$-module, then for any integer $n \ge 0$, there are canonical isomorphisms of $R$-modules $(\bigwedge^n M)^\vee \simeq \bigwedge^n (M^\vee)$, $( S^n(M))^\vee \simeq \Gamma(M^\vee)$. (Here, $N^\vee$ denote the $R$-linear dual of a locally free $R$-module $N$.)
	\item (Right exactness) The formations of the graded algebras $S^*(M)$, $\bigwedge_{\cl}^* M$ and $\Gamma_\cl^*(M)$ are right exact in the variable $M$ in the following sense: suppose we are given a short exact sequence of discrete $R$-modules:
		$$0 \to P \to M \xrightarrow{u} N \to 0.$$
Then there are induced short exact sequences
	\begin{align*}
	0 \to I_1 \to S^*(M) \xrightarrow{S^*(u)} S^*(N) \to 0, \\
	0 \to I_2 \to \bigwedge\nolimits_\cl^* M \xrightarrow{\bigwedge\nolimits_\cl^*(u)} \bigwedge\nolimits_\cl^* N \to 0, \\
	0 \to I_3 \to \Gamma_\cl^*(M) \xrightarrow{\Gamma_\cl^*(u)} \Gamma_\cl^*(N) \to 0,
	\end{align*}
where $I_1$ is the ideal of the commutative algebra $S^*(M)$ generated by $P$ (\cite[Chapter III, \S 6.2, Proposition 4]{Bou}), $I_2$ is the two-sided ideal of the algebra $\bigwedge_\cl^*(M)$ generated by $P$  (\cite[Chapter III, \S 7.2, Proposition 3]{Bou}), and $I_3$ is the ideal of the commutative algebra $\Gamma_\cl^*(M)$ generated by elements of the form $x^{(i)}$ for $x \in P$, $i \ge 1$ (\cite[Proposition IV.8, p. 284]{Ro}), respectively. 
\end{enumerate}

\begin{remark}
A word of {\em caution}: if $n \ge 2$, the functors $M \mapsto S^n(M), \bigwedge_\cl^n(M), \Gamma_\cl^n(M)$ are {\em not} right exact in the usual sense as they are {\em not additive}. However, they are ``right exact" in a ``non-abelian sense" that they commute with reflexive coequalizers, which is one of the underlying reasons they have a well-behaved non-abelian derived theory. We will explore their derived functors in the following subsections.
\end{remark}

 \subsection{Non-abelian derived functors} \label{sec:non-abelian} 
To define and study the derived functors of symmetric, exterior, and divided powers, we will use Lurie's general theory of non-abelian derived functors as developed in \cite[\S 5.5.8]{HTT}, which we review in this subsection.

To keep our exposition brief, we introduce the following notations:
\begin{itemize}
	\item We let $\Cat_\infty^\amalg$ (resp. $\Cat_\infty^{\pi}$) denote the subcategory of $\Cat_\infty$ spanned by those $\infty$-categories which admit finite coproducts (resp. finite products) and those functors which preserve finite coproducts (resp. finite products). For $\shC, \shD \in \Cat_\infty^\amalg$ (resp. $\Cat_\infty^{\pi}$), we let $\Fun^\amalg(\shC, \shD)$ (resp. $\Fun^\pi(\shC, \shD)$) denote the full subcategory of $\Fun(\shC, \shD)$ spanned by those functors which preserve finite coproducts (resp. finite products).
	\item We let $\Cat_\infty^\Sigma$ denote the subcategory of $\Cat_\infty$ spanned by those $\infty$-categories which admit small sifted colimits and those functors which preserve small sifted colimits. For $\shC, \shD \in \Cat_\infty^\Sigma$, we let $\Fun_\Sigma(\shC, \shD)$ denote the full subcategory of $\Fun(\shC, \shD)$ spanned by those functors which preserve small sifted colimits. 
	\item We let $\widehat{\Cat}_\infty$ denote the ``very large version" of $\Cat_\infty$, that is, it is the $\infty$-category of not necessarily small $\infty$-categories. We let $\widehat{\Cat}_\infty^{\rm L}$ denote the full subcategory of $\widehat{\Cat}_\infty$ spanned by those $\infty$-categories which admit all small colimits and those functors which preserve all small colimits. For $\shC, \shD \in \widehat{\Cat}_\infty^{\rm L}$, we let $\Fun^{\rm L}(\shC, \shD)$ denote the full subcategory of $\Fun(\shC, \shD)$ spanned by those functors which preserve all small colimits. 
	\item Similarly, we let $\widehat{\Cat}_\infty^?$ denote the ``very large version" of the $\infty$-category $\Cat_\infty^?$, where $? = \amalg, \pi, \Sigma$ (that is, $\widehat{\Cat}_\infty^? \subseteq \widehat{\Cat}_\infty$ is defined in a same way as the above-defined $\Cat_\infty^?$ with the only exception that its objects are not assumed to be small).
\end{itemize}

Let $\shC \in \Cat_\infty^\amalg$ be an $\infty$-category that admits finite coproducts, and let $\shP_{\Sigma}(\shC) = \Fun^{\pi}(\shC^\op, \shS)$ denote the $\infty$-category spanned by those functors which preserve finite products. We can think of $\shP_{\Sigma}(\shC)$ as obtained from $\shC$ by formally adjoining sifted colimit. More precisely:

\begin{proposition}[Lurie {\cite[\S 5.5.8]{HTT}}] \label{prop:nonab:derived} For any $\shC \in \widehat{\Cat}_\infty^\amalg$, there exists a {\em Yoneda functor} $j \colon \shC \to \shP_\Sigma(\shC)$, such that the natural transform $\{j \colon \shC \to \shP_\Sigma(\shC)\}_{\shC}$ exhibits the functor  
	$$(\shC \in \widehat{\Cat}_\infty^\amalg) \mapsto (\shP_\Sigma(\shC) \in \widehat{\Cat}_\infty^{\rm L})$$
 	as a left adjoint to the inclusion $\widehat{\Cat}_\infty^{\rm L}\subseteq \widehat{\Cat}_\infty^\amalg$. Moreover, given any $\shC \in \Cat_\infty^\amalg$:
	\begin{enumerate}
		\item  \label{prop:nonab:derived-1} The Yoneda functor $j \colon \shC \to \shP_\Sigma(\shC)$ is fully faithful and preserves finite coproducts.
		\item  \label{prop:nonab:derived-2} For any $\shD \in \widehat{\Cat}_\infty^\Sigma$, composition with $j$ induces an equivalence of $\infty$-categories
				$$\Fun_\Sigma(\shP_\Sigma(\shC), \shD) \to \Fun(\shC, \shD).$$
			Moreover, a functor $F \colon \shP_{\Sigma}(\shC) \to \shD$ belongs to $\Fun_\Sigma(\shP_\Sigma(\shC), \shD)$ if and only if $F$ is a left Kan extension of $f = F \circ j$ along $j$.
		\item  \label{prop:nonab:derived-3} For any $\shD \in \widehat{\Cat}_\infty^{\rm L}$, composition with $j$ restricts to an equivalence of $\infty$-categories
				$$\FunL(\shP_\Sigma(\shC), \shD) \to \Fun^{\amalg}(\shC, \shD).$$
			Combining with \eqref{prop:nonab:derived-2}, we obtain that for any given functor $F \in \Fun_\Sigma(\shP_\Sigma(\shC), \shD)$, $F$ preserves all small colimits if and only if $f = F \circ j$ preserves finite coproducts.
	\end{enumerate}
\end{proposition}

\begin{proof} This is a reformulation of \cite[Propositions 5.5.8.10, 5.5.8.15]{HTT} in view of the adjunction provided by \cite[Corollary 5.3.6.10]{HTT}.
\end{proof}

\begin{definition} In the situation of Proposition \ref{prop:nonab:derived} \eqref{prop:nonab:derived-2} or \eqref{prop:nonab:derived-3}, we will refer to the functor $F$ as the {\em derived functor} of $f = F|\shC$. By virtue of Proposition \ref{prop:nonab:derived}, the functor $F$ is essentially uniquely determined by $f$ and preserves sifted colimits.
\end{definition}

\subsection{Derived symmetric, exterior, and divided powers} \label{sec:derived_sym} 
In this subsection, we apply Lurie's machinery of \S \ref{sec:non-abelian} to the classical functors of \S \ref{sec:classical_sym}.

\begin{notation}[{\cite[\S 25.2.1]{SAG}}] \label{notation:SCRMod} Let $A$ be a simplicial commutative ring and let $A^\circ$ be its underlying $\EE_\infty$-ring. We let $\SCRModcn: = \CAlgDelta \times_{\CAlg} \Mod^\cn$ denote the $\infty$-category of pairs $(A, M)$, where $A$ is a simplicial commutative ring and $M$ is a connective $A^\circ$-module. Then the forgetful functor $q \colon \SCRModcn \to \CAlgDelta$ is a coCartesian ﬁbration which classifies the functor $(A \in \CAlgDelta) \mapsto (\Mod^\cn_A \in \hatCat)$. Let ${\rm PolyMod}^{\rm ff} \subseteq \SCRModcn$ denote the full subcategory spanned by elements of the form $(A = \ZZ[x_1, \ldots, x_m], M = A^n)$ for $m,n \ge 0$. Then the inclusion ${\rm PolyMod}^{\rm ff} \hookrightarrow \SCRModcn$ extends to an equivalence of $\infty$-categories $\shP_\Sigma({\rm PolyMod}^{\rm ff}) \simeq \SCRModcn$ (\cite[Proposition 25.2.1.2]{SAG}).
\end{notation}

\begin{construction}[Derived symmetric, exterior and divided powers; Lurie {\cite[\S 25.2.2]{SAG}}]\label{constr:sym.wedge.gamma} In the situation of Proposition \ref{prop:nonab:derived}, we let $\shC  ={\rm PolyMod}^{\rm ff}$, $\shD = \SCRModcn$ (Notation \ref{notation:SCRMod}) and let $f \colon {\rm PolyMod}^{\rm ff} \to {\rm PolyMod}^{\rm ff} \subseteq \SCRModcn$ be one of classical functors considered in \S \ref{sec:classical_sym}. Then we obtain the {\em derived} functor of $f$. More concretely, let $n \ge 0$ be an integer:
	\begin{enumerate}
		\item Let $f$ be the functor $(R, M) \mapsto (R, S_R^n(M))$, where $S_R^n(M)$ is the classical $n$th symmetric power of $M$ over $R$ (Definition \ref{def:cl:sym}). Then the derived functor of $f$ has the form $F(A,M) = (A, \Sym_A^n M)$, where $A \in \CAlgDelta$ and $\Sym_A^n M \in \Mod_A^\cn$. The connective $A$-module $\Sym_A^n M$ is called the {\em derived $n$th symmetric power of $M$ over $A$}.
		\item Let $f$ be the functor $(R, M) \mapsto (R, \bigwedge_{\cl, R}^n(M))$, where $\bigwedge_{\cl, R}^n(M))$ is the classical $n$th exterior power of $M$ over $R$ (Definition \ref{def:cl:wedge}). Then the derived functor of $f$ has the form $F(A,M) = (A, \bigwedge_A^n M)$, where $A \in \CAlgDelta$ and $\bigwedge_A^n M \in \Mod_A^\cn$. We will refer to the connective $A$-module $\bigwedge_A^n M$ as the {\em derived $n$th exterior power of $M$ over $A$}.
		\item Let $f$ be the functor $(R, M) \mapsto (R, \Gamma_{\cl, R}^n(M))$, where $\Gamma_{\cl, R}^n(M))$ is the classical $n$th divided power of $M$ over $R$ (Definition \ref{def:cl:Gamma}). Then the derived functor of $f$ has the form $F(A,M) = (A, \Gamma_A^n M)$, where $A \in \CAlgDelta$ and $\Gamma_A^n M \in \Mod_A^\cn$. We will refer to the connective $A$-module $\Gamma_A^n M$ as the {\em derived $n$th divdied power of $M$ over $A$}.
		\end{enumerate} 
\end{construction}

For readers' convenience, we list some of the basic properties of these functors that are frequently used in this paper; all these results in our context are due to Lurie:

\begin{proposition}[Lurie {\cite[\S 25.2]{SAG}}] \label{prop:symwedgegamma} The derived symmetric, exterior, and divided power functors defined in Construction \ref{constr:sym.wedge.gamma} satisfy the following properties:
\begin{enumerate}
	\item \label{prop:symwedgegamma-1} 
		(Base change) Let $A \to B$ be a morphism of simplicial commutative rings. Then for any connective $A$-module $M$ and an integer $n \ge 0$, there are canonical equivalences:
		\begin{align*}
			 B \otimes_A \Sym_A^n (M) & \xrightarrow{\sim} \Sym_B^n (B \otimes_A M); \\
			 B \otimes_A \bigwedge\nolimits_A^n (M) & \xrightarrow{\sim} \bigwedge\nolimits_B^n (B \otimes_A M); \\
			 B \otimes_A \Gamma_A^n (M) & \xrightarrow{\sim} \Gamma_B^n (B \otimes_A M).
		\end{align*}
	\item \label{prop:symwedgegamma-2}
	(Freeness) For any simplicial commutative ring $A \in \CAlgDelta$ and an integer $n \ge 0$, if the $A$-module $M$ is (locally) free of ﬁnite rank $r$, then $\Sym_A^n M$, $\bigwedge_A^n M$ and $\Gamma_A^n M$ are (locally) free of ﬁnite ranks ${n+r -1 \choose n}$, ${r \choose n}$, and respectively ${n+r -1 \choose n}$. If $M$ is flat, then $\Sym_A^n M$, $\bigwedge_A^n M$ and $\Gamma_A^n M$ are flat.
	\item \label{prop:symwedgegamma-3}
	(Flat modules over discrete rings) If $A= R$ is an ordinary commutative ring, and $M$ is a flat $A$-module, then there are canonical isomorphisms 
		$$ \Sym_R^n M \simeq \Sym_{\cl, R}^n M, \qquad  \bigwedge\nolimits_R^n M \simeq \bigwedge\nolimits_{\cl, R}^n M, \qquad  \Gamma_R^n (M) \simeq \Gamma_{\cl, R}^n (M).$$
(In this situation, we will often abuse notation and use $\Sym_R^n M$, $\bigwedge_A^n M$ and $\Gamma_A^n M$ to denote the corresponding classical constructions.)

	\item \label{prop:symwedgegamma-4}
	(Illusie's proposition) Let $A$ be an simplicial commutative ring, then for any connective $A$-module $M$ and an integer $n \ge 0$, there are canonical equivalences of connective $A$-modules $\Sym_A^n(\Sigma M) \simeq \Sigma^n \bigwedge_A^n (M)$ and $\bigwedge_A^n(\Sigma M) \simeq \Sigma^n \Gamma_A^n(M)$.
\end{enumerate}
\end{proposition}
\begin{proof} The assertion \eqref{prop:symwedgegamma-1} is \cite[Proposition 25.2.3.1]{SAG}, the assertion \eqref{prop:symwedgegamma-2} is \cite[Corollary 25.2.3.2, Corollary 25.2.3.3]{SAG}, the assertion \eqref{prop:symwedgegamma-3} is \cite[Proposition 25.2.3.4]{SAG}, and the assertion \eqref{prop:symwedgegamma-4} is \cite[Proposition 25.2.4.2]{SAG}. 
\end{proof}

\subsection{Derived symmetric algebras}\label{sec:dSym*}
In this subsection, we review the construction of derived symmetric algebras and study their basic properties.

\begin{definition}[Derived symmetric algebras; {\cite[Construction 25.2.2.6]{SAG}}] \label{def:dsym*} Let $\SCRModcn$ be defined as in Notation \ref{notation:SCRMod}, and let $\Fun(\Delta^1, \CAlgDelta)$ denote the functor $\infty$-category whose objects are morphisms of simplicial commutative rings $A \to B$. Then the forgetful functor  
	$${\rm forg} \colon \Fun(\Delta^1, \CAlgDelta) \to \SCRModcn, \quad (A \to B) \mapsto (A, B^\circ)$$
 admits a left adjoint functor of the following form:
 	$$\SCRModcn \to \Fun(\Delta^1, \CAlgDelta), \quad (A, M) \mapsto (A, \Sym_A^*(M)).$$
The $A$-algebra $\Sym_A^*(M) \in \CAlgDelta_A$ is called {\em derived symmetric algebra of $M$ over $A$}. By virtue of {\cite[Construction 25.2.2.6]{SAG}}, there is a canonical equivalence of $A$-modules
	$$\Sym_A^*(M) \simeq \bigoplus_{n \ge 0} \Sym_A^n(M).$$
\end{definition}
 
\begin{remark}[The underlying commutative algebra objects of derived symmetric algebras] \label{rmk:CAlgofSym} Let $A$ be a simplicial commutative ring, and let $\Theta \colon \CAlgDelta \to \CAlg_\ZZ^\cn$ denote the functor of taking underlying $\EE_\infty$-rings (\cite[Construction 25.1.2.1]{SAG}). Then $\Theta$ induces a forgetful functor
	$$\Theta_A \colon \CAlgDelta_A \to \CAlg^\cn_{\Theta(A)/} \simeq \CAlg(\Mod_A^\cn),$$
where the last equivalence is given by \cite[Corollary 3.4.1.7]{HA}. For an object $B \in \CAlgDelta_A$, we will refer to its image $\Theta_A(B) \in \CAlg(\Mod_A^\cn)$ as the {\em underlying commutative algebra object in $\Mod_A^\cn$ of $B$}. For any connective $A$-module $M$, we will generally abuse notations by using the notation $\Sym_A^*(M)$ to denote the derived symmetric algebra of $M$ over $A$, its underlying commutative algebra object in $\Mod_A^\cn$, and its underlying connective $A$-module. In particular, the underlying connective $A$-module $\Sym_A^*(M) \simeq \bigoplus_{n \ge 0} \Sym_A^n(M)$ is equipped with a multiplication map
	$$m \colon \Sym_A^*(M) \otimes_A \Sym_A^*(M) \to \Sym_A^*(M)$$
which is unital, commutative, and associative up to homotopy. Furthermore, the multiplication is {\em graded}, that is, $m$ restricts to a map of connective $A$-modules
	$$\Sym_A^n(M) \otimes_A \Sym_A^{n'}(M) \to \Sym_A^{n+n'}(M).$$
(To prove this assertion, observe that since the relative tensor product $\otimes_A \colon \Mod_A^\cn \times \Mod_A^\cn \to \Mod_A^\cn$ preserves sifted colimits separately in each variable, we might regard the multiplication $m$ as a functor from $\SCRModcn$ to $\shE$, where $\shE$ is defined as in Notation \ref{notation:EandD}. Then since $m$ preserves sifted colimits, it is determined by its restrictions to ${\rm PolyMod}^{\rm ff}$, for which the desired statements are straightforward.)
\end{remark}
 
\begin{remark}[Coproducts in $\SCRModcn$ and $\Fun(\Delta^1, \CAlgDelta)$]\label{rmk:Sym*:coproducts} Let $\shC_0 = {\rm PolyMod}^{\rm ff}$ denote the full subcategory of $\SCRModcn$ spanned by objects of the form $(A = \ZZ[x_1, \ldots, x_m], M = A^n)$, as in Notation \ref{notation:SCRMod}. Then $\shC_0$ admits finite coproducts given by the formula $(A, M=A^n) \coprod (B, N= B^t) = (A \otimes B, M \boxplus N)$,  where $A=\ZZ[x_1, \ldots, x_m]$, $B=\ZZ[y_1, \ldots, y_s]$, $A \otimes B \simeq \ZZ[x_1, \ldots, x_m, y_1, \ldots, y_s]$, and $M \boxplus N \simeq (A \otimes B)^{n+t}$. By virtue of Proposition \ref{prop:nonab:derived}, the coproduct structure of $\shC_0$ extends to a coproduct structure on $\SCRModcn \simeq \shP_\Sigma(\shC_0)$ for which  the inclusion $\shC_0 \hookrightarrow \SCRModcn$ preserves finite coproducts. The forgetful functor $\SCRModcn \to \CAlgDelta$ preserves finite coproducts, therefore coproducts in $\SCRModcn$ takes the form $(A, M) \coprod (B,N) = (A \otimes B, M \boxplus N)$, where $A \otimes B$ is the pushout of the diagram of simplicial commutative rings $A \leftarrow \ZZ \rightarrow B$, and $M \boxplus N$ is an $A \otimes B$-module. We will refer $M \boxplus N \in \Mod_{A \otimes B}$ as the {\em external direct sum of $M$ and $N$}. 

We denote the coproduct of two objects $(A \to A')$ and $(B \to B')$ in $\Fun(\Delta^1, \CAlgDelta)$ by $(A \to B) \coprod (A ' \to B' ) = (A \otimes A', B \boxtimes B')$, where $A \otimes B$ is the pushout of the diagram of simplicial commutative rings $A \leftarrow \ZZ \rightarrow A'$ as before, and $B \boxtimes B' \in \CAlgDelta_{A \otimes A'}$ denotes the simplicial commutative ring $B \otimes B'$ with the induced $A \otimes A'$-algebra structure. We will refer to $B \boxtimes B' \in \CAlgDelta_{A \otimes A'}$ as the {\em external tensor product of $B$ and $B'$}. 
\end{remark} 
 
\begin{definition} \label{def:perfectfp} Let $A$ be a simplicial commutative ring and $n \ge 0$ an integer. We say an $A$-module $M$ is {\em perfect} (resp., {\em perfect to order $n$}, in the sense of \cite[Proposition 2.5.7]{DAG}) if the functor $\Map_{\Mod_A}(M, \blank)$ preserve filtered colimits (resp., if the restriction of functor $\Map_{\Mod_A}(M, \blank)$ to $\tau_{\le n} \Mod_A$ preserves filtered colimits). We $M$ is {\em almost perfect} if it is perfect to order $n$ for all $n \ge 0$. Let $\phi \colon A \to B$ be a map of simplicial commutative rings. We say $\phi$ is {\em locally of finite presentation} (resp., {\em of finite presentation to order $n$}, in the sense of \cite[Proposition 3.1.5]{DAG}) if the functor $\Map_{\CAlgDelta_A}(B, \blank)$ preserve filtered colimits (resp., if the restriction of functor $\Map_{\CAlgDelta_A}(B, \blank)$ to $\tau_{\le n} \CAlgDelta_A$ preserves filtered colimits). We say $\phi$ is {\em almost of ﬁnite presentation} if it is of finite presentation to order $n$ for all $n \ge 0$.
\end{definition}
 
The following proposition is a list of basic properties of derived symmetric algebras:

\begin{proposition}\label{prop:Sym*} Let $A$ be a simplicial commutative ring and $M$ a connective $A$-module.
	\begin{enumerate}[leftmargin=*]
		\item \label{prop:Sym*-1} (Classical truncations)
		There is a canonical equivalence of discrete $\pi_0 A$-algebras 
			$$\pi_0(\Sym_A^*(M)) \simeq \pi_0(\Sym_{\pi_0 A}^* (\pi_0 M)) \simeq S^*_{\pi_0 A} (\pi_0 M),$$
		 where $S_{\pi_0 A}^*(\pi_0 M)$ denotes the classical symmetric algebra of the discrete module $\pi_0(M)$ over the discrete commutative ring $\pi_0 A$.
		\item \label{prop:Sym*-2} (Finiteness conditions)
		The map of simplicial commutative rings $A \to \Sym_A^*(M)$ is of finite presentation to order $n$ (resp.  almost of ﬁnite presentation, locally of finite presentation) if and only if the $A$-module $M$ is perfect to order $n$ (resp. almost perfect, perfect).
		\item \label{prop:Sym*-3} (Relative cotangent complexes)
		Let $L_{\Sym_A^*(M)/A}$ (resp. $L_{A/\Sym_A^*(M)}$) denote the relative cotangent complex of the map of simplicial commutative rings $A \to \Sym_A^*(M)$ (resp. the augmentation map $\Sym_A^*(M) \to A$). Then there are canonically equivalences $L_{\Sym_A^*(M)/A} \simeq M \otimes_{A}  \Sym_A^*(M) \in \Mod_{\Sym_A^*(M)}^{\rm cn}$ and $L_{A/\Sym_A^*(M)} \simeq \Sigma \, M \in \Mod_{A}^{\rm cn}$. 
		\item \label{prop:Sym*-4} (Base-change) The functor $(A, M) \in \SCRModcn \mapsto (A, \Sym_A^* M) \in \Fun(\Delta^1, \CAlgDelta)$ carries $q$-coCartesian edges of $\SCRModcn$ to $\mathrm{ev}_{\{0\}}$-coCartesian edges of $\Fun(\Delta^1, \CAlgDelta)$. In other words, for any map $A \to A'$ in $\CAlgDelta$ and any connective $A$-module $M$, there is a canonical equivalence $A' \otimes_A \Sym_A^*(M) \simeq \Sym_{A'}^*(A' \otimes_A M) \in \CAlgDelta_{A'}$.
		\item \label{prop:Sym*-5} (Direct sums)
		For $A, B \in \CAlgDelta$, $M \in \Modcn_A$, $N \in \Modcn_B$, there is a canonical equivalence $\Sym_{A \otimes B}^*(M \boxplus N) \simeq \Sym_A^*(M) \boxtimes \Sym_B^*(N)$ in $\CAlgDelta_{A \otimes B}$. In particular, for any connective $A$-modules $M$ and $M'$, there is a canonical equivalence of algebras $\Sym_A^*(M \oplus M') \simeq \Sym_A^*(M) \otimes_A \Sym_A^*(M') \in \CAlgDelta_{A}$.
	\end{enumerate}
\end{proposition}

\begin{proof} 
Assertion \eqref{prop:Sym*-5} follows from the combination of Remark \ref{rmk:Sym*:coproducts} and Proposition \ref{prop:nonab:derived}: let $\shC_0 = {\rm PolyMod^{ff}}$ be as defined in Remark \ref{rmk:Sym*:coproducts}, then it is direct to see that the restriction $(\Sym^*){|\shC_0} \colon \shC_0 \to \Fun(\Delta^1, \CAlgDelta)$ preserves finite coproducts. Therefore, by virtue of Proposition \ref{prop:nonab:derived} \eqref{prop:nonab:derived-3}, $\Sym^* \colon \SCRModcn \to \Fun(\Delta^1, \CAlgDelta)$ preserves finite coproducts. The last statement of \eqref{prop:Sym*-5} follows from base change along the multiplication map $A \otimes A \to A$.

We next prove assertion \eqref{prop:Sym*-4}. To show the canonical map $\alpha_{(M, A')} \colon A' \otimes_A \Sym_A^*(M) \to \Sym_{A'}^*(A' \otimes_A M)$ is an equivalence, arguing as the proof of \cite[Proposition 25.2.3.1]{SAG}, (since the functor $M \mapsto \alpha_{(M,A')}$ commutes with sifted colimits) we may first reduce to the case where $M = A \times_\ZZ M_0$, $M_0 \simeq \ZZ^m$. Then we have a commutative diagram
	$$
	\begin{tikzcd}
		& A' \otimes_A A \otimes_\ZZ \Sym_\ZZ^*(M_0) \ar{ld} \ar{rd} & \\
		A' \otimes_A \Sym_A^*(A \otimes_\ZZ M_0) \ar{rr}{\alpha_M} & & \Sym_{A'}^*(A' \otimes_A A \otimes_\ZZ M_0).
	\end{tikzcd}
	$$
To show $\alpha_M$ is an equivalence, it suffices to prove in the case where $A = \ZZ$. Since the functor $A' \mapsto \alpha_{M, A'}$ commutes with sifted colimits, we may further reduce to the case where $A' = \ZZ[x_1, \ldots, x_n]$. In this case, $\alpha_M$ is evidently an isomorphism of polynomial algebras.

For assertion \eqref{prop:Sym*-3}, the formula $L_{\Sym_A^*(M)/A} \simeq M \otimes_{A} \Sym_A^*(M)$ is \cite[Example 25.3.2.2]{SAG}. By virtue of \cite[Remark 25.3.2.5]{SAG}, the composition $A \to \Sym_A^*(M) \to A$ of simplicial commutative rings induces a 
fiber sequence
	$$A \otimes_{\Sym_A^*(M)} L_{\Sym_A^*(M)/A}  \to L_{A/A} \to L_{A/\Sym_A^*(M)}.$$
 The formula $L_{A/\Sym_A^*(M)} \simeq \Sigma \, M$ follows from the above fiber sequence and $L_{A/A} \simeq 0$.

For assertion \eqref{prop:Sym*-2}, if the map $A \to \Sym_A^*(M)$ is of finite presentation to order $n$ (resp.  almost of ﬁnite presentation, locally of finite presentation), then it follows from \cite[Proposition 3.2.14]{DAG} that $ L_{\Sym_A^*(M)/A}$ is perfect to order $n$ (resp., almost perfect, perfect) as an $\Sym_A^*(M)$-module. Consequently, $M = L_{\Sym_A^*(M)/A} \otimes_{\Sym_A^*(M)} A$ is perfect to order $n$ (resp., almost perfect, perfect) as an $A$-module. For the converse direction, from the canonical equivalence of spaces
	$$\Map_{\CAlgDelta_A}(\Sym_A^*(M), B) \simeq \Map_{\Mod_A}(M, B^\circ)$$
for all $B \in \CAlgDelta_A$ and the fact that $(\tau_{\le n} B)^\circ \simeq \tau_{\le n} (B^\circ)$, we deduce that if $M$ is perfect to order $n$ (resp. perfect), then $A \to \Sym_A^*(M)$ is of finite presentation to order $n$ (resp. locally of finite presentation).

Finally, we prove assertion \eqref{prop:Sym*-1}. Apply assertion \eqref{prop:Sym*-4} to the map $A \to \pi_0(A)$, we obtain an equivalence $\pi_0(A) \otimes_A \Sym_A^*(M) \simeq \Sym_{\pi_0 A}^*(\pi_0 A \otimes_A M)$. Consequently, we have
	$$\pi_0(\Sym_A^* M) \simeq \pi_0(\Sym_{\pi_0 A}^*(\pi_0 A \otimes_A M)).$$
Let $R = \pi_0 A$. It remains to show that for $N = \pi_0(A) \otimes_A M \in \Modcn_R$, there is a canonical equivalence $\pi_0(\Sym_R^* N) \simeq S_R^* (\pi_0 N)$, where $S_R^*(\pi_0 N)$ is the classical symmetric algebra. Let $P_\bullet$ be a simplicial resolution of $N$, where each $P_n$ is a projective $R$-module. 
Since the derived symmetric algebra functor $\Sym_R^*(\blank)$ preserve sifted colimits, we obtain equivalences 
	$$\Sym_R^*(M) \simeq |\Sym_R^*(P_\bullet)| \simeq |S_R^* (P_\bullet)|.$$
Consequently, $\pi_0 (\Sym_R^*(N))$ is isomorphic to the colimit of the diagram
	$$
	\begin{tikzcd}
		& S_R^*(P_1) \ar[shift left]{r}{S_R^*(d_1)} \ar[shift right]{r}[swap]{S_R^*(d_0)}  & S_R^*(P_0)
	\end{tikzcd}
	$$
(where $d_0, d_1 \colon P_1 \rightrightarrows P_0$ are the face maps). On the other hand, the colimit of the above diagram is canonically isomorphic to $S_R^*(\pi_0 N)$, by the right-exactness of the classical symmetric algebra functors (see, for example, \cite[Chapter III, \S 6.2, Proposition 4, p.499]{Bou}). 
\end{proof}

\begin{proposition} \label{prop:Sym*:cofib} Let $R$ be a simplicial commutative ring, and let $\rho \colon M' \to M$ be a map of connective $R$-modules. Then there is a pushout diagram in $\CAlgDelta_R$:
	$$
	\begin{tikzcd}
		\Sym_R^*(M') \ar{d} \ar{r}{\Sym_R^*(\rho)}& \Sym_R^*(M) \ar{d} \\
		R \ar{r} & \Sym_R^*({\rm cofib}(\rho)).
	\end{tikzcd}
	$$
\end{proposition}

\begin{proof} 
Being a left adjoint (Definition \ref{def:dsym*}), the functor $\Sym_R^*$ preserves small colimits. Consequently, the desired result follows from applying $\Sym_R^*$ to the pushout square
	$$
	\begin{tikzcd}
		M' \ar{d} \ar{r}{\rho}& M \ar{d} \\
		0 \ar{r} & \cofib(\rho).
	\end{tikzcd}
	$$
\end{proof}
\subsection{Universal fiber sequences associated to symmetric powers} \label{sec:univ:fib:seq} In this subsection, we review and generalizes Lurie's construction in his proof of Illusie's result \cite[Proposition 25.2.4.1]{SAG} to obtain universal fiber sequences associated to symmetric powers.

\begin{notation} \label{notation:EandD} 
Let $\shE$ denote the $\infty$-category $\Fun(\Delta^1, \SCRModcn) \times_{\Fun(\Delta^1, \CAlgDelta)} \CAlgDelta$ whose objects are pairs $(A, \rho \colon M' \to M)$, where $A$ is a simplicial commutative ring and $\rho$ is a morphism of connective $A$-modules. Let $\shD$ denote the full subcategory of the $\infty$-category $\Fun(\Delta^1 \times \Delta^1, \SCRModcn) \times_{\Fun(\Delta^1 \times \Delta^1, \CAlgDelta)} \CAlgDelta$ spanned by objects $(A, \alpha)$, where  $A$ is a simplicial commutative ring, and $\alpha$ is a fiber sequence of connective $A$-modules:
	$$
	\begin{tikzcd}
		M' \ar{r}{\rho} \ar{d} & M \ar{d}{\varphi} \\
		0 \ar{r} & M'' .
	\end{tikzcd}
	$$
As usual, we will abbreviate the above fiber sequence as $\alpha \colon M' \to M \to M''$. Let $\theta \colon \shD \to \shE$ denote the forgetful functor which carries $(A, \alpha \colon M' \xrightarrow{\rho} M \xrightarrow{\varphi} M'')$ to $(A, \rho \colon M' \to M)$. Then by virtue of \cite[Remark 1.1.1.7]{HA}, that is, by applying \cite[Proposition 4.3.2.15]{HTT} twice, we deduce that $\theta$ is a trivial Kan fibration. 

Let $\shE_0 \subseteq \shE$ be the full subcategory spanned by those pairs $(R, \rho \colon M' \to M)$, where $R$ is a polynomial ring $\ZZ[x_1, \ldots, x_k]$ and $\rho$ fits into a short exact sequence $0 \to M' \xrightarrow{\rho} M \to M'' \to 0$ of finitely generated free $R$-modules. Similarly, let $\shD_0 \subseteq \shD$ be the full subcategory spanned by those pairs $(R, \alpha)$, where $R$ is a polynomial ring $\ZZ[x_1, \ldots, x_k]$, and $\alpha$ is a short exact sequence $0 \to M' \xrightarrow{\rho} M \xrightarrow{\rho'} M'' \to 0$ of finitely generated free $R$-modules. Then the argument of the proof of \cite[Proposition 25.2.4.1]{SAG} shows that the inclusions $\shE_0 \hookrightarrow \shE$ and $\shD_0 \hookrightarrow \shD$ induce equivalences of $\infty$-categories $\shP_\Sigma(\shE_0) \simeq \shE$ and $\shP_\Sigma(\shD_0) \simeq \shD$, respectively.
\end{notation}

\begin{construction}[Koszul complexes] \label{constr:Koszul:S^n} Let $R$ be a discrete commutative ring, and suppose we are given a morphism $\rho \colon M' \to M$ of finite generated locally free $R$-modules. Let $S_R^* M$ denote the ordinary symmetric algebra generated by $M$ over $R$. Then there is canonically a Koszul complex of locally free $S_R^*(M)$-modules of finite rank:
	\begin{align}\label{eqn:Koszul:S^*}
	\cdots \to (\bigwedge\nolimits_R^2 M') \otimes_R (S_R^* M) \to  (\bigwedge\nolimits_R^1 M') \otimes_R (S_R^* M)  \to  (\bigwedge\nolimits_R^0 M') \otimes_R (S_R^* M)
	\end{align}
associated to the cosection $M' \otimes_R S_R^*(M) \to S_R^*(M)$ induced by $\rho$. Restricting to homogeneous elements of degree $n$, we obtain a complex of locally free $R$-modules of finite rank:
	\begin{align}\label{eqn:Koszul:S^n}
	\bS^n(R, \rho \colon M' \to M) \colon \quad \bigwedge\nolimits^n_R M'  \xrightarrow{d_n}  \cdots \xrightarrow{d_2} \bigwedge\nolimits_R^1 M'  \otimes_R (S_R^{n-1} M) \xrightarrow{d_1} S_R^n M.
	\end{align}
The construction $(R, \rho) \mapsto (R, \bS^n(R, \rho \colon M' \to M))$ determines a functor from the ordinary category of pairs $(R, \rho \colon M' \to M)$, where $R$ is a (discrete) commutative ring and $\rho$ is a morphism of finite generated locally free $R$-modules, to the ordinary category category of pairs $(R, M_*)$, where $R$ is a (discrete) commutative ring and $M_*$ is a non-negatively graded chain complexes of finitely generated locally free $R$-modules. By construction, the formation of $\bS^n(R, \rho \colon M' \to M)$ commutes with base change $R \to R'$ of discrete commutative rings. For any integer $n' \ge 0$, from the $S_R^*(M)$-module structure of \eqref{eqn:Koszul:S^*}, multiplication with elements of $S_R^{n'} M$ induces a morphism of non-negatively graded chain complexes of finitely generated locally free $R$-modules:
	$$\bS^n(R, \rho \colon M' \to M) \otimes_R (S_R^{n'} M) \to \bS^{n+n'}(R, \rho \colon M' \to M).$$
\end{construction}

\begin{remark} \label{rmk:S^d:explicit}By choosing local basis for the locally free modules $M'$ and $M$, respectively, the differentials $d_i$ of the complex $\bS^n(R, \rho \colon M' \to M)$ of \eqref{eqn:Koszul:S^n} can be described explicit as follows: for all local elements $x_1, \ldots, x_i \in M'$ and $y \in S_R^{d-i} M$, we have
	\begin{align*}
		d_i (x_{1} \wedge \cdots \wedge x_{i} \otimes y) = \sum_{j=1}^i (-1)^{j-1} (x_{1} \wedge \cdots  \widehat{x_{j}} \cdots \wedge x_{i} ) \otimes (\rho(x_{j}) \cdot y) \in \bigwedge\nolimits_R^{i-1} M' \otimes_R S_R^{d-i+1} M. 
	\end{align*}
\end{remark}

\begin{variant} There is an analogous construction for exterior algebra. Let $R$ be a discrete commutative ring, and suppose we are given a morphism $\rho \colon M' \to M$ of finite generated locally free $R$-modules. For any $n \ge 0$, there is a Koszul complex of exterior type	
	\begin{equation} \label{eqn:Koszul:wedge^n}
		 \bwedge^n(R, \rho \colon M' \to M) \colon \quad \Gamma_R^n(M') \xrightarrow{d_n'}  \cdots  \xrightarrow{d_2'} \Gamma_R^1(M') \otimes_R \bigwedge\nolimits_R^{n-1}(M)   \xrightarrow{d_1'} \bigwedge\nolimits_R^{n} M
	\end{equation}
which could be defined by the following formula:
	$$ \bwedge^n(R, \rho \colon M' \to M) = \Sigma^n \, (\bS^n(R, \rho^\vee \colon M^\vee \to M'^\vee))^\vee.$$
(Here $(\blank)^\vee = \Hom_R(\blank, R) = \Ext^0_R(\blank,R)$ denotes the operation of taking duals over $R$ in the classical sense.) Concretely, the differentials $d_j'$ of the complex \eqref{eqn:Koszul:wedge^n} are described by the following  formula: for all local elements $g \in \bigwedge^{d-j} M$ and $\varepsilon_1^{(\nu_1)} \cdots \varepsilon_{m'}^{(\nu_{m'})} \in \Gamma^{j}(M')$ (where $\varepsilon_1, \ldots, \varepsilon_{m'}$ is a local basis of $M'$, $\nu_i \ge 0$ and $\sum \nu_i = j$), we have
	$$d_j'(g \otimes \varepsilon_1^{(\nu_1)} \cdots \varepsilon_{m'}^{(\nu_{m'})}) = \sum_{i=1}^{m'}  (\rho(\varepsilon_i) \wedge g ) \otimes \varepsilon_1^{(\nu_1)}  \cdots \varepsilon_{i}^{(\nu_{i} - 1)} \cdots \varepsilon_{m'}^{(\nu_{m'})} \in  \bigwedge\nolimits_R^{d-j+1} M \otimes_R \Gamma_R^{j-1} M'.$$
The formation of $\bwedge^n(R, \rho \colon M' \to M)$ enjoys similar functorial properties as the formation of $\bS^n(R, \rho \colon M' \to M)$, and all the constructions and results of this subsection have direct analogue for the exterior powers. The complex $\bwedge^n(R, \rho \colon M' \to M)$ will not play an essential role in this paper; However, it is useful for computations in various situations.
\end{variant}

\begin{remark} The totality $\bS^*(R, \rho \colon M' \to M) = \bigoplus_{n \ge 0} \bS^n(R, \rho \colon M' \to M)$ has a canonical graded $R$-Hopf algebra structure and a bi-graded dga structure; The totality $\bwedge^*(R, \rho \colon M' \to M) = \bigoplus_{n \ge 0} \bwedge^n(R, \rho \colon M' \to M)$ has a canonical bigraded $R$-Hopf algebra structure and a bi-graded dga structure. We will not use these algebra structures in this paper; but they will play a role in the theory of derived Schur functors in our forthcoming paper.
\end{remark}
	
\begin{construction} \label{constr:left:Sym} Following Lurie \cite[proof of Proposition 25.2.4.1]{SAG}, we let $R$ be a discrete commutative ring, and consider a short exact sequence of finite generated free $R$-modules
	$0 \to M' \xrightarrow{\rho} M \to M'' \to 0.$
Then the Koszul complex of \eqref{eqn:Koszul:S^*} is an exact sequence of finite generated free $S_R^*(M)$-modules that resolves $S_R^* M''$. Restricting to homogeneous elements of degree $n$, we obtain that the complex $\bS^n(R, \rho \colon M' \to M)$ of \eqref{eqn:Koszul:S^n} is a resolution of the module $S_R^n M''$. Let $f_{i,n}(R, \rho)$ denote the image of the differential $d_i$ of $\bS^n(R, \rho \colon M' \to M)$, and let $\alpha_{i,n}(R, \rho)$ denote the short exact sequence:
	$$0 \to f_{i+1,n}(R, \rho) \xrightarrow{a_i} (\bigwedge\nolimits^i_R M') \otimes_R (S_R^{n-i} M) \xrightarrow{b_i} f_{i,n}(R, \rho) \to 0.$$
Then $a_{i-1} \circ b_{i} = d_i$, and we have canonical equivalences:
	$$f_{0,n}(R, \rho \colon M' \to M) \simeq S_R^n M'' , \qquad f_{n,n}(R, \rho \colon M' \to M) \simeq \bigwedge\nolimits_R^n M'.$$
For any integer $n' \ge 0$, multiplication with elements of $S_R^{n'} M$ induces functors 
	$$\alpha_{i,n}(R, \rho)\otimes_R S_R^{n'} M  \to \alpha_{i,n+n'}(R, \rho)$$ 
for which the induced map on middle terms coincides with the map induced by the multiplication map $S_R^{n-i}M \otimes_R S_R^{n'} M \to S_R^{n+n'-i} M$.
Since $\shP_\sigma(\shE_0) \simeq \shE$, and $\shD$ admits sifted colimits, by virtue of Proposition \ref{prop:nonab:derived}, for each integer $0 \le i \le n-1$, the functors $\shE_0 \to \shE$ and $\shE_0 \to \shD$,
	$$(R, \rho \colon M' \to M) \mapsto (R, d_{i+1} \text{~of~} \bS^n(R, \rho)) \qquad (R, \rho \colon M' \to M)  \mapsto (R, \alpha_{i,n}(R, \rho))$$
have essentially unique extensions $D_{i+1,n} \colon \shE \to \shE$ and $U_{i,n} \colon \shE \to \shD$,
	$$(A, \rho \colon M' \to M) \mapsto (A, d_{i+1,n}(A, \rho)) \qquad (A, \rho \colon M' \to M) \mapsto (A, \alpha_{i,n}(A, \rho) ),$$
respectively; here $d_{i+1,n}(A, \rho)$ is a morphism of connective $A$-modules:
		$$(\bigwedge\nolimits^{i+1}_A M') \otimes_R (\Sym_A^{n-i-1} M) \to (\bigwedge\nolimits^{i}_A M') \otimes_R (\Sym_A^{n-i} M),$$
and $\alpha_{i,n}(A, \rho)$ is a fibre sequence of connective $A$-modules:
	$$F_{i+1,n}(A, \rho) \xrightarrow{a_{i,n}(A, \rho)} (\bigwedge\nolimits^i_A M') \otimes_A (\Sym_A^{n-i} M) \xrightarrow{b_{i,n}(A, \rho)} F_{i,n}(A, \rho).$$
Moreover, there are canonical equivalences 
	$$d_{i,n}(A, \rho) \circ d_{i+1,n}(A,\rho) \simeq 0, \qquad a_{i,n}(A, \rho) \circ b_{i+1,n}(A, \rho) \simeq d_{i+1,n}(A,\rho),$$
	$$F_{0,n}(A, \rho \colon M' \to M) \simeq \Sym_A^n (\cofib(\rho)), \qquad F_{0,n}(A, \rho \colon M' \to M) \simeq \bigwedge\nolimits_A^n M'.$$
Furthermore, for any integer $n' \ge 0$, there are canonical morphisms of fiber sequences 
	$$\alpha_{i,n}(A, \rho) \otimes_A \Sym_A^{n'} M  \to 
\alpha_{i,n+n'}(A, \rho)$$
for which the induced map on middle terms coincides with the map induced by the multiplication map $\Sym_A^{n-i} M \otimes_A \Sym_A^{n'} M \to \Sym_A^{n+n'-i} M$ of Remark \ref{rmk:CAlgofSym}.
\end{construction}

\begin{remark}\label{rmk:d_i:Sym:bc} The functors $D_{i,n} \colon \shE \to \shE$, $(A, \rho) \mapsto (A, d_{i,n}(A, \rho))$, where $1 \le i \le n$, by construction commute with sifted colimits, and the diagrams
	$$
	\begin{tikzcd}
		\shE \ar{r}{D_{i,n}} \ar{d}{q} & \shE \ar{d}{q} \\
		\CAlgDelta \ar{r}{\id} & \CAlgDelta
	\end{tikzcd}
	$$
commute up to canonical equivalences (here the vertical arrow $q$ is the natural forgetful map the carries $(A, \rho)$ to $A$). Furthermore, \cite[Proposition 25.2.3.1]{SAG} implies that the functor $D_{i,n} \colon \shE \to \shE$ carries $q$-coCartesian edges to $q$-coCartesian edges. 
\end{remark}

\begin{remark}[Base change] \label{rmk:left:Sym:bc} The functors $U_{i,n} \colon \shE \to \shD$, $(A, \rho) \mapsto (A, \alpha_{i,n}(A, \rho))$, where $0 \le i \le n-1$, by construction commute with sifted colimits, and the diagrams
	$$
	\begin{tikzcd}
		\shE \ar{r}{U_{i,n}} \ar{d}{q} & \shD \ar{d}{q'} \\
		\CAlgDelta \ar{r}{\id} & \CAlgDelta
	\end{tikzcd}
	$$
commute up to canonical equivalences (here the vertical arrows $q$ and $q'$ are the natural forgetful maps). We claim that each functor $U_{i,n} \colon \shE \to \shD$ carries $q$-coCartesian edges to $q'$-coCartesian edges. That is, for every morphism $\phi \colon A \to B$ of simplicial commutative rings and every morphism $\rho \colon M' \to M$ of connective $A$-modules, the canonical map
	$$B \otimes_A \alpha_{i,n}(A, \rho)  \to \alpha_{i,n}(B, B \otimes_A \rho)$$
is an equivalence in $\shD \times_{\CAlgDelta} \{B\}$. In fact, for a given $n \ge 0$, we prove by induction on $0 \le i \le n-1$. In the case $i=0$, the equivalence $B \otimes_A \alpha_0(A, \rho) \simeq \alpha_0(B, B \otimes_A \rho)$ follows from the fact that the functor $\Mod_A \to \Mod_B$, $M \mapsto B \otimes_A M$ is exact, and the formations of $\Sym_A^n M$ and $\Sym_A^n(\cofib(\rho))$ commute with base change (see \cite[Proposition 25.2.3.1]{SAG}). For the induction step, the case of $(i+1)$ follows from the case $i$ by virtue of \textit{loc. cit.} and the fact that the final vertex of $\alpha_{i+1}(A, \rho)$ coincides with the initial vertex of $\alpha_{i,n}(A, \rho)$. 
\end{remark}

\begin{construction} \label{constr:right:Sym} Let $R$ be a discrete commutative ring, and suppose we are given a short exact sequence $0 \to M' \xrightarrow{\rho} M \to M'' \to 0$ of finite generated free $R$-modules. For integers $n \ge 0$ and $0 \le i \le n$, we let $g_{i,n}(R,\rho) \in \Mod_R^\cn$ denote the class of the ``brutal" truncation 
	$$(\bigwedge\nolimits^i_R M' ) \otimes_R (S_R^{n-i} M) \xrightarrow{d_i} \cdots \to M' \otimes_R S_R^{n-1} M \xrightarrow{d_1} S_R^n M$$
 of the complex \eqref{eqn:Koszul:S^n}. Then there are canonical equivalences
 	$$g_{0,n}(R, \rho) = S_R^n M, \qquad g_{0,n}(R, \rho) \simeq S_R^n M''.$$ 
Since the $R$-modules 
	$$\Ext^s_R\left( (\bigwedge\nolimits_R^{i} M') \otimes_R (S_R^{n-i} M), (\bigwedge\nolimits_R^{j} M') \otimes_R (S_R^{n-j} M) \right)$$
vanish for all $s \ne 0$ and $i,j \in \ZZ$ (here $M, M'$ are finite generated free $R$-modules), the exact sequence \eqref{eqn:Koszul:S^n} determines essentially unique fiber sequences
	$$\beta_{i,n}(R, \rho) \colon \quad \Sigma^{i} (\bigwedge\nolimits_R^{i+1} M') \otimes_R (S_R^{n-i-1} M) \to g_{i,n}(R, \rho) \to g_{i+1}(R, \rho).$$
For any integer $n' \ge 0$, multiplication with elements of $S_R^{n'} M$ induces functors 
	$$\beta_{i,n}(R, \rho) \otimes_R S_R^{n'} M \to \beta_{i,n+n'}(R, \rho)$$
for which the induced map on middle terms coincides with the map induced by the multiplication map $S_R^{n-i-1}M \otimes_R S_R^{n'} M \to S_R^{n+n'-i-1} M$. 
As in Construction \ref{constr:left:Sym}, by virtue of Proposition \ref{prop:nonab:derived}, the functors $\shE_0 \to \shD$,
	$$(R, \rho \colon M' \to M) \mapsto (R, \beta_{i,n}(R, \rho))$$
extend essentially uniquely to functors $V_{i,n} \colon \shE \to \shD$,
	$$(A, \rho \colon M' \to M) \mapsto (A, \beta_{i,n}(A, \rho) ),$$
where $\beta_{i,n}(A, \rho)$ is a fibre sequence of connective $A$-modules:
	$$\Sigma^{i} (\bigwedge\nolimits_A^{i+1} M') \otimes_A (\Sym_A^{n-i-1} M) \xrightarrow{c_{i,n}(A, \rho)} G_{i,n}(A, \rho) \to G_{i+1,n}(A, \rho).$$
Let $e_{i,n}(A, \rho)$ denote the morphism $G_{i+1,n}(A, \rho) \to \Sigma^{i+1} (\bigwedge\nolimits_A^{i+1} M') \otimes_R (\Sym^{n-i-1} M)$ obtained from $\beta_{i,n}(A, \rho)$. Then the composition of the natural maps
	$$\Sigma^{i} (\bigwedge\nolimits_A^{i+1} M') \otimes_A (\Sym_A^{n-i-1} M) \xrightarrow{c_{i,n}(A, \rho)} G_{i,n}(A, \rho)  \xrightarrow{e_{i-1,n}(A, \rho)} \Sigma^{i}  (\bigwedge\nolimits_A^{i} M') \otimes_A (\Sym_A^{n-i} M)$$
is canonically equivalent to $\Sigma^i \circ d_{i+1, n}(A, \rho)$, and there are canonical equivalences:
	$$G_{0,n}(A, \rho \colon M' \to M) \simeq \Sym_A^n M, \qquad G_{n,n}(A, \rho \colon M' \to M) \simeq \Sym_A^n (\cofib(\rho)).$$
For any integer $n' \ge 0$, there are canonical morphisms of functors 
	$$ \beta_{i,n}(A, \rho) \otimes_A \Sym_A^{n'} M \to 
\beta_{i,n+n'}(A, \rho)$$
for which the induced map on the initial vertices coincides with the map induced by the multiplication $\Sym_A^{n-i-1} M \otimes_A \Sym_A^{n'} M \to \Sym_A^{n+n'-i-1} M$ of Remark \ref{rmk:CAlgofSym}.
\end{construction}

\begin{remark}[Base change] \label{rmk:right:Sym:bc} For all $0 \le i \le n-1$, by construction the functors $V_{i,n} \colon \shE \to \shD$, $(A, \rho) \mapsto (A, \beta(A, \rho))$ commute with sifted colimits, and the following diagrams
	$$
	\begin{tikzcd}
		\shE \ar{r}{V_{i,n}} \ar{d}{q} & \shD \ar{d}{q'} \\
		\CAlgDelta \ar{r}{\id} & \CAlgDelta
	\end{tikzcd}
	$$
commute up to canonical equivalences (where the vertical arrows $q$ and $q'$ are the natural forgetful maps). By a similar argument as in Remark \ref{rmk:left:Sym:bc}, we deduce that each functor $V_{i,n} \colon \shE \to \shD$ carries $q$-coCartesian edges to $q'$-coCartesian edges. In other words, for every morphism $\phi \colon A \to B$ of simplicial commutative rings and every morphism $\rho \colon M' \to M$ of connective $A$-modules, the canonical map
	$$B \otimes_A \beta_{i,n}(A, \rho)  \to \beta_{i,n}(B, B \otimes_A \rho)$$
is an equivalence in $\shD \times_{\CAlgDelta} \{B\}$. 
\end{remark}

\begin{lemma}\label{lem:right:convolution}  Let $\shC$ be a stable $\infty$-category (or a prestable $\infty$-category in the sense of \cite[Deﬁnition C.1.2.1]{SAG}), let $n \ge 1$ be an integer, and suppose we are given a composable sequence of morphisms in $\shC$: 
	$$C_n \xrightarrow{d_n} C_{n-1} \xrightarrow{d_{n-1}} \cdots \to C_{1} \xrightarrow{d_1} C_0$$
for which the successive compositions satisfy $d_{i} \circ d_{i+1} \simeq 0$, $1 \le i \le n-1$. Assume that
	$$\Ext^{-k}_\shC(C_{i+j}, C_i) : = \pi_k (\Map_\shC(C_{i+j}, C_{i})) \simeq 0 \quad \text{for} \quad i \ge 0, j \ge 1, k \ge 1.$$
Then there are essentially unique objects $\{D_i \in \shC \}_{0 \le i \le n}$ that fit into essentially unique cofiber sequences $\beta_i \colon \Sigma^i C_{i+1} \xrightarrow{f_i} D_i \xrightarrow{g_i} D_{i+1}$ such that $D_0 = C_0$, $f_0 = d_1$, and $h_{i-1} \circ f_i \simeq \Sigma^i \, d_{i+1}$, where $h_i \colon D_{i+1} \to \Sigma^{i+1} C_{i+1}$ denotes the connecting morphism determined by $\beta_i$. 
\end{lemma}

\begin{proof} This result is well-known in the context of triangulated categories (see, for example, \cite[Lemma 1.5]{Orlov97}), and the same proof strategy works here. We will illustrate with the case $n=2$; the general situation follows from a similar argument and induction. Let $D_1$ be the cofiber of $d_1 \colon C_1 \to C_0$, let $\beta_0 \colon C_1 \xrightarrow{f_0=d_1} C_0 \xrightarrow{g_0} D_1$ denote the corresponding cofiber sequence, and let $h_0 \colon D_1 \to \Sigma C_1$ denote the essentially unique connecting morphism determined by $\beta_0$. From the fiber sequences of spaces
	$$\Map_\shC(\Sigma C_2, D_1) \xrightarrow{h_0 \circ} \Map_\shC(\Sigma C_2, \Sigma C_1) \xrightarrow{\Sigma d_1 \circ} \Map_{\shC}(\Sigma C_2, \Sigma C_0),$$
there exists an element $f_1 \in \Map_\shC(\Sigma C_2, D_1)$ lifting the element $\Sigma d_2 \in \Map_\shC(\Sigma C_2, \Sigma C_1)$. The space $\Omega \Map(\Sigma C_2, C_0)$ of such liftings is contractible by our assumption and Whitehead's theorem (\cite[p.16]{HTT}). Therefore, $f_1$ is determined up to contractible choices. Let $D_2$ be the cofiber of $f_1$, and $\beta_1$ the corresponding cofiber sequence. Then $D_2$ (resp. $\beta_1$) exists and is determined up to contractible choices. This proves the desired assertion.
\end{proof}

\begin{remark}
We see from the above proof Lemma \ref{lem:right:convolution} that the conditions of the lemma are much stronger than necessary (for example, only the conditions $\Ext_\shC^{-k}(C_2, C_0)\simeq 0$, $k \ge 1$, were used in the proof of the case $n=2$). However, the above version is enough for the purpose of applications in this paper; in a forthcoming paper, we plan to include a more systematical study of these systems and stronger versions of Lemma \ref{lem:right:convolution}.
\end{remark}

\begin{corollary}[Illusie] \label{cor:Illusie:S^n} Let $R$ be a discrete commutative ring, and let $\rho \colon M' \to M$ be a morphism of finitely generated locally free $R$-modules. Let $n \ge 0$ be an integer, and let $\bS^n(R, \rho \colon M' \to M)$ denote the Koszul complex \eqref{eqn:Koszul:S^n}. Then there is a canonical equivalence
	$$[\bS^n(R, \rho \colon M' \to M)] \simeq \Sym_R^n ( \cofib(M' \xrightarrow{\rho} M)) \in  \Mod_R^\cn$$
where $[\bS^n(R, \rho \colon M' \to M)]$ denotes the image of the complex $\bS^n(R, \rho \colon M' \to M)$ under the canonical equivalence $\shD_{\ge 0}(\Ch (\Mod_R^\heartsuit)) \simeq \Mod_R^\cn$ of \cite[Proposition 7.1.1.15]{HA}.
\end{corollary}
\begin{proof} By virtue of Lemma \ref{lem:right:convolution}, we have a canonical equivalence $[\bS^n(R, \rho \colon M' \to M)] \simeq G_{n,n}(R, \rho \colon M' \to M)$ (where $G_{i,n}(R, \rho)$ is defined as in Construction \ref{constr:right:Sym}). Then the desired assertion follows from the properties of the fiber sequences $\beta_i(A,\rho)$ of Construction \ref{constr:right:Sym}.
\end{proof}

\begin{variant}[Illusie] \label{cor:Illusie:wedge^n} In the same situation of as Corollary \ref{cor:Illusie:S^n}, let $\bwedge^n(R, \rho \colon M' \to M)$ denote the Koszul complex of exterior type \eqref{eqn:Koszul:wedge^n}. Then there is a canonical equivalence
	$$[\bwedge^n(R, \rho \colon M' \to M)] \simeq \bigwedge\nolimits_R^n ( \cofib(M' \xrightarrow{\rho} M)) \in  \Mod_R^\cn.$$
\end{variant}
\begin{proof} The proof is completely analogous to the proof of Corollary \ref{cor:Illusie:S^n}.
\end{proof}

\subsection{Filtrations of symmetric powers} \label{sec:filtration:fib:seq} 
In this subsection, we review and slightly generalizes Lurie's \cite[Construction 25.2.5.4]{SAG}. Let $R$ be a discrete commutative ring, and suppose we are given a short exact sequence of finite generated free $R$-modules:
	$$0 \to M' \xrightarrow{\varphi} M \to M'' \to 0.$$
For each non-negative integer $n$ and an integer $0 \le i \le n$, there are natural morphisms $(S_R^{n-i}M') \otimes_R (S_R^{i} M) \to S_R^{n} M$, $(\bigwedge_R^{n-i} M') \otimes_R (\bigwedge_R^{i} M) \to \bigwedge_R^{n} M$, and $(\Gamma_R^i M') \otimes_R (\Gamma_R^{n-i} M) \to \Gamma_R^{n} M$ given by the respective multiplication maps; We let $\overline{F}^i(R, \varphi) \subseteq S_R^{n} M$,  $\overline{G}^i(R, \varphi) \subseteq \bigwedge_R^{n} M$, and $\overline{H}^i(R, \varphi) \subseteq \Gamma_R^n M$ denote the corresponding images. Then we obtain finite filtrations
	\begin{align*}
		&0 =  \overline{F}^{-1}(R, \varphi) \subseteq \overline{F}^{0}(R, \varphi) \subseteq \cdots \subseteq  \overline{F}^{n}(R, \varphi) = S_R^n (M), \\
		&0 =  \overline{G}^{-1}(R, \varphi) \subseteq \overline{G}^{0}(R, \varphi) \subseteq \cdots \subseteq  \overline{G}^{n}(R, \varphi) = \bigwedge\nolimits_R^n (M), \\
		&0 =  \overline{H}^{-1}(R, \varphi) \subseteq \overline{H}^{0}(R, \varphi) \subseteq \cdots \subseteq  \overline{H}^{n}(R, \varphi) = \Gamma_R^n (M),
	\end{align*}
respectively, such that the induced maps to the respective successive quotients
	\begin{align*}
	S_R^{n-i}M' \otimes S_R^{i} M'' &\to \overline{F}^i(R, \varphi)/ \overline{F}^{i-1}(R, \varphi),\\
	 \bigwedge\nolimits_R^{n-i} M' \otimes \bigwedge\nolimits_R^{i} M'' & \to \overline{G}^i(R, \varphi)/ \overline{G}^{i-1}(R, \varphi), \\ 
	 \Gamma_R^{n-i} M' \otimes \Gamma_R^{i} M'' &\to  \overline{H}^i(R, \varphi)/ \overline{H}^{i-1}(R, \varphi)
	\end{align*}
are isomorphisms. In particular, for each $0 \le i \le n$, we obtain canonical fiber sequences 
	\begin{align*}
		&\alpha^i(R, \varphi) \colon \qquad \overline{F}^{i-1}(R, \varphi) \to \overline{F}^i(R, \varphi) \to (S_R^{n-i}M') \otimes_R (S_R^{i} M''), \\
		&\beta^i(R, \varphi) \colon \qquad \overline{G}^{i-1}(R, \varphi) \to \overline{G}^i(R, \varphi) \to (\bigwedge\nolimits_R^{n-i}M') \otimes_R (\bigwedge\nolimits_R^{i} M''),\\
		&\gamma^i(R, \varphi) \colon \qquad \overline{H}^{i-1}(R, \varphi) \to \overline{H}^i(R, \varphi) \to (\Gamma_R^{n-i}M') \otimes_R (\Gamma_R^{i} M'').
	\end{align*}
Let $\shE_0, \shE, \shD$ be the $\infty$-categories defined in  \S \ref{sec:univ:fib:seq}, Notation \ref{notation:EandD}, then by virtue of Proposition \ref{prop:nonab:derived}, the above functors extend essentially uniquely to functors
	$$\shE \to \Fun(\Delta^n, \SCRModcn) \times_{\Fun(\Delta^n, \CAlgDelta)} \CAlgDelta,$$
	\begin{align*}
		&(A, \varphi \colon M' \to M) \mapsto (A, F^0(A, \varphi) \xrightarrow{d_1} F^1(A, \varphi) \to \cdots \to F^n(A, \varphi)) \\
		&(A, \varphi \colon M' \to M) \mapsto (A, G^0(A, \varphi) \xrightarrow{e_1} G^1(A, \varphi) \to \cdots \to G^n(A, \varphi)) \\
		&(A, \varphi \colon M' \to M) \mapsto (A, H^0(A, \varphi) \xrightarrow{f_1} H^1(A, \varphi) \to \cdots \to H^n(A, \varphi))
	\end{align*}
and functors $\shE \to \shD$,
		\begin{align*}
		&(A, \varphi \colon M' \to M) \mapsto (A, \alpha^i(A, \varphi) \colon F^{i-1}(A, \varphi) \xrightarrow{d_{i}} F^i(A, \varphi) \to (\Sym_A^{n-i}M') \otimes_A (\Sym_A^{i} M'')), \\
		&(A, \varphi \colon M' \to M) \mapsto (A, \beta^i(A, \varphi) \colon  G^{i-1}(A, \varphi) \xrightarrow{e_{i}} G^i(A, \varphi) \to (\bigwedge\nolimits_A^{n-i}M' )\otimes_A (\bigwedge\nolimits_A^{i} M'')), \\
		&(A, \varphi \colon M' \to M) \mapsto (A, \gamma^i(A, \varphi) \colon H^{i-1}(A, \varphi) \xrightarrow{f_{i}} H^i(A, \varphi) \to (\Gamma_A^{n-i}M') \otimes_A (\Gamma_A^{i} M'')),
	\end{align*}
where $\alpha_{i,n}(A, \varphi)$, $\beta_{i,n}(A,\varphi)$, and $\gamma_i(A, \varphi)$ are cofiber sequences, and $M''$ is the cofiber of $\rho$. Moreover, we have canonical equivalences:
	\begin{align*}
		&F^0(A, \varphi) = \Sym^n_A M'  & F^n(A, \varphi) = \Sym^n_A M, \\
		&G^0(A, \varphi) = \bigwedge\nolimits^n_A M' & G^n(A, \varphi) =  \bigwedge\nolimits^n_A M, \\
		&H^0(A, \varphi) = \Gamma^n_A M' & H^n(A, \varphi) = \Gamma^n_A M.
	\end{align*}

\begin{example} Let $A$ be a simplicial commutative ring, and $\rho \colon M \to M'$  a morphism of connective $A$-modules such that $\cofib(A, \varphi)$ is a line bundle (that is, a locally free module of rank $1$). Then for each $n \ge 1$, we obtain a fiber sequence:
	$$\beta^1(A, \rho) \colon \qquad \bigwedge\nolimits_A^n M' \xrightarrow{\bigwedge^n \varphi} \bigwedge\nolimits_A^n M \to \left(\bigwedge\nolimits_A^{n-1} M'\right) \otimes_A  \cofib(A, \varphi).$$
\end{example}

\section{Quasi-coherent sheaves and algebras on prestacks}\label{sec:QStk}

In this section, we first review Lurie's machinery in \cite[\S 6.2.1]{SAG}, which allows one to ``globalize" constructions from affine cases to the cases of general prestacks (\S \ref{sec:shvprestacks}). In \S \ref{sec:QCoh}, we review the theory of quasi-coherent sheaves on prestacks. In \S \ref{sec:QCAlgDelta}, we apply the machinery of \S \ref{sec:shvprestacks} to construct the {\em $\infty$-category of quasi-coherent algebras $\QCAlgDelta(X)$} on a prestack $X$, study its basic properties and define {\em relative spectra}. Finally, in \S \ref{sec:sym:prestack} and \S \ref{sec:Sym*:prestack}, we apply the machinery of \S \ref{sec:shvprestacks} to globalize the constructions and results of \S \ref{sec:SymAlg} on symmetric, exterior, and divided powers, and symmetric algebras, respectively, to the case of prestacks.

\subsection{$\infty$-categories of sheaves of objects on prestacks} \label{sec:shvprestacks}
Let $q \colon \shC \to \CAlgDelta$ be a coCartesian fibration. In this subsection, we will apply the machinery developed by Lurie in \cite[\S 6.2.1]{SAG} to construct an $\infty$-category $\underline{\shC}(X)$ of ``sheaf of objects of $\shC$ on a prestack $X$". Recall that a prestack $X$ is functor from the $\infty$-category of simplicial commutative rings $\CAlgDelta$ to the $\infty$-category of spaces $\shS$. By virtue of \cite[Proposition 3.3.2.5]{HTT}, a prestack $X$ is classified by a left fibration which we denote by $\int_{\CAlgDelta} X \to \CAlgDelta$. By virtue of \cite[Construction 6.2.1.7]{SAG},  the coCartesian fibration $q \colon \shC \to \CAlgDelta$ determines a functor 
	$$\Phi(q) \colon \Fun(\CAlgDelta, \shS)^{\op} \to \widehat{\Cat}_\infty$$ 
which carries a prestack $X \in \Fun(\CAlgDelta, \shS)$ to the $\infty$-category
	$$\underline{\shC}(X) = \Fun_{/ \CAlgDelta}^{\rm CCart}(\int_{\CAlgDelta} X, \shC),$$
where $\Fun_{/ \shE}^{\rm CCart}(\int_{\CAlgDelta} X, \shC)$ denotes the $\infty$-category of functors from $\int_{\CAlgDelta} X$ to $\shC$ which commute with their projections to $\CAlgDelta$ and carry each edge of $\int_{\CAlgDelta} X$ to a coCartesian edge of $q \colon \shC \to \CAlgDelta$ (see \cite[Definition 6.2.1.1]{SAG}). 

\begin{remark} \label{rmk:shvprestacks:description} By virtue of \cite[Remark 6.2.1.8]{SAG}, more informally, $\underline{\shC}(X)$ can be described as the limit $\varprojlim_{A, \eta \in X(A)} \shC_{A}$, where $\shC_{A}$ is the fiber of $q$ over $A \in \CAlgDelta$. In other words, we can think of an object of $C \in \underline{\shC}(X)$ as a functor which assigns each $A \in \CAlgDelta$ and $A$-point $\eta \in X(A)$ an object $C_\eta \in \shC$. The assignment $C_\eta$ depend functorially on $A$ in the following strong sense: if $\phi \colon A \to B$ is a map of simplicial commutative rings, and $\eta' \in X(B)$ is the image of $\eta$ under $\phi$, then we obtain a q-coCartesian morphism $C_{\eta} \to C_{\eta'}$ in $\shC$.
\end{remark}

\begin{remark} \label{rmk:shvprestacks:properties} The above construction of $\underline{\shC}(X)$ has the following properties:
\begin{enumerate}
	\item   \label{rmk:shvprestacks:properties-1}
	A map $f \colon X \to Y$ of prestacks (i.e., a natural transformation from $X$ to $Y$) by construction determines a pullback functor $f^*_{\shC} \colon \underline{\shC}(Y) \to \underline{\shC}(X)$;
	\item  \label{rmk:shvprestacks:properties-2}
	If $F \colon \shC \to \shC'$ is a morphism of coCartesian fibrations over $\CAlgDelta$ (that is, $p' \colon \shC' \to \CAlgDelta$ is another coCartesian fibration, and $F$ is a functor that commutes with their respective projections to $\CAlgDelta$ up to canonically equivalences and carries $p$-coCartesian edges to $q$-coCartesian edges), then $F$ induces a natural transformation from $\Phi(q)$ to $\Phi(q')$. In particular, $F$ determines a functor $F_X \colon \underline{\shC}(X) \to \underline{\shC'}(X)$ for each prestack $X$, and the formation of $F_X$ commutes with base change (that is, for any map of prestacks $f \colon X \to Y$, there is a canonical equivalence of functors $F_Y \circ f_{\shC}^* \simeq f_{\shC'}^* \circ F_X$).
	\item  \label{rmk:shvprestacks:properties-3}
	(\cite[Proposition 6.2.1.9]{SAG}) Let $\chi \colon \CAlgDelta \to \widehat{\Cat}_\infty$ be the functor classified by the coCartesian fibration $q \colon \shC \to \CAlgDelta$, then $\Phi(q) \colon \Fun(\CAlgDelta, \shS)^{\op} \to \widehat{\Cat}_\infty$ is the right Kan extension of $\chi$ along the Yoneda embedding $j \colon \CAlgDelta \to \Fun(\CAlgDelta, \shS)^{\op}$. (Notice by \cite[Lemma 5.1.5.5]{HTT}, this is equivalent to $\Phi(q)$ commutes with small limits, and there is a canonical equivalence $\Phi(q) \circ j \simeq \chi$.) In particular, for any $A \in \CAlgDelta$, there is a canonical equivalence $\underline{\shC}(\Spec A) \simeq \shC_A$.
	\end{enumerate}
\end{remark}

\subsection{Quasi-coherent sheaves on prestacks} \label{sec:QCoh} This subsection reviews the construction and basic properties of the functor $\QCoh$. Many materials of this subsection parallel Lurie's exposition in \cite[\S 6.2-6.4]{SAG}, which studies the functor $\QCoh$ in the spectral setting. 

\begin{construction}[Quasi-coherent sheaves; compare with {\cite[Deﬁnition 6.2.2.1]{SAG}}] Let $q \colon \SCRMod \to \CAlgDelta$ denote the coCartesian fibration which classifies the functor $(A \in \CAlgDelta) \mapsto (\Mod_A \in \widehat{\Cat}_\infty)$. Apply the construction \S \ref{sec:shvprestacks} to $q$, we obtain a functor 
		$$\QCoh \colon  \Fun(\CAlgDelta, \shS)^\op \to \widehat{\Cat}_{\infty},$$
	which carries each $X  \in  \Fun(\CAlgDelta, \shS) $ to the $\infty$-category
		$$\QCoh(X) = \Fun_{/ \CAlgDelta}^{\rm CCart}(\int_{\CAlgDelta} X, \SCRMod).$$
For a prestack $X \in  \Fun(\CAlgDelta, \shS)$, we will refer to $\QCoh(X)$ as the {\em $\infty$-category of quasi-coherent sheaves on $X$} (or {\em $\infty$-category of quasi-coherent $\sO_X$-modules)}, and elements $\sF \in \QCoh(X)$ as {\em quasi-coherent sheaves on $X$} (or {\em quasi-coherent $\sO_X$-modules}).
\end{construction}

\begin{remark}[Compare with {\cite[Remark 6.2.2.7]{SAG}}] By virtue of \cite[Remark 6.2.1.8]{SAG} (reviewed in Remark \ref{rmk:shvprestacks:description}), more informally, we can describe $\QCoh(X)$ as the limit of stable $\infty$-categories $\varprojlim_{R \in \CAlgDelta, \eta \in X(R)} \Mod_R$ (hence $\QCoh(X)$ is a stable $\infty$-category). In other words, we can think of an object $\sF \in  \QCoh(X)$ as an assignment $(R \in \CAlgDelta, \eta \in X(R)) \mapsto (\sF(\eta) \in \Mod_R)$ which depends functorially on $R$ in the following strong sense: if $\phi \colon R \to R'$ is a map of simplicial commutative rings such that $\eta' \in X(R')$ is the image of $\eta$ under $\phi$, then we obtain a $q$-coCartesian morphism $(R, \sF(\eta)) \to (R', \sF(\eta'))$ in $\SCRMod$, which we will view as an equivalence of $R'$-modules $R' \otimes_R \sF(\eta) \xrightarrow{\sim} \sF(\eta')$. 
\end{remark}

\begin{remark} \label{rmk:QCoh:properties} By virtue of Remark \ref{rmk:shvprestacks:properties},  $\QCoh$ satisfies the following properties:
\begin{enumerate}
	\item A map $f \colon X \to Y$ of prestacks determines a pullback functor $f^* \colon \QCoh(Y) \to \QCoh(X)$. Moreover, (since $\QCoh$ preserves all small limits) $\QCoh(X)$ and $ \QCoh(Y)$ admit small colimits and $f^*$ preserves small colimits.
	\item $\QCoh \colon \Fun(\CAlgDelta, \shS)^{\op} \to \widehat{\Cat}_\infty$ is the right Kan extension of the functor $R \mapsto \Mod_R$ along the Yoneda embedding $\CAlgDelta \to \Fun(\CAlgDelta, \shS)^\op$. In particular, for any $A \in \CAlgDelta$, there is a canonical equivalence $\QCoh(\Spec A) \simeq \Mod_A$.
	\end{enumerate}
\end{remark}

\begin{remark}[Symmetric monoidal structure; compare with {\cite[\S 6.2.6]{SAG}}] Since the functor $(R \in \CAlgDelta) \mapsto (\Mod_R \in \widehat{\Cat}_\infty)$ factors as a composition	
	$$\CAlgDelta \xrightarrow{U} \CAlg(\widehat{\Cat}_\infty) \xrightarrow{\theta} \widehat{\Cat}_\infty,$$
		where $U$ assigns each $R \in \CAlgDelta$ to the symmetric monoidal $\infty$-category $\Mod_{R^\circ}^{\otimes}$ of \cite[\S 4.5.3]{HA}. Let $\QCoh^{\otimes} \colon \Fun(\CAlgDelta, \shC) \to  \CAlg(\widehat{\Cat}_\infty) $ be a right Kan extension of $U$, then $\QCoh^{\otimes}$ assigns to each functor $X \in \Fun(\CAlgDelta, \shC)$ a symmetric monoidal $\infty$-category $\QCoh(X)^{\otimes}$ whose underlying $\infty$-category can be identified with $\QCoh(X)$. We will denote the tensor product on $\QCoh(X)$ by $\sF \otimes_{\sO_X} \sF'$ (or simply by $\sF \otimes \sF'$ whenever there is no confusion). More informally, the tensor product on $\QCoh(X)$ is given point-wise by $(\sF \otimes_{\sO_X} \sF')(\eta) \simeq \sF(\eta) \otimes_R \sF'(\eta)$. The {\em unit} of $\QCoh(X)^{\otimes}$ is given by the {\em structure sheaf} $\sO_X$, which assigns to each point $\eta \in X(R)$ the module $R$ (regarded as a module over itself).
\end{remark}

\begin{theorem}[Lipman--Neeman, Lurie]\label{thm:Neeman-Lipman-Lurie} Let $f \colon X \to Y$ be a morphism of prestacks which is quasi-compact, quasi-separated relative derived scheme (or algebraic space). Let $f^* \colon \QCoh(Y) \to \QCoh(X)$ denote the pullback functor.
\begin{enumerate}[leftmargin=*]
	\item \label{thm:Neeman-Lipman-Lurie-1}
	\begin{enumerate}[label=(\theenumi\alph*), ref=\theenumi\emph{\alph*}]
		\item \label{thm:Neeman-Lipman-Lurie-1i}
		The pullback functor $f^*$ admits a right adjoint $f_* \colon \QCoh(X) \to \QCoh(Y)$, called pushforward functor, which preserves small colimits; 
		\item For every pullback diagram of prestacks
			$$
			\begin{tikzcd}
				X' \ar{d}{f'} \ar{r}{g'} & X \ar{d}{f} \\
				Y' \ar{r}{g} & Y,
			\end{tikzcd}
			$$
	the canonical base change transformation (the Beck--Chevalley transformation) $g^* f_* \to f'_* g'^*$ is an equivalence of functors from $\QCoh(X)$ to $\QCoh(Y')$.
		\item \label{thm:Neeman-Lipman-Lurie-1ii}
		For every pair of objects $\sF \in \QCoh(X)$ and $\sG \in \QCoh(Y)$, the canonical map $\sF \otimes f_* \sG \to f_*(f^* \sF \otimes \sG)$ is an equivalence.
	\end{enumerate}
	\item \label{thm:Neeman-Lipman-Lurie-2}
	If we furthermore assume that $f$ is proper, locally almost of finite presentation and locally of finite Tor-amplitude. Then:
		\begin{enumerate}[label=(\theenumi\alph*), ref=\theenumi\emph{\alph*}]
		\item \label{thm:Neeman-Lipman-Lurie-2i}
		The pushforward functor $f_* \colon \QCoh(X) \to \QCoh(Y)$ preserves perfect objects, and admits a right adjoint $f^! \colon \QCoh(Y) \to \QCoh(X)$ which preserves small colimits.
		\item \label{thm:Neeman-Lipman-Lurie-2ii}
		For every pullback diagram of prestacks
			$$
			\begin{tikzcd}
				X' \ar{d}{f'} \ar{r}{g'} & X \ar{d}{f} \\
				Y' \ar{r}{g} & Y,
			\end{tikzcd}
			$$
	the canonical base change transformation $g'^* f^! \to f'^!g'^*$ is an equivalence of functors from $\QCoh(Y)$ to $\QCoh(X')$.
		\item \label{thm:Neeman-Lipman-Lurie-2iii}
		For every pair of objects $\sF, \sG \in \QCoh(Y)$, the canonical map $f^!\sF \otimes f^* \sG \to f^!(\sF \otimes \sG)$ is an equivalence.
		\item \label{thm:Neeman-Lipman-Lurie-2iv}
		The pullback functor $f^*$ admits a left adjoint $f_! \colon \QCoh(X) \to \QCoh(Y)$ which preserves perfect objects and whose restriction to $\Perf(X)$ is given by the formula $f_!(\sF) = (f_* \sF^\vee)^\vee$.
		\item \label{thm:Neeman-Lipman-Lurie-2v}
		For every pullback diagram of prestacks
			$$
			\begin{tikzcd}
				X' \ar{d}{f'} \ar{r}{g'} & X \ar{d}{f} \\
				Y' \ar{r}{g} & Y,
			\end{tikzcd}
			$$
	the canonical base change transformation $f'_! g'^*\to g^*f_!$ is an equivalence of functors from $\QCoh(X)$ to $\QCoh(Y')$.
		\item \label{thm:Neeman-Lipman-Lurie-2vi}
		For every pair of objects $\sF \in \QCoh(X)$ and $\sG \in \QCoh(Y)$, the canonical map $f_!(f^* \sF \otimes \sG) \to \sF \otimes f_! \sG$ is an equivalence.
		\item \label{thm:Neeman-Lipman-Lurie-2vii}
		Let $\omega_f = \omega_{X/Y}$ denote the {\em relative dualizing sheaf}, that is, the object $f^! \sO_Y \in \QCoh(X)$. Then there are canonical natural isomorphisms of functors
				$$f^!(\blank) \simeq f^*(\blank) \otimes \omega_f \quad \text{and} \quad f_!(\blank) \simeq f_*(\blank \otimes \omega_f).$$
		\end{enumerate}
	\item \label{thm:Neeman-Lipman-Lurie-3}
	In the situation of \eqref{thm:Neeman-Lipman-Lurie-2}, if we furthermore assume that the relative dualizing sheaf $\omega_f$ is perfect, then each member of the adjunction sequence $f_! \dashv f^* \dashv f_*\dashv f^!$ preserves perfect objects, almost perfect objects, and locally bounded almost perfect objects (in particular, in the case where $X$ and $Y$ are locally Noetherian, each of the above functors preserves locally bounded coherent sheaves).
\end{enumerate}
\end{theorem}
\begin{proof} In the situation of classical algebraic geometry where $X$, $Y$ are quasi-compact, quasi-separated schemes (and the involved base change are Tor-independent), the assertions are consequences of the results of \cite{LN, Nee, Nee10, Lip}; see also our summary in \cite[Theorem 3.1]{J21}. In the context of spectral algebraic geometry, assertion \eqref{thm:Neeman-Lipman-Lurie-1} is a consequence of \cite[Proposition 6.3.4.1]{SAG} and \cite[Corollary 6.3.4.3]{SAG}, assertions \eqref{thm:Neeman-Lipman-Lurie-2i} through \eqref{thm:Neeman-Lipman-Lurie-2vii} follows from a combination of \cite[Proposition 6.4.2.1, Lemma 6.4.2.2, Proposition 6.4.5.3, and Proposition 6.4.5.4]{SAG}, and assertion \eqref{thm:Neeman-Lipman-Lurie-3} is a direct consequence of \eqref{thm:Neeman-Lipman-Lurie-2}. The same proof strategy of the above results work in our setting with minimal modifications.
\end{proof}

\subsection{Quasi-coherent algebras on prestacks and relative spectra} \label{sec:QCAlgDelta}  In this subsection, we apply the construction of \cite[\S6.2.1, Construction 6.2.1.7]{SAG} (which we have reviewed in \S \ref{sec:shvprestacks}) to study the $\infty$-categories of quasi-coherent algebras on prestacks. 

\begin{construction}[Quasi-coherent algebras] Let $\mathrm{ev}_{\{0\}} \colon \Fun(\Delta^1, \CAlgDelta) \to \CAlgDelta$ denote the evaluation map $(A \to B) \mapsto A$. Then $\mathrm{ev}_{\{0\}}$ is the coCartesian fibration which classifies the functor $(R \in \CAlgDelta) \mapsto (\CAlgDelta_{R} \in \widehat{\Cat}_\infty)$ (where for a map of simplicial commutative rings $R \to R'$, the induced functor $\CAlgDelta_R \to \CAlgDelta_{R'}$ is the base change map $A \mapsto A \otimes_R R'$ of \cite[Notation 25.1.4.4]{SAG}; In other words, an edge $(R, A) \to (R', A')$ is a $\mathrm{ev}_{\{0\}}$-coCartesian edge if and only if the induced map $A \otimes_R R' \to A'$ is an equivalence of $R'$-algebras). Apply the construction \S \ref{sec:shvprestacks} to the coCartesian fibration $\mathrm{ev}_{\{0\}}$, we obtain a functor 
		$$\QCAlgDelta \colon  \Fun(\CAlgDelta, \shS)^{\op} \to \widehat{\Cat}_{\infty},$$
	which carries each $X  \in  \Fun(\CAlgDelta, \shS)$ to an $\infty$-category
		$$\QCAlgDelta(X) = \Fun_{/ \CAlgDelta}^{\rm CCart}\big(\int_{\CAlgDelta} X, \Fun(\Delta^1, \CAlgDelta)\big).$$
For a prestack $X \in  \Fun(\CAlgDelta, \shS)$, we will refer to $\CAlgDelta(X)$ as the {\em $\infty$-category of quasi-coherent algebras over $X$} (or {\em $\infty$-category of quasi-coherent $\sO_X$-algebras}), and elements $\sF \in  \QCAlgDelta(X)$ as {\em quasi-coherent algebras over $X$} (or {\em quasi-coherent $\sO_X$-algebras}).	
\end{construction}

\begin{remark} By virtue of Remark \ref{rmk:shvprestacks:description}, more informally, we can describe $ \QCAlgDelta(X)$ as the limit $\varprojlim_{R \in \CAlgDelta, \eta \in X(R)} \CAlgDelta_R$. In other words, we can think of an object $\sA \in  \QCAlgDelta(X)$ as an assignment $(R \in \CAlgDelta, \eta \in X(R)) \mapsto (\sA(\eta) \in \CAlgDelta_R)$ which depends functorially on $R$ in the following strong sense: if $\phi \colon R \to R'$ is a map of simplicial commutative rings such that $\eta' \in X(R')$ is the image of $\eta$ under $\phi$, then we obtain a $\mathrm{ev}_{\{0\}}$-coCartesian morphism $(R, \sA(\eta)) \to (R', \sA(\eta'))$ in $\Fun(\Delta^1, \CAlgDelta)$, which we will view as an equivalence of simplicial $R'$-algebras $R' \otimes_R \sA(\eta) \xrightarrow{\sim} \sA(\eta')$. 
\end{remark}

\begin{remark} \label{rmk:QCAlgDelta:properties} By virtue of Remark \ref{rmk:shvprestacks:properties},  $\QCAlgDelta$ satisfies the following properties:
\begin{enumerate}
	\item A map $f \colon X \to Y$ of prestacks determines a pullback functor $f^*_{\rm alg} \colon \QCAlgDelta(Y) \to \QCAlgDelta(X)$. Moreover, since $\QCAlgDelta$ preserves all small limits, we obtain that $\QCAlgDelta(X)$ and $ \QCAlgDelta(Y)$ admit small colimits and $f^*_{\rm alg}$ preserves small colimits.
	\item $\QCAlgDelta \colon \Fun(\CAlgDelta, \shS)^{\op} \to \widehat{\Cat}_\infty$ is the right Kan extension of the functor $R \mapsto \CAlgDelta_R$ along the Yoneda embedding $\CAlgDelta \to \Fun(\CAlgDelta, \shS)^\op$. In particular, for any $A \in \CAlgDelta$, there is a canonical equivalence $\QCAlgDelta(\Spec A) \simeq \CAlgDelta_A$.
	\end{enumerate}
\end{remark}

\begin{proposition}[{Compare with \cite[Proposition 6.2.3.1]{SAG}}] \label{prop:QCAlg:fppf} Let $\Shv_{\rm fpqc}$ the full subcategory of $\Fun(\CAlgDelta, \shS)$ spanned by those functors which are sheaves with respect to the fpqc topology, and let $L \colon \Fun(\CAlgDelta, \shS) \to \Shv_{\rm fpqc}$ denote a left adjoint to the inclusion functor.
\begin{enumerate}
	\item \label{prop:QCAlg:fppf-1} The canonical maps $\QCoh \circ L \to \QCoh$ and $\QCAlgDelta \circ L \to \QCAlgDelta$ are equivalences of functors from $ \Fun(\CAlgDelta, \shS)^{\op}$ to $\widehat{\Cat}_\infty$. 
	\item \label{prop:QCAlg:fppf-2} The restrictions $\QCoh|_{\Shv_{\rm fpqc}^{\op}} \colon \Shv_{\rm fpqc}^{\op} \to \widehat{\Cat}_\infty$, and $\QCAlgDelta|_{\Shv_{\rm fpqc}^{\op}} \colon \Shv_{\rm fpqc}^{\op} \to \widehat{\Cat}_\infty$ preserve small limits, respectively.
\end{enumerate}
\end{proposition}
\begin{proof} The proof for the functor $\QCoh$ is completely analogous to the spectral case \cite[Proposition 6.2.3.1]{SAG}; we only need to prove the assertions for $\QCAlgDelta$. By our construction, $\QCAlgDelta$ preserves limits (see \cite[Lemma 6.2.1.12]{SAG}). By virtue of \cite[Proposition 1.3.1.7]{SAG}, to prove the assertion \eqref{prop:QCAlg:fppf-1}, we only need to show that the functor $(R \in \CAlgDelta) \mapsto (\CAlgDelta_R \in \widehat{\Cat}_\infty)$ satisfies flat descent: in other words, if $R \to R_\bullet$ is the \v{C}ech nerve of a faithfully flat map $R \to R_0$ of simplicial commutative rings, then the canonical map
	$$\theta \colon \CAlgDelta_{R} \to \varprojlim_{[n] \in \bDelta_s} \CAlgDelta_{R_n}$$
is an equivalence of $\infty$-categories, where $\bDelta_s$ is the subcategory of $\Delta$ whose morphisms are injective maps $[m] \to [n]$ of ﬁnite linearly ordered sets. Invoking \cite[Proposition 5.2.2.36]{HA}, it will suﬃce to verify the following assertions: 
	\begin{enumerate}
		\item[$(i)$] The functor $\theta$ is conservative. That is, if $\alpha \colon A \to B$ is a morphism in $\CAlgDelta_R$ such that each of the induced maps $\alpha_n \colon R_n \otimes_R A \to R_n \otimes_R B$ is an equivalence, then $\alpha$ is an equivalence. This follows from the fully faithfulness of $R \to R_0$.
		\item[$(ii)$]  Suppose we are given objects $A_n \in \CAlgDelta_{R_n}$ such that  for each injective map $[m] \to [n]$ in $\bDelta$, the (flat) map $R_m \to R_n$ induces an equivalence $A_n \simeq R_n \otimes_{R_m} A_m$,  then there exists a $R$-algebra $A = \varprojlim A_\bullet$ such that each of the natural maps $R_n \otimes_R A \to A_n$ is an equivalence. Since each map $A_m \to A_n$ is flat, the existence of $A = \varprojlim A_\bullet$ follows from fact that the identify functor $\CAlgDelta \to \CAlgDelta$ is a fpqc sheaf (see \cite[Theorem D.6.3.5]{SAG} for the corresponding result in spectral setting; The analogous statement implies that the fpqc topology on the $\infty$-category $\CAlg^{\Delta, \op}$ is subcanonical).
	\end{enumerate}
The assertion \eqref{prop:QCAlg:fppf-2} is a consequence of assertion \eqref{prop:QCAlg:fppf-1}.
\end{proof}

\begin{proposition}[Affine morphisms]\label{prop:affine} Let $f \colon Y \to X$ be an affine morphism in  $\Fun(\CAlgDelta, \shS)$.
	\begin{enumerate}
		\item \label{prop:affine-1} The pullback functor $f^*_{\rm alg} \colon \QCAlgDelta(X) \to \QCAlgDelta(Y)$ admits a right adjoint 
			$$f_{*}^{\rm alg} \colon  \QCAlgDelta(Y) \to \QCAlgDelta(X).$$
			Moreover, for each map $\Spec R \to X$, let $Y_R = Y \times_X \Spec R$ denote the fiber product, then the induced pullback diagram of $\infty$-categories
		\begin{equation*}
				\begin{tikzcd}
					\QCAlgDelta(X) \ar{r}{f^*_{\rm alg}} \ar{d}& \QCAlgDelta(Y) \ar{d}\\
					\QCAlgDelta(\Spec R) \ar{r}{(f_{R})_{\mathrm{alg}}^*}&  \QCAlgDelta(Y_R)
				\end{tikzcd}
			\end{equation*}
		is right adjointable.
		\item \label{prop:affine-2} 
		(Compare with \cite[Proposition 6.3.4.5]{SAG}) We let $\mathrm{Aff}_{/X}$ denote the full subcategory of $\Fun(\CAlgDelta, \shS)_{/X}$ spanned by affine morphisms $f \colon Y \to X$. Then the construction  
		$$((f \colon Y \to X) \in \mathrm{Aff}_{/X})  \mapsto (\sA(Y/X) = f_{*}^{\rm alg} (\sO_Y) \in \QCAlgDelta(X))$$ 
induces an equivalence of $\infty$-categories $\mathrm{Aff}_{/X} \simeq  \QCAlgDelta(X)^{\op}$.
		\item \label{prop:affine-3} 
		Let $f \colon Y \to X$, $g \colon Z \to X$ be morphisms in $\Fun(\CAlgDelta, \shS)$. Assume $f$ is affine and set $\sA = f_{*}^{\rm alg} (\sO_Y) \in \QCAlgDelta(X)$. Then the canonical map 
				$$\theta \colon \Map_{\Fun(\CAlgDelta, \, \shS)_{/X}}(Z, \Spec \sA) \to \Map_{\QCAlgDelta(X)}(g^*_{\rm alg} (\sA), \sO_Z)$$ 
is a homotopy equivalence.
	\end{enumerate}
\end{proposition}

\begin{proof} 
To prove the assertion \eqref{prop:affine-1}, we use a similar strategy as the proof of Proposition \ref{prop:Sym*:prestack}. Write $X$ as a colimit of a diagram $q \colon K \to \Fun(\CAlgDelta, \shS)$ where $q(s) \simeq \Spec A_s$ is affine. Then by our assumption $Y_s = \Spec A_s$, where $A_s$ is an $R_s$-algebra. For each edge $e \colon s \to s'$ in $K$, since $A_s \simeq A_{s'} \otimes_{R_{s'}} R_s$, the associated pullback diagram of $\infty$-categories 
			\begin{equation*}
				\begin{tikzcd}
						\QCAlgDelta(\Spec R_{s'}) \simeq \CAlgDelta_{R_{s'}}\ar{r}\ar{d}{q(e)^*_{\rm alg}}&    \QCAlgDelta(\Spec A_{s'}) \simeq \CAlgDelta_{A_{s'}} \ar{d}{q(e)^*_{\rm alg}}\\
				\QCAlgDelta(\Spec R_s) \simeq \CAlgDelta_{R_s} \ar{r}  & \QCAlgDelta(\Spec A_s)^\cn \simeq \CAlgDelta_{A_s}
				\end{tikzcd}
			\end{equation*}
	is evidently right adjointable. Then \cite[Corollary 4.7.4.18]{HA} implies the assertion \eqref{prop:affine-1}.

For assertion \eqref{prop:affine-2}, writing X as a colimit of representable functors again, we can reduce to the case where $X = \Spec R$. Each affine morphism $Y \to X$ has the form $Y = \Spec A$, where $R \to A$ is a morphism of simplicial commutative rings, and $f_{*}^{\rm alg} (A) = A \in \CAlgDelta_R$ by virtue of \eqref{prop:affine-1}. Then $\lambda$ is the evident homotopy inverse of the functor $\Spec \colon \CAlgDelta_R \to {\rm Aff}_{/\Spec R}^{\op}$.

To prove the assertion \eqref{prop:affine-3}, we use the same strategy as the proof of \cite[Lemma 9.2.1.2]{SAG}. When regarded as functors of $Z$, both the domain and codomain of $\theta$ carry colimits of functors to limits of spaces. Writing $Z$ as a colimit of corepresentable functors, we may reduce to the case $Z = \Spec R$. Replacing $X$ by $Z$ and $Y$ by $Y \times_X Z$, we may further reduce to the affine case where $Y = \Spec A$, $X=Z = \Spec R$, and $R \to A$ is a map of simplicial commutative rings. In this case, the desired result follows from assertion \eqref{prop:affine-2}.
\end{proof}

\begin{definition}[Relative Spectrum; compare with {\cite[Construction 2.5.1.3]{SAG}}] \label{def:relative:spec} Let $X \in \Fun(\CAlgDelta, \shS)$, and let $\lambda$ denote the homotopy equivalence of Proposition \ref{prop:affine} \eqref{prop:affine-2}. We let 
	$$\Spec_X \colon \QCAlgDelta(X) \to \mathrm{Aff}_{/X} $$
denote a fixed a homotopy inverse to the equivalence $\lambda$. Given a quasi-coherent algebra $\sA$ over $X$, we let $\pr \colon \Spec_X \sA \to X$ denote the corresponding affine morphism, and we will refer $\Spec_X \sA$ as the {\em relative spectrum of $\sA$ over $X$}. 
\end{definition}

Let $\pr \colon \Spec_X \sA \to X$ denote the natural projection, then by construction:
	\begin{enumerate}	
		\item There is a canonical equivalence $\sA \xrightarrow{\sim} \pr_*^{\rm alg}(\sO_{\Spec_X \sA}) \in \QCAlgDelta(X)$.
		\item For any base change $\phi \colon X' \to X$, there is a canonical identification 
			$$\Spec_{X'}(\phi^* \sA) = (\Spec_X \sA) \times_X X'.$$
	 	\end{enumerate}

\begin{remark}[Non-affine pushforwards of quasi-coherent algebras] In general, the pullback functor $f^*_{\rm alg}$ admits a right adjoint $f_{*}^{\rm alg} \colon \QCAlgDelta(Y) \to \QCAlgDelta(X)$ only under mild assumptions. For example, if $X = \Spec R$, then $f_{*}^{\rm alg} \colon \QCAlgDelta(Y) \to \CAlgDelta_A$ is the {\em simplicial global section functor}: regarding each object $\sA \in \QCAlgDelta(Y)$ as a functorial assignment $\eta \colon \Spec A \to X$, $\sA(\eta) \in \CAlgDelta_{A}$, then $f_{*}^{\rm alg}(\sA) = \varprojlim_{\eta} \sA(\eta) \in \CAlgDelta_{R}$. Beware that if $f$ is not affine, then generally the underlying sheaf of $f_{*}^{\rm alg}(\sA)$ is {\em not} equivalent to the pushforward $f_*(\sA)$ of the underlying sheaves, even in the case $\sA = \sO_Y$ (for example, $f_{*}^{\rm alg}(\sA)$ is always connective, but $f_*(\sA)$ might not). It is worth studying the functor $f^*_{\rm alg}$ for general non-affine morphisms, but we will not pursue this direction further in this paper.
\end{remark}

\subsection{Derived symmetric powers and universal fiber sequences on prestacks} \label{sec:sym:prestack}
In this subsection, we apply the construction of \S \ref{sec:shvprestacks} to globalize the constructions and results of \S \ref{sec:univ:fib:seq} and obtain universal fiber sequences associated to symmetric powers on prestacks. 

\begin{construction}\label{constr:symwedgegamma:prestack} Let $n \ge 0$ be an integer, and let $\Sym_A^n(M)$, $\bigwedge_A^n(M)$ and $\Gamma_A^n(M)$ be the derived symmetric powers, derived exterior power and derived divided powers, respectively, constructed in Construction \ref{constr:sym.wedge.gamma}. By virtue of \cite[Proposition 25.2.3.1]{SAG}, each of three functors $\Sym^n, \bigwedge^n, \Gamma^n \colon \SCRModcn \to \SCRModcn$ sends $q$-coCartesian edges to $q$-coCartesian morphisms. Therefore we could apply Remark \ref{rmk:shvprestacks:properties}\eqref{rmk:shvprestacks:properties-2} to these functors and obtain natural transformations of functors
	$$\Sym^n, \bigwedge\nolimits^n, \Gamma^n \colon \QCoh^\cn \to \QCoh^\cn$$
from $\QCoh^\cn \colon \Fun(\CAlgDelta, \shS)^\op \to \widehat{\Cat}_{\infty}$ to itself. Consequently, for any given prestack $X \in \Fun(\CAlgDelta, \shS)$ and a connective quasi-coherent sheaf $\sF$ on $X$, we obtain connective connective quasi-coherent sheaves $\Sym_X^n(\sF)\in \QCoh(X)^\cn$, $\bigwedge\nolimits_X^n(\sF) \in \QCoh(X)^\cn$ and $\Gamma_X^n(\sF) \in \QCoh(X)^\cn$, respectively. We will refer to the connective quasi-coherent sheaves $\Sym_X^n(\sF)$, $\bigwedge\nolimits_X^n(\sF)$, and $\Gamma_X^n(\sF)$ as the {\em derived $n$th symmetric powers}, {\em derived $n$th exterior powers}, and {\em derived $n$th divided powers} of $\sF$ over $X$, respectively. 
\end{construction}
\begin{remark}
	\begin{enumerate}
		\item The formations of $\Sym_X^n(\sF)$, $\bigwedge\nolimits_X^n (\sF)$ and $\Gamma_X^n(\sF)$ commute with base change of prestacks. In other words, for any map of prestacks $f \colon Y \to X$ and $\sF \in \QCoh(X)^\cn$, there are canonical equivalences of quasi-coherent $\sO_Y$-modules $f^* \Sym_X^n(\sF) \simeq \Sym_Y^n (f^* \sF)$, $f^* \bigwedge_X^n(\sF) \simeq \bigwedge_Y^n (f^* \sF)$, and $f^* \Gamma_X^n(\sF) \simeq \Gamma_Y^n (f^* \sF)$.
		\item Remark \ref{rmk:shvprestacks:properties} \eqref{rmk:shvprestacks:properties-3} implies that in the representable case $X= \Spec A$, there are canonical equivalences of functors $\Sym_{\Spec A}^n \simeq \Sym^n_A$, $\bigwedge_{\Spec A}^n \simeq \bigwedge^n_A$ and $\Gamma^n_{\Spec A} \simeq \Gamma^n_A$.
	\end{enumerate}
\end{remark}

\begin{construction}\label{constr:EandD:Sym:prestacks} 
Let $\shE$ and $\shD$ be the $\infty$-categories defined in \S \ref{sec:univ:fib:seq}. In particular, objects of $\shE$ are pairs $(A , \rho \colon M' \to M)$ where $A \in \CAlgDelta$ and $\rho$ is morphism of connective $A$-modules, and objects of $\shD$ are pairs $(A, \alpha)$ where $A \in \CAlgDelta$ and $\alpha$ is a fiber sequence of connective $A$-modules. We apply the construction of \S \ref{sec:shvprestacks} to the coCartesian fibrations $q_\shE \colon \shE \to \CAlgDelta$ and $q_\shD \colon \shD \to \CAlgDelta$, then we obtain functors 
	$$\Phi(q_\shE), \Phi(q_\shD) \colon \Fun(\CAlgDelta, \shS)^{\op} \to \widehat{\Cat}_\infty$$ 
which carry a prestack $X \in \Fun(\CAlgDelta, \shS)$ to the respective $\infty$-categories
	$$\underline{\shE}(X) = \Fun_{/ \CAlgDelta}^{\rm CCart}\big(\int_{\CAlgDelta} X, \shE\big), \qquad \underline{\shD}(X) = \Fun_{/ \CAlgDelta}^{\rm CCart}\big(\int_{\CAlgDelta} X, \shD\big).$$
Unwinding the definition, we obtain that the objects of $\underline{\shE}(X)$ are morphisms of connective quasi-coherent sheaves $\rho \colon \sM' \to \sM$ on $X$, and objects of  $\underline{\shD}(X)$ are fiber sequences $\alpha \colon \sM' \to \sM \to \sM''$ of connective quasi-coherent sheaves on $X$.
\end{construction}

\begin{construction} \label{constr:Postnikov:Sym:prestacks} Let $n \ge 1$ be a fixed integer. By virtue of Remarks \ref{rmk:left:Sym:bc} and \ref{rmk:right:Sym:bc}, we can apply Remark \ref{rmk:shvprestacks:properties} \eqref{rmk:shvprestacks:properties-2} to the functors $(A, \rho) \in \shE \mapsto (A, d_{i+1,n}(A, \rho)) \in \shE$, $(A, \rho)\in \shE \mapsto (A, \alpha_{i,n}(A, \rho)) \in \shD$ and $(A, \rho)\in \shE \mapsto (A,\beta_{i,n}(A, \rho)) \in \shD$ (where $0 \le i \le n-1$) of Constructions \ref{constr:left:Sym} and \ref{constr:right:Sym}, and obtain natural transformations 
	$$d_{i+1,n} \colon \Phi(q_{\shE}) \to \Phi(q_{\shE}), \qquad \alpha_i, \beta_i \colon \Phi(q_{\shE}) \to \Phi(q_{\shD}).$$
In other words, for any prestack $X \in \Fun(\CAlgDelta, \shS)$, we have functors
	\begin{align*}
	& \underline{\shE}(X) \to \underline{\shE}(X), &\quad (\rho \colon \sM' \to \sM) \mapsto d_{i+1,n}(X,\rho);\\
	& \underline{\shE}(X) \to \underline{\shD}(X), & \quad (\rho \colon \sM' \to \sM) \mapsto \alpha_{i,n}(X, \rho);\\
	&\underline{\shE}(X) \to \underline{\shD}(X), & \quad (\rho \colon \sM' \to \sM) \mapsto \beta_{i,n}(X, \rho).
	\end{align*}
Here, $d_{i+1,n}(X, \rho)$ is a morphism of connective quasi-coherent sheaves on $X$:
		$$d_{i+1,n}(X, \rho) \colon \quad (\bigwedge\nolimits^{i+1}_X \sM') \otimes_{\sO_X} (\Sym_X^{n-i-1} \sM) \to (\bigwedge\nolimits^{i}_X \sM') \otimes_{\sO_X} (\Sym_X^{n-i} \sM)$$
such that there are canonical equivalences $d_{i+1,n}(X, \rho) \circ d_{i,n}(X,\rho) \simeq 0$; $\alpha_{i,n}(X,\rho)$ and $\beta_{i,n}(X, \rho)$ are fiber sequences of connective quasi-coherent sheaves on $X$:
	\begin{align*}
		&\alpha_{i,n}(X, \rho) \colon & \sF_{i+1,n}(X, \rho) \to (\bigwedge\nolimits^i_X \sM' ) \otimes_{\sO_X} (\Sym_X^{n-i} \sM) \to \sF_{i,n}(X, \rho), \\
		&\beta_{i,n}(X, \rho) \colon &  \Sigma^{i} (\bigwedge\nolimits_X^{i+1} \sM') \otimes_{\sO_X} (\Sym_X^{n-i-1} \sM) \to \sG_{i,n}(X, \rho) \to \sG_{i+1,n}(X, \rho)
	\end{align*}
such that the composition of second edge of $\alpha_{i+1,n}(X,\rho)$ and the first edge of $\alpha_{i,n}(X, \rho)$:
	$$(\bigwedge\nolimits^{i+1}_X \sM' ) \otimes_{\sO_X} (\Sym_X^{n-i+1} \sM) \to \sF_{i+1,n}(X, \rho) \to (\bigwedge\nolimits^{i}_X \sM' )$$
is canonically equivalent to $d_{i+1,n}(X, \rho)$, and the composition of first edge of $\beta_{i,n}(X,\rho)$ and the connecting edge of $\beta_{i-1,n}(X, \rho)$:
	$$ \Sigma^{i} (\bigwedge\nolimits_X^{i+1} \sM') \otimes_{\sO_X} (\Sym_X^{n-i-1} \sM) \to \sG_{i,n}(X, \rho)  \to  \Sigma^{i} (\bigwedge\nolimits_X^{i} \sM') \otimes_{\sO_X} (\Sym_X^{n-i} \sM) $$
is canonically equivalent to $\Sigma^i \circ d_{i+1,n}(X, \rho)$. By construction, the formations of $d_{i,n}(X, \rho)$, $\alpha_{i,n}(X,\rho)$, and $\beta_{i,n}(X, \rho)$ commute with sifted colimits of $\rho$ and arbitary base change of prestacks $X' \to X$. Furthermore, there are canonical equivalences:
	\begin{align*}
	&\sF_{0,n}(X, \rho \colon \sM' \to \sM) \simeq \Sym_X^n (\cofib(\rho)), \qquad \sF_{n,n}(X, \rho \colon \sM' \to \sM) \simeq \bigwedge\nolimits_X^n \sM'; \\
	& \sG_{0,n}(X, \rho \colon \sM' \to \sM) \simeq \Sym_X^n \sM, \qquad \sG_{n,n}(X, \rho \colon \sM' \to \sM) \simeq \Sym_X^n (\cofib(\rho)).
	\end{align*}
Moreover, for any integer $n' \ge 0$, there are canonical morphisms of fiber sequences 
	$$\alpha_{i,n}(X, \rho) \otimes_{\sO_X} \Sym_X^{n'} \sM  \to 
\alpha_{i,n+n'}(X, \rho), \qquad   \beta_{i,n}(X, \rho) \otimes_{\sO_X} \Sym_X^{n'} \sM  \to \beta_{i,n+n'}(X, \rho)$$
for which the induced maps on the middle and initial vertices, respectively, coincide with the maps induced by the multiplication maps $\Sym_X^{n-i} \sM \otimes_{\sO_X} \Sym_X^{n'} \sM \to \Sym_X^{n+n'-i} \sM$ and $\Sym_X^{n-i-1} \sM \otimes_{\sO_X} \Sym_X^{n'} \sM \to \Sym_X^{n+n'-i-1} \sM$ of Remark \ref{rmk:CAlgofSym:prestack}, respectively.
\end{construction}

\subsection{Derived symmetric algebras over prestacks} \label{sec:Sym*:prestack} 
\begin{proposition}[Derived Symmetric Algebras]\label{prop:Sym*:prestack} Let $X \in \Fun(\CAlgDelta, \shS)$ be a prestack.
	\begin{enumerate}
		\item \label{prop:Sym*:prestack-1} The forgetful functor $\mathrm{forg} \colon \QCAlgDelta(X) \to \QCoh(X)^{\cn}$ admits a left adjoint functor $\Sym_X^* \colon \QCoh(X)^\cn \to \QCAlgDelta(X)$.  In other words, for all $\sF, \sF' \in \QCoh(X)^\cn$, the composition with the natural morphism quasi-coherent sheaves $\eta_\sF \colon \sF \to \Sym_X^*(\sF)$ induces a homotopy equivalence
 	$$\Map_{\QCAlgDelta(X)}(\Sym_X^* (\sF), \sF') \simeq \Map_{\QCoh(X)^\cn}(\sF, \sF').$$
		\item \label{prop:Sym*:prestack-2} For every map of prestacks $f \colon Y \to X$ in $\Fun(\CAlgDelta, \shS)$, the associated diagram of $\infty$-categories 
			\begin{equation*}
				\begin{tikzcd}
					\QCAlgDelta(X) \ar{r}{\mathrm{forg}} \ar{d}{f^*_{\rm alg}}& \QCoh(X)^\cn \ar{d}{f^*}\\
					\QCAlgDelta(Y) \ar{r}{\mathrm{forg}}&  \QCoh(Y)^\cn
				\end{tikzcd}
			\end{equation*}
		is left adjointable. In other words, the Beck--Chevalley transformation $\lambda_f \colon \Sym_Y^* \circ f^*_{\rm alg} \to f^* \circ \Sym_X^*$ is an equivalence of functors from $\QCoh(X)^\cn$ to $\QCAlgDelta(Y)$.
	\end{enumerate}
\end{proposition}

\begin{proof} We use the same strategy as Lurie's proof of \cite[Proposition 6.3.4.1]{SAG}. Write $X$ as a colimit of a diagram $q \colon K \to \Fun(\CAlgDelta, \shS)$ where $q(s) \simeq \Spec A_s$ is affine, $s \in K$, $K$ is a simplicial set. Then for each edge $e \colon s \to s'$ in $K$, the associated diagram of $\infty$-categories 
			\begin{equation*}
				\begin{tikzcd}
						\QCAlgDelta(\Spec A_{s'}) \simeq \CAlgDelta_{A_{s'}}\ar{r}{\mathrm{forg}}\ar{d}{q(e)^*_{\rm alg}}&    \QCoh(\Spec A_{s'})^\cn \simeq \Mod_{A_{s'}}^\cn\ar{d}{q(e)^*}\\
				\QCAlgDelta(\Spec A_s) \simeq \CAlgDelta_{A_s} \ar{r}{\mathrm{forg}}  & \QCoh(\Spec A_s)^\cn \simeq \Mod_{A_s}^\cn
				\end{tikzcd}
			\end{equation*}
	is left adjointable: this follows from the fact that formation of $\Sym^*_R$ commutes with base change. Since $\QCoh(X) = \varprojlim \QCoh(\Spec A_s)$ and $\QCAlgDelta(X) = \varprojlim \QCAlgDelta(\Spec A_s)$, by virtue of \cite[Corollary 4.7.4.18]{HA}, we obtain that $\mathrm{forg} \colon \QCAlgDelta(X) \to \QCoh(X)^{\cn}$ is left adjointable, and the diagram of \eqref{prop:Sym*:prestack-2} is left adjointable whenever $Y$ is affine. To prove the assertion \eqref{prop:Sym*:prestack-2} for a general morphism of prestacks $f \colon Y \to X$, let $\sF \in \QCoh(X)^\cn$, and we wish to show that the natural map $\lambda_f (\sF) \colon \Sym_Y^*(f^* \sF) \to f^* \Sym_X^*(\sF)$ is an equivalence. It suffices to show that for every map $h \colon \Spec A \to Y$, $h^* \lambda_f (\sF) $ is an equivalence. The composition $\Spec A \xrightarrow{h} Y \xrightarrow{f} X$ induces a commutative diagram 
	$$
	\begin{tikzcd}
		&  & h^* \, \Sym_Y^* (f^* \sF) \ar{rd}{h^*\lambda_{f} (\sF)}& \\
		& \Sym_{\Spec A}^* ( h^* \, f^* \sF) \ar{ru}{ \lambda_{h}(f^* \sF)} \ar{rr}{\lambda_{f\circ h} (\sF)}& & h^* \, f^* \Sym_X^* (\sF)
	\end{tikzcd}
	$$
where $\lambda_{f\circ h} (\sF)$ and $\lambda_{h}(f^* \sF)$ are equivalences since we have proved the assertion \eqref{prop:Sym*:prestack-2} in the case where the source $\Spec A$ is affine. Therefore $h^* \lambda_f (\sF) $ is an equivalence.
\end{proof}

\begin{remark} \label{rmk:Sym^*:globalization} Combining Proposition \ref{prop:Sym*:prestack} \eqref{prop:Sym*:prestack-2} and Proposition \ref{prop:Sym*} \eqref{prop:Sym*-4}, we obtain that the left adjoint functor $\Sym_X^*$ of Proposition \ref{prop:Sym*:prestack} agrees with the functor obtained by ``globalizing" the symmetric algebra $\Sym_A^*(M)$ of Definition \ref{def:dsym*}, that is, the functor obtained by applying Remark \ref{rmk:shvprestacks:properties} \eqref{rmk:shvprestacks:properties-2} to the functor $(A, M) \mapsto (A, \Sym_A^*M) \colon \SCRModcn \to \Fun(\Delta^1, \CAlgDelta)$. In particular,  if $X = \Spec R$, and $\sF \in \QCoh(X)$ is the sheaf associated to $M \in \Mod_R$ under the equivalence $\QCoh(X)\simeq \Mod_R$, then $\Sym_{\Spec R}^*( \sF)$ is canonically equivalent to the quasi-coherent $\sO_X$-algebra associated to the derived symmetric algebra $\Sym_R^* (M)$ under the equivalence $\QCAlgDelta(X) \simeq \CAlgDelta_R$. 
Moreover, for any prestack $X$ and any $\sF \in \QCoh(X)^\cn$, there are canonical equivalences of quasi-coherent sheaves
	$$\Sym_X^*(\sF) \simeq \bigoplus_{n \ge 0} \Sym_X^n (\sF),$$
where $\Sym_X^n(\sF)$ is the derived $n$th symmetric power of $\sF$ over $X$ (Construction \ref{constr:symwedgegamma:prestack}).
\end{remark}

\begin{remark}[The underlying commutative algebra objects of derived symmetric algebras]\label{rmk:CAlgofSym:prestack} Combining 
Remarks \ref{rmk:shvprestacks:properties} and \ref{rmk:CAlgofSym}, we obtain for any prestack $X$ a forgetful functor
	$$\Theta_X \colon \QCAlgDelta(X) \to \CAlg(\QCoh(X)^\cn).$$ 
For a quasi-coherent algebra $\sA \in \QCAlgDelta(X)$, we will refer to the image $\Theta_X(\sA)$ of $\sA$ in $\CAlg(\QCoh(X)^\cn)$ as the {\em underlying commutative algebra object in $\QCoh(X)^\cn$ of $\sA$}. 

For any connective quasi-coherent sheaf $\sF \in \QCoh(X)^\cn$, we will generally abuse notations by using $\Sym_X^*(\sF)$ to denote the derived symmetric algebra of $\sF$ over $X$ (as a quasi-coherent algebra), its underlying commutative algebra object $\Theta_X(\Sym_X^*(\sF))$, and its underlying connective quasi-coherent sheaf over $X$. In particular, the underlying quasi-coherent sheaf $\Sym_X^*(\sF) =  \bigoplus_{n \ge 0} \Sym_X^n (\sF)$ is equipped with a multiplication map
	$$m \colon \Sym_X^*(\sF) \otimes_{\sO_X} \Sym_X^*(\sF) \to \Sym_X^*(\sF)$$
which is unital, graded, and commutative and associative up to homotopy. 
\end{remark}

\begin{remark} Combining the machinery of \S \ref{sec:shvprestacks} (especially Remark \ref{rmk:shvprestacks:properties}) and Propositions \ref{prop:Sym*}, \ref{prop:Sym*:cofib}, we will obtain the global version of all the properties listed in Propositions \ref{prop:Sym*} and \ref{prop:Sym*:cofib} for the derived symmetric algebras over prestacks. For example:
	\begin{itemize}[leftmargin=*]
		\item Proposition \ref{prop:Sym*} \eqref{prop:Sym*-1} implies that if $X$ is a derived scheme, $\sF \in \QCoh(X)^\cn$, and let $X_\cl = (|X|, \pi_0(\sO_X))$ denote its underlying classical scheme. Then there is a canonical isomorphism of discrete quasi-coherent $\pi_0(\sO_X)$-algebras 
				$$\pi_0(\Sym_X^*(\sF)) \simeq \pi_0(\Sym_{X_\cl}^* (\pi_0 \sF)) \simeq S^*_{X_\cl} (\pi_0 \sF),$$
				where $S^*_{X_\cl} (\pi_0 \sF)$ denotes the classical symmetric algebra of the discrete quasi-coherent sheaf $\pi_0(\sF)$ over the classical scheme $X_\cl$.	
		\item Proposition \ref{prop:Sym*} \eqref{prop:Sym*-4} implies that if $f \colon Y \to X$ is a natural transformation of functors $X, Y \colon \CAlgDelta \to \shS$, and $\sF \in \QCoh(X)^\cn$, then there is a canonical equivalence of quasi-coherent algebras $f_{\rm alg}^* \Sym_X^*(\sF) \simeq \Sym_Y^*(f^* \sF)$.
		\item Proposition \ref{prop:Sym*} \eqref{prop:Sym*-5} implies that for any prestacks $X,Y$, and $\sF \in \QCoh(X)^\cn$, $\sG \in \QCoh(Y)^\cn$, there is a canonical equivalence of quasi-coherent algebras $\Sym_{X \times Y}^*(\sF \boxplus \sG) \simeq \Sym_X^*(\sF) \boxtimes \Sym_Y^*(\sG)$. In particular, for any $\sF, \sF' \in \QCoh(X)^\cn$, there is a canonical equivalence $\Sym_X^*(\sF \oplus \sF') \simeq \Sym_X^*(\sF) \otimes_{\sO_X} \Sym_X^*(\sF')$.
	\end{itemize}
We leave the details to the readers (as we will not directly use these statements in this paper). On the other hand, we will study the geometric (and essentially equivalent) version of these properties in the next section.
\end{remark}

\section{Derived projectivizations}
Let $X$ be a scheme and $\sF$ a discrete quasi-coherent sheaf on $X$, then Grothendieck's construction of the (classical) projectivization $\PP_{\cl}(\sF)$ of $\sF$ over $X$ can be described by any of the following equivalent ways:
\begin{enumerate}[label=(\alph*)]
	\item $\PP_{\cl}(\sF)$ is the homogeneous prime spectrum of the graded $\sO_X$-algebra $S_X^*(\sF)$, where $S_X^*(\sF)$ is the classical symmetric algebras of $\sF$ over $X$.
	\item $\PP_{\cl}(\sF)$ is the quotient $[(\VV_{\cl}(\sF) - \mathbf{0}_{\sF})/\GG_{m,X}]$, where $\VV_{\cl}(\sF)$ is the classical affine cone $ \Spec_{\cl, X}(S_X^*(\sF))$ and $\mathbf{0}_{\sF}$ is the zero section.
	\item $\PP_{\cl}(\sF)$ represents the functor $\CAlg^\heartsuit \to \shS{\rm et}$ over $X$ which carries a morphism of classical schemes $\eta \colon T=\Spec R \to X$, where $R \in \CAlg^\heartsuit$, to the set of equivalent classes of surjections $\pi_0(\eta^* \sF) \twoheadrightarrow \sL$, where $\sL$ is a line bundle on $T$.
\end{enumerate}
Consequently, there are three equivalent approaches to defining the derived projectivizations in the framework of derived algebraic geometry. This paper adopts an approach which is the derived counterpart of $(c)$ and close in spirit to Grothendieck's approach in \cite{EGAI}. 

Let $X$ be a prestack and $\sF$ a connective quasi-coherent sheaf over $X$. In \S \ref{sec:def:dproj:qcoh}, we define the {\em affine cone} $\VV(\sF)$ of $\sF$ over $X$ and study their properties (see Propositions \ref{prop:affinecone} and \ref{prop:affinecone:PB}). In \S \ref{sec:def:dproj:qcoh}, we define the {\em derived projectivization} $\PP(\sF)$ of $\sF$ over $X$ via a similar manner as the above $(c)$, and examine their properties, including finiteness (Proposition \ref{prop:proj:finite}) and functoriality (Proposition \ref{prop:proj-4,5}), their behaviors under quotient maps (Corollary \ref{cor:proj:closedimmersion}, Proposition \ref{prop:proj:PB}), and describe their relative cotangent complexes through the Euler fiber sequences which vastly generalize the classical Euler sequences (Theorem \ref{thm:proj:Euler}).

\begin{remark} The readers are referred to Hekking's paper \cite{He21} for a different approach which is more closely related to the above perspectives $(a)$ and $(b)$.
\end{remark}

\begin{remark} We present this section in a way such that all the above constructions and results (especially those in \S \ref{sec:def:dproj:qcoh}) are readily generalized to the situation of Grassmannians, flags, and their variants (see \cite{J20, J21}). See \cite{J22a} for details.
\end{remark}

\subsection{Affine cones} \label{sec:affine-cone} Let $X$ be a prestack $\CAlgDelta \to \shS$ and $\sF$ a connective quasi-coherent sheaf on $X$. For for each $A \in \CAlgDelta$ and each $\eta \colon T=\Spec A \to X$, 
the construction 
	\begin{align} \label{def:functor:vectorbundle}
	(\eta \colon T =\Spec A \to X) \mapsto (\Map_{\QCoh(T)}(\eta^* \sF, \sO_T) \in \shS)
	\end{align}
defines a prestack $F \in \Fun(\CAlgDelta, \shS)_{/X}$ with a natural projection $F \to X$. More informally, the prestack $F$ parametrizes the ``cosections of the sheaf $\sF$ over $X$".

\begin{proposition} Let $X$ be a prestack and let $\sF$ be a connective quasi-coherent sheaf on $X$. Then the functor \eqref{def:functor:vectorbundle} is representable by the relative affine derived scheme over $X$:
	$$\pr_{\VV(\sF)} \colon \VV(\sF) = \Spec (\Sym_X^* (\sF)) \to X.$$
The universal morphism of $\sO_{\VV(\sF)}$-modules $\pr_{\VV(\sF)}^* \sF \to \sO_{\VV(\sF)}$ is induced from the canonical morphism $\sF \to \Sym_X^*(\sF)$ by adjunction.
\end{proposition}

\begin{proof} We prove a slightly more general statement: there are canonical equivalences for all prestacks $T$ and natural transformations $\eta \colon T \to X$:
	\begin{align*}
		\Map_{\Fun(\CAlgDelta, \shS)_{/X}}(T, \VV(\sF)) \simeq \Map_{\QCAlgDelta(T)}(\eta^*_{\rm alg} \Sym_S^*(\sF), \sO_T)  \\
	\simeq \Map_{\QCAlgDelta(T)}(\Sym_T^*(\eta^* \sF), \sO_T) \simeq \Maps_{\QCoh(T)}(\eta^* \sF, \sO_T).
	\end{align*}
Here, the first equivalence follows from Proposition \ref{prop:affine} \eqref{prop:affine-3}, the second from Proposition \ref{prop:Sym*:prestack} \eqref{prop:Sym*:prestack-2}, and the last from Proposition \ref{prop:Sym*:prestack} \eqref{prop:Sym*:prestack-1}.
\end{proof}

\begin{definition} Let $X$ be a prestack and let $\sF$ be a connective quasi-coherent sheaf on $X$. We will refer to $\VV(\sF)$ as the {\em affine cone of $\sF$ over $X$}, and we will refer to the morphism of $\sO_{\VV(\sF)}$-modules $\pr_{\VV(\sF)}^* \sF \to \sO_{\VV(\sF)}$ as the {\em universal cosection morphism}.
\end{definition}

\begin{remark}[Grothendieck's dual convention] \label{rmk:Grothendieck:affinecone} Beware that we are following Grothendieck's convention, so that the affine cone $\VV(\sF)$ parametrizes the {\em cosections} of $\sF$ rather than sections of $\sF$. To put it slightly differently -- for simplicity of exposition, we assume now $\sF$ is connective and perfect -- for any morphism $\eta \colon T= \Spec A \to X$, the space of $T$-valued points of $\VV(\sF)$ over $\eta$ is the space of sections of the {\em dual} $(\sF_T)^\vee$ of $\sF_T$:
	$$\VV(\sF)(\eta) \simeq \Map_{\QCoh(T)}(\sF_T, \sO_T) \simeq \Map_{\QCoh(T)}(\sO_T, (\sF_T)^\vee) \simeq \Omega^{\infty} ((\sF_T)^\vee) \in \shS.$$
(Here, $\sF_T = \eta^* \sF$, and $\Omega^{\infty} \colon \Sp \to \shS$ is the zeroth space functor.) In particular, for any field $\kappa$ and a map $\eta \colon \Spec \kappa \to X$, the space of $\kappa$-valued points of $\VV(\sF)$ over $\eta$,
	$$\VV(\sF)(\eta) \simeq \Map_{\Mod_{\kappa}}(\sF_{\kappa}, \kappa) \simeq \Ext^0_\kappa( \pi_0 (\sF_\kappa), \kappa) \simeq \Map_{\Mod_{\kappa}}(\kappa, \sF_{\kappa}^\vee) \simeq \pi_0 ((\sF_\kappa)^\vee),$$
is (canonically homotopy equivalent to) a $\kappa$-vector space.
\end{remark}


\begin{remark}[Zero sections] \label{rmk:zerosections} By construction, there is a canonical equivalence from the space of sections of the projection $\pr \colon \VV(\sF) \to X$ to the space $\Map_{\QCoh(X)}(\sF, \sO_X)$. In particular, the zero morphism $0 \in \Map_{\QCoh}(\sF, \sO_X)$ corresponds to a canonical section $i_{\mathbf{0}} \colon X \to \VV(\sF)$. We will refer $i_{\mathbf{0}}$ as the {\em zero section} of the vector bundle $\pr \colon \VV(\sF) \to X$ and denote the image of $i_{\mathbf{0}}$ by $\mathbf{0}_{\sF} \subset \VV(\sF)$.
\end{remark}

\begin{example} If $\sF = \sO_X^{\oplus n}$, then $\Sym_X^*(\sO_X^{\oplus n}) = \sO_X[T_1, \ldots, T_n] = \sO_X \otimes_\ZZ \ZZ[T_1, \ldots, T_n]$ (where $T_i$ are  indeterminates) and we have canonical equivalence $\VV(\sO_X^{\oplus n}) \simeq X \times_\ZZ \AA_\ZZ^n$.
\end{example}

\begin{example}[Universal Hom space]  \label{eg:Homspaces} Let $X$ be a prestack, and let $\sF$ and $\sE$ be quasi-coherent sheaves on $X$ such that $\sF$ is connective and $\sE$ is a truncated perfect complex. Then the functor $H \colon \Fun(\CAlgDelta, \shS)_{/X} \to \shS$ which carries $\eta \colon T \to X$ to 
	$$H(\eta) = \Map_{\QCoh(T)}(\eta^* \sF, \eta^* \sE) \simeq \Map_{\QCoh(T)}(\eta^* (\sF \otimes_{\sO_X} \sE^\vee), \sO_T)$$
is representable by the relative affine derived scheme $\pr \colon \VV(\sF \otimes_{\sO_X} \sE^\vee) \to X$. We will let $\pr \colon |\sHom_X(\sF, \sE)| \to X$ denote the relative derived scheme $\VV(\sF \otimes_{\sO_X} \sE^\vee) \to X$ and refer it as the {\em universal Hom spaces of morphisms from $\sF$ to $\sE$ over $X$}. Let $\tau \colon \pr^* \sF \to \pr^* \sE$ denote the canonical map of quasi-coherent sheaves which corresponds to the tautological cosection map $\pr^*(\sF \otimes \sE^\vee) \to \sO_{\VV(\sF \otimes \sE^\vee)}$; We will refer $\tau$ as the {\em tautological morphism}. By construction, given $\eta \colon T \to X$, then each morphism $\sigma \colon \eta^* \sF \to \eta^* \sE$ in $\QCoh(T)$ is classified by a map $f \colon T \to |\sHom_X(\sF, \sE)|$ lifting $\eta$ such that there is a canonical equivalence $\sigma \simeq f^* \tau$.
\end{example}

\begin{definition}\label{def:perfectfp:prestacks} Let $X$ be a prestack and $\sF \in \QCoh(X)$. We say that $\sF$ is {\em perfect to order $n$ (in the sense of \cite[Definition 2.5.7]{DAG}; resp. 
 almost perfect, perfect)} if for every simplicial commutative ring $R$ and every point $\eta \in X(\eta)$, then $R$-module $\sF(\eta)$ has the same property in the sense of Definition \ref{def:perfectfp} (this is well-defined since the above properties are stable under base change \cite[Proposition 3.5.2]{DAG}). Let $f \colon X \to Y$ be a morphism of derived schemes. We say that $f$ is locally of finite presentation to order $n$ (resp. locally almost of finite presentation, locally of ﬁnite presentation), if for every pair of open immersions $\Spec B \subseteq X$, $\Spec A \subseteq Y$, where $A, B \in \CAlgDelta$, such that $f$ restricts to a morphism $\Spec B \to \Spec A$, $B$ is of finite presentation to order $n$ (resp. almost of finite presentation, locally of ﬁnite presentation) over $A$ in the sense of Definition \ref{def:perfectfp}. Let $f \colon X \to Y$ be a morphism of prestacks which is a relative derived scheme. We say that $f$ is {\em locally of finite presentation to order $n$ (resp. locally almost of finite presentation, locally of ﬁnite presentation)} if for every map $\eta \colon \Spec R \to Y$, where $R \in \CAlgDelta$, the induced morphism of derived schemes $X_R = X \times_{\Spec R} Y \to \Spec R$ has the same property in the above sense for derived schemes (this is well defined as these properties are stable under base change and local on the target with respect to {\'e}tale topology; see \cite[Propositions 3.5.2, 3.5.3]{DAG}).
\end{definition}

We summarize the basic properties of affine cones in the following proposition:
\begin{proposition} \label{prop:affinecone} Let $X$ be a prestack and $\sF$ a connective quasi-coherent sheaf on $X$. 
\begin{enumerate}[leftmargin=*] 
	\item \label{prop:affinecone-1} (Underlying classical schemes) If $X$ is a derived scheme, then $\VV(\sF)$ is a derived scheme, and the underlying classical scheme of $\VV(\sF)$ is canonically identified with the classical affine cone $\VV_\cl(\pi_0(\sF)) = \Spec_{X_\cl}^\cl \Sym_{\cl}^*(\pi_0 (\sF)) \to X_\cl$ of the sheaf $\pi_0(\sF) \in \QCoh(X_\cl)^\heartsuit$ over $X_\cl = t_0(X)$.
	\item \label{prop:affinecone-2} (Finiteness conditions) The projection map $\pr_{\VV(\sF)} \colon \VV(\sF) \to X$ is locally of finite presentation to order $n$ (resp., locally almost of ﬁnite presentation, locally of finite presentation) if and only if the quasi-coherent module $\sF \in \QCoh(X)^{\rm cn}$ is perfect to order $n$ (resp. almost perfect, perfect).
	\item \label{prop:affinecone-3} (Relative cotangent complex) The projection $\pr_{\VV(\sF)} \colon \VV(\sF) \to X$ admits a relative cotangent complex given by $L_{\VV(\sF)/X} = \pr_{\VV(\sF)}^* \sF \in \QCoh(\VV(\sF))^\cn$, and the inclusion map of zero section $i_{\mathbf{0}} \colon X \to \VV(\sF)$ (see Remark \ref{rmk:zerosections}) admits a relative cotangent complex given by $L_{X/\VV(\sF)} \simeq \Sigma \, \sF \in \QCoh(X)^\cn$. In particular, the projection $\pr_{\VV(\sF)} \colon \VV(\sF) \to X$ is smooth if and only if $\sF$ is a vector bundle, and is quasi-smooth if and only $\sF$ has perfect-amplitude contained in $[0,1]$.
	\item \label{prop:affinecone-4} (The formation of affine cones commutes with base change) For any maps of prestacks $f \colon X' \to X$, there is a canonical equivalence $\VV(f^* \sF) \simeq \VV(\sF) \times_X X'$. 
	\item \label{prop:affinecone-5}(Direct sums) For any pair of connective quasi-coherent sheaves $\sE, \sF \in \QCoh(X)^\cn$, there is a canonical equivalence $\VV(\sE \oplus \sF) \simeq \VV(\sE) \times_X \VV(\sF)$.
	 \item \label{prop:affinecone-6} Let $\varphi \colon \sE \to \sF$ be a map of connective quasi-coherent sheaves on $X$. Then $\varphi$ induces an affine morphism $\VV(\varphi^*) \colon \VV(\sF) \to \VV(\sE)$ such that the pullback of the universal cosection morphism on $\VV(\sE)$ is canonically equivalent to the universal cosection morphism on $\VV(\sF)$. Moreover, the morphism $\VV(\varphi^*)$ admits a relative cotangent complex which can be described by the formula
	 $$L_{\VV(\sF)/\VV(\sE)} =  \pr_{\VV(\sF)}^*({\rm cofib}\big(\varphi \colon \sE \to \sF)\big).$$
If $\varphi \colon \sE \to \sF$ is surjective (that is, the induced map $\pi_0(\sE) \to \pi_0(\sF)$ is an epimorphism in $\QCoh(X)^\heartsuit$), then $\VV(\varphi^*)$ is a closed immersion. In particular, if the fiber of $\varphi$ is a vector bundle on $X$, then $\iota$ is a quasi-smooth closed immersion (which is given by the canonical cosection of the vector bundle ${\rm fib}(\varphi)$).  
	\end{enumerate}
\end{proposition}
\begin{proof} As the assertions \eqref{prop:affinecone-1} through \eqref{prop:affinecone-5} are local, they follow directly from the corresponding properties of symmetric algebras of Proposition \ref{prop:Sym*} \eqref{prop:Sym*-1} through \eqref{prop:Sym*-5}, respectively. For the last assertion \eqref{prop:affinecone-6}, the map of quasi-coherent sheaves  $\varphi$ induces a map of quasi-coherent algebras $\Sym_X^*(\varphi) \colon \Sym_X^* \sE \to \Sym_X^* \sF$. Hence by Proposition \ref{prop:affine}, we obtain an affine map $\VV(\varphi^*) \colon \VV(\sF) \to \VV(\sE)$. By virtue of \cite[Proposition 3.2.12]{DAG}, the composition of maps $\VV(\sF) \to \VV(\sE) \to X$ induces a canonical cofiber sequence
	$$\iota^* L_{\VV(\sE)/X} \to L_{\VV(\sF)/X} \to L_{\VV(\sF)/\VV(\sE)}.$$ 
Hence the desired expression of $L_{\VV(\sF)/\VV(\sE)}$ follows from assertion \eqref{prop:affinecone-3}. If $\varphi \colon \sE \to \sF$ is sujrective, then $\Sym_X^*(\varphi) \colon \Sym_X^* \sE \to \Sym_X^* \sF$ is surjective by virtue of Proposition \ref{prop:Sym*} \eqref{prop:Sym*-1}, hence $\VV(\varphi^*)$ is a closed immersion. Finally, if $\fib(\varphi)$ is a vector bundle, then $L_{\VV(\sF)/\VV(\sE)} \simeq \Sigma \fib(\varphi)$ is a perfect complex of Tor-amplitude concentrated on degree $1$, hence $\VV(\varphi^*)$ is a quasi-smooth closed immersion.
\end{proof}

\begin{proposition} \label{prop:affinecone:PB}  Let $X$ be a prestack and $\sF' \xrightarrow{\rho'} \sF \xrightarrow{\rho''} \sF''$ a fiber sequence of connective quasi-coherent sheaves on $X$. Then there is a pullback diagram of prestacks:
	$$
	\begin{tikzcd}
		\VV(\sF'') \ar{d}[swap]{\pr_{\VV(\sF'')}} \ar{r}{\VV(\rho''^*)}& \VV(\sF) \ar{d}{\VV(\rho'^*)} \\
		X \ar{r}{i_{\mathbf{0}}} & \VV(\sF'),
	\end{tikzcd}
	$$
where $\VV(\rho'^*)$ and $\VV(\rho''^*)$ are the natural induced affine maps (see  Proposition \ref{prop:affinecone} \eqref{prop:affinecone-6}), and $i_{\mathbf{0}} \colon X \to \VV(\sF')$ is the inclusion of the zero section (see Remark \ref{rmk:zerosections}).
\end{proposition}

\begin{proof} We may reduce to the case where $X= \Spec R$, $R \in \CAlgDelta$. Then the assertion follows from Proposition \ref{prop:Sym*:cofib}. (Notice that this proposition provides an alternative proof of the assertions after ``moreover" of Proposition \ref{prop:affinecone} \eqref{prop:affinecone-6}.)
\end{proof}

\begin{remark} \label{rmk:affinecone:PB} In the same situation of Proposition  \ref{prop:affinecone:PB}, it follows from the pullback diagram of  Proposition  \ref{prop:affinecone:PB} that there is a pullback diagram of prestacks:
	$$
	\begin{tikzcd}
		\VV(\sF'') \ar{d}[swap]{\VV(\rho''^*)} \ar{r}{\VV(\rho''^*)}& \VV(\sF) \ar{d}{\iota_{\rho'}} \\
		\VV(\sF)  \ar{r}{\iota_{\mathbf{0}}} & \VV(\sF') \times_X \VV(\sF),
	\end{tikzcd}
	$$
where $\iota_{\mathbf{0}}, \iota_{\rho'}$ are sections of the projection $\VV(\sF') \times_X \VV(\sF) \to \VV(\sF)$ that classify the zero map $\pr_{\VV(\sF)}^* \sF' \to \sO_{\VV(\sF)}$ and the composition $\pr_{\VV(\sF)}^* \sF' \xrightarrow{\pr_{\VV(\sF)}^* \rho'} \pr_{\VV(\sF)}^* \sF \to \sO_{\VV(\sF)}$, respectively. 
 \end{remark}

As an immediate application of the above construction, we can formulate the theory of derived zero loci of a cosection in a relative simple form.

\begin{definition}[Derived zero locus of a cosection] \label{def:dzero} Let $X$ be a prestack, and let $\sF$ be a connective quasi-coherent sheaf on $X$. For a cosection $\rho \colon \sF \to \sO_X$ of $\sF$, we define the {\em derived zero locus of the cosection $\rho$} to be the essentially unique closed sub-prestack $\iota \colon Z (\rho) \to X$ of $X$ that fits into a pullback square of prestacks:
	\begin{equation*}
	\begin{tikzcd}
		Z(\rho) \ar{d}[swap]{\iota} \ar{r}{\iota} & X \ar{d}{i_{\rho}} \\
		X \ar{r}{i_{\mathbf{0}}} & \VV(\sF),
	\end{tikzcd}
	\end{equation*}
where $i_{\mathbf{0}}$ and $i_{\rho}$ are sections of the projection $\VV(\sF) \to X$ which classify the zero map $\sF \to \sO_X$ and the cosection $\rho \colon \sF \to \sO_X$, respectively. If $\sF$ is furthermore perfect, we will generally abuse notations by writing $Z(\rho) = Z(\rho^\vee)$ and referring to it as also the {\em derived zero locus of the section $\rho^\vee \colon \sO_X \to \sF^\vee$}.
 \end{definition}
 
\begin{remark}[Affine cases] Suppose $X = \Spec A$, where $A \in \CAlgDelta$, and $\sF$ corresponds to connective $A$-module $M$. By adjunction, each cosection map $\rho \colon M \to A$ corresponds to an essentially uniquely map of algebras ${\rm ev}_{\rho} \colon \Sym_A^*(M) \to A$. Notice that if $\rho$ is the zero cosection map, then ${\rm ev}_{\mathbf{0}}$ coincides with the canonical augmentation map $\Sym_A^*(M) \to A$. Then we can describe the derived zero locus of $\rho$ by the formula $Z(\rho) = \Spec (A \otimes_{\Sym_A^*(M)} A)$, where $ A \otimes_{\Sym_A^*(M)} A$ fits into a pushout square of simplicial commutative rings
 	$$
	\begin{tikzcd}
		\Sym_A^*(M) \ar{d}[swap]{\mathrm{ev}_{\rho}} \ar{r}{\mathrm{ev}_{\mathbf{0}}}& A \ar{d} \\
		A \ar{r} & A \otimes_{\Sym_A^*(M)} A.
	\end{tikzcd}
	$$
 \end{remark}
 
 \begin{remark}[Derived zero loci as relative spectra] \label{rmk:dzero:algebra} In the situation of Definition \ref{def:dzero}, we can describe the closed immersion (hence an affine morphism) $\iota \colon Z(\rho) \to X$ as the relative spectrum (Definition \ref{def:relative:spec}) of any of the following equivalent quasi-coherent $\sO_X$-algebras: 
 	\begin{align*}
	\sO_X \otimes_{{\rm ev}_0,  \Sym_X^*(\sF), {\rm ev}_\rho} \sO_X & \simeq \Sym_X^*(\cofib(\sF \xrightarrow{\id_\sF} \sF)) \otimes_{\Sym_X^*(\sF), {\rm ev}_\rho} \sO_X \\
	& \simeq \Sym_X^*(\cofib(\sF \xrightarrow{\rho} \sO_X)) \otimes_{\sO_X[T], T \mapsto 1} \sO_X,
	\end{align*}
where the last equivalence follows from the following pullback squares of prestacks:
	$$
	\begin{tikzcd}
		Z(\rho) \ar{r}{\iota} \ar{d}& X \ar{d}{i_{\mathbf{1}}} \\
		\VV(\cofib(\rho)) \ar{d}[swap]{\pr_{\VV(\cofib(\rho))}} \ar{r}{\VV(\rho''^*)}& \VV(\sO_X)\simeq \Spec_X(\sO_X[T]) \ar{d}{\VV(\rho^*)} \\
		X \ar{r}{i_{\mathbf{0}}} & \VV(\sF),
	\end{tikzcd}
	$$	
 where $T$ is an indeterminate, $\rho''$ is the canonical morphism $\sO_X \to \cofib(\rho)$, $i_{\mathbf{1}} \colon X \to \VV(\sO_X)$ is the map classifying the identity cosection $\id \colon \sO_X \to \sO_X$; the bottom square is a pullback square by virtue of Proposition \ref{prop:affinecone:PB}, and the outer rectangle is a pullback square by definition of $Z(\rho)$, hence the top square is also a pullback square.
 \end{remark}

 \begin{lemma}\label{lem:Koszul:fib} Let $X$ be a prestack, and let $\sF$ be a vector bundle on $X$ of (finite) rank $n >0$. Let $\iota \colon Z(\rho) \hookrightarrow X$ be the derived zero locus of a cosection $\rho \colon \sF \to \sO_X$. Then for any $\sE \in \QCoh(X)$, there is a sequence of quasi-coherent sheaves $\sG_i(X, \sE) \in \QCoh(X)$, $0 \le i \le n$, such that $\sG_{0}(X, \sE) = \sE$, $\sG_{n}(X, \sE) \simeq \iota_* \, \iota^* \sE$, and for which there are fiber sequences in $\QCoh(X)$ for all $0 \le i \le n$ (where we set $\sG_{-1}(X, \sE) = 0$):
 	$$\sG_{i-1}(X, \sE) \to \sG_i(X, \sE) \to \Sigma^i (\bigwedge\nolimits^i_X \sF) \otimes_{\sO_X} \sE.$$
Moreover, if $\sE \in \Perf(X)$, then $\sG_i(X, \sE) \in \Perf(X)$.
 \end{lemma}
 
 \begin{proof} It suffices to construct functorially the desired quasi-coherent sheaves and fiber sequences in the case where $X = \Spec A$, $A \in \CAlgDelta$; The general case for prestacks will follow from the procedure of \S \ref{sec:shvprestacks}. Since tensoring with $\sE$ is an exact functor, by setting $\sG_i(X, \sE) = \sG_i(X, \sO_X) \otimes_{\sO_X} \sE$ we may reduce to the case where $\sE = \sO_X$.
 
We consider the following variant of Construction \ref{constr:right:Sym}. Let $R$ be a discrete commutative ring, and suppose we are given a short exact sequence $0 \to M' \xrightarrow{\varphi} M \to M'' \to 0$ of finite generated free $R$-modules. For any integer $0 \le i \le n$, we let $g_{i}(R,\varphi) \in \Mod_{S_R^*M}^\cn$ denote the class of the ``brutal" truncation 
	$$(\bigwedge\nolimits^i_R M' ) \otimes_R (S_R^* M) \xrightarrow{d_i} \cdots \to M' \otimes_R S_R^* M \xrightarrow{d_1} S_R^* M$$
 of the complex \eqref{eqn:Koszul:S^*}. Then $g_i(R, \rho)$ fits into canonical fiber sequences of $S_R^* M$-modules:
	$$g_{i-1}(R, \varphi) \to g_{i}(R, \varphi) \to \Sigma^{i} (\bigwedge\nolimits_R^{i} M') \otimes_R (S_R^* M).$$
By a similar procedure as Construction \ref{constr:right:Sym}, the constructions $(R, \rho) \in \shE_0 \mapsto (S_R^*M, g_{i}(R,\varphi) \in \SCRModcn$ admit essentially unique left Kan extensions $(A, \rho) \in \shE \mapsto (\Sym_A^*(M), G_{i}(A, \varphi) ) \in \SCRModcn$, for which there are fiber sequences of connective $\Sym_A^* M$-modules
	$$G_{i-1}(A, \varphi) \to G_{i}(A, \varphi) \to \Sigma^{i} (\bigwedge\nolimits_A^{i} M') \otimes_A (\Sym_A^* M),$$
and the formation of $G_i(A, \varphi)$ commutes sifted colimits and base changes of rings. Now suppose $M$ is a locally free $A$-module of rank $n$, and we are given a map of $A$-modules $\rho \colon M \to A$ which determines a map of algebras ${\rm ev}_\rho \colon \Sym_A^*(M) \to A$. We set 
	$$\sG_{i}(\Spec A, \sO_{\Spec A}) :  = G_i(A, \id_M \colon M \to M) \otimes_{\Sym_A^*(M), \,\mathrm{ev}_\rho} A  \in \Mod_A^{\rm cn}.$$
Then the formation of $\sG_{i}(\Spec A, \sO_{\Spec A})$ is functorial on $\rho \colon M \to A$ and commutes with base changes. Moreover, there are fiber sequences of connective $A$-modules
	$$\sG_{i-1}(\Spec A, \sO_{\Spec A})  \to \sG_{i}(\Spec A, \sO_{\Spec A})  \to \Sigma^{i} (\bigwedge\nolimits_A^{i} M).$$
It remains to show that the canonical map of $A$-modules $\theta_A \colon \sG_{n}(\Spec A, \sO_{\Spec A}) \to \iota_* \sO_{\Spec A}$ is an equivalence. Since the assertion that $\theta_A$ is an equivalence is Zariski local on $A$, we may assume $M\simeq \ZZ^n \otimes_\ZZ A$, then $\theta_A \simeq \theta_{\ZZ} \otimes_\ZZ A$. It suffices to show that the canonical map $\theta'_\ZZ \colon G_n(\ZZ, \id \colon \ZZ^n \to \ZZ^n) \to \ZZ$ is an equivalence of $\ZZ[x_1, \ldots, x_n]$-modules (which implies $\theta_{\ZZ} \simeq \theta'_\ZZ \otimes_{\ZZ[x_1, \ldots, x_n]} \ZZ$ is an equivalence). But this follows from the exactness of the classical Koszul complex \eqref{eqn:Koszul:S^*} in the case $\rho = \id \colon \ZZ^n \to \ZZ^n$.
 \end{proof}

\begin{remark} \label{rmk:Koszul:fib} Let $A$ be a simplicial commutative ring and $\rho \colon M \to A$ a cosection of a locally free $A$-module $M$ of rank $n \ge 1$. For a map of connective $A$-modules $\varphi \colon N' \to N$, we let $G_i(A, \varphi)$ denote the $\Sym_A^*(N)$-module as defined in the proof of Lemma \ref{lem:Koszul:fib}, and for any integer $m \ge 0$, let $G_{i,m}(A, \varphi)$ denote the $A$-module defined as in Construction \ref{constr:right:Sym}. Since all these modules are constructed from the complex \eqref{eqn:Koszul:S^*}, by construction, for any $m \ge n$, there are canonical equivalences of $A$-modules:
	$$G_{i}(A, M \xrightarrow{\id_M} M) \otimes_{\Sym_A^*(M)} A \simeq G_{i}(A, M \xrightarrow{\rho} A) \otimes_{\Sym_A^*(A)} A \simeq G_{i,m}(A, M \xrightarrow{\rho} A). $$
Consequently, in the situation of Lemma \ref{lem:Koszul:fib}, for any given integer $m$ such that $m \ge n$, we can equivalently describe the quasi-coherent sheaves $\sG_i(X, \sE)$ by the formula
	$$\sG_i(X, \sE) \simeq \sG_{i,m}(X, \rho \colon \sF \to \sO_X) \otimes_{\sO_X} \sE,$$ 
where $\sG_{i,m}(X, \rho \colon \sF \to \sO_X)$ is defined as in Construction \ref{constr:Postnikov:Sym:prestacks}.
\end{remark}

\subsection{Derived projectivizations}\label{sec:def:dproj:qcoh}
By virtue of Proposition \ref{prop:affine} \eqref{prop:affine-2}, we will generally abuse notations by not distinguishing the $\infty$-category $\CAlgDelta$ and the opposite $\infty$-category of affine derived schemes. Similarly, if $X$ is a prestack and $A \in \CAlgDelta$, we will generally not distinguish the space of $A$-points $X(A)$ and the space of maps $T = \Spec A \to X$.

Let $X$ be a prestack $\CAlgDelta \to \shS$ and $\sF$ a connective quasi-coherent sheaf over $X$. For each $A \in \CAlgDelta$ and an $A$-point $\eta \colon T =\Spec A \to X$ of $X$, we let $P_{\sF}(\eta)$ denote the full subcategory of $\QCoh(T)_{\eta^* \sF/} \times_{\QCoh(T)}  \mathrm{Pic}(T)$ spanned by {\em surjective} maps $\eta^*\sF \to \sL$ (that is, the induced map $\pi_0(\eta^* \sF) \to \pi_0(\sL)$ is an epimorphism in $\QCoh(T)^\heartsuit$), where $\sL$ is a line bundle on $T$. Here, $\Pic(T)$ denotes the full subcategory of $\QCoh(T)^{\simeq}$ spanned by line bundles (see \cite[\S 19.2]{SAG} for more about Picard functors). Then $P_{\sF}(\eta)$ is a Kan complex, and the construction 
	$$(A \in \CAlgDelta, \eta \in X(A)) \mapsto (P_{\sF}(\eta) \in \shS)$$
defines a prestack $P_{\sF} \in \Fun(\CAlgDelta, \shS)_{/X}$ over $X$.

\begin{proposition} \label{prop:proj:represent}
Let $X$ be a prestack and let $\sF$ be a connective quasi-coherent sheaf on $X$. Then the prestack $P_{\sF} = \PP(\sF)$ is a relative derived scheme over $X$. 
\end{proposition}
 
 \begin{proof} It suffices to show that for each map of the form $\Spec R \to X$, where $R \in \CAlgDelta$, the pullback $P_{\sF} \times_{X} \Spec R$ is a derived scheme over $\Spec R$. Hence we may reduce to the case where $X = \Spec R$, $R \in \CAlgDelta$, and we wish to prove the functor $P_{\sF}$ is representable by a derived scheme over $X$. Since $\QCoh(X) \simeq \Mod_R$, we may abuse notations and regard $\sF \in \QCoh(X)$ as a connective $R$-module. Let $(t_i)_{i \in I}$ be a family of elements (of possibly infinite members) that generates the $\pi_0(R)$-module $\Ext_R^0(R, \sF)$. Then the family of maps $\{t_i \colon \sO_X \to \sF\}_{i \in I}$ induces a map $ \sO_X^{ \oplus I} \to \sF$ that is surjective on $\pi_0$. For each $i \in I$, we let $U_{i}$ denote the subfunctor of $P_{\sF}$
 such that for every $A \in \CAlgDelta$ and $\eta \colon T=\Spec A \to X$, $U_{i}(\eta)$ is the full subcategory of $P_{\sF}(\eta)$ spanned by those surjections $u \colon \eta^* \sF \to \sL$ for which the composite map 
 	$$\sO_T \xrightarrow{\eta^* t_i} \eta^* \sF \xrightarrow{u} \sL$$
is surjective on $\pi_0$.
The desired result will follow from the following assertions:
	\begin{enumerate}[label=$(\roman*)$, ref=$\roman*$]
		\item \label{proof:prop:proj:represent-i}
		For each $i$, the sub-prestack $U_{i}$ is representable by a derived affine scheme.
		\item \label{proof:prop:proj:represent-ii} 
		For each $i$, the inclusion $U_{i} \subseteq P_{\sF}$ is an open immersion.
		\item \label{proof:prop:proj:represent-iii}
		 The family of maps $\{U_{i} \to P_{\sF} \}_{i \in I}$ is jointly surjective.
	\end{enumerate}
	
We first prove assertion \eqref{proof:prop:proj:represent-i}. Let $U_i'$ denote the subfunctor of $U_i$ such that for each $\eta \colon T = \Spec R \to X$, $U_i'(\eta)$ is the space spanned by surjective cosections $u \colon \eta^* \sF \to \sO_T$ for which the composition $\sO_T \xrightarrow{\eta^* t_i} \eta^* \sF \xrightarrow{u} \sO_T$ is an isomorphism. On the other hand, for each $(u \colon \eta^* \sF \to \sL) \in U_i(\eta)$, the epimorphism $f_i \colon \sO_T \xrightarrow{\eta^* t_i} \eta^* \sF \xrightarrow{u} \sL$ between line bundles is necessarily an isomorphism. Since $\Pic(T)_{\sO_T/}$ is contractible, we obtain that the inclusion $U_i'(\eta) \subseteq U_i(\eta)$ is a homotopy equivalence, and $U_i'(\eta)$ is canonically homotopy equivalent to the fiber space
	$$\mathrm{fib}\left( \Map_{\QCoh(T)} (\eta^* \sF, \sO_T) \xrightarrow{\circ (\eta^* t_i)} \Map_{\QCoh(T)}(\sO_T, \sO_T)\right)$$
over the point $\id_{\sO_T}$. Consequently, we obtain a canonical equivalence between the prestack $U_i$ and the fiber of the morphism $\VV(t_i^*) \colon \VV(\sF) \to \VV(\sO_X) \simeq \AA_X^1$ over the section determined by $\id_{\sO_X}$. Therefore, $U_i$ is representable by an affine closed derived subscheme of $\VV(\sF)$. 

We next prove assertion \eqref{proof:prop:proj:represent-ii}. It suffices to show that for every affine derived scheme $Z = \Spec B$ and every map $g \colon Z \to P_{\sF}$, the functor
	\begin{equation} \label{eqn:proj:rep:Ui-Z}
	(\eta \colon T = \Spec A \to X) \mapsto (U_i(\eta) \times_{P_{\sF}(\eta)} Z(\eta) \in \shS)
	\end{equation}
is representable by a derived open subscheme of $Z$. Let $\sF_Z = g^* \sF$. By the definition of $P_{\sF}$, the morphism $g \colon Z \to P_\sF$ determines a unique surjection $u_Z \colon \sF_Z \to \sL$, where $\sL$ is a line bundle on $Z$. Let $f_{i}$ denote the composition map $\sO_Z \xrightarrow{g^* t_{i}} \sF_Z \xrightarrow{u_{Z}} \sL$. By working Zariski locally over $Z$, we may assume $\sL \simeq \sO_Z$ and regard $f_i$ as an element of $\pi_0(B)$. We let $U_{Z, i} = \Spec B[\frac{1}{f_i}]$ be the derived open subscheme of $Z = \Spec B$ defined by localization of $B$ with respect to $f_i$ (where $B[\frac{1}{f_i}] \to B$ is the localization map of simplicial commutative rings for the localization of $B$ with respect to $f_i \in \pi_0(B)$; see \cite[Proposition 4.1.18]{DAGV}). To show that $U_{i, Z} \subseteq Z$ represents the functor \eqref{eqn:proj:rep:Ui-Z}, it suffices to show the following assertion:
	\begin{enumerate}[label=$(*)$, ref=$*$]
		\item  \label{proof:prop:proj:represent-ii-*}
		Let $h \colon Y=\Spec C \to Z = \Spec B$ be any map of derived affine schemes. Then $h^*(f_i)$ is an epimorphism on $\pi_0$ if and only if $h$ factorizes through $U_{i, Z} \subseteq Z$.
	\end{enumerate}
The above assertion \eqref{proof:prop:proj:represent-ii-*} is a direct consequence of the natural homotopy equivalence 
	$$\Map_{\CAlgDelta}(B[\frac{1}{f_i}], C) \to \Map_{\CAlgDelta}^0(B, C)$$
of \cite[Proposition 4.1.18 (1)]{DAGV}, where $\Map_{\CAlgDelta}^0(B, C)$ denotes the union of those connected components of $\Map_{\CAlgDelta}(B, C)$ spanned by those maps $B \to C$ which carry $f_i \in \pi_0(B)$ to an invertible element of $\pi_0(C) (\simeq \pi_0(\sO_Y))$. 

To prove assertion \eqref{proof:prop:proj:represent-iii}, it suffices to show that for each map $g \colon Z = \Spec B \to P_{\sF}$ from an affine derived scheme $Z = \Spec B$, the family of derived open subschemes $\{U_{i, Z}\}_{i \in I}$ forms an Zariski open cover of $Z$, where $U_{i, Z}$ is the open subscheme of $Z$ that represents the functor \eqref{eqn:proj:rep:Ui-Z} of assertion \eqref{proof:prop:proj:represent-ii}. To prove this, we let $z \in |\Spec B| = |\Spec \pi_0(B)|$ be any point, and let $\kappa$ denote the residue field of $\pi_0(B)$ at $z$. Since the map $(f_{i, Z})_{i \in I} \colon \sO_Z^{\oplus I} \to \sL$ is surjective on $\pi_0$, we obtain that the induced map $(f_{i,Z} \otimes_{\pi_0(B)} \kappa)_{i \in I} \colon \kappa^{\oplus I} \to \kappa$ is a surjection of vector spaces. Therefore, there exists $i \in I$ such that $f_{i, Z} \otimes_{\pi_0(B)} \kappa$ is an isomorphism. It follows from the above property \eqref{proof:prop:proj:represent-ii-*} that $z \in U_{i, Z}$. This proves assertion \eqref{proof:prop:proj:represent-iii}.
 \end{proof}

\begin{definition} We let $\pr \colon \PP(\sF) \to X$ denote the relative derived scheme over $X$ which represents the prestack $P_{\sF}$, and refer to it as the {\em derived projectivization of $\sF$ (over $X$)}. We will let $\sO(1)$ (or $\sO_{\PP(\sF)}(1)$, if we want to address the dependence on $\PP(\sF)$) denote the universal line bundle on $\PP(\sF)$ and let $\rho \colon \pr^* \sF \to \sO(1)$ denote the tautological map that is surjective on $\pi_0$. We will refer $\rho$ as the {\em tautological quotient map} (or {\em tautological surjection}).
\end{definition}

\begin{remark}[Affine open cover $\{U_i\}$] \label{remark:proj:Ui} Let $X$ be a prestack and $\sF$ a connective quasi-coherent sheaf on $X$. The proof of Proposition \ref{prop:proj:represent} shows that whenever there exists a family of sections $\{t_i \colon \sO_X \to \sF\}_{i \in I}$ that is jointly surjective on $\pi_0$ \footnote{Notice by virtue of Proposition \ref{prop:proj-4,5} \eqref{prop:proj-5}, $\PP(\sF) \simeq \PP(\sF \otimes \sL)$ for any line bundle $\sL$, so that it suffices to find a generating family of sections for $\sF \otimes \sL$. For example, if $X$ is a quasi-projective derived scheme, up to twisting $\sF$ by a line bundle, we could always find such a generating family of sections for $\sF$.}, there is a Zariski open cover $\{U_i \subseteq \PP(\sF)\}_{i \in I}$ of $\PP(\sF)$, where $U_i \subseteq \PP(\sF)$ is the open subfunctor which is characterized by the following universal property:
	\begin{enumerate}[label=$(*')$, ref=$*'$]
		\item  \label{remark:prop:proj:represent*'}
		Let $h \colon Y \to \PP(\sF)$ be any map of prestacks. Then $h^*(\rho \circ (\pr^* t_i))$ is an epimorphism on $\pi_0$ if and only if $h$ factorizes through $U_{i} \subseteq \PP(\sF)$.
	\end{enumerate}
In particular, over $U_i$, the composite map of the restrictions
	$$f_i \colon \sO_{U_i} \xrightarrow{(\pr^*t_i)|_{U_i}} \pr_{U_i}^* \sF \xrightarrow{\rho|_{U_i}} \sO(1)|_{U_i}$$
is surjective on $\pi_0$, where $\pr_{U_i}$ is the composite map $U_i \subseteq \PP(\sF) \xrightarrow{\pr} X$. Therefore, $f_i$ is invertible, and we let $f_i^{-1}$ denote an inverse of $f_i$. (We may regard $f_i^{-1}$ as a section of $\sO(-1)|_{U_i}$.) From the proof of assertion \eqref{proof:prop:proj:represent-ii} of Proposition \ref{prop:proj:represent}, the cosection induced by the composition 
	$$\pr_{U_i}^* \sF \xrightarrow{\rho|_{U_i}} \sO(1)|_{U_i} \xrightarrow{f_i^{-1}} \sO_{U_i}$$
is classified by a closed immersion $s_{U_i} \colon U_i \to \VV(\sF)$, 
which fits into a pullback square
 	$$
	\begin{tikzcd}
	U_i \ar{r}{s_{U_i}}  \ar{d}[swap]{\pr_{U_i}} & \VV(\sF) \ar{d}{\VV(t_i^*)}\\
	X \ar{r}{s_X} & \VV(\sO_X),
	\end{tikzcd}
	$$
where $s_X \colon X \to \VV(\sO_X)$ is the section which classifies the identity map. If $X = \Spec R$ is an affine derived scheme, then $U_i$ is equivalent to the spectrum of the simplicial commutative $R$-algebra $R \otimes_{R[T]} \Sym_R^*(\sF)$, where $T$ is an indeterminate, $R[T] \simeq \Sym_R^*(\sO_X) \to \Sym_R^*(\sF)$ is the map induced by $t_i \colon \sO_X \to \sF$, and $R[T] \to R$ is the evaluation map $T \mapsto 1_R$.
\end{remark}

\begin{proposition} \label{prop:proj-classical} If $X$ is a derived scheme and $\sF$ a connective quasi-coherent sheaf on $X$. Then $\PP(\sF)$ is a derived scheme, and the underlying classical scheme of $\PP(\sF)$ is canonically identified with the classical projectivization $\PP_\cl(\pi_0(\sF)) = \Proj_{X_\cl}^\cl \Sym_{\cl}^*(\pi_0 (\sF)) \to X_\cl$ of the discrete sheaf $\pi_0(\sF) \in \QCoh(X_\cl)^\heartsuit$ over the classical scheme $X_\cl = (|X|, \pi_0(\sO_X))$.
\end{proposition}

\begin{proof} Let $\{U_i\}_{i \in I}$ be the Zariski open cover of $\PP(\sF)$ that constructed in the proof of Proposition \ref{prop:proj:represent}, then each $U_i$ is a relative derived affine schemes over $X$. By virtue of Proposition \ref{prop:affinecone} \eqref{prop:affinecone-1} and the proof of Proposition \ref{prop:proj:represent}, the family of classical relative affines schemes $\{(U_i)_\cl\}_{i \in I}$ over $X_{\cl}$ forms an Zariski open cover of the classical projectivization $\PP_{\cl}(\pi_0(\sF))$.
\end{proof}

\begin{remark}[Classical projectivization functors] \label{rmk:classical:proj:functor} Let $X$ be prestack and $\sF$ a connective quasi-coherent sheaf over $X$. Let $R \in \CAlg^\heartsuit$ be an ordinary commutative ring and let $\eta \colon T=\Spec R \to X$ be a map of prestacks (which necessarily factorizes through the inclusion $X_\cl \subseteq X$ if $X$ is a derived scheme). For any line bundle $\sL$ on $T$, there are canonical homotopy equivalences
	$$\Map_{\QCoh(T)}(\eta^* \sF, \sL) \simeq \Map_{\QCoh(T)^\heartsuit}(\pi_0(\eta^* \sF), \sL) \simeq \Map_{\QCoh(T)^\heartsuit}(\eta_{\cl}^* (\pi_0 \sF), \sL).$$
(Here $\eta_\cl^* = \pi_0 \circ \eta^*|_{\QCoh(X_\cl)^\heartsuit} \colon \QCoh(X_\cl)^\heartsuit \to \QCoh(T)^\heartsuit$ denotes the classical pullback of discrete sheaves.) Therefore, we deduce that the restriction of functor $\PP(\sF)$ to the ordinary category of commutative rings $\PP(\sF)|_{\CAlg^\heartsuit} \colon \CAlg^\heartsuit \to \shS$ is canonically equivalent to the classical projectivization functor $\PP_\cl(\pi_0 \sF)$ that carries each pair $(R \in \CAlg^\heartsuit, \eta \colon \Spec R \to X_\cl)$ to the space of line bundle quotients $\eta_\cl^* (\pi_0 \sF) \twoheadrightarrow \sL$ in $\QCoh(\Spec R)^\heartsuit$. In the case where $X$ is a derived scheme, this provides anther proof of Proposition \ref{prop:proj-classical}. 

Moreover, combined with Remark \ref{rmk:Grothendieck:affinecone}, we see that for any field $\kappa$ and a morphism $\eta \colon \Spec \kappa \to X$, the fiber product $\PP(\sF) \times_X \Spec \kappa$ is canonically equivalent to the classical projective space $\PP_\kappa(\pi_0(\eta^*\sF))$ of the $\kappa$-vector space $\pi_0(\eta^*\sF) \simeq \pi_0(\eta^*(\pi_0\sF))$, whose $\kappa$-valued points canonically correspond to the elements of the quotient set $(\Ext^0(\pi_0(\sF), \kappa) - \{0\})/\kappa^\times$. 
\end{remark}

 \begin{example} Let $X = \Spec \ZZ$ and let $\sF$ be represented by the complex $\ZZ \xrightarrow{0} \ZZ$. Then for any ordinary commutative ring $R \in \CAlg^\heartsuit$ and any map $\eta \colon \ZZ \to R$, the space of $\eta$-value points $\PP(\sF)(\eta)$ is homotopy equivalent to the singleton $\{f \colon \sF \otimes R \to R\}$, where $\pi_1(f) =0$ and $\pi_0(f) = \id_R$, that is, $\Map_{/\Spec \ZZ}(\Spec R, \PP(\sF)) \simeq \{*\}$. In other words, when restricting to the subcategory of ordinary commutative rings, the derived projectivization functor $\PP(\sF)$ is the same as the classical projectivization functor $\PP(\sF)_\cl = \Spec \ZZ$. The difference only happens when we test the values of $\PP(\sF)$ on {\em non-discrete} simplicial commutative rings. For example, let $\ZZ[\varepsilon] = \Sym_\ZZ^*(\Sigma \, \ZZ) = \ZZ \oplus \ZZ \varepsilon$ denote the ring of {\em derived dual numbers}, where we informally think of $\varepsilon$ as a generator at homological degree $1$ (which necessarily satisfies $\varepsilon^2 =0$). Let $\eta \colon \ZZ \to \ZZ[\varepsilon]$ denote the canonical map. Then the space $\PP(\sF)(\eta)$ is homotopy equivalent to the set
 	$$ \left\{\ZZ[\varepsilon] \oplus \Sigma (\ZZ[\varepsilon]) \xrightarrow{(\id, m \cdot \varepsilon)} \ZZ[\varepsilon] \right \}_{m \in \ZZ} \simeq \ZZ$$
which is not homotopy equivalent to the space $\PP(\sF)_\cl(\eta) = \{*\}$. (In fact, we could obtain from the later Example \ref{eg:prop:proj:PB} that there is a canonical equivalence $\PP(\sF) \simeq \Spec (\ZZ[\varepsilon])$.)
 \end{example}

\begin{corollary} Let $X$ be a prestack, and suppose $\sF \in \QCoh(X)^\cn$ is locally of finite type (that is, perfect to order $0$ in the sense of Definition \ref{def:perfectfp:prestacks}). Then $\pr\colon \PP(\sF) \to X$ is proper.
\end{corollary}
\begin{proof} The assertion being local, we may assume $X$ is representable by $\Spec R$, where $R$ is a simplicial commutative ring. The condition of the corollary implies that $\pi_0(\sF)$ is a finitely presented $\pi_0(R)$-module in the sense of classical commutative algebra. Therefore, the classical projectivization $\PP_{\cl}(\pi_0(\sF))$ is proper over $X_\cl$ (see \cite[\href{https://stacks.math.columbia.edu/tag/01WC}{Tag 01WC}]{stacks-project}). We deduce from Proposition \ref{prop:proj-classical} that $\PP(\sF) \to X$ is proper (see \cite[Remark 5.0.0.1]{SAG}).
\end{proof}

\begin{proposition} \label{prop:proj-4,5} Let $X$ be a prestack and $\sF$ a connective quasi-coherent sheaf on $X$.
\begin{enumerate}
	\item \label{prop:proj-4} (The formation of projectivizations commutes with base change) 
	Let $f \colon X' \to X$ be a morphism of prestacks. Then there is a natural equivalence $\PP(f^* \sF) \simeq \PP(\sF) \times_X X'$ such that the pullback of the universal line bundle (resp. the tautological quotient map) of $\PP(\sF)$ along $f \colon X' \to X$ is canonically equivalent to the universal line bundle (resp. the tautological quotient map) of $\PP(f^*\sF)$.
	\item \label{prop:proj-5} (Tensoring with line bundles) 
	Let $\sL$ be a line bundle on $X$. Then there is a natural equivalence $g \colon \PP(\sF) \simeq \PP(\sF \otimes \sL)$ such that there are canonical equivalences $g^* (\sO_{\PP(\sF \otimes \sL)}(1)) \simeq \sO_{\PP(\sF)}(1) \otimes \pr_{\PP(\sF)}^* \sL$ and 
		$$g^*\big(\pr_{\PP(\sF \otimes \sL)}^* (\sF\otimes \sL) \xrightarrow{\rho_{\PP(\sF \otimes \sL)}} \sO_{\PP(\sF \otimes \sL)}(1)\big) \simeq \big(\pr_{\PP(\sF)}^* \sF \xrightarrow{\rho_{\PP(\sF)}} \sO_{\PP(\sF)}(1)\big) \otimes \pr_{\PP(\sF)}^* \sL,$$
where $\rho_{\PP(\sF)}$ and $\rho_{\PP(\sF \otimes \sL)}$ are the tautological quotient maps.
\end{enumerate}
\end{proposition}
\begin{proof}
We first prove assertion \eqref{prop:proj-4}. For any $A \in \CAlgDelta$ and $\eta' \colon \Spec A \to X'$, we let $\eta \colon \Spec A \to X$ denote the composition $f \circ \eta'$. Then there is a equivalence of spaces from $\PP(f^*\sF)(\eta')$ to $\PP(\sF)(\eta) \times_{\{\eta\}} \{\eta'\}$ given by sending $(\eta'^*(f^* \sF) \to \sL)$ to $(\eta^* \sF \to \sL)$.

Assertion \eqref{prop:proj-5} is proved similarly. For any $A \in \CAlgDelta$ and $\eta \colon \Spec A \to X$, we define a map $g(\eta)$ of spaces from $\PP(\sF)(\eta)$ to $\PP(\sF \otimes \sL)(\eta)$ by sending $(\eta^* \sF \to \sL')$ to $(\eta^* (\sF \otimes \sL) \to (\sL' \otimes \sL))$. Each map $g(\eta)$ is an equivalence of spaces since it has an obvious inverse given by sending $(\eta^* (\sF \otimes \sL) \to \sL'')$ to $(\eta^* \sF \to \sL'' \otimes \sL^{-1})$.
\end{proof}

\begin{proposition}[Finiteness properties] \label{prop:proj:finite}
 Let $X$ be a prestack and $\sF$ a connective quasi-coherent sheaf on $X$. If $\sF$ is perfect to order $n$ for some integer $n \ge 0$, (resp. 
 almost perfect, perfect; see Definition \ref{def:perfectfp:prestacks}), then the morphism $\pr \colon \PP(\sF) \to X$ is locally of finite presentation to order $n$ (resp. locally almost of finite presentation, locally of ﬁnite presentation). 
\end{proposition}

\begin{proof} It suffices to prove in the representable case where $X = \Spec R$. Let $\{U_i \}_{i \in I}$ be the Zariski open cover of $\PP(\sF)$ associated to a generating family of sections $\{t_i \colon \sO_X \to \sF\}_{i \in I}$ constructed in the proof of Proposition \ref{prop:proj:represent}. As the desired assertion for the morphism $\PP(\sF) \to X$ is local on the source with respect to Zariski topology, it suffices to prove the assertion for each $U_i$. Consider the pullback square 
 	$$
	\begin{tikzcd}
	U_i \ar{r}{s_{U_i}}  \ar{d}[swap]{\pr |_{U_i}} & \VV(\sF) \ar{d}{\VV(t_i^*)}\\
	X \ar{r}{s_X} & \VV(\sO_X)
	\end{tikzcd}
	$$
of Remark \ref{remark:proj:Ui}. If $\sF$ is perfect to order $n$, (resp. 
 almost perfect, perfect), then the projection $\VV(\sF) \to X$ is locally of finite presentation to order $n$ (resp. locally almost of finite presentation, locally of ﬁnite presentation), by virtue of Proposition \ref{prop:affinecone} \eqref{prop:affinecone-6}. Since the projection $\VV(\sO_X)\simeq \AA_X^1 \to X$ is smooth and locally of finite presentation, it follows that the morphism $\VV(\sF) \to \VV(\sO_X)$ is locally of finite presentation to order $n$ (resp. locally almost of finite presentation, locally of ﬁnite presentation), whence $U_i \to X$ has the same property from base change.
\end{proof}

\begin{theorem}[Euler fiber sequences] \label{thm:proj:Euler} Let $X$ be a prestack and $\sF$ a connective quasi-coherent sheaf on $X$. Let $\pr \colon \PP(\sF) \to X$ be the derived projectivization of $\sF$. Then the projection $\pr \colon \PP(\sF) \to X$ admits a relative cotangent complex $L_{\PP(\sF)/X} $. Moreover, there is a natural fiber sequence of connective quasi-coherent sheaves 
		$$L_{\PP(\sF)/X}  \otimes_{\sO_{\PP(\sF)}} \sO(1) \to \pr^* \sF \xrightarrow{\rho} \sO(1)$$
on $\PP(\sF)$, where $\rho$ is the tautological quotient map. We will refer to the above sequence as the {\em Euler fiber sequence}. Consequently, if $\sF$ is almost perfect (resp. perfect, of Tor-amplitude $\le n$), then $L_{\PP(\sF)/X}$ is almost perfect (resp. perfect, of Tor-amplitude $\le n$). 
\end{theorem}

We will deduce the above theorem from a general result Proposition \ref{prop:cotangent:QCPerf} regarding cotangent complexes for maps of modules. Before stating the result, let us review the definition of relative cotangent complexes of morphisms between prestacks.

\begin{definition}[See {\cite[\S 3.2]{DAG}}] \label{def:relativecotangent} Let $f \colon X \to Y$ be a natural transformation between functors $X, Y \colon \CAlgDelta \to \shS$. We say that {\em $f$ admits a relative cotangent complex} if there exists an almost connective quasi-coherent sheaf $L_{X/Y} \in \QCoh(X)^{\rm acn}$ such that for each $A \in \CAlgDelta$ and $\eta \in X(A)$,  the functor $F_\eta \colon \Mod_A^\cn \to \shS$,
	$$F_\eta(M)  = {\rm fib} (X(A \oplus M)  \to X(A) \times_{Y(A)} Y(A \oplus M))$$
(where the fiber is taken over the point of $X(A) \times_{Y(A)} Y(A \oplus M))$ determined by $\eta$) is corepresentable by $\eta^* L_{X/Y}$. If the above condition is satisfied, the (essentially unique) quasi-coherent sheaf $L_{X/Y} \in \QCoh(X)^{\rm acn}$ is called the {\em relative cotangent complex of $f$}.
\end{definition}

\begin{remark}[Compare with {\cite[Remark 17.2.4.3]{SAG}}] \label{rmk:cotangent:ab} Unwinding the definition, we see that $f \colon X \to Y$ admits a relative cotangent complex if and only if the following conditions are satisﬁed:
	\begin{enumerate}
		\item[$(a)$] For each $A \in \CAlgDelta$ and $\eta \in X(A)$, the above defined functor $F_\eta \colon \Mod_A^{\rm cn} \to \shS$ is corepresented by an almost connective $A$-module $M_\eta$. Notice that by virtue of \cite[ Example 17.2.1.4]{SAG}, the almost connective module $M_\eta$ that corepresents the functor $F_\eta$, if exists, is uniquely determined up to contractible ambiguity. 
		\item[$(b)$] The module $M_\eta$ that corepresents $F_\eta$ is functorial on $A$ in the strong sense: for each map $A \to B$ in $\CAlgDelta$ and $\eta \in X(A)$, let $\eta' \in X(B)$ denote the image of $\eta$. Then the functor $F_{\eta'}$ is corepresented by $B \otimes_A M_\eta$. 
	\end{enumerate}
If both $(a)$ and $(b)$ are satisfied, then the cotangent complex $L_{X/Y} \in \QCoh(Y)^{\rm acn}$ is described by the formula $\eta^* L_X = M_\eta \in \Mod_A$ for $\eta \in X(A)$. 
\end{remark}

\begin{proposition}\label{prop:cotangent:QCPerf}Let $\Perf$ be the functor $(R \in \CAlgDelta) \mapsto ((\Mod_R^{\perf})^{\simeq} \in \shS)$, and let ${\rm QC}^{\rm acn}$ be the functor $(R \in \CAlgDelta) \mapsto ((\Mod_R^{\rm acn})^{\simeq} \in \widehat{\shS})$, where $\Mod_R^{\rm acn}$ is the full subcategory of $\Mod_R$ spanned by almost connective modules, and $(\blank)^\simeq$ is the operation of taking core of an $\infty$-category (that is, discarding non-invertible morphisms). Let $\Mod_{\Delta^1}$ be the functor $(R \in \CAlgDelta) \mapsto (\Fun(\Delta^1, \Mod_R)^{\simeq} \in \widehat{\shS})$ so that the objects of $\Mod_{\Delta^1}(R)$ are morphisms of $R$-modules $M \to N$. Let $F$ be the subfunctor of $\Mod_{\Delta^1}$ whose $R$-points are the full subcategories $F(R) \subseteq \Mod_{\Delta^1}(R)$ spanned by those $R$-module morphisms $M \to N$ for which $M$ is almost connective and $N$ is perfect. Then there is a natural projection map $F \to {\rm QC}^{\rm acn} \times \Perf$ that carries $(M \to N) \in F(R)$ to $(M, N) \in  {\rm QC}^{\rm acn}(R) \times\Perf(R)$.
	\begin{enumerate}
		\item \label{prop:cotangent:QCPerf-1} The projection $\alpha \colon F \to {\rm QC}^{\rm acn} \times \Perf$ admits a relative cotangent complex $L_{F/({\rm QC}^{\rm acn} \times \Perf)}$. Moreover, for each $R$-point $\eta \in F(R)$ that classifies a map of  $R$-modules $M \to N$, $\eta^*L_{F/({\rm QC}^{\rm acn} \times \Perf)}$ can be canonically identified with $M \otimes_R N^\vee$.
		\item \label{prop:cotangent:QCPerf-2} The prestack $\Perf$ admits an absolute cotangent complex $L_{\Perf}$. Moreover, for each $R$-point $\eta \in \Perf(R)$ that classifies a perfect $R$-module $N$, $\eta^*L_{\Perf}$ can be canonically identified with $\Sigma^{-1}\, N \otimes_R N^\vee$.
		\item \label{prop:cotangent:QCPerf-3} The projection $F \to {\rm QC}^{\rm acn}$ admits a relative cotangent complex $L_{F/  {\rm QC}^{\rm acn}}$. Moreover, for each $R$-point $\eta \in F(R)$ that classifies a map of $R$-modules $f\colon M \to N$, $\eta^*L_{F/ {\rm QC}^{\rm acn}}$ can be canonically identified with ${\rm fib}(M \xrightarrow{f} N)\otimes_R N^\vee$.
	\end{enumerate}
\end{proposition}

\begin{proof} We first prove \eqref{prop:cotangent:QCPerf-1}. For each $\eta \in F(R)$ that classifies a map of $R$-modules $M \to N$ for which $M$ is almost connective and $N$ is perfect, and $I \in \Mod_{R}^\cn$, let $G_\eta(I)$ denote the fiber 
	$${\rm fib}\left(F(R \oplus I) \to F(R) \times_{{\rm QC}^{\rm acn}(R) \times \Perf(R)} ({\rm QC}^{\rm acn}(R\oplus I) \times \Perf(R \oplus I)) \right)$$
over the point $(M,N)$. (Here $R \oplus I$ is the trivial square-zero extension of $R$ by $I$, see \cite[Construction 25.3.1.1]{SAG}). Unwinding the definition, we see that 
	\begin{align*}
		G_{\eta}(I)
		& \simeq \Map_{\Mod_{R \oplus I}}(M \otimes_R (R\oplus I), N \otimes_R (R\oplus I)) \times_{\Map_{\Mod_R}(M,N)} \{M \to N\} \\
		&\simeq \Map_{\Mod_{R}}(M, N \otimes_R (R\oplus I)) \times_{\Map_{\Mod_R}(M,N)} \{M \to N\}  \\
		& \simeq \Map_{\Mod_{R}}(M, N \otimes_R I) \simeq \Map_{\Mod_R} (M \otimes_R N^\vee, I).
	\end{align*}
Therefore, the functor $I \mapsto G_{\eta}(I)$ is almost corepresentable by $M \otimes_R N^\vee$. The module $M \otimes_R N^\vee$ is functorial on $R$ in the strong sense: for each $R \to R'$, $(M \otimes_{R} N^\vee) \otimes_{R} R' \simeq (M \otimes_R R') \otimes_R' (N \otimes_R R')^\vee)$. Thus, we have verified the conditions $(a)$ and $(b)$ of Remark \ref{rmk:cotangent:ab}, and we deduce that the projection map $\alpha \colon F \to {\rm QC}^{\rm acn} \times \Perf$ admits a relative cotangent complex which is described by $\eta^* L_{\alpha} \simeq G_{\eta}$. This proves  \eqref{prop:cotangent:QCPerf-1}.

The analogue of assertion \eqref{prop:cotangent:QCPerf-2} in the spectral setting is proved in \cite[Corollary 19.2.2.3, Remark 19.2.2.4]{SAG}. We now show that a similar approach works in the derived context. Let $R \in \CAlgDelta$, $\eta \in \Perf(R)$ that corresponds to a perfect $R$-module $N$. Let $H_\eta \colon \Mod_R^\cn \to \shS$ be the functor that carries $I$ to the space $\Perf(R\oplus I) \times_{\Perf(R)} \{ \eta\}$.
Let $d_0 \colon R \to R \oplus \Sigma I$ be the map of simplicial algebras that corresponds to the trivial derivation of $\Sigma I$ (see \cite[Definition 25.3.1.4]{SAG}), so that there is a pullback diagram
	$$
	\begin{tikzcd}
		R \oplus I \ar{r} \ar{d} & R  \ar{d}{d_0} \\
		R \ar{r}{d_0} & R\oplus \Sigma I.
	\end{tikzcd}
	$$
Since $d_0$ is surjective on $\pi_0$, we can apply \cite[Proposition 5.6.2]{DAG} and deduce that digram of $\infty$-categories induced by tensor functors
	$$
	\begin{tikzcd}
		\Mod_{R \oplus I}^{\rm acn}  \ar{r} \ar{d} & \Mod_R^{\rm acn}  \ar{d}{u^*} \\
		\Mod_R^{\rm acn} \ar{r} & \Mod_{R\oplus \Sigma I}
	\end{tikzcd}
	$$
is a pullback square. In other words, the induced map 
	\begin{equation}\label{eqn:prop:cotangent:pullback}
	\Mod_{R \oplus I}^{\rm acn} \to \Mod_R^{\rm acn} \times_{\Mod_{R \oplus I}} \Mod_R^{\rm acn}
	\end{equation}
is an equivalence. In particular, passing to homotopy fibers over $N$, we obtain that
	\begin{align*}
	 \Mod_{R \oplus I}^\cn \times_{\Mod_R} \{N\} 
		& \simeq \Map_{\Mod_{R \oplus \Sigma I}}(N \otimes_R (R \oplus \Sigma I), N \otimes_R(R \oplus \Sigma I)) \times_{\Map_{\Mod_R(N,N)}}  \{\id_N\}. \\
		&\simeq \Map_{\Mod_{R}}(N, N \otimes_R (R\oplus \Sigma I)) \times_{\Map_{\Mod_R}(N,N)} \{ \id_N \}  \\
		& \simeq \Map_{\Mod_{R}}(N, N \otimes_R \Sigma I) \simeq \Map_{\Mod_R} (\Sigma^{-1} \, N \otimes_R N^\vee, I).
	\end{align*}
(The first equivalence follows from the fact that any lifting of $\id_{N}$ to an endomorphism of $N \otimes_R(R \oplus \Sigma I)$ is automatically invertible; Compare with the argument before \cite[Proposition 19.2.2.2]{SAG}). Therefore the functor $I \mapsto H_{\eta}(I)$ is almost corepresentable by $\Sigma^{-1}\, N \otimes_R N^\vee$. $\Sigma^{-1} \, N \otimes_R N^\vee$ is evidently functorial on $R$ in the strong sense as in  \eqref{prop:cotangent:QCPerf-1}. This proves \eqref{prop:cotangent:QCPerf-2}.

We now prove assertion \eqref{prop:cotangent:QCPerf-3}. Notice that from the fiber sequence
	$$ \alpha^* L_{({\rm QC}^{\rm acn} \times \Perf) / {\rm QC}^{\rm acn}} \to L_{F/ {\rm QC}^{\rm acn}} \to L_{F/({\rm QC}^{\rm acn} \times \Perf)}$$
we immediate obtain that $L_{F/ {\rm QC}^{\rm acn}}$ exists and there is a canonical equivalence $\eta^*L_{F/ {\rm QC}^{\rm acn}} \simeq {\rm fib} (M \otimes_R N^\vee \xrightarrow{g} N \otimes_R N^\vee)$. To show that the above map $g$ coincides with the map $f \otimes \id_{N^\vee}$, we will give explicit descriptions of the above fiber sequence in our case.

Let $R \in \CAlgDelta$ and fix an $R$-point $\eta \in F(R)$ which corresponds to a map of $R$-modules $f \colon M \to N$, where $M$ is almost connective and $N$ is perfect. We let $F_\eta \colon \Mod_R^\cn \to \shS$ denote the functor corepresented by $\eta^*L_{F/ {\rm QC}^{\rm acn}}$, that is, for each $I \in \Mod_{R}^\cn$, $F_\eta(I)$ is the space
	$$F_\eta(I) : = {\rm fib}\left(F(R \oplus I) \to F(R) \times_{{\rm QC}^{\rm acn}(R) } ({\rm QC}^{\rm acn}(R\oplus I)\right).$$
Invoking \eqref{eqn:prop:cotangent:pullback} and arguing as the proof of assertion \eqref{prop:cotangent:QCPerf-2}, we can identify $F_\eta(I)$ with the space of commutative squares $\sigma \colon \Delta^1 \times \Delta^1 \to \Mod_{R \oplus \Sigma I}$, depicted as
	$$
	\begin{tikzcd}
		u^* M \ar{d}[swap]{\id_{u^* M}} \ar{r}{\widetilde{f}}& u^* N \ar{d}{\alpha_N} \\
		u^*M  \ar{r}{\widetilde{f}'}  & u^*N,
	\end{tikzcd}
	$$
which is a lifting of the degenerate commutative square $\sigma_0 \colon \Delta^1 \times \Delta^1 \to \Mod_R$ depicted as
	$$
	\begin{tikzcd}
		 M \ar{d}[swap]{\id_{M}} \ar{r}{f}&  N \ar{d}{\id_N} \\
		M  \ar{r}{f}  & N.
	\end{tikzcd}
	$$
Similarly, invoking \eqref{eqn:prop:cotangent:pullback} again, we can identify the space $Z = \Map_{\Mod_R}(M, N\otimes_R \Sigma I)$ with space of those maps $\widetilde{f} \colon u^* M \to u^* N$ in $\Mod_{R \oplus \Sigma I}$ which is a lifting $f$. If we let $Y = H_\eta(I) \simeq \Map_R(N, N \otimes \Sigma I)$, then there is a natural map of spaces $h = f^*\colon Y \to Z$ induced by pullback along $f$, that is, $h$ carries $\alpha_N \in Y$ to $\alpha_N \circ u^*f \in Z$. Then it follows the above description of $F_\eta(I)$ that it is homotopy equivalent to the homotopy fiber of $h \colon Y \to Z$,
	$${\rm hFib}(h) = Y \times_{\Fun(\{1\}, Z)} \times \Fun(\Delta^1, Z) \times_{\Fun(\{0\}, Z)} \{u^* f\}.$$
Notice that the space $G_\eta(I) \simeq \Map_{\Mod_R}(M, N \otimes_R I)\simeq \Omega \Map_{\Mod_R}(M, N\otimes_R \Sigma I)$ of assertion \eqref{prop:cotangent:QCPerf-1} is naturally identified with the loop space $\Omega Z$ of $Z$ (with base point $z_0$ that corresponds to $u^* f \in \Map_{\Mod_{R \oplus \Sigma I}}(u^* M, u^* N)$). Therefore, the natural inclusion map of spaces $G_\eta(I) \to F_\eta(I)$ coincide with the inclusion map $\Omega Z \to {\rm hFib}(h)$ that carries a loop $\gamma$ to $(y_0, \gamma)$, where $y_0 \in Y$ is the base point that corresponds to $u^* \id_N$. We thus deduce from the fiber sequence of spaces $\Omega Z \to {\rm hFib}(h) \to Y$ that there is a fiber sequence of functors
	$G_\eta \to F_\eta \to H_\eta.$
Since the space $Y= H_\eta(I)$ is naturally pointed (with base point $y_0$), there is a naturally induced fiber sequence of spaces (see, for example, \cite[Chapter 8, \S 6]{May})
	$$\Omega Y \xrightarrow{-\Omega h} \Omega Z \to {\rm hFib}(h),$$
where the map $-\Omega h \colon \Omega H_\eta(I) \to G_\eta(I)=\Omega Z$ is induced from $h=f^* \colon Y \to Z$; We could informally depict the map $-\Omega h$ as
	$$
	\begin{tikzcd}
	u^* N \arrow[r, bend left=50, ""{name=U, below}, "\id_{u^* N}"]
	\arrow[r, bend right=50, ""{name=D}, "\id_{u^* N}"{below}]
	& u^* N
	\arrow[phantom, from=U, to=D, " \tau"]
	\end{tikzcd}
	\mapsto
		\begin{tikzcd}
		u^* M \arrow[r, bend left=50, ""{name=LU, below}, "u^*f"]
	\arrow[r, bend right=50, ""{name=LD}, "u^*f"{below}]
	&u^* N \arrow[r, bend left=50, ""{name=U, below}, "\id_{u^* N}"]
	\arrow[r, bend right=50, ""{name=D}, "\id_{u^* N}"{below}]
	\arrow[phantom, from=LU, to=LD, " u^* \sigma_0"]
	& u^* N.
	\arrow[phantom, from=U, to=D, " \tau"]
	\end{tikzcd}
	$$
Therefore, we obtain a fiber sequence of functors $\Omega H_\eta \xrightarrow{-\Omega f^*} G_\eta \to F_\eta$. This implies the desired equivalence $\eta^*L_{F/ {\rm QC}^{\rm acn}} \simeq {\rm fib}(M \xrightarrow{f} N)\otimes_R N^\vee$ of assertion \eqref{prop:cotangent:QCPerf-3}.
\end{proof}

\begin{proof}[Proof of Theorem \ref{thm:proj:Euler}] Let $F'$ be the subfunctor of $F$ such that for each $R \in \CAlgDelta$, $F'(R)$ is the full subcategory of $F(R)$ spanned by those maps of $R$-modules $M \to N$ for which $N$ is a line bundle. Let $F''$ be the subfunctor of $F'$ such that for each $R \in \CAlgDelta$, $F''(R)$ is spanned by those maps of $R$-modules $M \to N$ which are surjective on $\pi_0$. By virtue of \cite[Proposition 19.2.4.5]{SAG}, we see that $F' \subseteq F$ is an open immersion. By the same argument as the proof of assertion \eqref{proof:prop:proj:represent-ii} of Proposition \ref{prop:proj:represent}, we see that the inclusion $F'' \subseteq F'$ is an open immersion. Therefore we deduce that $F''$ admits a relative cotangent complex over ${\rm QC}^{\rm acn}$ and the relative cotangent complex of $F'' \to {\rm QC}^{\rm acn}$ is given by the restriction of the relative cotangent complex of $F \to {\rm QC}^{\rm acn}$ to the open subfunctor $F'' \subseteq F$.

Let $P_\sF$ denote the functor represented by $\PP(\sF)$. Suppose we are given a pair of objects $(R \in \CAlgDelta, \eta \in P_{\sF}(R))$ which corresponds to a surjective map $\eta^* \sF \to \sL$ in $\Mod_R^\cn$. We let $\eta' \in X(R)$ denote the image of $\eta$ under the projection $P_\sF \to X$. Then $\eta^* \sF \to \sL$ corresponds to a point $\eta'' \in F''(R)$. For each $I \in \Mod_R^\cn$, we let $F''_\eta(I)$ denote the fiber 
	$${\rm fib}\left(P_\sF(R \oplus I) \to P_{\sF}(R) \times_{X(R)} X(R \oplus I) \right).$$ 
Unwinding the definition, we see that $I \mapsto F''_{\eta}(I)$ coincides with the functor 
	$$I \mapsto \Map_{\Mod_R}(\eta''^*L_{F''/{\rm QC}^{\rm acn}}, I).$$
We deduce from Proposition \ref{prop:cotangent:QCPerf}  \eqref{prop:cotangent:QCPerf-3} that $F''_\eta$ is corepresentable by the connective $R$-module $M_\eta = \fib(\eta^* \sF \to \sL) \otimes_R \sL^\vee$. Since the formation of $\eta \mapsto M_\eta$ is functorial on $R$ in the strong sense (that is, for any map of simplicial commutative rings $R \to A$, let $\eta_A \in P_\sF(A)$ denote the image of $\eta$, then $M_{\eta_A} \simeq M_\eta \otimes_R A$), the conditions $(a)$ and $(b)$ of Remark \ref{rmk:cotangent:ab} have been verified. Therefore, we obtain that $P_\sF \to X$ admits a relative cotangent complex which is characterized by the formula 
	$$L_{P_{\sF}/X} \simeq {\rm fib}(\pr^* \sF \to \sO(1)) \otimes_{\sO_{\PP(\sF)}} \sO(-1).$$
This proves Theorem \ref{thm:proj:Euler}.
\end{proof}

\begin{remark}[Cotangent complexes over $U_i$] If $X$ and $\sF$ satisfies the condition of Remark \ref{remark:proj:Ui}, then we let $\{U_i \subseteq \PP(\sF)\}_{i \in I}$ be the corresponding Zariski open cover of $\PP(\sF)$ associated to a generating family of sections $\{t_i \colon \sO_X \to \sF\}_{i \in I}$. As $U_i \to X$ is the base change of the affine morphism $\VV(t^*) \colon \VV(\sF) \to \VV(\sO_X)$ along the section $s_X \colon X \to \VV(\sO_X)$, by virtue of Proposition \ref{prop:affinecone} \eqref{prop:affinecone-6}, we obtain an equivalence
	$${\rm cofib}\big(t_{i, U_i} \colon \sO_{U_i} \to \pr_{U_i}^* \sF\big) \otimes \sO(-1)|_{U_i}  \simeq L_{U_i/X}.$$
Since the composition $f_i = \rho|_{U_i} \circ t_{i, U_i}$ is an isomorphism, by virtue of the property ($TR4$) of the stable $\infty$-category $\QCoh(U_i)$ (see \cite[Theorem 1.1.2.14]{HA}) we obtain a canonical equivalence
	$${\rm fib}\big(\rho|_{U_i} \colon \pr_{U_i}^* \sF \to \sO(1)|_{U_i}\big) \xrightarrow{\sim} {\rm cofib} \big(t_{i, U_i} \colon \sO_{U_i} \to \pr_{U_i}^* \sF\big).$$
Therefore we obtain an equivalence
	$$\varphi_i \colon {\rm fib}\big(\rho|_{U_i} \colon \pr_{U_i}^* \sF \to \sO(1)|_{U_i}\big)  \otimes \sO(-1)|_{U_i} \xrightarrow{\sim} L_{U_i/X}$$
over each $U_i$. However, unlike the classical situation, it seems to be difficult to explicitly glue these equivalences $\varphi_i$ on $U_i$ to an equivalence $\varphi$ on $\PP(\sF)$, as least for two reasons: first, as we are working with $\infty$-categories, the ``cocycle conditions" for gluing $\varphi$ are not conditions but potentially infinite additional data which we need to specify on double, triple, fourfold intersections, etc. Second, since we wish to glue in general non-discrete modules, it is difficult to specify concretely the morphisms $\varphi_i$ and their compatibility condition data. It is due to these difficulties that we adopt the above functorial approach in our proof of Theorem \ref{thm:proj:Euler}.
\end{remark}

\begin{corollary}[Closed immersions] \label{cor:proj:closedimmersion} Let $X$ be a prestack and let $\varphi \colon \sF \to \sE$ be a surjective map of quasi-coherent sheaves on $X$ (that is, the induced map $\pi_0(\sF) \to \pi_0(\sE)$ is an epimorphism in $\QCoh(X)^\heartsuit$). Then 
	\begin{enumerate}
		\item \label{cor:proj:closedimmersion-1} The sujrection $\varphi \colon \sF \to \sE$ induces a canonical closed immersion $\iota \colon \PP(\sE) \to \PP(\sF)$ such that the pullback of the tautological quotient map on $\PP(\sF)$ is canonically equivalent to the tautological quotient map on $\PP(\sE)$. 			
		\item \label{cor:proj:closedimmersion-2}The closed immersion $\iota \colon \PP(\sE) \to \PP(\sF)$ admits a relative cotangent complex $\LL_{\PP(\sE)/\PP(\sF)}$ which can be described by the formula
				$$L_{\PP(\sE)/\PP(\sF)} \simeq \pr_{\PP(\sE)}^*({\rm cofib}(\varphi \colon \sF \to \sE)) \otimes_{\sO_{\PP(\sE)}} \sO_{\PP(\sE)}(-1).$$
	\end{enumerate}
\end{corollary}
\begin{proof} 
For each $A \in \CAlgDelta$ and $\eta \colon \Spec A \to X$, $\varphi$ induces a natural map of spaces from $\PP(\sE)(\eta)$ to $\PP(\sF)(\eta)$ by sending $(\eta^* \sE \to \sL)$ to the composition $(\eta^* \sF \xrightarrow{\eta^* \varphi} \eta^* \sE \to \sL)$. This defines a natural transform $\iota \colon \PP(\sE) \to \PP(\sF)$ such that $\iota^*$ preserves the tautological quotients. To show $\iota$ is a closed immersion, we may reduce to the case where $X = \Spec R$. Let $\{t_i \colon \sO_X \to \sF\}_{i \in I}$ be a family of sections that is jointly surjective on $\pi_0$, then by our assumption on $\varphi$, $\{\varphi \circ t_i \colon \sO_X \to \sE\}_{i \in I}$ is also jointly surjective on $\pi_0$. As in the proof of Proposition \ref{prop:proj:represent}, we let $\{U_i \}_{i \in I}$ and $\{V_i \}_{i \in I}$ denote the Zariski open covers of $\PP(\sF)$ and $\PP(\sE)$  associated to $\{t_i \colon \sO_X \to \sF\}_{i \in I}$ and $\{\varphi \circ t_i \colon \sO_X \to \sE\}_{i \in I}$, respectively. Then by construction, $V_i$ is the inverse image of $U_i$ from the map $\iota \colon \PP(\sE) \to \PP(\sF)$, and it suffices to show each restriction $\iota|_{V_i} \colon V_i \to U_i$ is a closed immersion. The argument of Remark \ref{remark:proj:Ui} shows that $\iota|_{V_i} \colon V_i \to U_i$ is the base change of the map $\VV(\varphi^*) \colon \VV(\sE) \to \VV(\sF)$ along the section map $s_X \colon X \to \VV(\sO_X)$. By virtue of Proposition \ref{prop:affinecone} \eqref{prop:affinecone-6}, $\VV(\varphi^*)$ is a closed immersion, whence so is $\iota|_{V_i} \colon V_i \to U_i$. This proves assertion \eqref{cor:proj:closedimmersion-1}.

We now prove \eqref{cor:proj:closedimmersion-2}. By virtue of Theorem \ref{thm:proj:Euler}, we have a diagram in $\QCoh(\PP(\sE))$:
	\begin{equation*}
	\begin{tikzcd}
		\iota^* L_{\PP(\sF)/X}  \ar{r} \ar{d} &L_{\PP(\sE)/X} \ar{d} \ar{r} &  L_{\PP(\sF)/\PP(\sE)} \ar[dashed]{d} \\
		 (\pr_{\PP(\sE)}^* \sF) \otimes \sO(-1)  \ar{r}{\pr_{\PP(\sE)}^* \varphi} \ar{d}[swap]{\iota^* \rho_{\PP(\sF)} \otimes \sO(-1)}& (\pr_{\PP(\sE)}^*\sE) \otimes \sO(-1)  \ar{d}{\rho_{\PP(\sE)} \otimes \sO(-1)} \ar{r} &  {\rm cofib} (\pr_{\PP(\sE)}^* \varphi) \otimes \sO(-1)\\
		\iota^*\sO_{\PP(\sF)}	\ar{r}{\simeq} & \sO_{\PP(\sE)}  & 
	\end{tikzcd}
	\end{equation*}
In the above diagram: the first row is the fiber sequence of relative cotangent complexes obtained from the composition $\PP(\sE) \to \PP(\sF) \to X$ (see \cite[Proposition 3.2.12]{DAG}, which furthermore implies the existence of $L_{\PP(\sE)/\PP(\sF)}$). The first and second columns are the fiber sequences of Theorem \ref{thm:proj:Euler} which also implies the two solid squares are commutative. Therefore by virtue of the property ($TR4$) of the stable $\infty$-category $\QCoh(\PP(\sE))$ (see \cite[Theorem 1.1.2.14]{HA}), there exists essentially a unique vertical dashed arrow rending the above diagram commutative and inducing an equivalence between $\LL_{\PP(\sF)/\PP(\sE)}$ and ${\rm cofib} (\pr_{\PP(\sE)}^* \varphi) \otimes \sO(-1) \simeq \pr_{\PP(\sE)}^*({\rm cofib} (\varphi) )\otimes \sO(-1)$. 
\end{proof}

\begin{proposition} \label{prop:proj:PB}  Let $X$ be a prestack and $\sF' \xrightarrow{\varphi'} \sF \xrightarrow{\varphi''} \sF''$ a fiber sequence of connective quasi-coherent sheaves on $X$. Then there is a pullback diagram of prestacks:
	\begin{equation} \label{eqn:proj:PB:V}
	\begin{tikzcd}
		\PP(\sF'') \ar{d}[swap]{\iota_{\varphi''}} \ar{r}{\iota_{\varphi''}} & \PP(\sF) \ar{d}{i_{\varphi'}} \\
		\PP(\sF)  \ar{r}{i_{\mathbf{0}}} & \VV_{\PP(\sF)}(\pr_{\PP(\sF)}^* \sF' \otimes \sO_{\PP(\sF)}(-1)),
	\end{tikzcd}
	\end{equation}
where $\iota_{\varphi''} \colon \PP(\sF'')  \to \PP(\sF)$ is the closed immersion induced by $\varphi'' \colon \sF \to \sF''$ (see Corollary \ref{cor:proj:closedimmersion} \eqref{cor:proj:closedimmersion-1});  $i_{\mathbf{0}}$ and $i_{\varphi'}$ are sections of the projection $\VV_{\PP(\sF)}(\pr_{\PP(\sF)}^* \sF' \otimes \sO_{\PP(\sF)}(-1)) \to \PP(\sF)$ that classify the zero map $\pr_{\PP(\sF)}^* \sF' \otimes \sO_{\PP(\sF)}(-1) \to \sO_{\PP(\sF)}$ and the map $\rho_{\varphi'}$ obtained by tensoring the composition $\pr_{\PP(\sF)}^* \sF' \xrightarrow{\pr_{\PP(\sF)}^* \varphi'} \pr_{\PP(\sF)}^* \sF \xrightarrow{\rho_{\PP(\sF)}} \sO_{\PP(\sF)}(1)$ with $\sO_{\PP(\sF)}(-1)$, respectively. In other words, $\iota_{\varphi''}$ identifies $\PP(\sF'')$ as the derived zero locus of the canonical cosection $\rho_{\varphi'} \colon \pr^*_{\PP(\sF)} \sF' \otimes \sO_{\PP(\sF)}(-1) \to \sO_{\PP(\sF)}$ determined by $\varphi'$ in the sense of Definition \ref{def:dzero}.
\end{proposition}

\begin{proof} Since the composition map $\pr_{\PP(\sF'')}^* \sF' \to \pr_{\PP(\sF'')}^* \sF \to \pr_{\PP(\sF'')}^* \sF'' \to \sO_{\PP(\sF'')}(1)$ is nullhomotopic, we see that  the maps $i_{\varphi'} \circ \iota_{\varphi''}$ and $i_{\mathbf{0}} \circ \iota_{\varphi''}$ both classify the zero map $\pr_{\PP(\sF'')}^* \sF' \otimes  \iota_{\varphi''}^*\sO_{\PP(\sF)}(-1) \to \sO_{\PP(\sF'')}$, hence the diagram \eqref{eqn:proj:PB:V} commutes. 

To show that \eqref{eqn:proj:PB:V} is a pullback diagram, it suffices to prove the assertion in the case where $X = \Spec R$, $R \in \CAlgDelta$. Let $\{U_i  \subseteq \PP(\sF)\}_{i \in I}$ and $\{V_i \subseteq \PP(\sF'') \}_{i \in I}$ be the Zariski open covers as in the proof of Corollary \ref{cor:proj:closedimmersion}, that is, they are Zariski open covers of $\PP(\sF)$ and $\PP(\sF'')$ associated to some generating family of sections $\{t_i \colon \sO_X \to \sF\}_{i \in I}$ and $\{\varphi'' \circ t_i \colon \sO_X \to \sF''\}_{i \in I}$, respectively. Since the assertion is local on $\PP(\sF)$ with respect to Zariski topology, it suffices to show for each $i \in I$, the restriction of the diagram \eqref{eqn:proj:PB:V} to each $U_i$,
	\begin{equation*}
	\begin{tikzcd}
		V_i \ar{d}[swap]{\iota_{\varphi''}|V_i} \ar{r}{\iota_{\varphi''}|V_i} & U_i\ar{d}{i_{\varphi'}|U_i} \\
		U_i \ar{r}{i_{\mathbf{0}}|U_i} & \VV_{U_i}(\pr_{U_i}^* \sF' \otimes \sO_{\PP(\sF)}(-1)|_{U_i}),
	\end{tikzcd}
	\end{equation*}
is a pullback square. By virtue of Remark \ref{remark:proj:Ui}, $\sO_{\PP(\sF)}(-1)|_{U_i}$ is trivial, and the above diagram is a base change of the commutative square
	$$
	\begin{tikzcd}
		\VV(\sF'') \ar{d}[swap]{\VV(\rho''^*)} \ar{r}{\VV(\rho''^*)}& \VV(\sF) \ar{d}{i_{\varphi'}} \\
		\VV(\sF)  \ar{r}{i_{\mathbf{0}}} & \VV(\sF') \times_X \VV(\sF),
	\end{tikzcd}
	$$
which is a pullback square by virtue of Remark \ref{rmk:affinecone:PB}. This proves the desired assertion. 
\end{proof}

\begin{example} Let $X$ be a prestack, and let $\sF$ be a connective quasi-coherent sheaf on $X$, and let $\rho \colon \sF \to \sO_X$ be a cosection. Then the closed immerison $\PP(\cofib(\rho)) \to \PP(\sO_X)\simeq X$ identifies $\PP(\cofib(\rho))$ as the derived zero locus $Z(\rho)$ of the cosection $\rho$ (Definition \ref{def:dzero}).
\end{example}

\begin{example}\label{eg:prop:proj:PB} Let $X$ be a prestack, let $\sV$ be a vector bundle on $X$, and let $\sF = \cofib(\sV \xrightarrow{0} \sO_X)$ denote the cofiber of the zero cosection map. Then Proposition \ref{prop:proj:PB} implies that there is a canonical closed immersion $\iota \colon \PP(\sF) \hookrightarrow X$ which fits into a pullback diagram of prestacks:
	$$
	\begin{tikzcd}
		\PP(\sF) \ar{d}[swap]{\iota} \ar{r}{\iota}& X \ar{d}{i_{\mathbf{0}}} \\
		X \ar{r}{i_{\mathbf{0}}} & \VV(\sV).
	\end{tikzcd}
	$$
Combining with Proposition \ref{prop:affinecone:PB}, we obtain a canonical equivalence $\PP(\sF) \simeq \VV(\Sigma\,\sV)$. More generally, let $\sE$ be any connective quasi-coherent sheaf on $X$, and let $\sF = \Sigma \sV \oplus \sE$, where $\sV$ is a vector bundle on $X$. Then by virtue of Proposition \ref{prop:proj:PB}, the canonical map $\sE \to \sF$ induces a closed immersion  $\iota \colon \PP(\sF) \hookrightarrow \PP(\sE)$ which fits into a pullback diagram of prestacks:
	$$
	\begin{tikzcd}
		\PP(\sF) \ar{d}[swap]{\iota} \ar{r}{\iota}& \PP(\sE) \ar{d}{i_{\mathbf{0}}} \\
		\PP(\sE) \ar{r}{i_{\mathbf{0}}} & \VV_{\PP(\sE)}(\pr_{\PP(\sE)}^*\sV \otimes \sO_{\PP(\sE)}(-1)).
	\end{tikzcd}
	$$
Combining with Proposition \ref{prop:affinecone:PB}, we obtain a canonical equivalence 
	$$\PP(\sF) \simeq \VV_{\PP(\sE)}(\Sigma\, \sV \boxtimes \sO(-1))$$
where $\sV \boxtimes \sO(-1)$ denote the vector bundle $\pr_{\PP(\sE)}^*\sV \otimes \sO_{\PP(\sE)}(-1)$ on $\PP(\sE)$.
\end{example}

\section{Generalized Serre's theorem} 
Let $\kappa$ be a commutative ring, and $V = \kappa^{\oplus n+1}$ a free $\kappa$-module of rank $(n+1)$, where $n \ge 0$ is an integer. We let $\PP_\kappa^n = \PP_{\Spec \kappa}(V)$ denote the projective space over $\kappa$ of dimension $n$. Let $\sO(1)$ denote the universal line bundle on $\PP_\kappa^n$, let $\sO(d) = \sO(1)^{\otimes d}$ if $d \ge 0$ and $\sO(d) = (\sO(1)^\vee)^{\otimes -d}$ if $d <0$. A theorem of Serre explicitly computes the cohomologies of line bundles $\sO(d)$ over $\PP_\kappa^n$. More concretely: 

\begin{theorem}[Serre] \label{thm:Serre:Pn}
\begin{enumerate}[leftmargin=*]
			\item If $d \ge 0$, then there is a canonical equivalence $\H^*(\PP_\kappa^n; \sO(d)) \simeq S^d_\kappa(V)$. In particular, $\H^*(\PP_\kappa^n; \sO(d))$ is a free $\kappa$-module of rank $\binom{n+d}{n}$ concentrated in degree $0$.  (Here, $\H^*(\PP_\kappa^n; \sO(d))\in \Mod_\kappa$ denotes the total cohomology of $\sO(d)$ over $\PP_\kappa^n$.)
			\item If $d <0$, then there are canonical equivalences 
				$$\H^*(\PP_\kappa^n; \sO(d)) \simeq \Sigma^{-n}(S_\kappa^{-d-n-1} V)^\vee \simeq \Sigma^{-n} \Gamma_\kappa^{-d-n-1}(V^\vee)$$
				 (where $\Gamma_\kappa^m(V^\vee)$ is the divided power). In particular, if $-n \le d \le -1$, then $\H^*(\PP_\kappa^n; \sO(d))$ vanishes; if $d \le -n-1$, then $\H^*(\PP_\kappa^n; \sO(d)) $ is a degree shift of a free $\kappa$-module of rank $\binom{-d-1}{n}$ placed in cohomological degree $n$.
		\end{enumerate}
\end{theorem}

In \cite[Theorem 5.4.2.6]{SAG}, Lurie extends the above Serre's result in the framework of spectral algebraic geometry to projective spaces over all $\EE_\infty$-rings $\kappa$, which, in particular, implies the classical Serre's theorem \ref{thm:Serre:Pn} in the form as we formulate above. 

In this section, we generalize the above Serre's result in a different direction, that is, we show it holds {\em beyond} the cases of projective bundles. Let $X$ be a prestack and $\sF$ a quasi-coherent sheaf over $X$ of perfect-amplitude contained in $[0,1]$. We show that Serre's theorem holds for the derived projectivization $\pr \colon \PP(\sF) \to X$ (Theorem \ref{thm:Serre:O(d)}). Our version of Serre's theorem also includes the cases where the sheaves $\sF$ have non-positive ranks; in this case, the formula for the ``global sections" of $\sO_{\PP(\sF)}(d)$ reproduces (a derived version of) the classical Eagon--Northcott complexes (Remark \ref{rmk:ENcomplexes}).

\subsection{Serre's theorem for projective bundles}  \label{sec:Serre:bundle}
In this subsection, we present a global version of Serre's Theorem \ref{thm:Serre:Pn} for projectivizations of vector bundles; Our formulation closely follows Grothendieck's in the classical setting \cite[III, 2.1.15 \& 2.1.16]{EGA}. 

\begin{theorem}[Serre, Grothendieck]  \label{thm:Serre:bundle}
Let $X$ be a prestack, let $\sF$ be a vector bundle of rank $r \ge 1$ over $X$, and let $\det \sF = \bigwedge\nolimits^r \sF \in \Pic(X)$ denote the determinant line bundle of $\sF$ over $X$. Let $\pr \colon \PP(\sF) \to X$ denote the derived projectivization of $\sF$ over $X$, and let $\sO(1)$ denote the universal line bundle on $\PP(\sF)$. Let $d$ be an integer, and we set $\sO(d) = \sO(1)^{\otimes d}$ if $d \ge 0$ and $\sO(d) = (\sO(1)^\vee)^{\otimes -d}$ if $d <0$ as usual. Then:
\begin{enumerate}[leftmargin=*]
			\item \label{thm:Serre:bundle-1} If $d \ge 0$, then there is a canonical equivalence 
				$$\varphi_d \colon \Sym_X^d (\sF) \xrightarrow{\sim} \pr_*(\sO(d)),$$
				where $\Sym_X^d(\sF)$ denotes the $d$th symmetric power of $\sF$ over $X$ (and is locally free).
			\item \label{thm:Serre:bundle-2} If $d <0$, then the following holds:
				\begin{enumerate}
					\item \label{thm:Serre:bundle-2i}
					If $-r+1 \le d \le -1$, then $\pr_*(\sO(d)) \simeq 0$.
					\item \label{thm:Serre:bundle-2ii}
					If $d = -r$, then there is a canonical equivalence
							$$\pr_*(\Sigma^{r-1} \pr^* (\det \sF) \otimes \sO(-r) ) \xrightarrow{\sim} \sO_{X}$$
							which exhibits $\Sigma^{r-1} \pr^* (\det \sF) \otimes \sO(-r) $ as (canonically equivalent to) the relative dualizing sheaf $\omega_{\pr}$ of the projection $\pr \colon \PP(\sF) \to X$ (Theorem \ref{thm:Neeman-Lipman-Lurie} \eqref{thm:Neeman-Lipman-Lurie-2vii}).
					\item \label{thm:Serre:bundle-2iii}
					 More generally, for every integer $d \le -r$, there is a canonical equivalence
						$$\psi_d \colon \pr_*(\sO(d))  \xrightarrow{\sim} \Sigma^{1-r} (\Sym_X^{-d-r} \sF)^\vee \otimes (\det \sF)^\vee \simeq \Sigma^{1-r} (\Gamma_X^{-d-r} (\sF^\vee)) \otimes (\det \sF)^\vee.$$
			\end{enumerate}
		\end{enumerate}
\end{theorem}

\begin{proof} Let $\rho \colon \pr^* \sF \to \sO(1)$ denote the tautological quotient morphism, then the fiber of $\rho$ is locally free of rank $r-1$. If $d \ge 0$, then the map 
	$$\pr^* \Sym_{X}^d(\sF)\simeq \Sym_{\PP(\sF)}^d(\pr^* \sF) \xrightarrow{\Sym^d(\rho)}  \Sym_{\PP(\sF)}^d(\sO(1)) \simeq \sO(d)$$
induces a canonical map
	$$\varphi_d \colon \Sym_{X}^d (\sF) \to \pr_*(\sO(d)).$$
We wish to show that the cofiber of $\varphi_d$ is equivalent to zero. Since the formation of $\varphi_d$ commutes with base change (\S \ref{sec:sym:prestack} and Theorem \ref{thm:Neeman-Lipman-Lurie} \eqref{thm:Neeman-Lipman-Lurie-1}), we may reduce to the case where $X = \Spec R$, $R \in \CAlgDelta$. Since the condition ``the cofiber of $\varphi_d$ is equivalent zero" is local on $X$ with respect to Zariski topology, we may further reduce to the case where $\sF = R^{\oplus r}$. Using the fact that the formation of $\varphi_d$ commutes with base change again, we could further reduce to the case where $X = \Spec \ZZ$; in this case, the assertion follows from Theorem \ref{sec:Serre:bundle}. This proves the assertion \eqref{thm:Serre:bundle-1}. The assertion \eqref{thm:Serre:bundle-2i} can be proved in the same way. 

To prove assertion \eqref{thm:Serre:bundle-2ii}, we apply Construction \ref{constr:Postnikov:Sym:prestacks} to the twisted tautological quotient
	$$\rho(-1) = \rho \otimes \sO(-1) \colon \pr^* \sF \otimes \sO(-1) \to \sO$$
in the case where $n=r$. Then for all $0 \le i \le r$, there are canonically defined quasi-coherent sheaves $\sG_{i,r}(\PP(\sF), \rho(-1))$ which fit into canonical fiber sequences $\beta_{i,r}(\PP(\sF), \rho(-1))$:
	$$\Sigma^{i} \,\bigwedge\nolimits^{i+1} (\pr^* \sF) \otimes \sO(-i-1)  \to \sG_{i,r}(\PP(\sF), \rho(-1)) \to \sG_{i+1,r}(\PP(\sF), \rho(-1)),$$
for which there are canonical equivalences $\sG_{0,r}(\PP(\sF), \rho(-1)) \simeq \sO_{\PP(\sF)}$ and 
	$$\sG_{r,r}(\PP(\sF), \rho(-1)) \simeq \Sym^r (\cofib(\rho(-1))) \simeq \Sigma^r \bigwedge\nolimits^r \fib(\rho(-1)) \simeq 0.$$
In particular, combining with assertion \eqref{thm:Serre:bundle-2i}, the application of the exact functor $\pr_*$ to the fiber sequences $\{\beta_{i,r}(\PP(\sF), \rho(-1))\}_{0 \le i \le r-2}$ induces canonical equivalences
	$$\sO_{X} \simeq  \pr_* \sG_{0,r}(\PP(\sF), \rho(-1)) \simeq \pr_* \sG_{1,r}(\PP(\sF), \rho(-1)) \simeq \cdots \simeq \pr_* \sG_{r-1,r}(\PP(\sF), \rho(-1));$$
and the application of $\pr_*$ to $\beta_{r-1,r}(\PP(\sF), \rho(-1))$ induces a canonical equivalence 
	$$\pr_*(\Sigma^{r-1} (\bigwedge\nolimits^{r} \pr^* \sF) \otimes \sO(-r) ) \xrightarrow{\sim} \pr_* \sG_{r-1,r}(\PP(\sF), \rho(-1)) \simeq \sO_{X}.$$
This proves assertion \eqref{thm:Serre:bundle-2ii}. Finally, for any $d \le -r$, we let $e = -r -d \ge 0$. By applying \eqref{thm:Serre:bundle-1} to $\sO(e)$, we obtain a canonical equivalence:
	$$\varphi_e^\vee \colon (\pr_* \sO(e))^\vee \xrightarrow{\sim} (\Sym_X^e(\sF))^\vee.$$
By virtue of Theorem \ref{thm:Neeman-Lipman-Lurie} (\ref{thm:Neeman-Lipman-Lurie-2iv}, \ref{thm:Neeman-Lipman-Lurie-2vii}), there is a canonical equivalence
	$$ (\pr_* \sO(e))^\vee \simeq \pr_* (\sO(-e) \otimes \omega_{\pr}).$$
Composing with the equivalence of \eqref{thm:Serre:bundle-2ii}, we obtain a canonical equivalence
	\begin{align*}
	\psi_d \colon \pr_*(\sO(d)) & \xrightarrow{\sim}  \pr_* (\sO(-e) \otimes \omega_{\pr}) \otimes \Sigma^{1-r} (\det \sF)^\vee   \\
	&\simeq   (\pr_* \sO(e))^\vee \otimes  \Sigma^{1-r} (\det \sF)^\vee \xrightarrow{\varphi_e^\vee} (\Sym_X^e(\sF))^\vee  \otimes  \Sigma^{1-r} (\det \sF)^\vee.
	\end{align*}
This proves assertion \eqref{thm:Serre:bundle-2iii}.
\end{proof}

\begin{remark}[Grothendieck--Serre duality for projective bundles] \label{rmk:GSforbundle} The above proof of assertion Theorem \ref{thm:Serre:bundle} \eqref{thm:Serre:bundle-2iii} implies that, for any $d \ge 0$, the composite map
	$$\pr_*(\sO(d)) \otimes \pr_*( \Sigma^{r-1} \sO(-d-r) \otimes \pr^* (\det \sF) ) \to \pr_*(\Sigma^{r-1} \pr^* (\det \sF) \otimes \sO(-r)) \xrightarrow{\sim} \sO_X$$
is a duality datum in the symmetric monoidal $\infty$-category $\QCoh(X)^\otimes$ (in the sense of {\cite[Definition 4.6.1.7]{HA}}). This statement is usually referred to as the {\em Grothendieck--Serre duality} for the projective bundle $\pr \colon \PP(\sF) \to X$.
\end{remark}

\subsection{Generalized Serre's theorem} \label{sec:Serre}
In this subsection, we generalize the above Theorem \ref{thm:Serre:bundle} of Serre and Grothendieck beyond the cases of projective bundles. 

\begin{lemma}\label{lem:Serre:O(d)} Let $X$ be a prestack, let $\sigma \colon \sW \to \sV$ be a map between vector bundles on $X$, where $\sW$ and $\sV$ are vector bundles of ranks $m$ and $n$, respectively. Let $\sF$ denote the cofiber of $\sigma$ and set $r = n - m$. Let $\pr \colon \PP(\sF) \to X$ and $q \colon \PP(\sV) \to X$ denote the derived projectivizations. Let $\varphi \colon q^* \sW \to \sO_{\PP(\sV)}(1)$ denote the composite map of $q^* \sigma$ with the tautological quotient $\rho_{\sV} \colon q^* \sV \to \sO_{\PP(\sV)}(1)$, and let $\iota \colon \PP(\sF) \to \PP(\sV)$ denote the closed immersion induced by the surjection $\sV \to \sF$ (Corollary \ref{cor:proj:closedimmersion}). Then:
\begin{enumerate}[leftmargin=*]
	\item \label{lem:Serre:O(d)-1} The projection $\pr \colon \PP(\sF) \to X$ is proper, quasi-smooth of relative dimension $r$.
	\item \label{lem:Serre:O(d)-2} For any integer $d$, we let $\sG_i(\PP(\sV);d)$ denote the quasi-coherent sheaf 
		$$\sG_{i,m}(\PP(\sV), q^* \sW \xrightarrow{\varphi} \sO_{\PP(\sV)}(1)) \otimes \sO_{\PP(\sV)}(d-m)$$
		for all $0 \le i \le m$, where $\sG_{i,m}(\PP(\sV), \varphi)$ is defined as in Construction \ref{constr:Postnikov:Sym:prestacks}. Then there are canonical equivalences $\sG_{0}(\PP(\sV); d) \simeq \sO_{\PP(\sV)}(d)$ and $\sG_{m}(\PP(\sV); d) \simeq \iota_*(\sO_{\PP(\sF)}(d))$. Furthermore, there are canonical fiber sequences $\beta_{i}(\PP(\sV); d)$:
 	$$\sG_{i-1}(\PP(\sV); d) \to \sG_i(\PP(\sV); d) \to \Sigma^i (\bigwedge\nolimits^i q^*\sW) \otimes \sO(d-i).$$
	\item \label{lem:Serre:O(d)-3} Let $\omega_{\pr}$ denote the relative dualizing sheaf of $\pr$. There is a canonical equivalence
		$$\Sigma^{r-1} {\pr}^*((\bigwedge\nolimits^m \sW)^\vee \otimes \bigwedge\nolimits^n \sV) \otimes \sO_{\PP(\sF)}(-r) \xrightarrow{\sim} \omega_{\pr}.$$
\end{enumerate}
\end{lemma}

\begin{remark} \label{rmk:lemma:Serre:O(d)} We could informally think of assertion \eqref{lem:Serre:O(d)-2} as the heuristic  statement that ``$\iota_*(\sO_{\PP(\sF)}(d))$ is canonically resolved by the Koszul complex $\bS^m(\PP(\sV), \varphi) \otimes \sO_{\PP(\sV)}(m-d)$:
	$$ \bigwedge\nolimits^m q^*\sW \otimes \sO_{\PP(\sV)}(-m+d) \xrightarrow{d_m'} \cdots  \to q^*\sW \otimes \sO_{\PP(\sV)}(d-1) \xrightarrow{d_1'} \sO_{\PP(\sV)}(d)."$$
This heuristic statement is rigirous whenever $X$ is classical and the cosection $\varphi \colon q^* \sW \to \sO_{\PP(\sV)}(1)$ is Koszul-regular. In general, we could think of the canonical fiber sequences of \eqref{lem:Serre:O(d)-2} as ``remembering" the extra information of how the complex is ``assembled" from $d_i'$.
\end{remark}

\begin{proof}[Proof of Lemma \ref{lem:Serre:O(d)}]
Assertion \eqref{lem:Serre:O(d)-1} is a consequence of Theorem \ref{thm:proj:Euler}. Assertion \eqref{lem:Serre:O(d)-2} is a consequence of Proposition \ref{prop:proj:PB}, Lemma 
\ref{lem:Koszul:fib} and Remark \ref{rmk:Koszul:fib}. Finally, the fiber sequence $\beta_{m}(\PP(\sV); -r)$ induces a canonical map
	$$\iota_* \sO_{\PP(\sF)}(-r) \simeq \sG_{m}(\PP(\sV);-r) \to \Sigma^m \bigwedge\nolimits^m(q^*\sW) \otimes \sO_{\PP(\sV)}(-n).$$
Applying the exact functor $q_*$ and using Theorem \ref{thm:Serre:bundle} \eqref{thm:Serre:bundle-2ii}, we obtain a canonical map
	$$\pr_*(\Sigma^{r-1} \sO_{\PP(\sF)}(-r)) \to \bigwedge\nolimits^m \sW \otimes q_*(\Sigma^{n-1} \sO_{\PP(\sV)}(-n)) \xrightarrow{\sim} \bigwedge\nolimits^m \sW \otimes (\bigwedge\nolimits^n \sV)^\vee$$
which determines a canonical map
	$$\theta \colon \Sigma^{r-1} {\pr}^*((\bigwedge\nolimits^m \sW)^\vee \otimes \bigwedge\nolimits^n \sV) \otimes \sO_{\PP(\sF)}(-r) \to \omega_{\pr}.$$
As the formation of $\theta$ commutes with base change, arguing as the proof of Theorem \ref{thm:Serre:bundle} \eqref{thm:Serre:bundle-1}, we may reduce to case where $X$ is classical, $\sV = \sO_X^{\oplus n}$, $\sW = \sO_X^{\oplus m}$, and $\varphi \colon q^* \sW \to \sO_{\PP(\sV)}(1)$ is Koszul-regular. Then the desired equivalence follows from $\omega_{\pr} \simeq \omega_{\iota} \otimes \omega_{q}$ (\cite[Corollary 6.4.2.8]{SAG}), Theorem \ref{thm:Serre:bundle} \eqref{thm:Serre:bundle-2ii}, and the formula for relative dualizing sheaves of  Koszul-regular closed immersions (see \cite[\href{https://stacks.math.columbia.edu/tag/0BR0}{Tag 0BR0}]{stacks-project} or \cite[Example 6.4.2.9]{SAG}).
\end{proof}

The following theorem\footnote{Zhao has also obtained a version of the generalized Serre's theorem  \ref{thm:Serre:O(d)} in \cite[Lemma 3.9]{Z20} (in the form the equivalence $[\mathbf{EN}_d(X, \sigma)] \simeq \pr_* (\sO(d))$ of Remark \ref{rmk:ENcomplexes}).} generalize Serre's Theorem \ref{thm:Serre:bundle} to the case of derived projectivizations of perfect complexes of Tor-amplitude $[0,1]$:

\begin{theorem}[Generalized Serre's theorem]
 \label{thm:Serre:O(d)}
Let $X$ be a prestack, let $\sF$ be a quasi-coherent sheaf on $X$ of perfect-amplitude contained in $[0,1]$ and constant rank $r$. Let $\pr \colon \PP(\sF) \to X$ denote the derived projectivization of $\sF$ over $X$, let $\sO(1)$ denote the universal line bundle on $\PP(\sF)$, and let $\rho \colon \pr^* \sF \to \sO(1)$ denote the tautological quotient map. Then the projection $\pr \colon \PP(\sF) \to X$ is proper, quasi-smooth of relative virtual dimension $r-1$. Furthermore, let $d$ be an integer, and we set $\sO(d) = \sO(1)^{\otimes d}$ if $d \ge 0$ and $\sO(d) = (\sO(1)^\vee)^{\otimes -d}$ if $d <0$. Then
\begin{enumerate}[leftmargin=*]
			\item \label{thm:Serre:O(d)-1} If $r \ge 1$, then the following holds:
				\begin{enumerate}
					\item  \label{thm:Serre:O(d)-1i}
					If $d \ge 0$, then there is a canonical equivalence 
							$$\varphi_d \colon \Sym_X^d (\sF) \xrightarrow{\sim} \pr_*(\sO(d)),$$
				where $\Sym_X^d(\sF)$ denotes the $d$th derived symmetric power of $\sF$ over $X$.
					\item  \label{thm:Serre:O(d)-1ii} 
					If $-r+1 \le d \le -1$, then $\pr_*(\sO(d)) \simeq 0$.
					\item  \label{thm:Serre:O(d)-1iii} 
					If $d = -r$, then 
							$\pr_*(\Sigma^{r-1} \sO(-r))$
							is a line bundle on $X$. We \textbf{define} the determinant line bundle of $\sF$ over $X$ by the formula 
								$$\det \sF = (\pr_*(\Sigma^{r-1} \sO(-r)))^\vee \in \Pic(X).$$
							Then the canonical equivalence
							$$\pr_*(\Sigma^{r-1} \pr^* (\det \sF) \otimes \sO(-r) ) \simeq  \sO_{X}$$
		exhibits $\Sigma^{r-1} \pr^* (\det \sF) \otimes \sO(-r) $ as canonically equivalent to the relative dualizing sheaf $\omega_{\pr}$ of the projection $\pr \colon \PP(\sF) \to X$ (Theorem \ref{thm:Neeman-Lipman-Lurie} \eqref{thm:Neeman-Lipman-Lurie-2vii}).		
					\item \label{thm:Serre:O(d)-1iv}   More generally, for every integer $d \le -r$, there are canonical equivalences
						$$\psi_d \colon \pr_*(\sO(d))  \xrightarrow{\sim} \Sigma^{1-r} (\Sym_X^{-d-r} \sF)^\vee \otimes (\det \sF)^\vee. 
						$$
				\end{enumerate}
						\item \label{thm:Serre:O(d)-2} If $r \le 0$, the following holds:
				\begin{enumerate}
					\item	 \label{thm:Serre:O(d)-2i} 
					If $d \ge -r +1$, then there is a canonical equivalence
				$$\varphi_d \colon \Sym_{X}^d (\sF) \xrightarrow{\sim}  \pr_*(\sO(d)).$$
					\item \label{thm:Serre:O(d)-2ii}
					if $d = -r$, then the cofiber of the canonical map $\Sigma^{r-1} \varphi_{-r}$ is a line bundle on $X$, where 
						$\varphi_{-r}$ is the canonical map $\Sym_X^{-r} \sF \to  \pr_*(\sO(-r))$. We \textbf{define} the determinant line bundle of $\sF$ over $X$ by the formula 
						$$\det \sF = \Sigma^{1-r} \cofib\big( \Sym_X^{-r} \sF \xrightarrow{\varphi_{-r}} \pr_*(\sO(-r)) \big)^\vee \in \Pic(X).$$
					Then the canonical map
						$$\pr_*(\Sigma^{r-1} \pr^* (\det \sF) \otimes \sO(-r) ) \to  \sO_{X}$$
					exhibits $\Sigma^{r-1} \pr^* (\det \sF) \otimes \sO(-r) $ as canonically equivalent to the relative dualizing sheaf $\omega_{\pr}$ of the projection $\pr \colon \PP(\sF) \to X$ (Theorem \ref{thm:Neeman-Lipman-Lurie} \eqref{thm:Neeman-Lipman-Lurie-2vii}).	
					\item \label{thm:Serre:O(d)-2iii}
					In general, if $0 \le d \le -r$, then there is a canonical fiber sequence
				$$\Sym_{X}^d (\sF) \xrightarrow{\varphi_d} \pr_*(\sO(d)) \xrightarrow{\psi_d} \Sigma^{1-r} (\Sym_X^{-r-d} \sF )^\vee \otimes_{\sO_X}  (\det \sF)^{\vee}.$$
					\item \label{thm:Serre:O(d)-2iv}
					If $d \le -1$, then there is a canonical equivalence
				$$\psi_d \colon \pr_*(\sO(d)) \xrightarrow{\sim}  \Sigma^{1-r} (\Sym_X^{-r-d} \sF )^\vee \otimes_{\sO_X}  (\det \sF)^{\vee}.$$
			\end{enumerate}
		\end{enumerate}
\end{theorem}

\begin{remark}[Determinant line bundles] Let $X$ be a prestack, and $\sF$ a quasi-coherent sheaf on $X$ of perfect-amplitude contained in $[0,1]$. It is well known that $\sF$ has a well-defined line bundle $\det \sF$. However, the usual methods of defining $\det \sF$ are either by an {\em ad hoc} gluing process or by deducing its existence from the theory of determinant of vector bundles via the equivalence of the $K$-theories of $\Vect(X)$ and $\Perf(X)$. We could regard our definition of $\det{\sF}$ in \eqref{thm:Serre:O(d)-1iii} and \eqref{thm:Serre:O(d)-2ii} as providing an {\em explicit} and {\em definite} model for the determinant line bundle of $\sF$. Our definition is justified by the following fact: by virtue of Lemma \ref{lem:Serre:O(d)} \eqref{lem:Serre:O(d)-3}, for any morphism $f \colon Y \to X$ of prestacks such that $f^* \sF$ is represented by the cofiber of a map $\sigma \colon \sW \to \sV$ between vector bundles over $Y$, there is a canonical equivalence
	$$(\det \sW)^\vee \otimes_{\sO_Y} (\det \sV)  \to f^* (\det \sF) \simeq \det(f^* \sF).$$
 \end{remark}

\begin{proof}[Proof of Theorem \ref{thm:Serre:O(d)}]
Let $\rho \colon \pr^* \sF \to \sO(1)$ denote the tautological quotient morphism, then for all $d \ge 0$, the map 
	$$\pr^* \Sym_{X}^d(\sF)\simeq \Sym_{\PP(\sF)}^d(\pr^* \sF) \xrightarrow{\Sym^d(\rho)}  \Sym_{\PP(\sF)}^d(\sO(1)) \simeq \sO(d)$$
induces a canonical map
	$$\varphi_d \colon \Sym_{X}^d (\sF) \to \pr_*(\sO(d)).$$
We first prove assertion \eqref{thm:Serre:O(d)-1i}. By the same argument of the proof of Theorem \ref{thm:Serre:bundle} \eqref{thm:Serre:bundle-1}, we may reduced to the case where $X = \Spec R$, where $R \in \CAlgDelta$, and $\sF$ is represented by the cofiber of a map $\sigma \colon \sW \to \sV$, where $\sW =R \otimes \ZZ^{\oplus m}$ and $\sV = R \otimes \ZZ^{\oplus n}$. By virtue of the fact that the formation of $\varphi_d$ commutes with base change, and Example \ref{eg:Homspaces}, we may reduce to the case where $X = |\sHom_\ZZ(\ZZ^{\oplus m}, \ZZ^{\oplus n})| \simeq \AA_\ZZ^{mn}$, and $\sF$ is represented by the cokernel of the tautological map $\sigma \colon \sO_X^{\oplus m} \to \sO_X^{\oplus n}$. Let $q \colon \PP(\sV) \to X$ denote the projection and let $\iota \colon \PP(\sF) \to \PP(\sV)$ denote the closed immersion induced by $\sV \twoheadrightarrow \sF$ as in Lemma \ref{lem:Serre:O(d)}. In particular, the composite map $\varphi \colon q^* \sW \to q^*\sV \to \sO_{\PP(\sV)}(1)$ is Koszul-regular, and for any $d \in \ZZ$, the discrete sheaf $\iota_*(\sO(d))$ is resolved by the Koszul complex $\bS^m(\PP(\sV), \varphi) \otimes \sO(m-d)$, which, in this case, is a genuine acyclic complex of vector bundles (Remark \ref{rmk:lemma:Serre:O(d)}):
	\begin{equation}\label{eqn:thm:Serre:O(d):Koszul}
	(\bigwedge\nolimits^m q^*\sW) \otimes \sO_{\PP(\sV)}(-m+d) \xrightarrow{d_m'} \cdots  \to (q^*\sW) \otimes \sO_{\PP(\sV)}(d-1) \xrightarrow{d_1'} \sO_{\PP(\sV)}(d).
	\end{equation}
To prove assertion \eqref{thm:Serre:O(d)-1i}, we observe that the canonical map $\rho \colon \pr^* \sF \to \sO(1)$ induces a composite map $q^* \sF \to \iota_* \pr^* \sF \to \iota_* \sO(1)$ which fits into a commutative diagram:
	$$
	\begin{tikzcd}
		q^*\sW \ar[equal]{d} \ar{r}{q^*\sigma}& q^*\sV \ar{d} \ar{r} & q^* \sF \ar{d} \\
		q^*\sW \ar{r}{\varphi} & \sO_{\PP(\sV)}(1) \ar{r} & \iota_* \sO(1).
	\end{tikzcd}
	$$	
Therefore, the composite map $q^* \Sym^d_X(\sF) \to \iota_* (\Sym^d (\pr^* \sF)) \to \iota_* \sO(d)$ induced from the canonical map $\Sym^d (\pr^* \sF) \to \sO(d)$ is compatible with the map of complexes
	$$\bS^d(\PP(\sV), q^*\sigma) \to \bS^d(\PP(\sV), \varphi) \to \bS^m(\PP(\sV),\varphi) \otimes \sO(m-d).$$
If we assume $d \ge 0$ and $r = n- m \ge 1$, then $d -m \ge -m \ge -n +1$. If we pushforward along $q$, then by virtue of Theorem \ref{thm:Serre:bundle}\eqref{thm:Serre:bundle-1}, \eqref{thm:Serre:bundle-2i} we obtain an equivalence
	$$\bS^d(X, \sigma \colon \sW \to \sV) \simeq q_* \left((\bS^m(\PP(\sV), q^* \sW \to \sO_{\PP(\sV)}(1))) \otimes \sO(m-d) \right)$$
which is compatible with the map $\varphi_d \colon \Sym_X^d (\sF) \to q_* \iota_* (\sO(d)) = \pr_* (\sO(d))$. Since $X$ is classical, the natural map $\bS^d(X, \sigma) \to \Sym^d_X(\sF)$ induces an equivalence $[\bS^d(X, \sigma) ] \simeq \Sym^d_X(\sF)$ (Corollary \ref{cor:Illusie:S^n}), we obtain that $\varphi_d$ is an equivalence. This proves assertion \eqref{thm:Serre:O(d)-1i}; the same argument proves \eqref{thm:Serre:O(d)-1ii} and \eqref{thm:Serre:O(d)-2i}. 
 
We next prove assertion \eqref{thm:Serre:O(d)-1iii}. Since the condition that ``$\pr_*(\Sigma^{r-1} \sO(-r))$ is a line bundle" is local on $X$, and the formation of $\pr_*(\Sigma^{r-1} \sO(-r))$ commutes with base change, we may reduce to the above case where $X = \AA_\ZZ^{mn}$ and $\sF$ is represented by the cofiber of the tautological map between vector bundles $\sigma \colon \sW = \sO_X^{\oplus m} \to \sV = \sO_X^{\oplus n}$. Setting $d = -r$ in \eqref{eqn:thm:Serre:O(d):Koszul} and using Theorem \ref{thm:Serre:bundle} \eqref{thm:Serre:bundle-2i}, \eqref{thm:Serre:bundle-2ii}, we obtain a canonical equivalence
		$$\pr_*(\Sigma^{r-1} \sO(-r)) \xrightarrow{\sim}  (\bigwedge\nolimits^m \sW) \otimes q_*(\Sigma^{n-1} \sO_{\PP(\sV)}(-n)) \xrightarrow{\sim} (\bigwedge\nolimits^m \sW) \otimes (\bigwedge\nolimits^n \sV)^\vee$$
This proves that $\pr_*(\Sigma^{r-1} \sO(-r))$ is a line bundle on $X$. The rest of assertion \eqref{thm:Serre:O(d)-1iii} follows from Lemma \ref{lem:Serre:O(d)} \eqref{lem:Serre:O(d)-3}.

Similarly, to prove assertion \eqref{thm:Serre:O(d)-2ii}, we may also reduce to the above case where $X = \AA_\ZZ^{mn}$ and $\sF$ is the cofiber of the tautological map $\sigma \colon \sO_X^m \to \sO_X^n$, and we set $d = -r$ in \eqref{eqn:thm:Serre:O(d):Koszul}. Let $\sG_i(\PP(\sV); -r)$ denote the sheaf constructed in Lemma \ref{lem:Serre:O(d)} \eqref{lem:Serre:O(d)-2} (which in this case canonically represents the $i$th brutal truncation of \eqref{eqn:thm:Serre:O(d):Koszul}), then the proof of assertion \eqref{thm:Serre:O(d)-1i} implies that there is a canonical equivalence
	$q_* (\sG_{-r}(\PP(\sV); -r)) \simeq \Sym_X^{-r}(\sF).$
By virtue of Theorem \ref{thm:Serre:bundle} \eqref{thm:Serre:bundle-2i} and Lemma \ref{lem:Serre:O(d)} \eqref{lem:Serre:O(d)-2}, there are canonical equivalences 
	$$q_* (\sG_{m-1}(\PP(\sV); -r)) \simeq \cdots q_* (\sG_{-r}(\PP(\sV); -r)) \qquad q_* (\sG_{m}(\PP(\sV); -r)) \simeq \pr_*(\sO(-r)),$$
and a canonical fiber sequence
	$$q_* (\sG_{m-1}(\PP(\sV); -r))  \to q_* (\sG_{m}(\PP(\sV); -r))  \to q_*(\Sigma^m (\bigwedge\nolimits^m q^*\sW) \otimes \sO(-n))$$
for which the first map is canonically equivalent to $\varphi_{-r}$. Therefore, invoking Theorem \ref{thm:Serre:bundle} \eqref{thm:Serre:bundle-2ii}, we obtain a canonical equivalence
	$$\cofib\big(\varphi_{-r} \colon \Sym_X^{-r}(\sF) \to \pr_*(\sO(-r))\big) \simeq \Sigma^{1-r} (\bigwedge\nolimits^m \sW) \otimes (\bigwedge\nolimits^n \sV)^\vee.$$
The rest of assertion \eqref{thm:Serre:O(d)-2ii} now follows from Lemma \ref{lem:Serre:O(d)} \eqref{lem:Serre:O(d)-3}.

The assertions \eqref{thm:Serre:O(d)-1iv} and \eqref{thm:Serre:O(d)-2iv} can be proved in the same way as the proof of Theorem \ref{thm:Serre:bundle} \eqref{thm:Serre:bundle-2iii}, that is, by setting $e = -r -d \ge 0$, we let $\psi_d$ denote the composite map:
	\begin{align*}
	 \pr_*(\sO(d)) & \xrightarrow{\sim}  \pr_* (\sO(-e) \otimes \omega_{\pr}) \otimes \Sigma^{1-r} (\det \sF)^\vee   \\
	&\simeq   (\pr_* \sO(e))^\vee \otimes  \Sigma^{1-r} (\det \sF)^\vee \xrightarrow{\varphi_e^\vee} (\Sym_X^e(\sF))^\vee  \otimes  \Sigma^{1-r} (\det \sF)^\vee
	\end{align*}
where the first equivalence follows from the expressions of $\omega_\pr$ of \eqref{thm:Serre:O(d)-1ii} and \eqref{thm:Serre:O(d)-2iii}, and the second equivalence is the Grothendieck duality (Theorem \ref{thm:Neeman-Lipman-Lurie} \eqref{thm:Neeman-Lipman-Lurie-2iv}\eqref{thm:Neeman-Lipman-Lurie-2vii}). Then $\psi_d$ is equivalence since $\varphi_e^\vee$ is an equivalence, which follow from assertions \eqref{thm:Serre:O(d)-1i} and \eqref{thm:Serre:O(d)-2i}, respectively.

Finally, it only remains to prove assertion \eqref{thm:Serre:O(d)-2iii}. By setting $e = -r -d \ge 0$, we let $\psi_d$ be defined by the same formula as above. Then we have a composable pair of maps
		$$\Sym_{X}^d (\sF) \xrightarrow{\varphi_d} \pr_*(\sO(d)) \xrightarrow{\psi_d} \Sigma^{1-r} (\Sym_X^{-r-d} \sF )^\vee \otimes_{\sO_X}  (\det \sF)^{\vee}$$
	whose formation commutes with base change. To prove that the above sequence forms a fiber sequence, we may again reduce to the above case where $X = \AA_\ZZ^{mn}$, $\sF$ is the cofiber of the tautological map $\sigma \colon \sO_X^m \to \sO_X^n$, and $\iota_*(\sO(d))$ is resolved by the Koszul complex \eqref{eqn:thm:Serre:O(d):Koszul}. We consider the following (shifted) truncations of the Koszul complex \eqref{eqn:thm:Serre:O(d):Koszul}:
\begin{align*}
	&\mathbf{G}_d(\PP(\sV), \varphi) \colon   &(\bigwedge\nolimits^{d} q^*\sW) \ \xrightarrow{d_{d}'} \cdots  \xrightarrow{d_2'} (q^*\sW) \otimes \sO_{\PP(\sV)}(d-1) \xrightarrow{d_1'}  \sO_{\PP(\sV)}(d), \\
	& \mathbf{E}_d(\PP(\sV), \varphi)\colon  & (\bigwedge\nolimits^{d+n-1} q^*\sW) \otimes \sO_{\PP(\sV)}(-n+1) \xrightarrow{d_{d+n-1}'} \cdots \xrightarrow{d_{d+1}'}  (\bigwedge\nolimits^{d+1} q^*\sW) \otimes \sO_{\PP(\sV)}(-1)\\
	&\mathbf{H}_e(\PP(\sV), \varphi)\colon  & (\bigwedge\nolimits^{m} q^*\sW) \otimes \sO_{\PP(\sV)}(d-m) \xrightarrow{d_{m}'} \cdots \xrightarrow{d_{d+n+1}'}  (\bigwedge\nolimits^{d+n} q^*\sW) \otimes \sO_{\PP(\sV)}(-n).
\end{align*}
(Here, the first terms of all three complexes are placed in degree $0$.) Then there are canonical isomoprhisms of complexes of vector bundles:
	\begin{align*}
	\mathbf{G}_d(\PP(\sV), \varphi) \simeq  \bS^{d}(\PP(\sV), \varphi) \qquad
	(\mathbf{H}_e(\PP(\sV), \varphi))^\vee \simeq \Sigma^{-e} \, (q^* \det \sW)^\vee \otimes \bS^{e}(\PP(\sV), \varphi) \otimes \sO(n)
	\end{align*}
(where $\bS^{d}(\PP(\sV), \varphi)$ denotes the Koszul complex defined by formula \eqref{eqn:Koszul:S^n} as usual, and we let $(\mathbf{H}_e(\PP(\sV), \varphi))^\vee$ denote the term-wise dual of the complex $\mathbf{H}_e(\PP(\sV), \varphi)$). Moreover, there are canonical fiber sequences determined by the original complex \eqref{eqn:thm:Serre:O(d):Koszul}:
\begin{align*} 
	& \mathbf{G}_d(\PP(\sV), \varphi) \to \widetilde{\mathbf{G}}_d(\PP(\sV), \varphi)\to \Sigma^{d+1} \,\mathbf{E}_d(\PP(\sV), \varphi), \\	
	& \widetilde{\mathbf{G}}_d(\PP(\sV), \varphi) \to \iota_* (\sO(d)) \to \Sigma^{d+n} \,\mathbf{H}_e(\PP(\sV), \varphi).
	\end{align*}
By virtue of Theorem \ref{thm:Serre:bundle} \eqref{thm:Serre:bundle-2i}, we obtain a canonical equivalence 
	$$q_*[\mathbf{G}_d(\PP(\sV), \varphi) ] \simeq q_*[ \widetilde{\mathbf{G}}_d(\PP(\sV), \varphi)]$$
 and hence a canonical fiber sequence
	\begin{align*} 
	q_*\left[\mathbf{G}_d(\PP(\sV),\varphi)\right] \xrightarrow{\varphi_d'} \pr_*(\sO(d)) \xrightarrow{\psi_d'} q_* \left[\Sigma^{n+d} \,\mathbf{H}_e(\PP(\sV), \varphi) \right].
	\end{align*}
The proof of assertion \eqref{thm:Serre:O(d)-1i} implies that there is a canonical equivalence $q_*[\mathbf{G}_d(\PP(\sV), \varphi)] \simeq \Sym_X^d(\sF)$ such that the map $\varphi_d'$ is canonically equivalent to the map $\varphi_d \colon \Sym_X^d(\sF) \to \pr_*(\sO(d))$. Similarly, combing with Grothendieck duality (Theorem \ref{thm:Neeman-Lipman-Lurie} \eqref{thm:Neeman-Lipman-Lurie-2iv}\eqref{thm:Neeman-Lipman-Lurie-2vii})
	$$(q_* \left[\Sigma^{n+d} \,\mathbf{H}_e(\PP(\sV), \varphi) \right])^\vee \simeq q_*(\Sigma^{e-m} [\mathbf{H}_e(\PP(\sV), \varphi)]^\vee \otimes \omega_q) \simeq  \Sigma^{r-1} \det \sF \otimes q_*[\bS^e(\PP(\sV), \varphi)],$$
there is a canonical equivalence 
	$$(q_* \left[\Sigma^{n+d} \,\mathbf{H}_e(\PP(\sV), \varphi) \right])^\vee \simeq  \Sigma^{r-1} \det \sF \otimes \Sym_X^e(\sF)$$
for which the composite map
	\begin{align*}
	\pr_*(\sO(d)) &\simeq (\pr_* \sO(e))^\vee \otimes  \Sigma^{1-r} (\det \sF)^\vee \\&\xrightarrow{\varphi_e^\vee} \Sigma^{1-r} (\det \sF)^\vee \otimes (\Sym_X^e(\sF))^\vee \simeq  q_* \left[\Sigma^{n+d} \,\mathbf{H}_e(\PP(\sV), \varphi) \right]
	\end{align*}
is canonically equivalent to $\psi_d'$. This proves assertion \eqref{thm:Serre:O(d)-2iii}. \end{proof}

\begin{remark}[{Grothendieck--Serre duality}] Let $\sF$ be a quasi-coherent sheaf on a prestack $X$ of perfect-amplitude in $[0,1]$ and rank $r$, and let $\pr \colon \PP(\sF) \to X$ denote the derived projectivization. Then the above proof of assertions \eqref{thm:Serre:O(d)-1iv} and \eqref{thm:Serre:O(d)-2iv} of Theorem \ref{thm:Serre:O(d)} shows that, for any $d \in \ZZ$, the composite map
	$$\pr_*(\sO(d)) \otimes \pr_*\big( \Sigma^{r-1} \sO(-d-r) \otimes \pr^* (\det \sF)\big) \to \pr_*(\Sigma^{r-1} \pr^* (\det \sF) \otimes \sO(-r)) \xrightarrow{\sim} \sO_X$$
is a duality datum in the symmetric monoidal $\infty$-category $\QCoh(X)^\otimes$ (in the sense of {\cite[Definition 4.6.1.7]{HA}}; notice this statement implies that if one of the two tensor factors is zero then the other is zero). We could regard this statement as the generalization of the {Grothendieck--Serre duality} for the projective bundles (Remark \ref{rmk:GSforbundle}). \end{remark}

\begin{remark}[Eagon--Northcott complexes]\label{rmk:ENcomplexes}  Let $X$ be a prestack, and let $\sF$ be a quasi-coherent sheaf on $X$ of perfect-amplitude contained in $[0,1]$. If $\sF$ has rank $r \le 0$, then we could regard the fiber sequences of Theorem \ref{thm:Serre:O(d)} \eqref{thm:Serre:O(d)-2iii} as the {\em derived version} of the classical Eagon--Northcott complexes. 

More concretely, if $X$ is a classical scheme, and $\sF$ is represented by the cofiber of a map $\sigma \colon \sW \to \sV$ between vector bundles on $X$, where $\rk \sW = m$, $\rk \sV = n$, and $r = n -m$. Then for any integer $d$, we define a complex $\mathbf{EN}_d(X, \sigma \colon \sW \to \sV)$ of vector bundles on $X$ by the following formula: if $r = n - m >0$, then we define
			$$
			\mathbf{EN}_d(X, \sigma) =
			\begin{cases}
				\bS^{-r-d}(X, \sigma)^\vee  \otimes_{\sO_X} (\det \sF)^\vee [1-r]& \text{if} \quad d \le -1;\\
				\bS^d(X, \sigma)& \text{if}\quad d \ge 0.
			\end{cases}
			$$	
(Here, $\bS^{d}(X, \sigma)$ denotes the complex defined by formula \eqref{eqn:Koszul:S^n} as usual, $[\blank]$ denotes the degree shift, and $(\blank)^\vee$ denotes the term-wise dual of a complex.) If $r = n - m \le 0$, then we define
			$$
			\mathbf{EN}_d(X, \sigma) =
			\begin{cases}
				\bS^{-r-d}(X, \sigma)^\vee  \otimes_{\sO_X} (\det \sF)^\vee [1-r]& \text{if} \quad d \le -1;\\
				 \bS^{-r-d}(X, \sigma)^\vee  \otimes_{\sO_X} (\det \sF)^\vee  [-r] \xrightarrow{\varepsilon} \bS^d(X, \sigma) &  \text{if} \quad 0 \le d \le -r; \\
				\bS^d(X, \sigma)& \text{if}\quad d \ge -r+1,
			\end{cases}
			$$	
where in the second formula, the rightmost term of the complex
	$\bS^{-r-d}(X, \sigma)^\vee  \otimes_{\sO_X} (\det \sF)^\vee [-r]$,
	$$\bigwedge\nolimits^{-r-d} \sW^\vee \otimes \det \sW \otimes (\det \sV)^\vee  \simeq \bigwedge\nolimits^{n + d} \sW \otimes (\det \sV)^\vee,$$
is placed at homological degree $d+1$, and $\varepsilon$ is the composite map
	$$\bigwedge\nolimits^{n + d} \sW \otimes (\det \sV)^\vee \xrightarrow{\wedge^n (\sigma^\vee)}  \bigwedge\nolimits^{n + d} \sW \otimes \bigwedge\nolimits^n (\sW^\vee) \xrightarrow{\llcorner} \bigwedge\nolimits^{d} \sW,$$
where the last map is the inner product of exterior algebras (\cite[Chapter III, \S 11.7]{Bou}). The complexes $\{\mathbf{EN}^d(X, \sigma)\}_{d \in \ZZ}$ are usually referred to as the {\em Eagon--Northcott complexes}. By our construction, the sequence of complexes $\{\mathbf{EN}_d(X, \sigma)\}_{d \in \ZZ}$ satisfies the following duality:
	$$\mathbf{EN}_d(X, \sigma) \otimes_{\sO_X} \det \sF \simeq  \left(\mathbf{EN}_{-r-d}(X, \sigma) \right)^\vee.$$
By virtue of the equivalences $[\bS^d(X, \sigma)] \simeq \Sym_X^d(\sF)$ and $[\bS^{-r-d}(X, \sigma)] \simeq \Sym_X^{-r-d}(\sF)$ of Corollary \ref{cor:Illusie:S^n}, and our proof of Theorem \ref{thm:Serre:O(d)} \eqref{thm:Serre:O(d)-2iii},
it is not hard to see that the complex $\mathbf{EN}_d(X, \sigma)$ canonically represents the cofiber of the map 
		$$\Sigma^{-r} (\Sym_X^{-r-d} \sF )^\vee \otimes_{\sO_X}  (\det \sF)^{\vee} \to \Sym_{X}^d (\sF) $$
induced by the fiber sequence of Theorem \ref{thm:Serre:O(d)} \eqref{thm:Serre:O(d)-2iii}. In particular, Theorem \ref{thm:Serre:O(d)} \eqref{thm:Serre:O(d)-2iii} implies that (in the case where $X$ is a classical scheme) there are canonical equivalences
	$$[\mathbf{EN}_d(X, \sigma)] \simeq \pr_* (\sO(d))$$
for all $d \in \ZZ$. Historically, in Kempf's approach (reviewed in \cite[\S B.2]{Laz04}), the above equivalence is taken as the {\em definition} of the Eagon--Northcott complex. We refer the readers to \cite[\S A.2.6.1]{Ei} and \cite[\S B.2]{Laz04} for more about Eagon--Northcott complexes. 
\end{remark}

The Generalized Serre's theorem is {\em not} true in general for classical projectivizations, even in the case where $\sF$ has positive rank and homological dimension one, and in degree $d=1$, as is shown in the following example:

\begin{example}[{\cite[\href{https://stacks.math.columbia.edu/tag/01OC}{Tag 01OC}]{stacks-project}}] Let $X = \Spec A$, $A= \kappa[u, v, s_1, s_2, t_1, t_2]/I$, where $\kappa$ is a field and $I = (-us_1 + vt_1 + ut_2, vs_1 + us_2 - vt_2, vs_2, ut_1)$.  Let $\overline{u}$ and $\overline{v}$ denote the class of $u$ and $v$ in $A$, respectively. Let $\sF$ be the cofiber of the map
	$$\sO_X \xrightarrow{(\overline{u}, \overline{v})} \sO_X^{\oplus 2}.$$
(Equivalently, $\sF$ is the sheaf corresponds to the module $M = (Ax \oplus Ay)/A(\overline{u}x + \overline{v}y)$.) Then the computation of {\cite[\href{https://stacks.math.columbia.edu/tag/01OC}{Tag 01OC}]{stacks-project}} shows that the classical projectivization $\PP_\cl(\sF)$ over $X$ violates the Serre's theorem. More concretely, {\em loc. cit.} shows that the canonical map
	$$\sH^0(\sF) \to  \RR^0 \pr_{\PP_\cl(\sF) \,*}(\sO_{\PP_\cl(\sF)}(1))$$
is not injective, which, in particular implies that the canonical map 
	$$\sF \to \pr_{\PP_\cl(\sF) \,*}(\sO_{\PP_\cl(\sF)}(1))$$
is not an equivalence. On the other hand, we know from Theorem \ref{thm:Serre:O(d)} \eqref{thm:Serre:O(d)-1i} that for the derived projectivization $\pr_{\PP(\sF)} \colon \PP(\sF) \to X$, the canonical map
	$$\varphi_d \colon \Sym_X^d(\sF) \to \pr_{\PP(\sF) \,*} (\sO_{\PP(\sF)}(d))$$
is an equivalence for all $d \ge 0$.
\end{example}

The next example shows that, the Generalized Serre's theorem does not hold for ``three-term complexes of vector bundles", even in very nice cases such as Koszul complexes:

\begin{example} Let $X = \AA^2 = \Spec \ZZ[x,y]$, let $p$ be the zero locus of defined by the maximal ideal $(x,y)$, and let $\sF = \sO_p$ denote the structure sheaf of $p$. Then $\sF$ can be resolved by the Koszul complex on $(x,y)$. The classical projectivization $\PP_\cl(\sF) \simeq \{p\}$, however, $\PP(\sF)$ is equipped with a non-trivial derived structure. Therefore $\pr_*(\sO_{\PP(\sF)}(1))$ is not isomorphic to $\sF \simeq \sO_p$ but is equipped with non-trivial higher homological information.
\end{example}

The above two examples seem to suggest that the generalized Serre's Theorem as we formulate in Theorem \ref{thm:Serre:O(d)} for sheaves of perfect amplitude contained in $[0,1]$ might be the most general form we could expect. However, there are possible ways to generalizing these results to sheaves of higher Tor-amplitude, if we consider variants of the derived projectivization constructions. For example, in the case of three-term complexes, if a given complex possesses certain symmetries, then we might consider the {\em reduced} derived structure on the projectivization (similar to the reduced obstruction theory in Donaldson--Thomas theory); it is then reasonable to expect that (variants of ) Serre's theorem to hold in these cases.

\section{Beilinson's relations}
Let $\kappa$ be a field and $V = \kappa^{\oplus n+1}$ a $\kappa$-vector space of rank $(n+1)$, where $n \ge 0$ is an integer. We let $\PP_\kappa^n = \PP_{\Spec \kappa}(V)$ denote the projective space over $\kappa$ of dimension $n$. Let $\sO(1)$ denote the universal line bundle on $\PP_\kappa^n$, and for an integer $d$, we let set $\sO(d) = \sO(1)^{\otimes d}$ if $d \ge 0$ and $\sO(d) = (\sO(1)^\vee)^{\otimes -d}$ if $d <0$ as usual. We let $\Omega^1 \simeq L_{\PP_\kappa^n/\kappa}$ denote algebraic cotangent bundle of $\PP_\kappa^n$, let $\Omega^i = \bigwedge\nolimits^i \Omega^1$ for $i \ge 0$, and let $\Omega^i(j) = \Omega^i \otimes \sO(j)$ for $i \ge 0$ and $j \in \ZZ$.

We may formulate Beilinson's computations in \cite{Be} in the following way:

\begin{description}[leftmargin=*]		
	\item[Beilinson's relations] The two sequences $\{\sO, \sO(1), \ldots, \sO(n)\}$ and $\{\Sigma^{n} \Omega^n(n), \ldots, \Sigma\,\Omega^1(1), \sO\}$ form a pair of (full) dual exceptional collections over $\kappa$. Moreover, for all integers $0 \le i,j \le n$, the following ``Beilinson's relations" hold:
	\begin{align*}
		& \Hom_\kappa^*(\sO(i), \sO(j)) = S_\kappa^{j-i} (V) \qquad
		 \Hom_\kappa^*(\Omega^i(i), \Omega^j(j)) =\big(\bigwedge\nolimits^{i-j}_{\kappa} V\big)^\vee\\
		&\Hom_\kappa^*(\sO(i), \Sigma^j \,\Omega^j(j)) = \delta_{i,j} \cdot \kappa,
	\end{align*}
where $\Hom_\kappa^*(\blank, \blank)$ denotes the $\Mod_\kappa$-valued $\Hom$ objects for quasi-coherent sheaves on $\PP_\kappa^n$, and $\delta_{i,j}$ is Kronecker delta function, that is, $\delta_{i,j} = 1$ if $i=j$ and $\delta_{i,j}=0$ if $i \ne j$.
\end{description}

\begin{remark} A generalization of Beilinson's result to the cases of all projectivizations of vector bundles over quasi-compact, quasi-separated schemes (together with a generalized version of Orlov's Theorem) can be found in \cite[Theorem B.3]{J21}.
\end{remark}

In this section, we prove a derived version of Beilinson's relations for all derived projectivizations of quasi-coherent sheaves of perfect-amplitude contained in $[0,1]$ over a perfect stack $X$ (Proposition \ref{prop:PG:dualexc}, Corollary \ref{cor:PVdot:relations}). The results and proofs of \S \ref{sec:beilinson} are most suitably formulated using the framework of relative exceptional sequences and mutation theory, as developed in the classical absolute cases in \cite{Bo, BK, Go} (see also \cite[\S 3.1]{BLM+} and \cite[\S 3.11]{J21} for the classical relative cases), which we discuss in \S \ref{sec:exc:mut}.

\subsection{Exceptional sequences and mutation functors} \label{sec:exc:mut} In this subsection, we will study the theory of relative exceptional sequences and mutations for linear $\infty$-categories.

\begin{definition} Let $X$ be a prestack. We will refer to the $\infty$-category $\Mod_{\QCoh(X)}(\Pr^\St)$ as the {\em $\infty$-category of linear stable $\infty$-categories over $X$} (where $\Pr^\St$ is the $\infty$-category of presentable stable $\infty$-categories). 
We will refer to the objects of $\Mod_{\QCoh(X)}(\Pr^\St)$ as {\em $X$-linear stable $\infty$-categories}, and morphisms in $\Mod_{\QCoh(X)}(\Pr^\St)$ as {\em $X$-linear functors}. The $\infty$-category $\Mod_{\QCoh(X)}(\Pr^\St)$ is equipped with a symmetric monoidal structure given by the relative tensor product $\otimes_{\QCoh(X)}$ over $\QCoh(X)$, where $\QCoh(X)$ is regarded as a commutative algebra object of $\Pr^\St$. For any $\shC \in \Mod_{\QCoh(X)}(\Pr^\St)$, we denote the action functor of $\QCoh(X)$ on $\shC$ by $\otimes \colon \QCoh(X) \times \shC \to \shC$, $(\sF, C) \mapsto \sF \otimes C$. 
\end{definition} 

 For technical reason, it is helpful to restrict our attention to {\em perfect stacks} $X$, which guarantees that we have a well behaved theory of $X$-linear categories.
\begin{enumerate}
	\item[$(\star)$] Throughout this subsection, we will assume that $X$ is a {\em perfect stack}.
\end{enumerate}

\begin{definition}[{Compare with \cite[Definition 9.4.4.1]{SAG}}] \label{def:perfstacks} We say a prestack $X \colon \CAlgDelta \to \shS$ is a {\em perfect stack} if it satisfies the following conditions:
	\begin{enumerate}
		\item $X$ is a quasi-geometric stack in the sense of (the derived version) of \cite[Definition 9.1.0.1]{SAG} (that is, $X$ satisfies descent with respect to the fpqc topology, the diagonal map $\delta \colon X \to X \times X$ is quasi-affine, and there exists a simplicial commutative ring $A \in \CAlgDelta$ and a faithful flat morphism $u \colon \Spec A \to X$).
		\item The $\infty$-category $\QCoh(X)$ is compactly generated and the structure sheaf $\sO_X$ is a compact object of $\QCoh(X)$.
	\end{enumerate}
Notice that given condition $(1)$, condition $(2)$ is equivalent to:
\begin{enumerate}
	\item[$(2')$] The inclusion $\Perf(X) \subseteq \QCoh(X)$ 
	 extends to an equivalence of $\infty$-categories 
	 	$$\Ind(\Perf(X)) \simeq \QCoh(X).$$
\end{enumerate}
\end{definition}

\begin{remark}
Our Definition \ref{def:perfstacks} is slightly different from the definition of a perfect stack in \cite{BFN10, GRI}. All these notions of perfect stacks work equally fine for all the discussions of this section; we choose the above Definition \ref{def:perfstacks} so that the following is true:
\begin{itemize}
	\item Any quasi-compact, quasi-separated derived scheme 
	$X$ is a perfect stack.
\end{itemize}
The above assertion can be deduced from \cite[Proposition 9.6.1.1]{SAG}, since the underlying spectral scheme $X^\circ$ of $X$ is a quasi-compact, quasi-separated spectral algebraic space. On the other hand, \cite{BFN10} requires semi-separateness as part of the definition of a perfect stack.
\end{remark}

\begin{proposition} \label{prop:perfectstacks} Let $X$ be a perfect stack (Definition \ref{def:perfstacks}). Then $\QCoh(X)$ is locally rigid (in the sense of {\cite[Definition D.7.4.1]{SAG}}). In particular, the following holds:
	\begin{enumerate}
		\item \label{prop:perfectstacks-1} The functor $\lambda = \Mapsp_{\QCoh(X)}(\sO_X, \blank) \colon \QCoh(X) \to \Sp$ exhibits $\QCoh(X)$ as a Frobenius algebra object of  $\Pr^\St$ (in the sense of \cite[Definition 4.6.5.1]{HA}, that is, the composition $\QCoh(X) \otimes \QCoh(X) \to \QCoh(X) \xrightarrow{\lambda} \Sp$ is a duality datum in the symmetric monoidal $\infty$-category $\Pr^\St$).
		\item \label{prop:perfectstacks-2} $\QCoh(X)$ is smooth when regarded as an algebra object of $\Pr^\St$ (in the sense of \cite[Definition 4.6.4.13]{HA}, that is, $\QCoh(X)$ is dualizable when regarded as a module over $\QCoh(X)^\op \otimes \QCoh(X)$).
		\item \label{prop:perfectstacks-3} The forgetful functor $\Mod_{\QCoh(X)}(\Pr^\St) \to \Pr^\St$ is compatible with duality. In other words, suppose we are given an object $\shC \in\Mod_{\QCoh(X)}(\Pr^\St)$. Then any duality datum $e \colon \shC \otimes_{\QCoh(X)} \shD \to \QCoh(X)$ in $\Mod_{\QCoh(X)}(\Pr^\St)$ induces a duality datum in $\Pr^\St$ through the composition
			$$\shC \otimes \shD \to \shC \otimes_{\QCoh(X)} \shD \xrightarrow{e} \QCoh(X) \xrightarrow{\lambda} \Sp;$$
		On the other hand, for any duality datum $\overline{e} \colon \shC \otimes \shD \to \Sp$ in $\Pr^\St$, the $\infty$-category $\shD$ admits an action of $\QCoh(X)$ for which $\overline{e}$ factorizes as a composition 
			$$\shC \otimes \shD \to \shC \otimes_{\QCoh(X)} \shD \xrightarrow{e} \QCoh(X) \xrightarrow{\lambda} \Sp,$$
		where $e$ is a duality datum in $\Mod_{\QCoh(X)}(\Pr^\St)$.
	\end{enumerate}
\end{proposition}
\begin{proof} It follows from our definition of a perfect stack that $\QCoh(X)$ is locally rigid. Then the assertions \eqref{prop:perfectstacks-1} and \eqref{prop:perfectstacks-2} are formal consequences of the locally rigidity of $\QCoh(X)$ (\cite[Propositions D.7.5.1, D.7.7.1]{SAG}). The assertion \eqref{prop:perfectstacks-3} is a consequence of \eqref{prop:perfectstacks-1} and \eqref{prop:perfectstacks-2} (\cite[Remark 4.6.5.15]{HA}). Notice that the assertions \eqref{prop:perfectstacks-1} through \eqref{prop:perfectstacks-3} are true even under the a priori weaker assumption that $X$ is weakly perfect: see \cite[Proposition 9.4.2.1, Corollaries 9.4.3.4, 9.4.3.5, 9.4.3.6]{SAG}.
\end{proof}

\begin{construction} \label{constr:mappingobject}Let $X$ be a perfect stack, and let $\shC \in \Mod_{\QCoh(X)}(\Pr^\St)$ be an $X$-linear stable $\infty$-category. Assume that $\shC$ is compactly generated. Therefore by virtue of \cite[Proposition D.7.2.3]{SAG}, $\shC$ is dualizable as an object of $\Pr^\St$, and we can identify its dual with the $\infty$-category $\Ind(\shC_c^\op)$, where $\shC_c$ is the full subcategory of $\shC$ spanned by compact objects. Then Proposition \ref{prop:perfectstacks} \eqref{prop:perfectstacks-3} implies that $\shC$ is also dualizable as a $\QCoh(X)$-module; We let 
	$$e \colon \Ind(\shC_c^\op) \otimes_{\QCoh(X)} \shC \to \QCoh(X)$$
denote the duality datum in $\Mod_{\QCoh(X)}(\Pr^\St)$ obtained from Proposition \ref{prop:perfectstacks} \eqref{prop:perfectstacks-3}. In particular, the restriction of $e$ induces a functor
	$$\Mapsp_{\shC/X} \colon \shC_c^{\op} \times \shC \to \QCoh(X) \qquad (C, D) \mapsto \Mapsp_{\shC/X}(C,D)$$
	which is characterized by the following universal property: there is a canonical {\em evaluation} map $\Mapsp_{\shC/X}(C,D) \otimes C \to D$ such that, for any $\sF \in \QCoh(X)$, the composition
	$$\Map_{\QCoh(X)}(\sF, \Mapsp_{\shC/X}(C,D)) \to  \Map_{\shC}(\sF \otimes C, \Mapsp_{\shC/X}(C,D)\otimes C) \to \Map_{\shC}(\sF \otimes C, D)$$
is a homotopy equivalence. For any pair of object $(C, D) \in \shC_c^{\op} \times \shC$, we will refer to $\Mapsp_{\shC/X}(C,D)$ as the {\em mapping object} of $C$ and $D$ in $\QCoh(X)$. We will sometimes simplify the notation by denoting $\Mapsp_{X}(C,D) = \Mapsp_{\shC/X}(C,D)$. Notice that from our construction, given any compact object $C \in \shC_c$, the induced functor 
	$$\Mapsp_{\shC/X}(C, \blank) \colon \shC \to \QCoh(X)$$
 is $\QCoh(X)$-linear and preserves all small colimits; For any $D \in \shC$, the induced functor 
 	$$\Mapsp_{\shC/X}(\blank, D) \colon \shC_c^\op \to \QCoh(X)$$
 is $\Perf(X)$-linear and preserves all finite colimits (hence is an exact functor). Moreover, for any triple of objects $\sF \in \Perf(X)$, $C \in \shC_c$ and $D \in \shC$, there is a canonical equivalence
 	$$\Mapsp_{\shC/X}(\sF \otimes C, D) \simeq \Mapsp_{\shC/X}(C, \sF^\vee \otimes D)$$
for which the application of the functor $\Map_{\QCoh(X)}(\sO_X, \blank)$ induces a homotopy equivalence
	$$\Map_{\shC}(\sF \otimes C, D) \simeq \Map_{\shC}(C, \sF^\vee \otimes D)$$
which exhibits $\sF^\vee \otimes D$ as an exponential object of $D$ by $\sF$ (see \cite[Lemma 4.6.1.5]{HA}). 
\end{construction}

\begin{example} \label{eg:qcqs:mappingobject} Let $f \colon Y \to X$ be a map between perfect stacks which is a relative quasi-compact, quasi-separated derived scheme. Then the pullback functor $f^* \colon \QCoh(X) \to \QCoh(Y)$ is symmetric monoidal and exhibits $\QCoh(Y)$ as an $X$-linear stable $\infty$-category. We let $f_* \colon \QCoh(Y) \to \QCoh(X)$ denote the right adjoint of $f^*$. For any perfect object $\sF \in \Perf(Y)$, we let $\sF^\vee$ denote the dual of $\sF$ in $\QCoh(Y)$. Then, for all triple of objects $\sF \in \Perf(Y)$, $\sG \in \QCoh(Y)$ and $\sE \in \QCoh(X)$, the composite map 
	$$f^*f_*(\sF^\vee \otimes_{\sO_Y} \sG) \otimes_{\sO_Y} \sF \to \sF^\vee \otimes_{\sO_Y} \sG \otimes_{\sO_Y} \sF \simeq  \sF^\vee \otimes_{\sO_Y} \sF \otimes_{\sO_Y} \sG  \to \sG $$
together with the canonical homotopy equivalences
	$$\Map_{\QCoh(X)}(\sE, f_*(\sF^\vee \otimes_{\sO_Y} \sG)) \simeq \Map_{\QCoh(Y)}(f^* \sE, \sF^\vee \otimes_{\sO_Y} \sG) \simeq \Map_{\QCoh(Y)}(f^*\sE \otimes_{\sO_Y} \sF, \sG)$$
exhibit $f_*(\sF^\vee \otimes_{\sO_Y} \sG)$ as a mapping object $\Mapsp_{\QCoh(Y)/X}(\sF, \sG) \in \QCoh(X)$. 
\end{example}

\begin{definition}[Relative exceptional objects] \label{def:relexcseq} Let $\shC$ be an $X$-linear stable $\infty$-category which is compactly generated. Let $E$ be a compact object of $\shC$, and let $\Mapsp_{\shC/X}(E,E)$ be the mapping object of $E$ with itself (Construction \ref{constr:mappingobject}), and let $\sO_X \to \Mapsp_{\shC/X}(E,E)$ be the canonical map in $\QCoh(X)$ which classifies the identity map $\id_E \colon E \to E$. We say an object $E \in \shC$ is a {\em relative exceptional object (over $X$)} if $E$ is compact and the canonical map $\sO_X \to \Mapsp_{\shC/X}(E,E)$ is an equivalence. We say a sequence of objects $(E_1, E_2, \ldots, E_n)$ in $\shC$ is a {\em relative exceptional sequence (of length $n$ over $X$)} if each $E_i \in \shC$ is a relative exceptional object, and $\Mapsp_{\shC/X}(E_i, E_j) \simeq 0$ for all $i > j$. We will also refer to a relative exceptional sequence $(E_1, E_2)$ of $\shC$ of length $2$ as a {\em relative exceptional pair (over $X$)}. 
\end{definition}

\begin{lemma} Let $\shC$ be a $X$-linear stable $\infty$-category which is compactly generated, and let $E \in \shC$ be a relative exceptional object over $X$. Then the functor $\alpha_E \colon (\sF \in \QCoh(X)) \mapsto (\sF \otimes E \in \shC)$ is fully faithful and induces fully faithful embeddings of $\infty$-categories
	$$\alpha_E|{\Perf(X)} \colon \Perf(X) \hookrightarrow \shC_c \qquad \alpha_E \colon \QCoh(X) \hookrightarrow \shC.$$
Moreover, $\alpha_E$ admits a right adjoint given by the functor $\Mapsp_{\shC/X}(E, \blank) \colon \shC \to \QCoh(X)$.
\end{lemma}

\begin{proof} For any pair of perfect objects $\sF, \sG \in \QCoh(X)$, by applying the functor $\Map_{\shC/X}(\sO_X, \blank)$ to the canonical equivalences
	$$\Mapsp_{\shC/X}(\sF \otimes E, \sG \otimes E) \simeq \Mapsp_{\QCoh(X)}(\sF, \sG) \otimes_{\sO_X} \Mapsp_{\shC/X}(E,E) \simeq \Mapsp_{\QCoh(X)}(\sF, \sG),$$
we see that the exact functor $\alpha_E|\Perf(X)$ is fully faithful. Moreover, by virtue of the equivalence $\Map_\shC(\sF \otimes E, \blank) \simeq \Map_\shC(E, \sF^\vee \otimes \blank)$, we see that $\alpha_E$ preserves compact objects. Therefore, by virtue of \cite[Proposition 5.3.5.11 (1)]{HTT}, we deduce that $\alpha_E \colon \QCoh(X) \to \shC$ is fully faithful. Furthermore, for any pair of objects $C \in \shC$ and $\sF \in \QCoh(X)$, the evaluation map $\Mapsp_{\shC/X}(E,C) \otimes E \to C$ together with the induced homotopy equivalence 
		$$\Map_{\QCoh(X)}(\sF, \Mapsp_{\shC/X}(E,C)) \simeq \Map_{\shC}(\sF \otimes E, C)$$
exhibits $\Mapsp_{\shC/X}(E, \blank)$ as a right adjoint functor of $\alpha_E$. 
\end{proof}

\begin{definition}[Proper linear $\infty$-categories]Let $\shC$ be an $X$-linear stable $\infty$-category (where $X$ is a fixed perfect stack). We say that  $\shC$ is {\em proper over $X$} (or  $\shC$ is a {\em proper $X$-linear $\infty$-category}), if $\shC$ is compactly generated and, for each pair of compact objects $C,D \in \shC$, the mapping object $\Mapsp_{\shC/X}(C,D)$ (Construction \ref{constr:mappingobject}) is a perfect object of $\QCoh(X)$.
\end{definition}

By virtue of \cite[Proposition 11.1.3.3]{SAG}, (by applying the global section functor) the above definition is equivalent to Lurie's definition of a proper quasi-coherent stack in the case where $X$ is a quasi-compact, quasi-separated derived algebraic space. 

\begin{example}\label{eg:qperf:mappingobject:dual} In the situation of Example \ref{eg:qcqs:mappingobject}, if we furthermore assume that $f$ is proper, locally almost of ﬁnite presentation and locally of finite Tor-amplitude. Then $f_*$ preserves perfect objects (see \cite[Theorem 6.1.3.2]{SAG}). In particular, we deduce that $\QCoh(Y)$ is a {\em proper} $X$-linear stable $\infty$-category. Furthermore, by virtue of Theorem \ref{thm:Neeman-Lipman-Lurie} \eqref{thm:Neeman-Lipman-Lurie-2i}\eqref{thm:Neeman-Lipman-Lurie-2iii}\eqref{thm:Neeman-Lipman-Lurie-2vii}, 
$f_*$ admits a right adjoint $f^!$ which satisfies $f^!(\blank) \simeq f^*(\blank) \otimes \omega_f$, where $\omega_f = f^! (\sO_X) \in \QCoh(Y)$ is the {relative dualizing sheaf}, and the relative Serre--Grothendieck duality implies that for all perfect objects $\sF, \sG \in \Perf(Y)$, there is a canonical equivalence
	$$\Mapsp_{\QCoh(Y)/X}(\sF, \sG \otimes \omega_f) \simeq \Mapsp_{\QCoh(Y)/X}(\sG,\sF)^\vee$$
of quasi-coherent sheaves on $X$ (see, for example, \cite[Remark 11.1.5.4]{SAG}).
\end{example}

\begin{definition}[Mutation functors] \label{def:mutations} Let $\shC$ be a $X$-linear stable $\infty$-category which is compactly generated, and let $E \in \shC$ be a relative exceptional object over $X$. The {\em left mutation functor (passing through $E$)}, denoted by $\LL_E \colon \shC \to \shC$, is the colimit-preserving functor which carries each object $A \in \shC$ to the cofiber of the evaluation map $\Mapsp_{\shC/E}(E, A) \otimes E \to A$. In other words, for each $A \in \shC$, there is a canonical fiber sequence
		\begin{equation}\label{eqn:left:mut}
			\Mapsp_{\shC/X}(E,A) \otimes E \to A \to \LL_E(A).		
		\end{equation}
If we furthermore assume that $\shC$ is proper over $X$, then $\LL_E$ restricts to an exact functor $\shC_c \to \shC_c$, which we shall also denote by $\LL_E$. Moreover, for each compact object $A \in \shC$, the evaluation map $\Mapsp_{\shC/X}(A,E) \otimes A \to E$ classifies a coevaluation map $A \to \Mapsp_{\shC/X}(A,E)^\vee \otimes E$. The {\em right mutation functor (passing through $E$)}, denoted by $\RR_E \colon \shC_c \to \shC_c$, is the exact functor which carries each compact object $A \in \shC$ to the fiber of the coevaluation map $A \to \Mapsp_{\shC/X}(A,E)^\vee \otimes E$. In other words, there is a canonical fiber sequence
		\begin{equation}\label{eqn:right:mut}
			\RR_E(A) \to A \to \Mapsp_{\shC/X}(A,E)^\vee \otimes E.	
		\end{equation}
\end{definition}

The following is a list of fundamental properties about mutation functors (compare with the classical results in \cite{Bo, BK, Go} and our exposition in \cite[\S 3.11]{J21}):

\begin{lemma}\label{lem:mut:relexc} Let $\shC$ be a proper $X$-linear stable $\infty$-category, let $E \in \shC$ be a relative exceptional object over $X$, and let $A$ be any compact object of $\shC$. We let $\Mapsp_{X} = \Mapsp_{\shC/X}$ denote the mapping object functor of Construction \ref{constr:mappingobject}. Then:
\begin{enumerate} 
	\item \label{lem:mut:relexc-1} There are canonical equivalences 
		$$\Mapsp_{X}(E, \LL_E(A)) \simeq 0 \qquad  \Mapsp_{X}(\RR_E(A), E) \simeq 0.$$
	\item \label{lem:mut:relexc-2} If $\Mapsp_{X}(A,E) \simeq 0$, then there is a canonical equivalence 
		$$\Mapsp_{X}(\LL_E(A), E) \simeq \Sigma^{-1} \Mapsp_{X}(E,A)^\vee.$$
	Furthermore, for any compact object $B \in \shC$, there are bi-functorial equivalences 
	 	$$\Mapsp_{X}(A,B) \xrightarrow{\sim} \Mapsp_{X}(A, \LL_E(B)) \xleftarrow{\sim} \Mapsp_{X}(\LL_E(A), \LL_E(B)).$$
	\item \label{lem:mut:relexc-3} If $\Mapsp_{X}(E,A) \simeq 0$, then there is a canonical equivalence 
		$$\Mapsp_{X}(E,\RR_E(A)) \simeq \Sigma^{-1} \Mapsp_{X}(A,E)^\vee.$$
	Furthermore, for any compact object $B \in \shC$, there are bi-functorial equivalences 
	 	$$\Mapsp_{X}(B,A) \xrightarrow{\sim} \Mapsp_{X}(\RR_E(B), A) \xleftarrow{\sim} \Mapsp_{X}(\RR_E(B), \RR_E(A)).$$
	\item \label{lem:mut:relexc-4} If $\Mapsp_{X}(A,E) \simeq 0$, then there is a canonical equivalence $\RR_E \circ \LL_E (A) \simeq A$. Similarly, if $\Mapsp_{X}(E,A) \simeq 0$, then there is a canonical equivalence $\LL_E \circ \RR_E (A) \simeq A$.
\end{enumerate}
\end{lemma}

\begin{proof} The first equivalence of \eqref{lem:mut:relexc-1} follows from applying the exact functor $\Mapsp_{X}(E, \blank)$ to the fiber sequence \eqref{eqn:left:mut}, and the second equivalence of \eqref{lem:mut:relexc-1} follows from applying the exact functor $\Mapsp_{X}(\blank, E)$ to the fiber sequence \eqref{eqn:right:mut}. We now prove \eqref{lem:mut:relexc-2}; the proof of \eqref{lem:mut:relexc-3} is similar. Since $\shC$ is a stable $\infty$-category, \eqref{eqn:left:mut} induces a canonical fiber sequence
	\begin{equation}\label{eqn:left:mut:Sigma}
	A \to \LL_E(A) \to \Sigma \, \Mapsp_{X}(E,A) \otimes E
	\end{equation}
(which is determined up to a contractible space of choices) for which the application of the exact functor $\Mapsp_X(\blank, E)$ induces an equivalence $\Mapsp_{X}(\LL_E(A), E) \simeq \Sigma^{-1} \Mapsp_{X}(E,A)^\vee$. Applying the exact functor $\Mapsp_{X}(A, \blank)$ to the fiber sequence
	$$\Mapsp_X(E,B) \otimes E \to B \to \LL_E(B),$$
we obtain a canonical equivalence $\Mapsp_{X}(A,B) \simeq \Mapsp_{X}(A, \LL_E(B))$; By applying the exact functor $\Mapsp_X(\blank, \LL_E(B))$ to \eqref{eqn:left:mut} and using the equivalence $\Mapsp_{X}(E, \LL_E(B)) \simeq 0$ of  \eqref{lem:mut:relexc-1}, we obtain a canonical equivalence $\Mapsp_{X}(A,\LL_E(B)) \simeq \Mapsp_{X}(\LL_E(A), \LL_E(B))$.

Finally, we prove \eqref{lem:mut:relexc-4}. We will prove in the case where $\Mapsp_{X}(A,E) \simeq 0$; The other case is similar. From the way how we obtain the equivalence $\Mapsp_{X}(\LL_E(A), E) \simeq \Sigma^{-1} \Mapsp_{X}(E,A)^\vee$ of \eqref{lem:mut:relexc-2}, we see that the coevaluation map $\LL_E(A) \to \Mapsp_X(\LL_E(A), E)^\vee \otimes E$ is equivalent to the morphism $\LL_E(A) \to \Sigma \Mapsp_X(E,A) \otimes E$ in the fiber sequence \eqref{eqn:left:mut:Sigma}. Therefore, the fiber sequence \eqref{eqn:left:mut:Sigma}  exhibits $A$ as the right mutation of $\LL_E(A)$ passing through $E$.
\end{proof}

\begin{remark} In the situation of Lemma \ref{lem:mut:relexc}, if we let ${}^\perp \langle E \rangle \subseteq \shC_c$ (resp., $\langle E \rangle^\perp \subseteq \shC_c$) denote the full subcategory of $\shC_c$ spanned by those objects $C$ for which $\Mapsp_{X}(C,E) \simeq 0$ (resp., $\Mapsp_{X}(E,C) \simeq 0$). Then ${}^\perp \langle E \rangle$ and $\langle E \rangle^\perp$ are $\Perf(X)$-linear stable subcategories of $\shC_c$. Therefore  Lemma \ref{lem:mut:relexc} implies that the restrictions of the mutation functors $\LL_E$ and $\RR_E$ induces mutually inverse $\Perf(X)$-linear equivalences of stable $\infty$-categories
	$$\LL_{E}\,|_{{}^\perp \langle E \rangle} \colon {}^\perp \langle E \rangle  \xrightarrow{\sim} \langle E \rangle^\perp \qquad \RR_{E}\,|_{\langle E \rangle^\perp} \colon \langle E \rangle^\perp \xrightarrow{\simeq} {}^\perp \langle E \rangle.$$
\end{remark}

\begin{lemma} \label{lem:mut:relexcpair} Let $\shC$ be a proper $X$-linear stable $\infty$-category, let $(E,F)$ be a relative exceptional pair of objects of $\shC$ over $X$. We let $\Mapsp_{X} = \Mapsp_{\shC/X}$ denote the mapping object functor of Construction \ref{constr:mappingobject}. Then:
	\begin{enumerate}
	\item \label{lem:mut:relexcpair-1} Both $(\LL_E(F), E)$ and $(F, \RR_F (E))$  are relative exceptional pairs over $X$.
	\item \label{lem:mut:relexcpair-2} There are canonical equivalences
		$$\Mapsp_{X}(\LL_E(F), E) \simeq \Sigma^{-1} \Mapsp_{X}(E,F)^\vee \simeq \Mapsp_{X}(F, \RR_F(E)).$$
	\item \label{lem:mut:relexcpair-3} There are canonical equivalences
		$$\Mapsp_{X}(F, \LL_E(F)) \simeq \sO_X \qquad \Mapsp_{X}(\RR_F(E), E)\simeq \sO_X.$$
	\item \label{lem:mut:relexcpair-4} There are canonical natural isomorphisms of functors from $\shC_c$ to itself:
		$$\LL_{\LL_E(F)} \circ \LL_E \simeq \LL_E \circ \LL_F \qquad \RR_{\RR_F(E)} \circ \RR_F \simeq \RR_F \circ \RR_E.$$
	\end{enumerate}
\end{lemma}

\begin{proof} We first prove assertion \eqref{lem:mut:relexcpair-1}. It follows from the canonical equivalence 
	$$\Mapsp_{X}(\LL_E(F), \LL_E(F)) \simeq \Mapsp_{X}(F,F) \simeq \sO_X$$
of Lemma \ref{lem:mut:relexc} \eqref{lem:mut:relexc-2} that $\LL_E(F)$ is a relative exceptional object; similarly, Lemma \ref{lem:mut:relexc} \eqref{lem:mut:relexc-3} implies that $\RR_F(E)$ is a relative exceptional object. Moreover, Lemma \ref{lem:mut:relexc} \eqref{lem:mut:relexc-1} implies $(\LL_E(F), E)$ and  $(F, \RR_F (E))$ are relative exceptional pairs. 
The assertion \eqref{lem:mut:relexcpair-2} follows from Lemma \ref{lem:mut:relexc} \eqref{lem:mut:relexc-2} and \eqref{lem:mut:relexc-3}.
To prove assertion \eqref{lem:mut:relexcpair-3}, if we apply the second equation of Lemma \ref{lem:mut:relexc} \eqref{lem:mut:relexc-2} to the case $A= B = F$, then we obtain canonical equivalences
	$$\Mapsp_{X}(F, \LL_E(F)) \xleftarrow{\sim} \Mapsp_{X}(F, F) \xleftarrow{\sim} \sO_X;$$
Similarly, by applying Lemma \ref{lem:mut:relexc} \eqref{lem:mut:relexc-3},  
we obtain canonical equivalences
	$$\Mapsp_{X}(\RR_F(E), E) \xleftarrow{\sim} \Mapsp_{X}(E, E) \xleftarrow{\sim} \sO_X.$$
We now prove the equivalence $\LL_{\LL_E(F)} \circ \LL_E \simeq \LL_E \circ \LL_F$ of assertion  \eqref{lem:mut:relexcpair-4}; the other case is similar. For any $A \in \shC_c$, consider the following commutative diagram of evaluation maps
	$$
	\begin{tikzcd}
		\Mapsp_{X}(F,A) \otimes  \Mapsp_{X}(E,F) \otimes E \ar{d} \ar{r} & \Mapsp_{X}(F,A) \otimes F \ar{d}  		\\
		\Mapsp_{X}(E,A) \otimes E \ar{r} & A. 
	\end{tikzcd}
	$$
By virtue of $(TR2)$ of the stable $\infty$-category $\shC_c$ (see \cite[Theorem 1.1.2.14]{HA}), the above commutative diagram extends to a commutative diagram of fiber sequences
	$$
	\begin{tikzcd}
		\Mapsp_{X}(F,A) \otimes  \Mapsp_{X}(E,F) \otimes E \ar{d} \ar{r}& \Mapsp_{X}(F,A) \otimes F \ar{d}  \ar{r}  & \Mapsp_{X}(F,A)  \otimes \LL_E(F) \ar{d} \\
		\Mapsp_{X}(E,A) \otimes E \ar{r} & A \ar{r} & \LL_E(A)
	\end{tikzcd}
	$$
By virtue of the canonical equivalence 
	$$\Mapsp_{X}(E, \Mapsp_{X}(F,A) \otimes F) \simeq \Mapsp_{X}(F,A) \otimes  \Mapsp_{X}(E,F),$$ 
we can identify the first vertical arrow of the above diagram with the composition of the functor $\Mapsp_{X}(E, \blank)$ and the evaluation map $\Mapsp_{X}(F,A) \otimes F \to A$. By virtue of the equivalence $\Mapsp_{X}(\LL_E(F), \LL_E(A)) \simeq \Mapsp_{X}(F,A)$ of  Lemma \ref{lem:mut:relexc} \eqref{lem:mut:relexc-2}, we can identify the last vertical arrow of the above diagram with the evaluation map  
	$$\Mapsp_{X}(\LL_E(F), \LL_E(A)) \otimes \LL_E(F) \to \LL_E(A).$$
By virtue of $(TR4)$ of the stable $\infty$-category $\shC_c$ (see \cite[Theorem 1.1.2.14]{HA}), we obtain that the above commutative diagram induces a fiber sequence
	$$\Mapsp_{X}(E, \LL_F(A)) \otimes E \to \LL_F(A) \to \LL_{\LL_E(F)}(\LL_E(A))$$
which exhibits $\LL_{\LL_E(F)}(\LL_E(A))$ as the left mutation of $\LL_F(A)$ passing through $E$. \end{proof}

\begin{definition} \label{def:dualrelexc} Let $\shC$ be a proper $X$-linear stable $\infty$-category, and let $E_\bullet=(E_1, E_2, \ldots, E_n)$ be a relative exceptional sequence of $\shC$ over $X$. We define the {\em left dual (relative exceptional) sequence $D_\bullet$ (over $X$)} of $E_\bullet$ by the following iterated left mutation process:
	$$D_1 := E_1, \qquad D_{i} := \LL_{E_1} \circ \ldots \circ \LL_{E_{i-1}} (E_i) \quad \text{for $2 \le i \le n$}.$$
Similarly, we define the {\em right dual (relative exceptional) sequence $D_\bullet$ (over $X$)} of $E_\bullet$ by the following iterated right mutation process: 
	$$F_n := E_n, \qquad F_{i} := \RR_{E_{n}} \circ \ldots \circ \RR_{E_{i+1}} (E_i)  \quad \text{for $1 \le i \le n-1$}. $$
\end{definition}

 \begin{corollary}\label{cor:dulexcseq} Let $\shC$ be a proper $X$-linear stable $\infty$-category, and let $E_\bullet=(E_1, E_2, \ldots, E_n)$ be a relative exceptional sequence of $\shC$ over $X$. Then both the left dual sequence $D_\bullet = (D_n ,\ldots, D_1)$ and the right dual sequence $F_\bullet =(F_1, \ldots, F_n)$ of $E_\bullet$ are relative exceptional sequences    over $X$, for which there are canonical equivalences (``relations of dual basis"):
  	$$
	\Mapsp_{\shC/X}(E_i, D_j) \simeq  \begin{cases} \sO_X & \text{if $i = j$}. \\
	0 & \text{ if $i\ne j$}. \end{cases}
	\qquad
	\Mapsp_{\shC/X}(F_i, E_j) \simeq \begin{cases} \sO_X &\text{if $i = j$}. \\
	0 & \text{if $i = j$}. \end{cases}
	$$
 \end{corollary}
\begin{proof} This follows from an iterated application of Lemmas \ref{lem:mut:relexc} and \ref{lem:mut:relexcpair}.
\end{proof}

\begin{remark} \label{rmk:excseqspan} In the situation of Corollary \ref{cor:dulexcseq}, we let $\langle \Perf(X) \otimes E_\bullet \rangle$ denote the smallest stable subcategory of $\shC_c$ which contains all elements of the form $\sF \otimes E_i$, where $\sF \in \Perf(X)$ and $1 \le i \le n$, and let $\langle \Perf(X) \otimes D_\bullet \rangle$ and $\langle \Perf(X) \otimes F_\bullet \rangle$ be defined similarly. Then there are canonical identifications of subcategories
	$$\langle \Perf(X) \otimes E_\bullet \rangle = \langle \Perf(X) \otimes D_\bullet \rangle = \langle \Perf(X) \otimes F_\bullet \rangle \subseteq \shC_c.$$
This follows from an iterated application of the fiber sequences \eqref{eqn:left:mut} and \eqref{eqn:right:mut}.
\end{remark}

\subsection{Beilinson's relations}\label{sec:beilinson}
In this subsection, we present the derived version of Beilinson's relations for projectivizations of sheaves of perfect-amplitude contained in $[0,1]$.

\begin{proposition}[Dual exceptional sequences] \label{prop:PG:dualexc} Let $X$ be a perfect stack in the sense of Definition \ref{def:perfstacks}, and let $\sF$ be a quasi-coherent sheaf on $X$ of perfect-amplitude contained in $[0,1]$ and rank $\delta \ge 1$. We let $\pr \colon \PP(\sF) \to X$ denote the derived projectivization of $\sF$ and let $\sO(1)$ denote the universal line bundle. Then the sequence of line bundles on $\PP(\sF)$:
	$$\sO, \sO(1), \cdots, \sO(\delta-1),$$
forms a relative exceptional sequence on $\PP(\sF)$ over $X$ (in the sense of Definition \ref{def:relexcseq}), whose left dual relative exceptional sequence (Definition \ref{def:dualrelexc}) is given by
	$$\Sym_{\PP(\sF)}^{\delta-1}(\Sigma \, L_{\PP(\sF)/X}(1)) , \cdots, \Sym_{\PP(\sF)}^{i}(\Sigma \,L_{\PP(\sF)/X}(1)), \cdots, \Sigma \, L_{\PP(\sF)/X}(1), \sO,$$
where $L_{\PP(\sF)/X}$ is the relative cotangent complex and $L_{\PP(\sF)/X}(1) =L_{\PP(\sF)/X} \otimes \sO(1)$. 
\end{proposition}

\begin{remark} From the canonical equivalences
	$$\Sym_{\PP(\sF)}^{i}(\Sigma \,L_{\PP(\sF)/X}(1)) \simeq \Sigma^i \bigwedge\nolimits_{\PP(\sF)}^i (L_{\PP(\sF)/X}(1))$$
for all $i \ge 0$, we obtain that the following sequence 
 	$$\bigwedge\nolimits_{\PP(\sF)}^{\delta-1}(L_{\PP(\sF)/X}(1)), \cdots, \bigwedge\nolimits_{\PP(\sF)}^i (L_{\PP(\sF)/X}(1)), \cdots,  L_{\PP(\sF)/X}(1), \sO$$
also forms a relative exceptional sequence on $\PP(\sF)$ over $X$.
\end{remark}

\begin{proof} It follows directly from Theorem \ref{thm:Serre:O(d)} \eqref{thm:Serre:O(d)-1i} and \eqref{thm:Serre:O(d)-1ii} that the sequence $\sO, \sO(1), \cdots, \sO(\delta-1)$ forms a relative exceptional sequence of $\PP(\sF)$ over $X$; It only remains to compute its left dual relative exceptional sequence. Let $\rho$ denote the tautological quotient $\pr^* \sF \to \sO(1)$. Then by virtue of Theorem \ref{thm:proj:Euler}, there is a canonical equivalence $\cofib(\rho) \simeq \Sigma \, L_{\PP(\sF)/X}(1)$. If we apply Construction \ref{constr:Postnikov:Sym:prestacks} to the map $\rho$, then we obtain quasi-coherent sheaves $\sG_{i,n}(\PP(\sF), \rho)$ for all integers $n \ge 0$ and $0 \le i \le n$ which fit into canonical fiber sequences
	$$\beta_{i,n}(\PP(\sF), \rho) \colon \qquad \Sigma^{i} \,\bigwedge\nolimits^{i+1} (\pr^* \sF) \otimes \sO(n-i-1)  \to \sG_{i,n}(\PP(\sF), \rho) \to \sG_{i+1,n}(\PP(\sF), \rho),$$
and for which there are canonical equivalences:
	$$\sG_{0,n}(\PP(\sF), \rho) \simeq \sO(n), \qquad \sG_{n,n}(\PP(\sF), \rho) \simeq \Sym^n (\Sigma \,  L_{\PP(\sF)/X}(1)).$$
Moreover, for any integer $n' \ge 0$, since the functor $\otimes \sO(n')$ is an exact equivalence, the canonical morphisms from Construction \ref{constr:Postnikov:Sym:prestacks},
	\begin{align}\label{eqn:prop:dual:PVdot:beta:O(1)}
	\beta_{i,n}(\PP(\sF), \rho) \otimes \sO(n') \to \beta_{i,n+n'}(\PP(\sF), \rho),
	\end{align}
are equivalences of fiber sequences. Now for each $n \ge 1$, the canonical natural transformation $\id_{\pr^* \sF} \to \rho$, which can be depicted as a commutative diagram:
	$$
	\begin{tikzcd}
		\pr^* \sF \ar{r}{\id} \ar{d}{\id} & \pr^* \sF \ar{d}{\rho} \\
		\pr^* \sF \ar{r}{\rho} & \sO(1),
	\end{tikzcd}
	$$
induces canonical morphisms of fiber sequences 
	$$\beta_{i,n}(\PP(\sF), \pr^* \id_{\sF}) \to \beta_{i,n}(\PP(\sF), \rho).$$
By virtue of Theorem \ref{thm:Serre:O(d)} \eqref{thm:Serre:O(d)-1i} and the functoriality of the fiber sequences $\beta_{i,n}$, by applying the exact functor $\pr_*$ we obtain canonical equivalences of fiber sequences:
	\begin{align}\label{eqn:prop:dual:X:beta:equiv}
		\beta_{i,n}(X,  \sF \xrightarrow{\id_\sF} \sF) \simeq \pr_* \left(\beta_{i,n}(\PP(\sF), \pr^* \xrightarrow{\rho} \sO(1)) \right).
	\end{align}
To finish the proof of the proposition, we only need to show the following assertion:
	\begin{itemize}
		\item[(*)] For for every pair of integers $n \ge 1$ and $0 \le i \le n-1$, the above fiber sequence $\beta_{i,n}(\PP(\sF), \rho)$ is canonically equivalent to the mutation fiber sequence \eqref{eqn:left:mut} for the left mutation of $\sG_{i,n}(\PP(\sV_\bullet), \rho)$ passing through $\sO(n-i-1)$.
	\end{itemize}
	
We prove $(*)$ by induction on $n$. Assume $n=1$, $i=0$, then by construction the fiber sequence $\beta_{0,1}(\PP(\sV_\bullet), \rho)$ is canonically equivalent with the shifted Euler fiber sequence 
	$$\pr^* \sF \xrightarrow{\rho} \sO(1) \to \Sigma \,  L_{\PP(\sF)/X}(1).$$
By virtue of Theorem \ref{thm:Serre:O(d)} \eqref{thm:Serre:O(d)-1i}, there is a canonical equivalence of fiber sequences
	$$\pr_*(\beta_{0,1}(\PP(\sV_\bullet), \rho)) \simeq (\sF \xrightarrow{\id} \sF \to 0)$$
which canonically identifies $\beta_{0,1}(\PP(\sV_\bullet), \rho)$ with the fiber sequence \eqref{eqn:left:mut} for the left mutation of $\sO(1)$ through $\sO$.

Assume that the assertion $(*)$ holds for an integer $n \ge 1$ and all integers $0 \le i \le n-1$. We wish to show that the assertion $(*)$ holds for $n+1$ and all integers $0 \le i \le n$. From the equivalences \eqref{eqn:prop:dual:PVdot:beta:O(1)} for the case $n'=1$, we deduce from induction hypothesis that the assertion $(*)$ holds for all $0 \le i \le n-1$; In particular, there are canonical equivalences
	$$\sG_{1, n+1}(\PP(\sF), \rho) \simeq \sG_{0,n}(\PP(\sF), \rho) \otimes \sO(1) \simeq \Sym^n(\Sigma \,  L_{\PP(\sF)/X}(1)) \otimes \sO(1).$$
It remains to prove the case where $i=n$, that is, to show that the fiber sequence
	$$\beta_{n,n+1}(\PP(\sF), \rho) \colon \quad \Sigma^n \bigwedge\nolimits^{n+1}(\pr^* \sF) \to \Sym^n(\Sigma \,  L_{\PP(\sF)/X}(1)) \otimes \sO(1) \to \Sym^{n+1}(\Sigma \,  L_{\PP(\sF)/X}(1))$$
can be canonically identified with the mutation fiber sequence \eqref{eqn:left:mut} for the left mutation of $\Sym^n(\Sigma \,  L_{\PP(\sF)/X}(1)) \otimes \sO(1)$ passing through $\sO$. Invoking the equivalence \eqref{eqn:prop:dual:X:beta:equiv}, we obtain that the fiber sequence $\pr_*(\beta_{n,n+1}(\PP(\sF), \rho))$ is canonically equivalent to
	$$\beta_{n,n+1}(X, \id_{\sF}) \colon \qquad \Sigma^n \bigwedge\nolimits^{n+1}(\sF) \xrightarrow{\sim} \sG_{n,n+1}(X, \id_{\sF}) \to 0.$$ 
From the adjunction between $\pr^*$ and $\pr_*$, we obtain from the above equivalence that the fiber sequence $\beta_{n,n+1}(\PP(\sF), \rho) $ exhibits $\Sym^{n+1}(\Sigma\, L_{\PP(\sF)/X}(1))$ as the left mutation of $\sG_{1, n+1}(\PP(\sF), \rho)$ through $\sO$. By induction, the assertion $(*)$ holds for all $n \ge 1$, $0 \le i \le n-1$.
\end{proof}

\begin{corollary}[Beilinson's relations] \label{cor:PVdot:relations} Assume we are in the same situation as Proposition \ref{prop:PG:dualexc}, then for all integers $0 \le i,j \le \delta-1$, there are canonical equivalences:
	\begin{equation}\label{eqn:cor:PVdot:relations}
	\left\{
	\begin{split}
		& \Mapsp_X \big(\sO(i), \sO(j) \big) \simeq \Sym_X^{j-i} (\sF), \\ 
		& 
		\Mapsp_X\big(\sO(i), \Sigma^j \bigwedge\nolimits_{\PP(\sF)}^j (L_{\PP(\sF)/X}(1))\big) \simeq \delta_{i,j} \cdot \sO_X, \\
		& \Mapsp_X\big(\bigwedge\nolimits_{\PP(\sF)}^i ( L_{\PP(\sF)/X}(1)), \bigwedge\nolimits_{\PP(\sF)}^j ( L_{\PP(\sF)/X}(1))\big) \simeq (\bigwedge\nolimits_X^{i-j} \sF)^\vee.
	\end{split}
	\right.
	\end{equation}
Here, $\Mapsp_X = \Mapsp_{\PP(\sF)/X}$ denotes the mapping object functor on $\PP(\sF)$ over $X$ of Construction \ref{constr:mappingobject}, $(\blank)^\vee$ denotes the dual of a perfect object in $\QCoh(X)$, the quasi-coherent sheaves $\Sym^d(\sE)$ and $\bigwedge^d(\sE)$ are understood as zero whenever $d<0$, and $\delta_{i,j}$ denotes the Kronecker delta function, that is, $\delta_{i,j} = 1$ if $i=j$ and $\delta_{i,j}=0$ if $i \ne j$.
\end{corollary}

\begin{proof} For simplicity of notations, in this proof, we set 
	$$\sR =  L_{\PP(\sF)/X} \otimes \sO(1)$$
and omit the subscript $\PP(\sF)$ in all the expressions of the form $\Sym_{\PP(\sF)}^i(\sE)$.

The first relation of \eqref{eqn:cor:PVdot:relations} follows from Theorem \ref{thm:Serre:O(d)} \eqref{thm:Serre:O(d)-1i}. The second relation of \eqref{eqn:cor:PVdot:relations} is a consequence of Proposition \ref{prop:PG:dualexc} and Corollary \ref{cor:dulexcseq}. To prove the third relation of \eqref{eqn:cor:PVdot:relations}, we only need to show that, for all $0 \le i,j \le \delta-1$, there are canonical equivalences
	\begin{equation}\label{eqn:cor:PVdot:relation:Sym}
	\Mapsp_X(\Sym^i(\Sigma \, \sR), \Sym^j(\Sigma \, \sR)) \simeq (\Sym_X^{i-j} (\Sigma \, \sF))^\vee
	\end{equation}
We prove \eqref{eqn:cor:PVdot:relation:Sym} by induction on $i$. The base case $i=0$ follows from Proposition \ref{prop:PG:dualexc} and Corollary \ref{cor:dulexcseq}. Now assume that \eqref{eqn:cor:PVdot:relation:Sym} holds for the case $(i-1)$, where $i \ge 1$ (thus we may assume $\delta \ge 2$). We wish to prove \eqref{eqn:cor:PVdot:relation:Sym} for the case $i$. By virtue of Remark \ref{rmk:excseqspan} and Proposition \ref{prop:PG:dualexc}, $\Sym^{i-1}(\Sigma \, \sR)$ belongs to the full subcategory
	$$\big\langle \pr^* (\Perf(X)) \otimes \sO, \cdots, \pr^*( \Perf(X)) \otimes \sO(i-1) \big \rangle.$$
From the equivalences $\Mapsp_X(\sO, \sO(-1)) \simeq \cdots \simeq \Mapsp_X(\sO(i-1), \sO(-1)) \simeq 0$, we obtain
	$$\Mapsp_X(\Sym^{i-1}(\Sigma \, \sR) \otimes \sO(1), \sO) \simeq \Mapsp_X(\Sym^{i-1}(\Sigma \, \sR), \sO(-1)) \simeq 0.$$
The fiber sequence $\beta_{i-1, i}(\PP(\sF), \rho)$ of the proof of Proposition \ref{prop:PG:dualexc} induces a fiber sequence
	\begin{equation}\label{eqn:cor:PVdot:relation:fib}
	\Sym^{i-1}(\Sigma \, \sR) \otimes \sO(1) \to \Sym^{i} (\Sigma \, \sR) \to \pr^* (\Sym_X^{i} (\Sigma \, \sF )).
	\end{equation}
By applying the exact functor $\Mapsp_X(\blank, \sO)$ to \eqref{eqn:cor:PVdot:relation:fib}, we obtain a canonical equivalence
	$$\Mapsp_X(\Sym^{i} (\Sigma \, \sR), \sO) \simeq ( \Sym_X^{i} (\Sigma \, \sF ))^\vee.$$	
This proves \eqref{eqn:cor:PVdot:relation:Sym} in the case where $j=0$. Assume $j \ge 1$, then $\Mapsp_X(\sO, \Sym^j(\Sigma \, \sR)) \simeq 0$. By applying $\Mapsp_X(\blank, \Sym^j(\Sigma \, \sR))$ to \eqref{eqn:cor:PVdot:relation:fib}, we obtain a canonical equivalence
	$$\Mapsp_X(\Sym^i(\Sigma \, \sR), \Sym^j(\Sigma \, \sR)) \simeq \Mapsp_X(\Sym^{i-1}(\Sigma \, \sR) \otimes \sO(1) ,  \Sym^j(\Sigma \, \sR)).$$
By virtue of the vanishing
	\begin{align*}
	&\Mapsp_X(\Sym^{i-1}(\Sigma \, \sR) \otimes \sO(1),  \pr^* (\Sym_X^{j} (\Sigma \, \sF )))  \\
	\simeq \,& \Sym_X^{j} (\Sigma \, \sF ) \otimes_{\sO_X} \Mapsp_X(\Sym^{i-1}(\Sigma \, \sR) \otimes \sO(1), \sO) \simeq 0,
	\end{align*}
if we apply the exact functor $\Mapsp_X(\Sym^{i-1}(\Sigma \, \sR) \otimes \sO(1), \blank)$ to the fiber sequence	
	$$\Sym^{j-1}(\Sigma \, \sR) \otimes \sO(1) \to \Sym^{j} (\Sigma \, \sR) \to \pr^* (\Sym_X^{j} (\Sigma \, \sF ))$$
(which is \eqref{eqn:cor:PVdot:relation:fib} but with $i$ replaced by $j$), we obtain canonical equivalences
	\begin{align*}
	&\Mapsp_X(\Sym^{i-1}(\Sigma \, \sR) \otimes \sO(1) ,  \Sym^j(\Sigma \, \sR)) \\
	 \simeq \, &\Mapsp_X(\Sym^{i-1}(\Sigma \, \sR) \otimes \sO(1) ,  \Sym^{j-1}(\Sigma \, \sR) \otimes \sO(1)) \\
	\simeq \, &\Mapsp_X(\Sym^{i-1}(\Sigma \, \sR),  \Sym^{j-1}(\Sigma \, \sR)). 
	\end{align*}
By induction, we obtain that \eqref{eqn:cor:PVdot:relation:Sym} holds for $i$ and all $1 \le j \le \delta-1$. 
\end{proof}

\section{Semiorthogonal decompositions}
For a prestack $Y$, we let $\Dqc(Y)$ denote the homotopy category of the $\infty$-category $\QCoh(Y)$. Since $\QCoh(Y)$ is a stable $\infty$-category, $\Dqc(Y)$ is a triangulated category. The category $\Dqc(Y)$ is often referred to as the derived category of quasi-coherent sheaves on $Y$.

Let $X$ be a (classical) Cohen--Macaulay scheme, and let $\sF$ be a discrete finite type quasi-coherent sheaf on $X$ of homological dimension $\le 1$ and rank $\delta \ge 1$. 
Assume furthermore that condition \eqref{eqn:condition:PK} of \S \ref{sec:classical} is satisfied (so that $\PP(\sF) \simeq \PP_\cl(\sF)$), then it follows immediate from Proposition \ref{prop:PG:dualexc} of last section that, for any $i \in \ZZ$, the functor
	$$\Psi_i =\pr^*(\blank) \otimes \sO(i) \colon \Dqc(X) \to \Dqc(\PP(\sF))$$
is fully faithful, and for any given integer $k \in \ZZ$, the essential images of $\Psi_i$ induced from the relative exceptional sequence $\sO(k+1), \ldots, \sO(k+\delta)$ of $\PP(\sF)$ over $X$,
	$$\Dqc(X) \otimes \sO(k+1), \cdots, \Dqc(X) \otimes \sO(k+\delta),$$
form a semiorthogonal sequence which spans a full triangulated subcategory of $\Dqc(\PP(\sF))$.

\begin{enumerate}[leftmargin=*]
	\item (Orlov's theorem) If $\sF$ is locally free of rank $\delta \ge 1$, then a theorem of Orlov \cite{Orlov92} implies that the above semiorthogonal sequence spans the whole category $\Dqc(\PP(\sF))$. In other words, there is a semiorthogonal decomposition
		$$\Dqc(\PP(\sF)) = \big\langle \Dqc(X) \otimes \sO(k+1), \cdots, \Dqc(X) \otimes \sO(k+\delta) \big\rangle.$$
	(See also \cite[Theorem 6.7]{BS},  \cite[Theorem 3.3]{Kh20}, \cite[Theorem B.3]{J21}.)
	\item (The projectivization formula) In general, if $\sF$ is not locally free, the above semiorthogonal sequence no longer spans the whole category $\Dqc(\PP(\sF))$. In \cite{JL18}, we show that the ``correction" is precisely provided by the category of another projectivization $\PP_\cl(\sExt^1_{\sO_X}(\sF, \sO_X))$, whose support is the non-locally-free locus of $\sF$. More precisely, under the above condition \eqref{eqn:condition:PK} (so that $\PP(\sF) \simeq \PP_\cl(\sF)$ and $\PP(\sExt^1_{\sO_X}(\sF, \sO_X)) \simeq \PP_\cl(\sExt^1_{\sO_X}(\sF, \sO_X))$), there is a semiorthogonal decomposition
		$$\Dqc(\PP(\sF)) = \big\langle \Dqc(\PP(\sExt^1_{\sO_X}(\sF, \sO_X))), ~ \Dqc(X) \otimes \sO(1), \cdots, \Dqc(X) \otimes \sO(\delta) \big\rangle.$$

\end{enumerate}

\begin{remark} The projectivization formula in the above form can be found in \cite[Theorem 3.4]{JL18}, \cite[Theorem 6.16]{J21}; see also \cite[Theorem 5.5]{Kuz07} and \cite[Theorem 4.6.11]{Tod3}. 
\end{remark}

In this section, we prove a generalization of the above projectivization formula (Theorem \ref{thm:structural}), which completely describes the $\infty$-category $\QCoh(\PP(\sF))$ for a quasi-coherent sheaf $\sF$ of perfect-amplitude contained in $[0,1]$ and rank $\delta \ge 0$ over a prestack $X$, in terms of the $\infty$-categories $\QCoh(X)$ and $\QCoh(\PP(\Sigma \sF^\vee))$, where $\Sigma \sF^\vee$ is the ``shifted dual" of $\sF$.

Theorem \ref{thm:structural} has many new applications even in the classical situation, and we consider such examples in \S \ref{sec:classical}.  For the purpose of applications, in \S \ref{sec:classical}, we first provide criteria for the derived projectivizations and derived symmetric products to be classical (Lemmata  \ref{lem:criterion:P:classical} and \ref{lem:criterion:Sym:classical}). Then we apply the results in each situation; we only the list lower dimensional examples in each situation, but all these results have analogues in higher dimensions.

\subsection{The Fourier--Mukai functors and semiorthogonal decompositions}
In this subsection, we present the main result of this section 

\begin{notation}\label{notation:SOD:setup}
Let $X$ be a prestack, and let $\sF$ be a quasi-coherent sheaf of perfect-amplitude contained in $[0,1]$ on $X$ of rank $\delta \ge 0$. We consider the following derived projectivizations
	$$\pr_+ \colon \PP_+ = \PP(\sF) \to X \qquad \pr_- \colon \PP_- = \PP(\Sigma\,\sF^\vee) \to X.$$
We let $\rho_+ \colon \pr_+^* \sF \to \sO_+(1)$ and $\rho_- \colon \pr_-^*( \Sigma \sF^\vee) \to \sO_-(1)$ denote the tautological quotient maps on $\PP_+$ and $\PP_-$, respectively. We denote the (derived) fiber product of $\PP_\pm$ over $X$ by
	$$Z = \PP(\sF) \times_X \PP(\Sigma \sF^\vee)$$
and let $\pr_{Z,+} \colon Z \to \PP_+$ and $\pr_{Z,-} \colon Z \to \PP_-$ denote the respective projections. For a pair of quasi-coherent sheaves $\sM_+ \in \QCoh(\PP_+)$ and $\sM_- \in \QCoh(\PP_-)$, we let $\sM_+ \boxtimes_X \sM_-$ denote the external tensor product $\pr_{Z,+}^* \sM_+ \otimes_{\sO_Z} \pr_{Z,-}^* \sM_- \in \QCoh(Z)$. 
\end{notation}

\begin{lemma}\label{lem:dproj:incidence} In the situation of Notation \ref{notation:SOD:setup}, the following constructions of a pair $(\widehat{Z}, \iota_{\widehat{Z}})$, where $\widehat{Z}$ is a prestack and  $\iota_{\widehat{Z}} \colon \widehat{Z} \to Z$ is a quasi-smooth closed immersion, are equivalent:
	\begin{enumerate}
		\item We let $\widehat{Z}$ denote the derived projectivization $\PP_{\PP_+} (\Sigma \, \fib(\rho_+)^\vee)$, and let $\iota_{\widehat{Z}} \colon \widehat{Z} \to Z$ denote the closed immersion induced by the surjective map $\Sigma (\pr_+^* \sF)^\vee \to \Sigma \, \fib(\rho_+)^\vee$ on $\PP_+$ (Corollary \ref{cor:proj:closedimmersion}), where $Z$ is canonically identified with $\PP_{\PP_+}(\Sigma (\pr_+^* \sF)^\vee)$.
		\item We let $\widehat{Z}$ denote the derived projectivization $\PP_{\PP_-}(\Sigma \, \fib(\rho_-)^\vee)$, and let $\iota_{\widehat{Z}} \colon \widehat{Z} \to Z$ denote the closed immersion induced by the surjective map $ \pr_-^* \sF \to \Sigma \, \fib(\rho_-)^\vee$ on $\PP_-$ (Corollary \ref{cor:proj:closedimmersion}), where $Z$ is canonically identified with $\PP_{\PP_-}(\pr_-^* \sF)$.
		\item We let $\widehat{Z}$ together with the closed immersion $\iota_{\widehat{Z}} \colon \widehat{Z} \to Z$ be defined by the following pullback square:
	\begin{equation*} 
	\begin{tikzcd}
		\widehat{Z}\ar{d}[swap]{\iota_{\widehat{Z}}} \ar{r}{\iota_{\widehat{Z}}} & Z \ar{d}{i_{\mathbf{1}}} \\
		Z \ar{r}{i_{\mathbf{0}}} & \VV_Z(\Sigma \,\sO_+(-1) \boxtimes_X \sO_-(-1)),
	\end{tikzcd}
	\end{equation*}
where $i_{\mathbf{0}}$ is the section of $\VV_Z(\Sigma \, \sO_+(-1) \boxtimes_X \sO_-(-1)) \to Z$ that classifies the zero map, and $i_{\mathbf{1}}$ is the section of $\VV_Z(\Sigma \, \sO_+(-1) \boxtimes_X \sO_-(-1)) \to Z$ that classifies the composition
	$$\Sigma \, \sO_+(-1) \boxtimes_X \sO_-(-1) \xrightarrow{\Sigma \, \rho_+^\vee \boxtimes \rho_-^\vee} \pr_+^*(\sF^\vee) \boxtimes_X \pr_-^*(\sF) \simeq (\pr_{Z}^*\sF)^\vee \otimes_{\sO_Z} (\pr_Z^*\sF) \xrightarrow{\rm ev} \sO_Z$$
(here, $\pr_{Z} \colon Z \to X$ is the natural projection, and ${\rm ev}$ is the evaluation map).
	\end{enumerate}
\end{lemma}
 
 \begin{proof} Notice that there are canonical fiber sequences
	$$\Sigma \sO_+(-1) \to \Sigma (\pr_+^* \sF)^\vee \to \Sigma \, \fib(\rho_+)^\vee \quad \text{and} \quad \Sigma \sO_-(-1) \to  \pr_-^* \sF \to \Sigma \, \fib(\rho_-)^\vee,$$
on $\PP_+ = \PP(\sF)$ and $\PP_- = \PP(\Sigma \, \sF^\vee)$, respectively. Therefore, both the equivalences of the constructions ``$(1) \iff (3)$" and ``$(2) \iff (3)$" follow from Proposition \ref{prop:proj:PB} and the following general fact: let $\shC$ be a symmetric monoidal $\infty$-category with unit $1 \in \shC$. Let $C, D \in \shC$ be dualizable objects with (right) duals $C^\vee$ and $D^\vee$, and evaluation maps $e_C \colon C^\vee \otimes C \to 1$ and $e_D \colon D^\vee \otimes D \to 1$, respectively. Let $f \colon C \to D$ be a map in $\shC$, and let $f^\vee \colon D^\vee \to C^\vee$ be the dual map. Then the following two composite maps
		\begin{align*}
			D^\vee \otimes C \xrightarrow{f^\vee \otimes \id} C^\vee \otimes C \xrightarrow{e_C} 1 \quad \text{and} \quad
			D^\vee \otimes C \xrightarrow{\id \otimes f} D^\vee \otimes D \xrightarrow{e_D} 1
		\end{align*}
	are canonically homotopy equivalent to each other.
 \end{proof}

 \begin{notation} \label{notation:hatZ} In the situation of Notation \ref{notation:SOD:setup}, we let $(\widehat{Z}, \iota_{\widehat{Z}} \colon \widehat{Z} \hookrightarrow Z)$ denote the closed sub-prestack of $Z$ constructed in Lemma \ref{lem:dproj:incidence}, and refer to is as the {\em universal incidence locus} for $\PP(\sF)$ and $\PP(\Sigma \, \sF^\vee)$ over $X$. We let $r_+  \colon \widehat{Z} \to \PP_+=\PP(\sF)$ and $r_- \colon \widehat{Z} \to \PP_- =\PP(\Sigma \, \sF^\vee)$ denote the respective projections, and we will refer to the following commutative diagram
	\begin{equation} \label{diagram:Corr}
	\begin{tikzcd}[row sep= 2.5 em, column sep = 4 em]
		\widehat{Z} \ar{rd}{\widehat{\pr}}\ar{r}{r_+}  \ar{d}[swap]{r_-} &  \PP_{+} \ar{d}{\pr_+}\\
		\PP_{-} \ar{r}{\pr_-} & X
	\end{tikzcd}
	\end{equation}
as the {\em incidence diagram} (for the derived projectivizations $\PP(\sF)$ and $\PP(\Sigma \, \sF^\vee)$ over $X$).
 \end{notation}
 
The main result of this section is the following:

\begin{theorem}\label{thm:structural} 
Let $X$ be a prestack and $\sF$ a quasi-coherent sheaf of perfect-amplitude contained in $[0,1]$ on $X$ of rank $\delta \ge 0$, and we let everything be defined as in Notations \ref{notation:SOD:setup} and \ref{notation:hatZ}. Then the exact functor induced by the incidence diagram \eqref{diagram:Corr},
	$$\Phi = r_{+,*} \, r_-^* \colon \QCoh(\PP(\Sigma \, \sF^\vee)) \to \QCoh(\PP(\sF)),$$
and its restriction to the subcategory of perfect objects,
	$$\Phi|\Perf (\PP(\Sigma \, \sF^\vee))  \colon \Perf(\PP(\Sigma \, \sF^\vee)) \to \Perf(\PP(\sF)),$$
are fully faithful. If $\delta = 0$, then $\Phi$ and $\Phi|\Perf (\PP(\Sigma \, \sF^\vee)) $ induce equivalences of $\infty$-categories
	$$\QCoh(\PP(\Sigma \, \sF^\vee)) \simeq \QCoh(\PP(\sF)) \qquad  \Perf(\PP(\Sigma \, \sF^\vee)) \simeq \Perf(\PP(\sF)).$$
If $\delta \ge 1$, then for any  $i \in \ZZ$, the functor 
	$$\Psi_i = \pr_{\PP(\sF)}^*(\blank) \otimes \sO_{\PP(\sF)}(i) \colon \QCoh(X) \to \QCoh(\PP(\sF))$$
and its restriction to the subcategory of perfect object,
	$$\Psi_i |\Perf(X) =  \pr_{\PP(\sF)}^*(\blank) \otimes \sO_{\PP(\sF)}(i) \colon \Perf(X) \to \Perf(\PP(\sF)),$$
are fully faithful. Furthermore, there are induced semiorthogonal decompositions
 	\begin{align*}
		\QCoh(\PP(\sF)) & = \big\langle \Phi (\QCoh(\PP(\Sigma \, \sF^\vee))), ~ \Psi_1(\QCoh(X)), \ldots, \Psi_{\delta}(\QCoh(X)) \big\rangle \\
		\Perf(\PP(\sF)) &= \big\langle \Phi (\Perf(\PP(\Sigma \, \sF^\vee))), ~ \Psi_1(\Perf(X)), \ldots, \Psi_{\delta}(\Perf(X)) \big\rangle.
	\end{align*}
 \end{theorem}
 
 \begin{remark} By virtue of Theorem \ref{thm:Neeman-Lipman-Lurie} \eqref{thm:Neeman-Lipman-Lurie-2} and \eqref{thm:Neeman-Lipman-Lurie-3}, the functors $\Phi$ and $\Psi_i$ of Theorem \ref{thm:structural} admit both left and right adjoints, and they and their left and right adjoints preserve perfect objects, almost perfect objects, and locally bounded almost perfect objects. Consequently, there are also induced similar semiorthogonal decompositions for the subcategories of almost perfect objects, and locally bounded almost perfect objects.
 \end{remark}
 
 \begin{remark}[Additive invariants] \label{rmk:additive} As a consequence of Theorem \ref{thm:structural}, if $\EE$ is any additive invariant (in the sense of \cite[Definition 6.1]{BGT}, from the $\infty$-category of stable $\infty$-categories, with exact functors between them, to a stable presentable $\infty$-category $\shD$), then there is a canonical equivalence
 	$$\EE(\Perf(\PP(\sF))) \simeq \EE(\Perf(\PP(\Sigma \sF^\vee))) \oplus \big(\bigoplus_{i=1}^\delta \nolimits
\EE(\Perf(X)) \big).$$
For example, we could take $\EE$ to be algebraic
K-theory, topological Hochschild homology, or topological cyclic homology; see \cite{BGT} for more details about additive invariants.
 \end{remark}
 
 \begin{remark}[Derived blowups] \label{rmk:blowup} In the situation of Theorem \ref{thm:structural}, if $\rank \sF = 1$, then the derived projectivization $\PP(\sF)$ could be regarded as the {\em derived blowup} (in the sense of \cite{KR18, He21}) of $X$ along the derived first degeneracy locus of $\sF$. This point of view is justified in the classical situation (\cite[\S 3.1.2]{JL18}, \cite[Lemma 5.1]{J20}, and \cite[Lemma 2.25]{J21}), and we expect similar results to be true in the derived context. 
 \end{remark}
 
 \begin{proof}[Proof of Theorem \ref{thm:structural}]
We let $\Phi_0 = \Phi$ and $\Phi_i = \Psi_i$ where $1 \le i \le \delta$, let $\Phi_i^L$ and $\Phi_i^R$ denote the left and right adjoint functors of $\Phi_i$, respectively. Then to prove Theorem \ref{thm:structural}, it suffices to prove the following assertions:
\begin{enumerate}
	\item (Fully-faithfulness) The cofiber of the unit $\id \to \Phi_i^R \circ \Phi_i$ is equivalent to zero for all $i$.
	\item (Semiorthogonality) The composite map $\Phi_i^R \circ \Phi_j$ is equivalent to zero for all $i >j$.
	\item (Generation) For every object $\sE \in \QCoh(\PP(\sF))$, the cofiber of the canonical map from $\sE$ to its left mutation passing through $\{\Im \Phi_i\}_{0 \le i \le \delta}$ is equivalent to zero.
\end{enumerate}
The formations of the above assertions $(1)$ through $(3)$, regarded as conditions on the pair $(X,\sF)$, commute with base change and are local on $X$ with respect to Zariski topology; The analogous statement is also true for the subcategories of perfect objects. Therefore, to prove Theorem \ref{thm:structural}, we might reduce to the local case where $X = \Spec R$, $R \in \CAlgDelta$, and $\sF$ is represented by the cofiber of a map of vector bundles $\sigma \colon \sW \to \sV$ on $X$. The proof of Theorem \ref{thm:structural} in the local situation is given in the next subsection.
\end{proof}

\begin{remark} The above conditions $(1)$ through $(3)$ of the pair $(X, \sF)$ are local on $X$ even with respect to {\em flat} topology. Consequently, the formation of the semiorthogonal decomposition of Theorem \ref{thm:structural} commutes with arbitrary base change and satisfies flat descent. See also \cite{Kuz11} and \cite{AE, BS} for general results on base change and flat descent of semiorthogonal decompositions.
\end{remark}
 
 \subsection{The local situation} In this subsection, we present the proof of Theorem \ref{thm:structural} in the local situation. Assume that we are in the situation of Notations \ref{notation:SOD:setup} and \ref{notation:hatZ}, and assume furthermore that  $X = \Spec R$, $R \in \CAlgDelta$, and $\sF$ is represented by the cofiber of a map of vector bundles $\sigma \colon \sW \to \sV$ on $X$, where $n = \rank \sV$, $m = \rank \sW$, $\delta = n - m \ge 0$. By virtue of Proposition \ref{prop:proj:PB}, there are closed immersions 
 	$$\iota_+ \colon \PP_+ =\PP(\sF) \hookrightarrow \PP(\sV) \qquad \iota_- \colon \PP_- = \PP(\Sigma \sF^\vee) \hookrightarrow \PP(\sW^\vee).$$
By virtue of Corollary \ref{cor:proj:closedimmersion}, the projections $r_\pm \colon \widehat{Z} \to \PP_\pm$ of diagram \eqref{diagram:Corr} are quasi-smooth, and factorize as the respective compositions
	\begin{equation}\label{diag:fact}
	\begin{tikzcd} \widehat{Z} \ar[hook]{r}{j_-} \ar{d}{r_-} & \PP_- \times_X \PP(\sV)\ar{dl}{q_-} \\
	\PP_-
	\end{tikzcd}	
	\qquad\text{and}\qquad
	\begin{tikzcd} \widehat{Z} \ar[hook]{r}{j_+} \ar{d}{r_+} & \PP_+ \times_X \PP(\sW^\vee) \ar{dl}{q_+} \\
	\PP_+
	\end{tikzcd}
	\end{equation}
Furthermore, by virtue of Proposition \ref{prop:proj:PB}, there are respective pullback squares:
	\begin{equation*} 
	\begin{tikzcd}
		\widehat{Z} \ar{d}[swap]{j_-} \ar{r}{j_-} & \PP_- \times_X \PP(\sV) \ar{d} \\
		\PP_- \times_X \PP(\sV)  \ar{r} & \VV_{\PP_- \times_X \PP(\sV)}(\iota_-^*\sR_-^\vee \boxtimes_X \sO_{\PP(\sV)}(-1)),
	\end{tikzcd}
	\end{equation*}
and 
	\begin{equation*} 
	\begin{tikzcd}
		\widehat{Z} \ar{d}[swap]{j_+} \ar{r}{j_+} & \PP_+ \times_X \PP(\sW^\vee) \ar{d} \\
		\PP_+ \times_X \PP(\sW^\vee)  \ar{r} & \VV_{\PP_+ \times_X \PP(\sW^\vee)}(\iota_+^*\sR_+^\vee \boxtimes_X \sO_{\PP(\sW^\vee)}(-1) ).
	\end{tikzcd}
	\end{equation*}

\begin{definition}
Let $\shC$ be a stable $\infty$-category. We say a collection $\shE$ of objects of $\shC$ {\em generates} $\shC$ if it satisfies the following condition: let $C \in \shC$ be an object having the property that $\Ext_\shC^n(E,C)\simeq 0$ for all $E \in \shE$ and $n \in \ZZ$, then $C \simeq 0$. Consequently, let $\shC_0$ be a stable subcategory of $\shC$ for which the inclusion $i \colon \shC_0 \hookrightarrow \shC$ admits a right adjoint. If $\shC_0$ contains a collection $\shE$ of objects which generates $\shC$ in the above sense, then $\shC_0 = \shC$.
\end{definition}
	 
\begin{lemma} \label{lem:P_pm:generator} In the above situation, we fix any $k\in \ZZ$, and let $\sR_+ =\Omega_{\PP(\sV)/X}(1)$ and $\sR_- = \Omega_{\PP(\sW^\vee)/X}(1)$ denote the twisted relative cotangent bundles of $\PP(\sV)$ and $\PP(\sW^\vee)$, respectively. 
\begin{enumerate}[leftmargin=*]
	\item $\QCoh(\PP_+)$ is generated by any of the following collections of vector bundles:
		$$\{\sO_+(i)\}_{k \le i \le k+n-1}, \qquad \{ \iota_+^*(\bigwedge\nolimits^i\sR_+)\}_{0 \le i \le n-1}, \qquad \text{or}  \qquad\{\iota_+^*(\bigwedge\nolimits^i \sR_+^\vee) \}_{0 \le i \le n-1}.$$
	\item $\QCoh(\PP_-)$ is generated by any of the following collections of vector bundles:
		$$\{\sO_-(i)\}_{k \le i \le k+n-1}, \qquad \{\iota_-^*(\bigwedge\nolimits^i \sR_-)\}_{0 \le i \le m-1},  \qquad \text{or}  \qquad \{\iota_-^*(\bigwedge\nolimits^i \sR_-^\vee) \}_{0 \le i \le m-1}.$$
\end{enumerate}
\end{lemma}
\begin{proof} This follows from the fact that $\QCoh(X)$ is generated by $\sO_X$ in the case where $X = \Spec R$ (\cite[Theorem 7.1.2.1]{HA}), the Beilinson's relative exceptional sequences of Proposition \ref{prop:PG:dualexc} are full in the case of projective bundles (which follows from the usual Orlov's theorem; see also \cite[Proposition B.3]{J21}), and that the morphisms $\iota_\pm$ are affine.
\end{proof}

The following is the analogue of the key lemma \cite[Lemma 5.6]{J21} in the case of projectivizations (Here, we only present half of all the cases, for these we will need in our proof).

\begin{lemma} \label{lem:key} In the above situation (where $\delta = n -m \ge 0$):
\begin{enumerate}
	\item \label{lem:key-1} Let $\Phi = r_{+\,*} \, r_{-}^* \colon \QCoh(\PP_-) \to \QCoh(\PP_+)$. Then for any $\sE \in \Perf(\PP(\sW^\vee))$, there is a series of fiber sequences in $\QCoh(\PP_+)$ for all $0 \le i \le n-1$:
	$$ \sG_{i-1}(\sE) \to \sG_{i}(\sE) \to \iota_+^*(\bigwedge\nolimits^i \sR_+^\vee) \otimes_{\sO_{\PP_+}} \pi_+^* (\sM_i(\sE)),$$
where $\sG_i(\sE) \in \Perf(\PP_+)$, $\sG_{-1}(\sE) = 0$, $\sG_{n-1}(\sE) \simeq \Phi(\iota_-^* \sE)$, and
	$$\sM_i(\sE) = \Sigma^i \, \Mapsp_{X}(\sO_{\PP(\sW^\vee)}(i), \sE) \in \Perf(X).$$
	\item  \label{lem:key-2} Let $\Phi^L = r_{-\,!} \, r_{+}^* \colon \QCoh(\PP_+) \to \QCoh(\PP_-)$ denote the left adjoint of $\Phi$. Then for any $\sE \in \Perf(\PP(\sV))$, there is a series of fiber sequences in $\QCoh(\PP_-)$ for all $0 \le i \le m-1$:
	$$ \iota_-^*(\bigwedge\nolimits^i \sR_-) \otimes_{\sO_{\PP_-}} \pi_-^* (\sN^i(\sE)) \to \sG^{i}(\sE) \to \sG^{i-1}(\sE),$$
where $\sG^i(\sE) \in \Perf(\PP_-)$, $\sG^{-1}(\sE) = 0$, $\sG^{m-1}(\sE) \simeq \Phi^L(\iota_+^* \sE)$, and
	$$\sN^i(\sE) = \Omega^i (\Mapsp_{X}(\sO_{\PP(\sV)}(i), \sE^\vee)^\vee) \in \Perf(X).$$
	In particular:
		\begin{enumerate}
			\item  \label{lem:key-2-1}If we set $\sE = \bigwedge^i \sR_+^\vee$, $0 \le i \le n-1$, then  there are canonical equivalences
					$$\Phi^L(\iota_+^* (\bigwedge^i \nolimits \sR_+^\vee)) \simeq \begin{cases}   \iota_-^*(\bigwedge\nolimits^i \sR_-) & \text{if} \quad  0 \le i \le m-1; \\  0, & \text{if} \quad m \le i \le n-1. 
		\end{cases} $$
			\item  \label{lem:key-2-2} If an object $\sE \in \Perf(\PP(\sV))$ having the property that 
		$$\Mapsp_X(\sO_{\PP(\sV)}(j), \sE^\vee) \simeq 0 \quad \text{for} \quad j=0, \ldots, m-1.$$
	Then $\Phi^L(\iota_+^*\sE) \simeq 0$.	
		\end{enumerate}
\end{enumerate}
\end{lemma}

\begin{proof} We will only prove assertion \eqref{lem:key-1}; the proof of assertion  \eqref{lem:key-2} is similar (and simpler). There are commutative diagrams:
	$$
	\begin{tikzcd} \widehat{Z} \ar[hook]{r}{j_+} \ar{d}[swap]{r_+} & \PP_+ \times_X \PP(\sW^\vee) \ar{dl}{q_+} \\
	\PP_+
	\end{tikzcd}
	\qquad \text{and} \qquad
	\begin{tikzcd}
		\widehat{Z} \ar{d}[swap]{r_-} \ar[hook]{r}{j_+}& \PP_+ \times_X \PP(\sW^\vee) \ar{d}{p_+} \\
		\PP_- \ar{r}{\iota_-} & \PP(\sW^\vee),
	\end{tikzcd}
	$$
where $p_+ \colon \PP_+ \times_X \PP(\sW^\vee) \to \PP(\sW^\vee)$ denotes the natural projection. Therefore, we obtain
	$$\Phi(\iota_-^* \sE) = r_{+\,*} r_{-}^* \sE \simeq q_{+\,*} j_{+\,*} j_+^* (\sO_{\PP_+} \boxtimes_X \sE).$$
Apply Lemma \ref{lem:Koszul:fib} to the quasi-smooth immersion $j_+$ (which corresponds to a cosection of $\iota_+^*\sR_+^\vee \boxtimes_X \sO_{\PP(\sW^\vee)}(-1)$), we obtain a sequence of fiber sequences on $\PP_+ \times_X \PP(\sW^\vee)$, $0 \le i \le n-1$:
$$\overline{\sG}_{i-1}(\sE) \to \overline{\sG}_{i}(\sE) \to \Sigma^i \,\iota_+^*(\bigwedge\nolimits^i \sR_+^\vee) \boxtimes_X (\sO_{\PP(\sW^\vee)}(-i) \otimes_{\sO_{\PP(\sW^\vee)}} \sE)$$
for which $\overline{\sG}_{-1}(\sE)\simeq 0 $ and $\overline{\sG}_{n-1}(\sE) \simeq j_{+\,*} \, j_+^* (\sO_{\PP_+} \boxtimes_X \sE)$. From the pullback square
	$$
	\begin{tikzcd}
		\PP_+ \times_X \PP(\sW^\vee) \ar{d}[swap]{q_+} \ar{r}{p_+} & \PP(\sW^\vee) \ar{d}{\pr_{\PP(\sW^\vee)}} \\
		\PP_+ \ar{r}{\pi_+} & X,
	\end{tikzcd}
	$$
we obtain that there are canonical equivalences
	\begin{align*}
		& q_{+*} \big(\Sigma^i \, \iota_+^*(\bigwedge\nolimits^i \sR_+^\vee) \boxtimes_X (\sO_{\PP(\sW^\vee)}(-i) \otimes_{\sO_{\PP(\sW^\vee)}} \sE) \big)  \\
		& \simeq  \Sigma^i   \iota_+^*(\bigwedge\nolimits^i \sR_+^\vee)  \otimes_{\sO_{\PP_+}} \pi_+^* \big(\pr_{\PP(\sW^\vee) \,*} (\sO_{\PP(\sW^\vee)}(-i) \otimes_{\sO_{\PP(\sW^\vee)}} \sE) \big) \\
		& \simeq \iota_+^*(\bigwedge\nolimits^i \sR_+^\vee)  \otimes_{\sO_{\PP_+}}  \pi_+^*  \big( \Sigma^i\, \Mapsp_X(\sO_{\PP(\sW^\vee)}(i),\sE) \big).
	\end{align*} 
Therefore, by pushing-forward the above fiber sequence along $q_{+}$ and setting $\sG_i(\sE) = q_{+ \,*} \overline{\sG}_i(\sE) \in \Perf(\PP_+)$,  we obtain the desired fiber sequences of assertion \eqref{lem:key-1}.
\end{proof}

\begin{proof}[Proof of Theorem \ref{thm:structural} in the local situation]
We first show that the functor $\Phi$ is fully faithful. By virtue of Lemma \ref{lem:P_pm:generator}, it suffices to show that the natural transformation 
	$\Phi^L \circ \Phi (\sE_i) \to \sE_i$
is an equivalence for all $\sE_i = \iota_-^* (\bigwedge^i \sR_-)$, where $0 \le i \le m-1$. By virtue of Lemma \ref{lem:key} \eqref{lem:key-1} (applied to $\sE = \bigwedge^i \sR_-$) we obtain canonical equivalences $\sG_{n-1}(\bigwedge^i \sR_-) \simeq \Phi(\sE_i)$ and 
		\begin{align*}
		& \sG_0( \bigwedge\nolimits^i \sR_-) \simeq \ldots \simeq \sG_{i-1}(\bigwedge\nolimits^i \sR_-) \simeq 0; \\
		& \sG_i( \bigwedge\nolimits^i \sR_-) \simeq \ldots \simeq \sG_{m-1}(\bigwedge\nolimits^i \sR_-) \simeq \iota_+^*(\bigwedge\nolimits^i \sR_+^\vee).
		\end{align*}
	Furthermore, for $0 \le j \le n-m-1$, there are fiber sequences in $\QCoh(\PP_+)$:
	$$ \sG_{m-1+j}(\bigwedge\nolimits^i \sR_-) \to \sG_{m+j}(\bigwedge\nolimits^i \sR_-) \to \iota_+^*(\bigwedge\nolimits^{m+j} \sR_+^\vee) \otimes_{\sO_{\PP_+}}  \pi_+^* (\sM_{m+j}(\bigwedge\nolimits^i \sR_-)).$$	
By virtue of Lemma \ref{lem:key} \eqref{lem:key-2-1} (applied to the case where $\sE = \bigwedge\nolimits^{m+j} \sR_+^\vee$), we obtain 
	$$\Phi^L \big( \iota_+^*(\bigwedge\nolimits^{m+j} \sR_+^\vee) \otimes_{\sO_{\PP_+}}  \pi_+^* (\sM_{m+j}(\bigwedge\nolimits^i \sR_-))\big) \simeq 0 \quad \text{for} \quad 0 \le j \le n-m-1.$$
Therefore, we obtain canonical equilvalences
	$$\Phi^L \circ \Phi(\sE_i) = \Phi^L( \sG_{n-1}(\bigwedge^i \sR_-)) \simeq \ldots \simeq \Phi^L(\sG_{m-1}(\bigwedge^i \sR_-)) \simeq \Phi^L(\iota_+^*(\bigwedge\nolimits^i \sR_+^\vee)) \simeq \sE_i,$$
where the last equivalence follows from Lemma \ref{lem:key} \eqref{lem:key-2-1}. This proves that $\Phi$ is fully faithful.

Next, we prove that the essential images of $\Phi$ and $\Psi_j$ are semiorthogonal, where $1 \le j \le \delta$. For all $\sA \in \QCoh(X)$ and $\sB \in \QCoh(\PP_-)$, we wish to show that the mapping space
	$$\Map_{\QCoh(\PP_+)}(\Psi_j(\sA), \Phi(\sB)) \simeq \Map_{\QCoh(\PP_-)} (\Phi^L (\pi_+^*(\sA) \otimes \sO_+(j)), \sB) $$
is contractible. Since $\Phi^L (\pi_+^*(\sA) \otimes \sO_+(j) \simeq \pi_+^*(\sA) \otimes \Phi^L(\iota_+^* (\sO_{\PP(\sV)}(j)))$, it suffices to show  
	$$\Phi^L(\iota_+^* (\sO_{\PP(\sV)}(j))) \simeq 0 \qquad \text{for all} \qquad 1 \le j \le n-m.$$
This follows directly from Lemma \ref{lem:key} \eqref{lem:key-2-2}. 

Finally, let $\shC$ denote the smallest stable subcategory of $\QCoh(\PP_+)$ that contains the essential images $\Im \Phi$ and $\{\Im \Psi_j\}_{1 \le j \le \delta}$. We wish to show that $\shC = \QCoh(\PP_+)$. By virtue of Lemma \ref{lem:P_pm:generator}, it suffices to show that for each $0 \le i \le n-1$, $\iota_+^* (\bigwedge^i \sR_+^\vee) \in \shC$. Since $\sR_+$ is a vector bundle on $\PP(\sV)$ of rank $(n-1)$, we obtain that for all $0 \le j \le \delta-1$, 
	$$\bigwedge\nolimits^{m+j} \sR_+^\vee \simeq (\bigwedge\nolimits^{\delta-1-j} \sR_+) \otimes \sO_{\PP(\sV)}(1) \otimes (\pr_{\PP(\sV)}^* \det \sV).$$
By virtue of Proposition \ref{prop:PG:dualexc}, the sequence of vector bundles 
	$$\bigwedge\nolimits^{\delta-1} \sR_+, \ldots, \sR_+, \sO_{\PP(\sV)}$$
is left dual to the relative exceptional sequence $\sO_{\PP(\sV)}, \sO_{\PP(\sV)}(1), \ldots, \sO_{\PP(\sV)}(\delta-1)$. Therefore, by virtue of Remark \ref{rmk:excseqspan}, each $\bigwedge\nolimits^{m+j} \sR_+^\vee$ is an iterated extensions of objects of the form $\sO_{\PP(\sV)}(s) \otimes (\pr_{\PP(\sV)}^* \sF_s)$, where $1 \le s \le \delta$ and $\sF_s \in \QCoh(X)$. In particular, we obtain that for all $\sM \in \QCoh(X)$,
	$$\iota_+^* (\bigwedge\nolimits^{m+j} \sR_+^\vee) \otimes_{\sO_{\PP_+}} (\pi_+^* \sM) \in \shC \quad \text{for} \quad 0 \le j \le \delta-1. $$
On the other hand, if we consider $\Phi(\sE_i)$, where $\sE_i = \iota_-^*(\bigwedge^i \sR_-)$, then $\sG_{n-1}(\bigwedge^i \sR_-) \simeq \Phi(\sE_i) \in \shC$ by  construction. By inverse induction and the fact that $\shC$ is stable under the formation of fibers, we obtain that $\sG_{m+k-1} (\bigwedge^i \sR_-) \in \shC$ for all $0 \le k \le \delta$. In particular, we obtain 
	$$\iota_+^*(\bigwedge\nolimits^i \sR_+^\vee) \simeq \sG_{m-1} (\bigwedge\nolimits^i \sR_-)  \in \shC$$
for all $0 \le i \le m-1$. This proves $\shC = \QCoh(\PP_+)$. Combined with the proof of the preceding subsection, this completes the proof of Theorem \ref{thm:structural}.
\end{proof}

\subsection{Applications to classical examples} \label{sec:classical}
In this subsection, we apply Theorem \ref{thm:structural} to classical examples. We only list examples in lower dimensions in each situation; all these results have direct analogues in higher dimensions. 

\subsubsection{Criteria for derived projectivizations and symmetric powers being classical} We first give criteria for the derived projectivizations and symmetric powers being classical. For simplicity of exposition, we make the following assumption throughout this subsection:

\begin{enumerate}[label=$(\star)$, ref=$\star$]
	\item \label{assumption:CM} We let $X$ be an integral Cohen--Macaulay scheme, let $\sW$ and $\sV$ be vector bundles on $X$ of rank $m$ and $n$ where $1 \le m \le n$, and let $\sigma \colon \sW \to \sV$ be a map of $\sO_X$-modules which is injective at the generic point of $X$. Let $\sF = \cofib(\sW \xrightarrow{\sigma} \sV) \in \Perf(X)^\cn$. 
\end{enumerate}
Under assumption \eqref{assumption:CM}, $\Sigma \sF^\vee$ is represented by the shifted dual complex 
	$$\Sigma \sF^\vee \simeq  \cofib(\sV^\vee \xrightarrow{\sigma^\vee} \sW^\vee) \in \Perf(X)^\cn.$$
(Thus by possibly replacing $\sF$ with $\Sigma \sF^\vee$, there is no loss of generality by assuming $m \le n$.)

\begin{remark}The assumption \eqref{assumption:CM} is {\em not} necessary: all the results of this section remain true for a general scheme $X$ and a discrete finite type quasi-coherent sheaf $\sF$ on $X$ of homological dimension $\le 1$, as long as one replaces the conditions on codimensions of degeneracy loci with conditions on polynomial depths of the corresponding Fitting ideals; see \cite[\S 2]{J21}.
\end{remark}

\begin{notation}\label{notation:deg:loci}  Under assumption \eqref{assumption:CM}, for each integer $i \ge 0$, we let $X_i \subseteq X$ denote the closed subscheme of $X$ defined by all $(m+1-i) \times (m+1-i)$-minors of the map $\sigma$; Equivalently, $X_i$ is {\em $i$-th degeneracy locus} of the discrete sheaf $\Coker(\sigma)$, characterized by the following property: for each point $x \in X_i$ (closed or not), $\rank (\Coker(\sigma)|_{x}) \ge n - m + i$. Therefore we have a sequence of closed subschemes of $X$:
	$$ \emptyset = X_{m+1}  \subseteq X_m \subseteq \cdots \subseteq X_{i} \subseteq \cdots \subseteq X_{1} \subseteq X.$$
For a nonempty closed subscheme $Z \subseteq X$, we let $\codim_Z(X)$ denote the minimal of the codimensions of irreducible components of $Z$ inside $X$ (since $X$ is Noetherian and equidimensional, $\codim_Z(X)$ is well-defined). In particular, for an integer $d$, $\codim_X(Z) \ge d$ holds if and only if $\codim_{X}(Z_j) \ge d$ holds for each irreducible component $Z_j$ of $Z$. 
\end{notation}

\begin{remark}[Expected codimensions] In the situation of Notation \ref{notation:deg:loci}, a theorem of Macaulay--Eagon--Northcott implies that for any irreducible component $Z$ of $X_i$, 
	$$\codim_X(Z) \le (n-m+i)i \qquad i=1,2,\ldots, m.$$
Moreover, a theorem of Hochster--Eagon implies that, if $\codim_X(X_i)$ achieves the maximal possible value $(n-m+i)i$ (that is, $X_i$ is equidimensional and each irreducible component of $X_i$ has codimension $(n-m+ i) i$ in $X$), then $X_i$ is a Cohen-Macaulay subscheme (see, for example, \cite[Exercise 10.9, \S 18.5, \& Theorem 18.18]{Ei} for these statements for commutative rings, and \cite[\S 2.2]{J21} for a summary in the geometric situation). The maximal possible value $(n-m+i)i$ for $\codim_X(X_i)$ is often referred to as the {\em expected codimension of $i$-th degneracy locus $X_i$ in $X$}; it is achieved in the situation when $\sigma$ is a ``generic" map.
\end{remark}

First, we give criteria for the derived projectivizations to be classical.

\begin{lemma} \label{lem:criterion:P:classical} Under assumption \eqref{assumption:CM} and let $X_i$ be defined as in Notation \ref{notation:deg:loci}. Then:
\begin{enumerate}
	\item $\PP(\sF)$ is classical if and only if the following condition is satisfied:
		\begin{equation} \label{eqn:condition:PF} \tag{$a$}
		\codim_X(X_i) \ge i \quad \text{for} \quad i=1, 2, \ldots, m. 
		\end{equation}
		If $\rk \sF = n -m \ge 1$, then $\PP(\sF)$ is classical and irreducible if and only if the following condition is satisfied:
		\begin{equation}  \label{eqn:condition:irr:PF} \tag{$a'$}
		\codim_X(X_i) \ge i +1 \quad \text{for} \quad i=1, 2, \ldots, m.
		\end{equation}	

	\item $\PP(\Sigma \sF^\vee)$ is classical if and only if the following condition is satisfied:
		\begin{equation} \label{eqn:condition:PK}  \tag{$b$}
		\codim_X(X_i) \ge n-m+i  \quad \text{for} \quad i=1, 2, \ldots, m.
		\end{equation}
		$\PP(\Sigma \sF^\vee)$ is classical and irreducible if and only if the following condition is satisfied:
		\begin{equation} \label{eqn:condition:irr:PK}  \tag{$b'$}
		\codim_X(X_1) = n-m+ 1 \quad \text{and} \quad \codim_X(X_i) \ge n-m+i+1 \quad \text{for} \quad i=2, \ldots, m.
		\end{equation}		
	\item Let $\widehat{Z}$ be the universal incidence locus of $\PP(\sF)$ and $\PP(\Sigma \sF^\vee)$ over $X$ of Notation \ref{notation:hatZ}. Then $\widehat{Z}$ is classical if and only if the following condition is satisfied:
		\begin{equation} \label{eqn:condition:hatZ}  \tag{$c$}
		\codim_X(X_i) \ge n-m+2i - 1  \quad \text{for} \quad i=1, 2, \ldots, m.
		\end{equation}
		$\widehat{Z}$ is classical and irreducible if and only if the following condition is satisfied:		
		\begin{equation}  \label{eqn:condition:irr:hatZ}  \tag{$c'$}
		\codim_X(X_1) = n-m+ 1 \quad \text{and} \quad \codim_X(X_i) \ge n-m+2i \quad \text{for} \quad i=2, \ldots, m.
		\end{equation}	
		(Notice that if $\widehat{Z}$ is classical, then $\PP(\sF)$ and $\PP(\Sigma \sF^\vee)$ are classical, and $\widehat{Z}$ is canonically equivalent to the classical fiber product $(\PP(\sF) \times_X \PP(\Sigma \sF^\vee))_\cl$.) 
\end{enumerate}
\end{lemma}

\begin{proof} We first prove the assertions about $\PP(\sF)$. By virtue of Proposition \ref{prop:proj:PB}, the closed immersion $\PP(\sF) \hookrightarrow \PP(\sV)$ is the derived zero locus of a section $s$ of the vector bundle $\sW^\vee \otimes \sO_{\PP(\sV)}(1)$. Since $\PP(\sV)$ is an integral Cohen--Macaulay scheme, therefore for each irreducible component $Z$ of $\PP_\cl(\sF)$, $\codim_{\PP(\sV)}(Z) \le m$. Moreover, $\PP(\sF)$ is classical if and only if $s$ is a regular section, if and only if $\codim_{\PP(\sV)}(\PP_\cl(\sF)) \ge m$.
Since the restriction of the projection from the classical projectivization $\PP_\cl(\sF) \to X$ is smooth with fiber $\PP^{n-m+i-1}$ over the subscheme $X_i \backslash X_{i+1}$ (Remark \ref{rmk:classical:proj:functor}), it follows that $\codim_{\PP(\sV)}(\PP_\cl(\sF)) \ge m$ if and only if condition \eqref{eqn:condition:PF} is satisfied. If we assume $n - m \ge 1$ and $\PP(\sF)$ is classical, then $\pr^{-1}(X \backslash X_1)$ is a $\PP^{n-m-1}$-bundle over $X \backslash X_1$. Hence $\pr^{-1}(X \backslash X_1)$ has codimension $m$ inside $\PP(\sV)$. Therefore, in this case, $\PP(\sF)$ is irreducible if and only if \eqref{eqn:condition:irr:PF} is satisfied. The same argument proves the assertions about $\PP(\Sigma \sF^\vee)$. By virtue of Lemma \ref{lem:dproj:incidence}, the assertions about $\widehat{Z}$ can be proved in a similar manner.
\end{proof}

\begin{remark} In the situation of Lemma \ref{lem:criterion:P:classical}:
\begin{itemize}
	\item If $\rk \sF = n - m  \ge 1$, then we have the following implications:
	$$\eqref{eqn:condition:PF} \impliedby \eqref{eqn:condition:irr:PF} \impliedby \eqref{eqn:condition:PK} \impliedby \eqref{eqn:condition:irr:PK} \impliedby \eqref{eqn:condition:hatZ} \impliedby \eqref{eqn:condition:irr:hatZ}.$$
In general, each of the above implications ``$\impliedby$" is {\em strict} (for example, if $\rk \sF$ and $m$ are large enough, say, if $\rk \sF \ge 2$ and $m \ge 3$). 
	\item If $\rk \sF = n-m = 0$,  then we have the following implications:	$$\eqref{eqn:condition:PF} \iff \eqref{eqn:condition:PK} \impliedby \eqref{eqn:condition:irr:PK} \impliedby \eqref{eqn:condition:hatZ} \impliedby \eqref{eqn:condition:irr:hatZ}.$$	
In general, each of the above implications ``$\impliedby$" is {\em strict} (for example, if $m \ge 3$). 
\end{itemize}
	It is easy to construct numerous counter-examples for the inverse of eahc of the above implications ``$\impliedby$", and we will consider some of such examples in the next subsection.
\end{remark}

Next, we give criteria for the derived symmetric powers to be classical.

\begin{lemma} \label{lem:criterion:Sym:classical} Under assumption \eqref{assumption:CM} and let $X_i$ be defined as in Notation \ref{notation:deg:loci}. Then:
\begin{enumerate}
	\item \label{lem:criterion:Sym:classical-1} Let $d \ge 0$ be an integer. Then the canonical map 
		$$\Sym^d_{X}(\sF) \to \pi_0 (\Sym^d_{X}(\sF)) \simeq S^d_{X}(\Coker(\sigma))$$ 
	is an equivalence if and only if the following condition is satisfied:
		$$\codim_X(X_i) \ge i \quad \text{for} \quad i=1, 2, \ldots, \min\{m, d\}.$$
	\item \label{lem:criterion:Sym:classical-2} Let $d \ge 0$ be an integer.
		\begin{enumerate}
			\item \label{lem:criterion:Sym:classical-2i}If $d \ge n-m+1$. Then the canonical map 
		$$\Sym^d_{X}(\Sigma \sF^\vee) \to \pi_0 (\Sym^d_{X}(\Sigma \sF^\vee)) \simeq S^d_X(\Coker(\sigma^\vee))$$ 
	is an equivalence if and only if the following condition is satisfied:
		$$\codim_X(X_i) \ge n-m+i \quad \text{for} \quad  i=1, 2, \ldots, \min\{m, d\}.$$
			\item \label{lem:criterion:Sym:classical-2ii} If $0 \le d \le n-m$. Assume that the following condition is satisfied:
		$$\codim_X(X_1) = n-m+1.$$	 
	Then there are canonical equivalences
		$$
		\pi_i(\Sym^d_X(\Sigma \sF^\vee)) \simeq 
		\begin{cases} 
		 \Sym^d_{\cl}(\Coker(\sigma^\vee)) & i=0. \\
		 \bigwedge\nolimits^{n-m-d}_{\cl}(\Coker(\sigma)) \otimes (\det \sF)^\vee & i = d.\\
		 0 & \text{otherwise}.
		\end{cases}
		$$
		\end{enumerate}
\end{enumerate}
\end{lemma}

\begin{proof}
We first prove assertion \eqref{lem:criterion:Sym:classical-1}. We will assume $n > m$, since the case $n=m$ can be reduced to assertion \eqref{lem:criterion:Sym:classical-2}. By virtue of Corollary \ref{cor:Illusie:S^n}, $\Sym_X^d (\sF)$ is canonically represented by the complex of vector bundles $\bS^{d}(X, \sW \xrightarrow{\sigma} \sV)$ (defined by formula \eqref{eqn:Koszul:S^n}). For each $0 \le i \le \min\{m,d\}$, we let $F_i$ denote the $i$th term of the complex $\bS^{d}(X, \sigma)$, let $d_i \colon F_i \to F_{i-1}$ denote the differential of $\bS^{d}(X, \sigma)$, and let $r_i = \sum_{j=i}^d (-1)^{j-i} \rank F_j$ be the ``expected rank" of $d_i$. Let $\sI_{X_i}$ denote the ideal of $X_i \subseteq X$, and let $I_{r_i}(d_i)$ denote the ideal generated by $(r_i \times r_i)$-minors of the differential $d_i$ (see \cite[\S 20.2]{Ei}). From the same argument as the proof of \cite[Theorem A2.1]{Ei}, we know that
	$$\sqrt{I_{r_i}(d_i)} = \sqrt{\sI_{X_i}} \quad \text{for} \quad i=1, 2, \ldots, \min\{m, d\}.$$
(This can also be deduced from the general result \cite[Theorem 6.1$(b)$]{AT19}). In particular, 
	$$\depth(I_{r_i}(d_i)) = \depth(\sI_{X_i}) = \codim_X(X_i)$$
since $X$ is Cohen--Macaulay. On the other hand, from Buchsbaum--Eisenbud's acyclicity criterion (\cite[Theorem 20.9]{Ei}), $\bS^{d}(X, \sigma)$ is a resolution of $S^d_X(\Coker(\sigma))$ if and only if
	$$\depth(I_{r_i}(d_i)) \ge i \quad \text{for} \quad i=1,2,\ldots, \min\{m,d\}.$$
This proves assertion \eqref{lem:criterion:Sym:classical-1} in the case where $n > m$.

We next prove assertion \eqref{lem:criterion:Sym:classical-2i}.
By virtue of Corollary \ref{cor:Illusie:S^n}, $\Sym_X^d (\Sigma \sF^\vee)$ is canonically represented by $\bS^{d}(X, \sV^\vee \xrightarrow{\sigma^\vee} \sV^\vee)$. For each $0 \le i \le \min\{m,d\}$, we let $F_i'$ denote the $i$th term of the complex $\bS^{d}(X, \sigma^\vee)$, let $d_i' \colon F_i' \to F_{i-1}'$ denote the differential, and let $r_i' = \sum_{j=i}^d (-1)^{j-i} \rank F_j'$. From the proof of \cite[Theorem A2.1]{Ei}, we know that
	$$
	\sqrt{I_{r_i'}(d_i')} = \begin{cases}
		\sqrt{\sI_{X_1}} &  \text{if} \quad 1 \le i \le n-m+1 \\
		\sqrt{\sI_{X_{i-n+m}}} &  \text{if} \quad n-m+2 \le i \le \min\{m, d\}.
	\end{cases}
	$$
(This can also be deduced from \cite[Theorem 6.1 $(a)$ \& $(b)$]{AT19}.) Therefore,
	$$
	\depth(I_{r_i'}(d_i')) = \begin{cases}
		\codim_X(X_1) &  \text{if} \quad 1 \le i \le n-m+1 \\
		\codim_X(X_{i-n+m})&  \text{if} \quad n-m+2 \le i \le \min\{m, d\},
	\end{cases}
	$$
since $X$ is Cohen--Macaulay. By Buchsbaum--Eisenbud's acyclicity criterion again, we know that $\bS^{d}(X, \sigma^\vee)$ is a resolution of $S_X^d(\Coker(\sigma^\vee))$ if and only if
	$$\depth(I_{r_i'}(d_i')) \ge i \quad \text{for} \quad i=1,2,\ldots, \min\{m,d\}.$$
This proves assertion \eqref{lem:criterion:Sym:classical-2i} (and assertion \eqref{lem:criterion:Sym:classical-1} in the case where $n = m$).

Finally, to prove assertion \eqref{lem:criterion:Sym:classical-2ii}, we consider the Eagon--Northcott complex $\mathbf{EN}_d(X, \sigma)$ of Remark \ref{rmk:ENcomplexes}. Similarly as above, the desired assertions follow from \cite[Theorem A2.1]{Ei} (where the complex $\shC^i$ there is our $\mathbf{EN}_i(X, \sigma)$) and Buchsbaum--Eisenbud's acyclicity criterion.
\end{proof}

\subsubsection{Examples violating \eqref{eqn:condition:PF}} By virtue of Lemma \ref{lem:criterion:P:classical}, as long as the pair $(X, \sF)$ in the situation of \eqref{assumption:CM} violates the condition \eqref{eqn:condition:PF}, $\PP(\sF)$ is not classical. For example:

\begin{example} \label{eg:A^1:PFcl} Let $X = \AA^1 = \Spec \ZZ[x]$, and $\sF = \sO_{\mathbf{0}}^{\oplus k}$, where $\mathbf{0} \in \AA^1$ is the closed point define by the equation $x=0$, $k \ge 1$ is an integer, and $\sO_{\mathbf{0}}$ is the structure sheaf of $\mathbf{0}$ regarded as a discrete $\sO_X$-module. In this case,
	$(X_1)_{\red} = \cdots = (X_{k-1})_{\red} = X_k = \{\mathbf{0}\}.$
Therefore, by virtue of Lemma \ref{lem:criterion:P:classical} and Lemma \ref{lem:criterion:Sym:classical}, we obtain:

\begin{itemize}
	\item If $k =1$, then
	$\PP(\sO_{\mathbf{0}})\simeq \PP_\cl(\sO_{\mathbf{0}}) \simeq \{\mathbf{0}\} \subseteq \AA^1$
is classical, and there is a canonical isomorphism 
	$\Sym_{\AA^1}^d(\sO_{\mathbf{0}}) \simeq \Sym_{\cl, \AA^1}^d(\sO_{\mathbf{0}})$
for all integers $d \ge 0$;
	\item If $k \ge 2$, then $\PP(\sO_{\mathbf{0}})$ is not classical, and the canonical map 
	$\Sym_{\AA^1}^d(\sO_{\mathbf{0}}^{\oplus k}) \to \Sym_{\cl, \AA^1}^d(\sO_{\mathbf{0}}^{\oplus k})$
is not an equivalence as long as $d \ge 2$. 
\end{itemize}
More concretely, we can identify the classical projectivization with a projective space
	$$\PP_\cl(\sO_{\mathbf{0}}^{\oplus k}) \simeq \PP_\ZZ^{k-1} \to \{\mathbf{0}\} \subseteq \AA^1.$$
On the other hand, by virtue of Proposition \ref{prop:proj:PB}, we can identify the derived projectivization $\PP(\sO_{\mathbf{0}}^{\oplus k})$ as the derived zero locus of $\AA^1 \times \PP^{k-1} = \Spec \ZZ[x] \times \Proj \ZZ[X_1, \ldots, X_k]$ defined by the $k$ equations $xX_1 = x X_2 = \cdots = x X_k = 0$. This derived zero locus has  underlying classical scheme $\{\mathbf{0}\} \times \PP^{k-1}$ defined by the equation $x=0$; but as long as $k \ge 2$, it is equipped with a non-trivial derived structure.
\end{example}

\subsubsection{Examples satisfying \eqref{eqn:condition:PF} but violating \eqref{eqn:condition:irr:PF} and \eqref{eqn:condition:PK}} 
By virtue of Lemma \ref{lem:criterion:P:classical}, as long as the pair $(X, \sF)$ in the situation of \eqref{assumption:CM} satisfies condition \eqref{eqn:condition:PF} but violate condition \eqref{eqn:condition:irr:PF} and \eqref{eqn:condition:PK}, $\PP(\sF)$ is classical but reducible, and $\PP(\Sigma \sF^\vee)$ is not classical. There are numerous examples. 

\begin{example}\label{eg:A^1:PKcl} Let $X = \AA^1 = \Spec \ZZ[x]$, and let $\sF$ denote the cokernel of the map 
	$$\sO_{\AA^1} \xrightarrow{(x,x)^T} \sO_{\AA^1}^{\oplus 2}.$$
\begin{itemize}[leftmargin=*]
	\item By virtue of Lemma \ref{lem:criterion:P:classical}, $\PP(\sF)$ is classical but $\PP(\Sigma \sF^\vee)$ is not. 
	\item By virtue of Lemma \ref{lem:criterion:Sym:classical}, the canonical map
		$\Sym_{\AA^1}^d(\sF) \to \Sym_{\cl, \AA^1}^d(\sF)$
		is an equivalence for every $d \ge 0$, but the canonical map $\Sym^d(\Sigma \sF^\vee) \to \Sym_\cl^d(\sO_{\mathbf{0}})$ is {\em not} an equivalence as long as $d \ge 1$ (the case $d \ge 2$ is Lemma \ref{lem:criterion:Sym:classical}  \eqref{lem:criterion:Sym:classical-2i} and the case $d=1$ is obvious). 
\end{itemize}
More concretely, $\PP(\sF) \simeq \AA^1 \coprod_{\rm pt} \PP^1$ is a nodal curve which is the union of $\AA^1$ and $\PP^1$ along a point, and $\PP(\Sigma \sF^\vee)$ is isomorphic the derived zero locus in $\AA^1$ which is ``cut out by the equation $(x=0)$ {\em twice}" (Proposition \ref{prop:proj:PB}). By virtue of Example \ref{eg:prop:proj:PB}, $\PP(\Sigma \sF^\vee)\simeq \Spec \ZZ[\varepsilon]$, where $\ZZ[\varepsilon] = \Sym_\ZZ^*(\Sigma \, \ZZ) \simeq \ZZ \oplus \ZZ \varepsilon$ is the ring of {\em derived} dual numbers (where we informally think of $\varepsilon$ as the generator in degree $1$; beware that $\ZZ[\varepsilon]$ is {\em not} isomorphic to the classical dual number $\ZZ[x]/(x^2)$). In particular, $\PP(\Sigma \sF^\vee)$ has underlying classical scheme $\{\mathbf{0}\} = \Spec \ZZ \simeq \Spec \ZZ[x]/(x)$, but with a non-trivial derived structure given by the extra equation $(x=0)$. 

Moreover, by virtue of Theorem \ref{thm:structural}, we obtain semiorthogonal decompositions
	\begin{align*}
	\QCoh\big(\AA^1 \coprod\nolimits_{\rm pt} \PP^1\big)  & = \big\langle \QCoh(\AA^1), ~\QCoh(\Spec  \ZZ[\varepsilon] ) \big \rangle, \\
	\Perf\big(\AA^1 \coprod\nolimits_{\rm pt} \PP^1\big)  & = \big\langle \Perf(\AA^1), ~\Perf(\Spec \ZZ[\varepsilon]) \big\rangle.
	\end{align*}
\end{example}

We could consider the following global version of Example \ref{eg:A^1:PKcl}:

\begin{example}[Reducible schemes]\label{eg:DinX}
Let $X$ be a scheme, let $D \subseteq X$ be an effective Cartier divisor, and let $f \in \H^0(X, \sO(D))$ be the section determined by $D$. Let $\sF = \sO_X \oplus \sO_D$. Then:
	\begin{itemize}
		\item By virtue of Proposition \ref{prop:proj:PB}, $\PP(\sF) \subseteq X \times \PP^1 = X \times \Proj \ZZ[X_0, X_1]$ is the divisor defined by the equation $(X_1 f = 0)$. Spelling this out, we see that $\PP(\sF)$ is classical and is union of $X$ and $\PP_D^1 = D \times \PP^1$ along the divisor $D$.
		\item By virtue of Example \ref{eg:prop:proj:PB}, we obtain $\PP(\Sigma \sF^\vee) \simeq \VV_{D}(\Sigma \, \sO_D(-D))$ is not classical (and is equipped with a nontrivial derived structure everywhere).
		More concretely, in any open subset $U \subseteq X$ for which $\sO_X(-D)|_U$ is a trivial line bundle, there is an equivalence $\PP_U(\Sigma \sF^\vee) \simeq (D \cap U) \times \Spec \ZZ[\varepsilon]$, where $\ZZ[\varepsilon] = \Sym_\ZZ^*(\Sigma \ZZ)$ is the ring of derived dual numbers.
		More informally, we could think of $\VV_{D}(\Sigma \, \sO_D(-D))$ as the total space ${\rm Tot}(\shN_{D/X}[-1])$ of $(-1)$-shifted normal bundle of $D$ in $X$, in the sense \cite{ACH, PTVV}. 
	\end{itemize} 
By virtue of Theorem \ref{thm:structural}, we obtain semiorthogonal decompositions
	\begin{align*}
	\QCoh\big(X \coprod\nolimits_{D} \PP_D^1\big)  & = \big\langle \QCoh(X), ~\QCoh(\VV_{D}(\Sigma \, \sO_D(-D))) \big \rangle, \\
	\Perf\big(X \coprod\nolimits_{D} \PP_D^1\big)  & = \big\langle \Perf(X), ~\Perf(\VV_{D}(\Sigma \, \sO_D(-D))) \big\rangle.
	\end{align*}
\end{example}	
	
\begin{example}[Attaching $\PP^1$'s to smooth points of curves]\label{eg:rational_tail}
In the above Example \ref{eg:DinX}, if we let $X=C$ be a reduced curve over a field $\kappa$ (that is, a reduced separated finite type scheme over $\kappa$ whose irreducible components all have dimension $1$), and let $p \in C(\kappa)$ be a non-singular rational point, and let $\sF = \sO_C \oplus \sO_p$, then we obtain semiorthogonal decompositions
	\begin{align*}
	\QCoh\big(C \coprod\nolimits_{p} \PP^1\big)  & = \big\langle \QCoh(C), ~\QCoh(\Spec \kappa[\varepsilon] ) \big \rangle, \\
	\Perf\big(C \coprod\nolimits_{p} \PP^1\big)  & = \big\langle \Perf(C), ~\Perf(\Spec \kappa[\varepsilon] ) \big\rangle,
	\end{align*}
where $\kappa[\varepsilon] = \Sym_\kappa(\Sigma \, \kappa) = \ZZ[\varepsilon] \otimes_\ZZ \kappa$ is the ring of derived dual numbers over $\kappa$; These semiorthogonal decompositions agree with the results in \cite[Example 5.2, Theorem 5.12]{Kuz21}. 

We could also repeat the above process: set $C_0 = C$, $p_0 = p$, and let $C_i$ be a curve obtained by attaching a $\PP^1$ to a nonsingular $\kappa$-point $p_{i-1} \in C_{i-1}(\kappa)$ of $C_{i-1}$. Then for an integer $n \ge 1$, $C_n$ is the union of $C$ with finite number of trees of $\PP^1$'s for which the total number of $\PP^1$ is $n$. We therefore obtain semiorthogonal decompositions:
\begin{align*}
	\QCoh(C_n)  & = \big\langle \QCoh(C), ~\text{$n$-copies of $\QCoh(\Spec \kappa[\varepsilon])$} \big \rangle, \\
	\Perf(C_n)  & = \big\langle \Perf(C), ~\text{$n$-copies of $\Perf(\Spec \kappa[\varepsilon])$} \big\rangle.
	\end{align*}
However, the morphism spaces of the latter components depend on the configurations of the trees of $\PP^1$'s. For example, in the case where $C_n$ is the union of $C$ and a chain $\Gamma_{n}$ of $\PP^1$'s, the endomorphism algebra of the latter component (of a choice of tilting bundle on $\Gamma_n$) is  described explicitly in \cite[Theorem 2.1]{Bur}. In general, the endomorphism algebras of ``the contributions" of the new attached trees of $\PP^1$'s can be described similarly; see \cite[Theorem 4.9]{KPS}.
\end{example}

We could also consider the following variant\footnote{I thank Michel van Garrel for asking about the relationship between the above example and the schemes in the construction of deformation to normal cones.} of Example \ref{eg:DinX} :
 \begin{variant}[Deformation to normal cones] \label{eg:normalcones}
 Similar to Example \ref{eg:DinX}, we let $X$ be a scheme and let $D \subseteq X$ be an effective Cartier divisor. Now we consider $\sF' = \sO_X \oplus \sO_D(D)$ instead of $\sF = \sO_X \oplus \sO_D$. Then $\PP(\sF') = X \coprod_{D} \PP_D(\shN \oplus \sO_D)$, where $\shN = \shN_{D/X}$ is the normal bundle, and $\PP(\Sigma \, \sF'^\vee) = D \times \Spec \ZZ[\varepsilon]$. Therefore, Theorem \ref{thm:structural} implies there are semiorthogonal decompositions 
	\begin{align*}
	\QCoh\big(X \coprod\nolimits_{D} \PP_D(\shN \oplus \sO_D)\big)  & = \big\langle \QCoh(X), ~\QCoh(D) \otimes \QCoh(\Spec \ZZ[\varepsilon])  \big \rangle, \\
	\Perf\big(X \coprod\nolimits_{D} \PP_D(\shN \oplus \sO_D)\big)  & = \big\langle \Perf(X), ~\Perf(D) \otimes \Perf(\Spec \ZZ[\varepsilon]) \big\rangle.
	\end{align*}
Notice that the reducible scheme $X \coprod_{D} \PP_D(\shN \oplus \sO_D)$ in the above formulae is precisely the special fiber appearing in the construction of deformation to normal cones applied to $D \subseteq X$. More concretely, let $M = \Bl_{D \times \{\infty\}} (X \times \PP^1)$ be the total space in the construction of deformation to normal cones of \cite[\S 5.1]{Ful}, and let $\rho \colon M \to \PP^1$ be the natural projection. Then for $t \ne \infty \in \PP^1$, $\rho^{-1}(t) \simeq X$, and $\rho^{-1}(\infty) \simeq X \coprod_{D} \PP_D(\shN \oplus \sO_D) \simeq \PP_X(\sF')$.
\end{variant}

\begin{remark}[Derived blowups] \label{rmk:egs:blowup} In view of Remark \ref{rmk:blowup}, we could view the semiorthogonal decompositions in Examples \ref{eg:DinX} and \ref{eg:normalcones} as {\em blowup formulae}, where $\PP(\sF)=X \coprod\nolimits_{D} \PP_D^1$ and $\PP(\sF')=X \coprod\nolimits_{D} \PP_D(\shN \oplus \sO_D)$ are the {derived blowups} (in the sense of \cite{KR18, He21}) of $X$ along the virtual codimension two closed derived subschemes $\PP(\Sigma \sF^\vee)={\rm Tot}(\shN_{D/X}[-1]) \subseteq X$ and $\PP(\Sigma \sF'^\vee)=D \times \Spec \ZZ[\varepsilon] \subseteq X$, respectively.
\end{remark}

If a variety $X$ over a field $\kappa$ has locally complete intersection singularities, then the cotangent complex $L_{X/\kappa}$ has Tor-amplitude $\le 1$. In many situations, the pair $(X, \sE= L_{X/\kappa})$ satisfies \eqref{eqn:condition:PF}. For example, we could consider the cases of nodal curves:

\begin{example}[Contracting rational bridges]\label{eg:rational_bridge}
Let $\kappa$ be a field, and let $C \subseteq \AA_\kappa^2 = \Spec \kappa [x,y]$ be the nodal curve defined by the equation $xy= 0$. Then the cotangent complex $\sF = L_{C/\kappa} \simeq \Omega_{C/\kappa}^1$ is represented by the complex of locally free sheaves:
	$$\left[\sO_C \xrightarrow{(y,x)^T} \sO_C {\rm d}x \oplus \sO_C {\rm d}y\right].$$
In this case, by virtue of Proposition \ref{prop:proj:PB}, $\PP(L_{C/\kappa})$ is the classical reducible scheme defined the equation $yX_0 + x X_1 =0$ in $C \times \Proj \kappa[X_0, X_1]$. 
In particular, $\widetilde{C} = \PP(\Omega_{C/\kappa}^1)$ is the total transform of $C \subseteq \AA^2$ along the blowup $\Bl_{0} \AA^2 \to \AA^2$, where $o =(0,0)$ is the origin. Then projection morphism $\pi \colon \widetilde{C} \to C$ is an isomorphism over $C \backslash \{o\}$ and the exceptional locus $E = \pi^{-1}(\o)$
is a projective line $\PP^1_\kappa$ (called the rational bridge) with multiplicity {\em two}\footnote{I thank Martin Kalck for bringing my attention to the multiplicity issue.}. The underlying morphism of reduced schemes $\widetilde{C}^{\rm red} =\PP(\Omega_{C/\kappa}^1)^{\rm red} \to C$ is the contraction of the rational bridge $\PP_\kappa^1$. 
On the other hand, $\PP(\Sigma L_{C/\kappa}^\vee)$ is the derived zero locus in $\AA^2$ cut out by the three equations $x=0, y=0$ and $xy=0$, hence we obtain $\PP(\Sigma L_{C/\kappa}^\vee)\simeq \Spec (\kappa[\varepsilon] \otimes_{\kappa} \kappa[\delta])$, where $\kappa[\varepsilon] = \Sym_\kappa(\Sigma \, \kappa) = \ZZ[\varepsilon] \otimes_\ZZ \kappa$ is the ring of derived dual numbers over $\kappa$ and  $\kappa[\delta] = \kappa[\delta]/(\delta^2)$ denote the ring of usual dual numbers. Theorem \ref{thm:structural} implies semiorthogonal decompositions
	\begin{align*}
	\QCoh(\widetilde{C})  & = \big\langle \pi^* \QCoh(C), ~\QCoh\big(\Spec (\kappa[\varepsilon] \otimes_{\kappa} \kappa[\delta])\big) \big \rangle, \\
	\Perf(\widetilde{C})  & = \big\langle \pi^* \Perf(C), ~\Perf\big(\Spec (\kappa[\varepsilon] \otimes_{\kappa} \kappa[\delta])\big) \big\rangle.
	\end{align*}
One can generalize this example by considering any nodal curve $C/\kappa$ (or a family of nodal curves $\sC/S$) and the projectivization $\widetilde{C}:=\PP(\Omega_{C/\kappa}^1) \to C$ (or $\widetilde{\sC}:=\PP(L_{\sC/S}) \to \sC$).
\end{example}

\begin{remark}[Stabilization maps] 
We could regard the formulae of Examples \ref{eg:A^1:PKcl} and \ref{eg:rational_tail} as describing the behavior of the derived categories under ``contracting rational tails" (\cite[\href{https://stacks.math.columbia.edu/tag/0E3G}{Tag 0E3G}]{stacks-project}).
On the other hand, the formula of Example \ref{eg:rational_bridge} describes the behavior of the derived categories under ``contracting non-reduced rational bridges" (see \cite[\href{https://stacks.math.columbia.edu/tag/0E7M}{Tag 0E7M}]{stacks-project} for more about the contraction of rational bridges). Notice that a recent result of Kuznetsov and Shinder \cite[Corollary 6.8]{KS22b} implies that there are no semiorthogonal decompositions for the underlying {\em reduced} scheme $\widetilde{C}^{\rm red}$ in Example \ref{eg:rational_bridge} (if $C$ is projective). 
These results altogether provide a comprehensive picture of the behavior of derived categories of nodal curves under stabilization maps,  since the general stabilization maps of nodal curves are composed of operations of contracting rational tails and rational bridges (see, for example, \cite[\href{https://stacks.math.columbia.edu/tag/0E7Q}{Tag 0E7Q}]{stacks-project}).
\end{remark}

We can also consider higher-dimensional variants of Example \ref{eg:A^1:PKcl}. For example:

\begin{example}\label{eg:A^2:(x,y)} Let $X = \AA^2 = \Spec \ZZ[x,y]$, and let $\sF$ denote the cokernel of the map 
	$$\sO_{\AA^2}^{\oplus 2} \xrightarrow{ \big(\begin{smallmatrix}
x & 0  \\
0 & y \\
x & y
\end{smallmatrix} \big)}  \sO_{\AA^2}^{\oplus 3}.$$
Then $\rank \sF = 1$, $\codim_X(X_1) = 1$, $\codim_X(X_2) = 2$, so by virtue of  Lemma \ref{lem:criterion:P:classical}, $\PP(\sF)$ is classical but $\PP(\Sigma \sF^\vee)$ is not. More concretely:
	\begin{itemize}
		\item  $\PP(\sF) \subseteq \AA^2 \times \PP^2 = \Spec \ZZ[x,y] \times \Proj \ZZ[X,Y,Z]$ is the classical closed subscheme defined by the equations $xX + xZ = 0$ and $y Y + y Z = 0$. Unwinding the definition, we see that $\PP(\sF)$ has four irreducible components:
			$$\PP(\sF) = \AA^2 \cup (\AA^1_{x} \times \PP^1) \cup  (\{(0,0)\} \times \PP^2) \cup (\AA^1_{y} \times \PP^1).$$
		\item $\PP(\Sigma^\vee \sF) \subseteq \AA^2 \times \PP^1 = \Spec \ZZ[x,y] \times \Proj \ZZ[U, V]$ is the derived closed subscheme defined by the quations $xU = 0$, $y V=0$, and $xU+yV = 0$.  Unwinding the definition, we see that the underlying classical scheme of $\PP(\Sigma^\vee \sF)$ is the blowup of the nodal curve defined by $(xy = 0)$ along the nodal point, that is,
			$$\PP(\Sigma^\vee \sF)_\cl = \AA^1 \coprod_{\rm pt} \PP^1 \coprod_{\rm pt} \AA^1.$$
			But over each closed point, $\PP(\Sigma^\vee \sF)$ is equipped with a non-trivial derived structure which is isomorphic to that of the ring of derived dual numbers $\Spec \ZZ[\varepsilon]$: 
			$$\PP(\Sigma^\vee \sF) \simeq (\AA^1 \coprod_{\rm pt} \PP^1 \coprod_{\rm pt} \AA^1) \times \Spec \ZZ[\varepsilon].$$
	\end{itemize}
Theorem \ref{thm:structural} implies there are semiorthogonal decompositions
	\begin{align*}
	&\QCoh \big( \AA^2 \cup (\AA^1 \times \PP^1) \cup  (\{(0,0)\} \times \PP^2) \cup (\AA^1 \times \PP^1) \big)\\
	  & \qquad  \qquad = \big\langle \QCoh(\AA^2), ~\QCoh\big( \AA^1 \coprod_{\rm pt} \PP^1 \coprod_{\rm pt} \AA^1\big) \otimes \QCoh(\Spec  \ZZ[\varepsilon] ) \big \rangle, \\
	&\Perf \big( \AA^2 \cup (\AA^1 \times \PP^1) \cup  (\{(0,0)\} \times \PP^2) \cup (\AA^1 \times \PP^1) \big) \\
	&  \qquad  \qquad= \big\langle \Perf(\AA^2), ~\Perf\big(  \AA^1 \coprod_{\rm pt} \PP^1 \coprod_{\rm pt} \AA^1\big) \otimes \Perf(\Spec \ZZ[\varepsilon]) \big\rangle.
	\end{align*}
\end{example}

\subsubsection{Examples satisfying \eqref{eqn:condition:irr:PF} but violating  \eqref{eqn:condition:PK}}  
By virtue of Lemma \ref{lem:criterion:P:classical}, as long as the pair $(X, \sF)$ in the situation of \eqref{assumption:CM} satisfies condition \eqref{eqn:condition:irr:PF} but violates condition \eqref{eqn:condition:PK}, $\PP(\sF)$ is classical and irreducible, but $\PP(\Sigma \sF^\vee)$ is not classical. For example:

\begin{example} \label{eg:A^2:(x,y,0)} 
Let $X = \AA^2 = \Spec \ZZ[x,y]$, and let $\sF$ denote the cokernel of the map 
	$$\sO_{\AA^2} \xrightarrow{(x,y,0)^T} \sO_{\AA^2}^{\oplus 3}.$$
Then $\PP(\sF)$ is the quadric cone $Q =\{xX + y Y =0\}  \subseteq \AA^2 \times \PP^2 = \Spec \ZZ[x,y] \times \Proj \ZZ[X,Y,Z]$, which is an irreducible threefold with a generic $\PP^1$-bundle structure over $X = \AA^2$ and with an $A_1$-singularity at the point $((0,0), [0,0,1]) \in Q$. By virtue of Proposition \ref{prop:proj:PB} again, $\PP(\Sigma \sF^\vee) \simeq \Spec \ZZ[\varepsilon]$, where $\ZZ[\varepsilon] = \Sym_\ZZ^*(\Sigma \,\ZZ)$ is the ring of derived dual numbers. Therefore, by virtue of Theorem \ref{thm:structural} we obtain semiorthogonal decompositions
	\begin{align*}
	\QCoh(Q)  & = \big\langle \QCoh(\Spec  \ZZ[\varepsilon] ), ~~ \QCoh(\AA^2) \otimes \sO_Q(1), \QCoh(\AA^2) \otimes \sO_Q(2)  \big \rangle, \\
	\Perf(Q)  & = \big\langle \Perf(\Spec \ZZ[\varepsilon]), ~ ~ \Perf(\AA^2) \otimes \sO_Q(1), \Perf(\AA^2) \otimes \sO_Q(2) \big\rangle.
	\end{align*}
We may replace $X=\AA^2$ by any reduced surface $S$ over a field $\kappa$ and $\sF$ by $\sO_S \oplus \sI_p$, where $\sI_p$ is the ideal sheaf of a non-singular rational point $p \in S(\kappa)$. The above semiorthogonal decompositions are compatible with results in \cite[\S 7.1]{Kaw19}, \cite[Theorem 4.6]{Xie21} and \cite[Theorem 3.6]{PS21} on nodal threefolds. 
\end{example}

Similar to Example \ref{eg:DinX}, we may consider the global version of Example \ref{eg:A^2:(x,y,0)}:
\begin{example}\label{eg:YinX} Let $X$ be a scheme and $Y \subseteq X$ be a regularly immersed closed subscheme of codimension $2$. Let $\sF = \sO_X \oplus \sI_Y$, where $\sI_Y$ is the ideal sheaf of $Y$. 
\begin{itemize}
	\item $\PP(\sF)$ is the classical scheme which is a $\PP^1$-bundle over $X \backslash Y$, $\PP^2$-bundle over $Y$, and has $\AA_1$-singularity along a unique section of the $\PP^2$-bundle over $Y$. More concretely, let $U \subseteq X$ be any open subscheme for which $Y \cap U \subset U$ is defined by the equation $(f_1 = f_2 = 0)$, where $(f_1, f_2) \in \sO_X(U)$ is a regular sequence. Then by virtue of Proposition \ref{prop:proj:PB},  $\PP_U(\sF) \subseteq U \times \PP^2 = U \times \Proj \ZZ[X_0, X_1, X_2]$ is an effective divisor defined by the equation $(f_1 X_1 + f_2 X_2) = 0$.
	\item By virtue of Example \ref{eg:prop:proj:PB}, $\PP(\Sigma \sF^\vee) \simeq \VV_Y(\Sigma (\det \shN_{Y/X})^\vee)$ which we could informally think of as the total space $(\det \NN_{Y/X}) [-1]$ of the $(-1)$-shifted determinant bundle of normal bundle, where $\shN_{Y/X}$ is the normal bundle of $Y$ inside $X$. More concretely, over the open subsets $U \subseteq X$ of the previous bullet,
	$\PP_U(\Sigma \sF^\vee)$ is equivalent to $(Y \cap U) \times \Spec \ZZ[\varepsilon]$, where $\ZZ[\varepsilon] = \Sym_\ZZ^*(\Sigma \ZZ)$ as usual. 
\end{itemize}
By virtue of Theorem \ref{thm:structural} we obtain semiorthogonal decompositions
	\begin{align*}
	\QCoh(\PP(\sF))  & = \big\langle \QCoh(\VV_Y(\Sigma (\det \shN_{Y/X})^\vee)), ~~ \QCoh(X) \otimes \sO(1), \QCoh(X) \otimes \sO(2)  \big \rangle, \\
	\Perf(\PP(\sF))  & = \big\langle \Perf(\VV_Y(\Sigma (\det \shN_{Y/X})^\vee)), ~ ~ \Perf(X) \otimes \sO(1), \Perf(X) \otimes \sO(2) \big\rangle.
	\end{align*}
\end{example}

\subsubsection{Examples satisfying \eqref{eqn:condition:PK} or even \eqref{eqn:condition:irr:PK} but violating  \eqref{eqn:condition:hatZ}} 
By virtue of Lemma \ref{lem:criterion:P:classical}, as long as the pair $(X, \sF)$ in the situation of \eqref{assumption:CM} satisfies condition \eqref{eqn:condition:PK} (resp. \eqref{eqn:condition:irr:PK}) but violates condition  \eqref{eqn:condition:hatZ},   $\PP(\sF)$ and $\PP(\Sigma \sF^\vee)$ are classical (resp. classical and irreducible) but the correspondence scheme $\widehat{Z}$ is not classical. 

\begin{example} Let $X= \AA^2 = \Spec \ZZ[x, y]$, and let $\sF$ be the cokernel of the map 
	$$\sO_X^{\oplus 2} \xrightarrow{ \big(\begin{smallmatrix}
x & 0  \\
0 & y \\
\end{smallmatrix} \big)} \sO_X^{\oplus 2}.$$
By virtue of Lemma \ref{lem:criterion:P:classical}, both $\PP(\sF)$ and $\PP(\Sigma \sF^\vee)$ are classical and reducible (in fact, they are both equivalent to the blowup of $\{xy = 0\} \subseteq \AA^2$ along the origin $\mathbf{0}$), but the universal incidence space $\widehat{Z}$ is {\em not} classical. Furthermore, Theorem \ref{thm:structural} implies that the diagram $\PP(\sF) \xleftarrow{r_+} \widehat{Z} \xrightarrow{r_-}  \PP(\Sigma \sF^\vee)$ induces equivalences of $\infty$-categories
	$$r_{+,*} \, r_-^* \colon \QCoh(\PP(\Sigma \, \sF^\vee)) \simeq \QCoh(\PP(\sF)) \qquad r_{+,*} \, r_-^* \colon \Perf(\PP(\Sigma \, \sF^\vee)) \simeq \Perf(\PP(\sF)).$$
(Moreover, the same argument of \cite[Theorem 4.1]{JL18} shows that the composition of equivalences $(r_{-,*} \, r_+^*) \circ (r_{+,*} \, r_-^*)$ is a non-trivial spherical twist in the sense of \cite{ST, AL}.)
\end{example}

\begin{example}\label{eg:A^4:birational} Let $X =\AA^4 = \Spec \ZZ[x,y,z,w]$ and let $\sF$ be the cokernel of the map 
	$$\sO_X^{\oplus 3} \xrightarrow{\sigma = \left(\begin{smallmatrix} x & y & w \\ z & x & y \\ w & z & 0 \end{smallmatrix} \right)} \sO_X^{\oplus 3}.$$
Then $\codim X_1 = 1$ and $\codim X_2 = X_3 = 4$. By virtue of Lemma \ref{lem:criterion:P:classical}, $\PP(\sF)$ and $\PP(\Sigma \sF^\vee)$ are classical, irreducible threefolds (which are both birational equivalent to the determinantal variety $X_1 = \{ - xw^2 + y^2 w+ z^2w -x y z =0\} \subseteq \AA^4$), and $\widehat{Z}$ is {\em not} classical. Furthermore, the diagram $\PP(\sF) \xleftarrow{r_+} \widehat{Z} \xrightarrow{r_-}  \PP(\Sigma \sF^\vee)$ induces equivalences of $\infty$-categories
	$$r_{+,*} \, r_-^* \colon \QCoh(\PP(\Sigma \, \sF^\vee)) \simeq \QCoh(\PP(\sF)) \qquad r_{+,*} \, r_-^* \colon \Perf(\PP(\Sigma \, \sF^\vee)) \simeq \Perf(\PP(\sF)).$$
(Moreover, the same argument of \cite[Theorem 4.1]{JL18} shows that the composition of equivalences $(r_{-,*} \, r_+^*) \circ (r_{+,*} \, r_-^*)$ is a non-trivial spherical twist in the sense of \cite{ST, AL}.)
\end{example}

There are also similar examples where $X$ is an irreducible sixfold, $\sigma$ is a $4 \times 4$-matrix, and $\PP(\sF)$ and $\PP(\Sigma \sF^\vee)$ are classical, irreducible fivefolds which are connected by a flop, but the universal incidence space $\widehat{Z}$ which induces the derived equivalence is not classical.

\subsubsection{Examples satisfying \eqref{eqn:condition:hatZ} but possibly violating \eqref{eqn:condition:irr:hatZ}}

\begin{example}[Symmetric matrices] The examples of generic symmetric matrices of size $k \times k$, where $k \ge 3$, are typical situations where condition \eqref{eqn:condition:hatZ} is satisfied but condition \eqref{eqn:condition:irr:hatZ} is not. We consider a concrete case, where $X = \AA^{6}= \Spec \ZZ[a,b,c,e,d,f]$ is the space of symmetric $3 \times 3$-matrices, and $\sigma \colon \sO_X^{\oplus 3} \to \sO_X^{\oplus 3}$ is the tautological map
	$$\sigma = \left( \begin{array}{ccc} a & b & c \\ b & d & e \\ c & e & f \end{array} \right).$$
Then $\codim_X (X_1) = 1$, $\codim_X(X_2) = 3$, and $\codim_X(X_3) = 6$. Let $Z = X_1$ denote the symmetric determinantal variety. Then both $\PP(\sF)$ and $\PP(\Sigma \sF^\vee)$ are (classical, irreducible) resolutions of $Z$. More concretely, they are both isomorphic to the subscheme $\widetilde{Z}$ of $\AA^6 \times \PP^2 =  \Spec \ZZ[a,b,c,e,d,f] \times \Proj \ZZ[X,Y,Z]$ defined by the equations 
	\begin{align*}
		aX + bY + cZ = 0, \qquad
		bX + dY + e Z = 0, \qquad
		c X + eY + f Z = 0.
	\end{align*}
The universal incidence space $\widehat{Z}$ is classical, with two irreducible components, and isomorphic to the classical fiber product $\widetilde{Z} \times_Z^\cl \widetilde{Z}$. Moreover, by virtue of Theorem \ref{thm:structural}, the diagram $\widetilde{Z} \xleftarrow{r_+} \widetilde{Z} \times_Z^\cl \widetilde{Z} \xrightarrow{r_-}  \widetilde{Z}$ induces auto-equivalences of $\infty$-categories
	$$r_{+,*} \, r_-^* \colon \QCoh(\widetilde{Z} ) \simeq \QCoh(\widetilde{Z} ) \qquad r_{+,*} \, r_-^* \colon \Perf(\widetilde{Z} ) \simeq \Perf(\widetilde{Z} ).$$
By virtue of \cite[Theorem 4.1]{JL18}, the above auto-equivalences are non-trivial, and the composition $(r_{-,*} \, r_+^*) \circ (r_{+,*} \, r_-^*)$ is a (non-trivial) spherical twist in the sense of \cite{ST, AL}.
\end{example}

\begin{example}[Symmetric products of curves, {\cite[\S 5.5]{Tod2}}] Let $C$ be a complex smooth projective curve of genus $g \ge 1$ and let $0 \le \delta \le g-1$ be an integer. Let $C^{(g-1-\delta)}$ and $C^{(g-1+\delta)}$ denote the $(g-1-\delta)$-th and $(g-1+\delta)$-th symmetric products of $C$, respectively. Let $X$ denote Picard variety $\Pic^{g-1+\delta}(C)$ of line bundles of degree ($g-1+\delta$) on $C$. By virtue of \cite[\S 3.1.3]{JL18}, there exists a quasi-coherent sheaf $\sF$ of perfect-amplitude contained in $[0,1]$ and rank $\delta$ over $X$, for which there are canonical identifications 
	$$\pr_+ ={\rm AJ} \colon C^{(g-1+\delta)} = \PP(\sF)  \to X \qquad \pr_- ={\rm AJ}^\vee \colon C^{(g-1-\delta)} = \PP(\Sigma^\vee \sF) \to X,$$
where ${\rm AJ}$ is the Abel--Jacobi map which carries an effective divisor $D$ on $C$ to $\sO(D)$, and ${\rm AJ}^\vee$ is the Serre-dual version of ${\rm AJ}$ which carries $D$ to $\sO(K_C - D)$.
\begin{itemize}
	\item By virtue of \cite[Lemma B.1]{JL18}, the condition \eqref{eqn:condition:hatZ} is always satisfied by the pair $(X,\sF)$, that is, the universal incidence space $\widehat{Z}$ is classical, and is isomorphic to the classical fiber product of $C^{(g-1-\delta)}$ and $C^{(g-1+\delta)}$ over $X$.
	\item Suppose $g \ge 3$. Then by virtue of the proof of \cite[Corollary 5.1]{J19}, we see that the condition \eqref{eqn:condition:irr:hatZ} is satisfied (that is, $\widehat{Z}$ is irreducible) if $C$ is not hyperelliptic. 
\end{itemize}
The converse of the second assertion is almost true in the following sense: 

\begin{itemize}
	\item Suppose $g \ge 3$, then as long as $C$ is hyperelliptic and $0 \le \delta \le g-3$, the condition \eqref{eqn:condition:irr:hatZ} is violated, that is, $\widehat{Z}$ is reducible.
	(To prove this assertion, observe that any complete linear series $g_{g-1-\delta}^1$ on the hyperelliptic curve $C$ is of the form $g_{2}^1 + p_1 + \cdots + p_{g-\delta-3}$, where $p_i \in C(\CC)$. Therefore the Brill--Noether locus $W_{g-1-\delta}^1(C)\neq \emptyset$ and $\dim(W_{g-1-\delta}^1(C))= g-\delta-3$, that is, the degeneracy locus $X_2 \neq \emptyset$ satisfies $\codim_X(X_2)= \delta + 3 < \delta + 2\cdot 2$, which violates \eqref{eqn:condition:irr:hatZ}. See \cite[Chapter IV, \S 3]{ACGH} for the involved notations.) 
\end{itemize}

Since the condition \eqref{eqn:condition:hatZ} is satisfied by all curves $C$, Theorem \ref{thm:structural} in this case reproduces Toda's result \cite[Corollary 5.11]{Tod2} (see also \cite[Corollary 3.10]{JL18}, \cite{BK19}):
	\begin{align*}
	\QCoh(C^{(g-1+\delta)}) &= \big \langle \Phi(\QCoh(C^{(g-1-\delta)})), ~\QCoh(\Pic^{g-1+\delta}(C))(1), \cdots, \QCoh(\Pic^{g-1+\delta}(C))(\delta)\big \rangle, \\
	\Perf(C^{(g-1+\delta)}) &= \big \langle \Phi(\Perf(C^{(g-1-\delta)})), ~\Perf(\Pic^{g-1+\delta}(C))(1), \cdots, \Perf(\Pic^{g-1+\delta}(C)(\delta)\big \rangle
	\end{align*}
where $\Phi = r_{+,*} \, r_-^*$ is the functor induced by the diagram $C^{(g-1+\delta)} \xleftarrow{r_+} \widehat{Z}  \xrightarrow{r_-} C^{(g-1-\delta)}$ of the classical fiber product. Therefore, to understand the $\infty$-category $\QCoh(C^{(g-1+\delta)})$ is equivalent to understand the $\infty$-category $\QCoh(C^{(g-1-\delta)})$, the morphism algebra of the relative exceptional sequence $\{\sO_{C^{(g-1+\delta)}}(i)\}_{1 \le i \le \delta}$, and the morphism spaces from $\QCoh(C^{(g-1-\delta)})$ to $\sO_{C^{(g-1+\delta)}}(i)$. By a result of Lin \cite{Lin}, the $\infty$-category $\QCoh(C^{(g-1-\delta)})$ is indecomposable. On the other hand, the results on Beilinson's relations (Corollary \ref{cor:PVdot:relations}) completely describe the morphism algebra of the part $\{\sO_{C^{(g-1+\delta)}}(i)\}_{1 \le i \le \delta}$: for all $1 \le i, j \le \delta$,
	$$\Mapsp_{\Pic^{g-1+\delta}(C)}(\sO_{C^{(g-1+\delta)}}(i), \sO_{C^{(g-1+\delta)}}(j)) \simeq \Sym^{j-i}_{\Pic^{g-1+\delta}(C)} ( \sF) \simeq S^{j-i}_{\Pic^{g-1+\delta}(C)} (\pi_0(\sF))$$
(the last equivalence follows from Lemma \ref{lem:criterion:Sym:classical}). Moreover, we can describe the dual relative exceptional sequences in terms of the sheaf of relative differentials $\Omega_{{\rm AJ}}^i$ (Proposition \ref{prop:PG:dualexc}).
\end{example}

\begin{remark}Notice that in the situation of \eqref{assumption:CM}, as long as the pair $(X, \sF)$ satisfies condition \eqref{eqn:condition:hatZ}, all the three spaces $\PP(\sF)$, $\PP(\Sigma \sF^\vee)$ and $\widehat{Z}$ are classical. Therefore, the classical results in \cite[Theorem 3.4]{JL18}, \cite[Theorem 6.16]{J21} are already applicable to these cases. Moreover,  \eqref{eqn:condition:hatZ} (or even \eqref{eqn:condition:irr:hatZ}) is satisfied by a generic map $\sigma \colon \sW \to \sV$ if $\sW^\vee \otimes \sV$ is globally generated. The new information we get in these cases from the current paper is the generalized Serre's theorem \ref{thm:Serre:O(d)}, and Beilison's relations (Proposition \ref{prop:PG:dualexc}, Corollary \ref{cor:PVdot:relations}). Furthermore, by virtue of Lemmata  \ref{lem:criterion:P:classical} and \ref{lem:criterion:Sym:classical}, we can formulate the output of these results in terms of classical algebraic geometry (although the  proofs require the machinery of derived algebraic geometry).
\end{remark}

\section{Hecke correspondence moduli as derived projectivizations}\label{sec:Hecke}

Many moduli problems naturally fit into the framework of derived projectivizations.  
We will consider one such example in this section: the Hecke correspondence moduli. To motivate our discussion, we first consider the following two examples:
	
\begin{example}[Hilbert schemes of points on surfaces] \label{eg:Hilbn}
Let $X$ be a smooth complex surface, and for an integer $n \ge 0$, let $\Hilb_n$ denote the moduli space of Hilbert scheme of points on $X$, that is, $\Hilb_n$ parametrizes ideal sheaves $\sI_n \subseteq \sO_X$ for which $\sO_X/\sI_n$ is a zero dimensional sheaf of length $n$. We let $\Hilb_{n,n+1}$ denote the {\em nested} Hilbert scheme which parametrizes the pairs $(\sI_n \supseteq \sI_{n+1})$, where $\sI_n \in \Hilb_n$, $\sI_{n+1} \in \Hilb_{n+1}$, for which there is an exact sequence $0 \to \sI_{n+1} \to \sI_n  \to \CC_{x} \to 0$ for some closed point $x \in X(\CC)$. There is a diagram:
		$$
		\begin{tikzcd}
				& \Hilb_{n,n+1} \ar{ld}[swap]{p_+} \ar{d}{p_X} \ar{rd}{p_-} & \\
				\Hilb_{n+1} & X & \Hilb_{n}
		\end{tikzcd}
		$$		
where $p_\pm$ are the forgetful maps which carry $(\sI_n \supseteq \sI_{n+1})$ to $\sI_{n+1}$ and $\sI_n$, respectively, and $p_X$ is the map which carries $(\sI_n \supseteq \sI_{n+1})$ to $x \in X$, where $\sI_n/\sI_{n+1} \simeq \CC_x$. We let $Z_n \subseteq \Hilb_n \times X$ denote the universal subscheme. From the work of \cite{ES}, we know that
	\begin{itemize}
		\item The projection $\pi_- =(p_- \times p_X) \colon \Hilb_{n,n+1} \to \Hilb_n \times X$ identifies $\Hilb_{n,n+1}$ as the projectivization $\PP(\sI_{Z_n})$, where $\sI_{Z_n}$ is the ideal sheaf of $Z_n$.
		\item The projection $\pi_+= (p_+ \times p_X) \colon \Hilb_{n-1,n} \to \Hilb_n \times X$ identifies $\Hilb_{n-1,n}$ as the projectivization $\PP(\sExt^1(\sI_{Z_n}, \sO_{\Hilb_n \times X}) \simeq \PP(\omega^\cl_{Z_n})$, where $\omega^\cl_{Z_n}$ is the classical dualizing sheaf of the Cohem--Macaulay scheme $Z_n$.
	\end{itemize}
In particular, the situation fits into the picture of the classical version (\cite[Theorem 3.1]{JL18}) of the projectivization formula Theorem \ref{thm:structural}, and there is a semiorthogonal decomposition:
	$$\Dqc(\Hilb_{n,n+1})  = \langle \Dqc(\Hilb_{n-1,n}), ~ \Dqc(\Hilb_{n} \times X) \otimes \sO(1) \rangle.$$
See \cite{BK19} (or \cite[\S 3.1.4]{JL18}, \cite[\S 5.2]{J19}) for more details.
\end{example}

\begin{example}[One-point modifications of Gieseker stable sheaves on surfaces \cite{NegShuffle, NegHecke}] \label{eg:Gieseker} Let $X$ be a smooth complex projective surface with an ample divisor $H \subseteq X$. We fix a pair of invariants $(r, c_1) \in \NN \times H^2(X,\ZZ)$ such that $r >0$ and $\gcd(r, c_1 \cdot H)=1$, and let $\shM$ be the disjoint union $\bigsqcup_{c_2 \ge \frac{r-1}{2r} c_1^2} \shM_{r,c_1,c_2}(X)$, where $\shM_{r,c_1,c_2}(X)$ is the moduli space of Gieseker stable (discrete) sheaves with numerical invariants $(r, c_1, c_2)$. We consider the moduli space $\foZ$ which parametrizes pairs of Gieseker stable sheaves $(\sF' \subset \sF)$ for which $\sF, \sF' \in \shM$ and $\sF/\sF' \simeq \CC_x$, where $x \in X(\CC)$ is a closed point and $\CC_x$ is the skyscraper sheaf supported at $x$. Then the moduli $\foZ$ fits into a similar diagram as in Example \ref{eg:Hilbn},
			$$
			\begin{tikzcd}
				& \foZ \ar{ld}[swap]{p_+} \ar{d}{p_X} \ar{rd}{p_-} & \\
				\shM & X & \shM,
			\end{tikzcd}
			$$
where $p_+, p_-$ and $p_X$ are the forgetful maps which carry $(\sF' \subset \sF)$ to $\sF'$, $\sF$ and $x$, respectively. Similar to Example \ref{eg:Hilbn}, the forgetful maps $\pi_\pm =(p_\pm \times_S p_X)$ are (derived) projectivizations of sheaves of perfect-amplitude in $[0,1]$. See \cite{NegShuffle, NegHecke} for more details.
\end{example}

This section considers a simultaneous generalization of the above two examples. More concretely, let $X \to S$ be a family of derived schemes over a fixed prestack $S$, and let $\foM$ denote the moduli prestack of connective quasi-coherent sheaves on $X$ over $S$.

In \S \ref{sec:Hecke:one-step}, we will define a prestack $\foZ$ of ``one-point modifications of connective sheaves on $X$ over $S$", which parametrizes exact triangles of the form
	$$\sF' \to \sF \to x_* \sL$$
where $\sF', \sF \in \foM$, $\sL$ is a line bundle, and $x \colon S \to X$ is a section of $X \to S$. We show that under mild assumptions ($X \to S$ is separated), the forgetful map $\pi_- \colon \foZ \to \foM \times_S X$ is always a {\em derived projectivization} (Proposition \ref{prop:pi_-:foZ:proj}). Moreover, if we restrict our attention to the moduli of  ``one-point modifications of perfect objects on $X$ over $S$ with a specified Tor-amplitude", and suppose that $X \to S$ is smooth and separated, then both the restrictions of the forgetful functors $\pi_\pm \colon  \foZ \to \foM \times_S X$ are derived projectivizations (Proposition \ref{prop:pi_pm:foZ:proj}). 

In \S \ref{sec:Hecke:general}, we consider the moduli $\foZ_\lambda$ of general ``flags of one-point modifications of sheaves on $X$ over $S$", where $\lambda$ is a set partition which encodes the combinatorial information of the flags. We show that the moduli of general flags also fit into similar derived projectivization pictures as above (Propositions \ref{prop:pi_-:foZ_lambda:proj} and \ref{prop:pi_pm:foZ_lambda:proj}). 

In \S 
\ref{sec:Hecke:surface}, we focus on the case where $X \to S$ is a family of smooth, separated derived surfaces. In this case, the above forgetful functors are derived projectivizations of sheaves of perfect-amplitude in $[0,1]$. In particular, we could apply all the results of this paper to this situation; see Corollary \ref{cor:Hecke:surface} for more detail.

\begin{remark}
This section's exposition and notations closely follow Negu{\c{t}}'s works \cite{NegW, NegHecke, NegShuffle} which study the shuffle algebra structures and Hecke correspondences for the moduli of Gieseker stable sheaves on surfaces; see also \cite{PZ20, Z21} for more results in this direction. We expect the results of this paper to be helpful for the study of similar algebraic structures for the general prestacks of perfect objects of Tor-amplitude $[0,1]$ on a surface $X$ or a family $X \to S$ of surfaces. Furthermore, the general framework of this section allows us to consider the similar structures on the moduli spaces with other types of stability conditions. For example, it allows us to study algebraic structures on the moduli spaces $\foM_{\sigma}(X/S)$ of Bridgeland (semi-)stable objects on a family $X \to S$ of surfaces, where $X \to S$ is equipped with (a family of) Bridgeland stability conditions $\sigma$ in the sense of 
 \cite{Bri07, Bri08, AB12, BM11, BM14, BLM+}. We wish to explore these aspects in the future.
\end{remark}

\subsection{Hecke correspondences: the case of one-step flags}  \label{sec:Hecke:one-step}
Let $X$ be a prestack over the fixed base prestack $S$. In this subsection, we construct moduli functors $\foZ$ (resp. $\foZ^{\perf}$) of  ``one-step flag of one-point modifications of connective quasi-coherent sheaves (resp. perfect coherent sheaves)" on $X$ and study their basic properties. 

\begin{construction} \label{construction:foZ} Let $X$ be a prestack over $S$. Consider the following constructions:
	\begin{enumerate}[leftmargin=*]
		\item We let $\foM$ denote the functor $\CAlgDelta \to \shS$ over $S$ which carries each map $\eta \colon T = \Spec R \to S$ to the space $\foM(\eta) =\QCoh(X \times_S T)^{\cn, \,\simeq}$. We let $\sU \in \QCoh(X \times_S \foM)^\cn$ denote the universal connective quasi-coherent sheaf on $X \times_S \foM$. 
		\item We let $\foZ \colon \CAlgDelta \to \shS$ denote the moduli functor over $S$ which carries each map $\eta \colon T = \Spec R \to S$ to the full subspace of 
	$$\Fun(\Delta^1, \QCoh(X \times_S T)^\cn)^{\simeq} \times \Maps_{/T}(T, X \times_S T) \times \Pic(T)^{\simeq}$$
spanned by those elements 
	$$((\sF' \xrightarrow{\phi} \sF); \, x \colon T  \to X \times_S T ; \, \sL \in \Pic(T))$$
for which there is a fiber sequence 
	$$\sF' \xrightarrow{\phi} \sF \to x_* \sL$$
in $\QCoh(X \times_S T)$. (Here, $\Maps_{/T}(T, X \times_S T)$ denotes the space of sections $x \colon T \to X \times_S T$ of the natural projection $X \times_S T \to T$.) We could informally think of $\foZ$ as the ``moduli of one-step flag of one-point modifications of connective quasi-coherent sheaves on $X$ over $S$". We denote the universal element of $\foZ$ by 
	$$((\sU' \xrightarrow{\phi_{\rm univ}} \sU); \, x_{\rm univ};  \, \sL_{\rm univ}),$$
where $\phi_{\rm univ} \colon \sU' \to \sU$ is a morphism in $\QCoh(X \times_S \foZ)^{\cn}$, $x_{\rm univ} \colon \foZ \to X \times_S \foZ$ is a section of the  projection $X \times_S \foZ \to \foZ$, $\sL_{\rm univ} \in \Pic(X \times_S \foZ)$ is a line bundle, and for which there is a universal fiber sequence in $\QCoh(X \times_S \foZ)$:
		$$\sU' \xrightarrow{\phi_{\rm univ}} \sU \to (x_{\rm univ})_* \sL_{\rm univ}.$$
		\item We consider the following natural forgetful maps
		$$
		\begin{tikzcd}
				& \foZ \ar{ld}[swap]{p_+} \ar{d}{p_X} \ar{rd}{p_-} & \\
				\foM & X & \foM,
		\end{tikzcd}
		$$		
		where $p_+$, $p_-$, and $p_X$ are the respective forgetful maps which carry each element
			$$\zeta = ((\sF' \xrightarrow{\phi} \sF); \, x \colon T  \to X \times_S T ; \, \sL \in \Pic(T)) \in \foZ(\eta),$$
		where $\eta \colon T=\Spec R \to S$, to the elements $p_+(\zeta) = \sF' \in \foM(\eta)$, $p_-(\zeta) =\sF \in \foM(\eta)$, and the composition $p_X(\zeta) = (T \xrightarrow{x} X \times_S T \xrightarrow{\pr_X} X) \in X(\eta)$, respectively.
		
		\item We consider the following maps:
		\begin{align*}
			\pi_+ = (p_+ \times_S p_X) \colon \foZ \to \foM \times_S X, \qquad
			\pi_- = (p_- \times_S p_X) \colon \foZ \to \foM \times_S X.
		\end{align*}
(Therefore $\pi_+$ is the map which carries the element $\zeta=((\sF' \xrightarrow{\phi} \sF); x; \sL )$ to $(\sF', \pr_X \circ x)$, and $\pi_+$ is the map which carries $\zeta=((\sF' \xrightarrow{\phi} \sF); \, x ; \, \sL)$ to $(\sF, \pr_X \circ x)$.)

	\end{enumerate}
\end{construction}

The following result relates the functors $\foZ$ with derived projectivizations.

\begin{proposition}[Compare with {\cite[Definition 2.5, Proposition 2.8]{NegShuffle}}] \label{prop:pi_-:foZ:proj} In the situation of Construction \ref{construction:foZ}, suppose that the map $X \to S$ is separated and a relative derived scheme, then the projection map $\pi_- \colon \foZ \to \foM \times_S X$ identifies $\foZ$ with the derived projectivization $\PP_{\foM \times_S X} (\sU)$ and for which there is a canonical identification $\sO(1) \simeq \sL_{\rm univ}$.
\end{proposition}

\begin{proof} For any map $\eta \colon T = \Spec R \to S$, where $R \in \CAlgDelta$,  and a map $\widetilde{\eta} = (f \times_S g) \colon T \to \foM \times_S X$ over $\eta$, we set $\sF = (\id_X \times_S f)^* (\sU) \in \QCoh(X \times_S T)^\cn$ and $x = (g \times_S \id_T) \colon T \to X \times_S T$, so that $\widetilde{\eta}^* \,\sU \simeq x^* \sF$. Then the space of liftings $T \to \foZ$ of $\widetilde{\eta}$ is the space of pairs $( (\sF' \xrightarrow{\phi} \sF) \in \QCoh(X \times_S T)^\cn_{/\sF}, \sL \in \Pic(T))$ for which there is a fiber sequence $\sF' \xrightarrow{\phi} \sF \to x_* \sL$. Such a space is homotopy equivalent to the space of pairs $(\sL \in \Pic(T), (\sF \xrightarrow{\varphi} x_* \sL) \in \QCoh(X \times_S T)_{\sF/})$ for which $\varphi$ is surjective on $\pi_0$ (and a homotopy inverse map is given by setting $\sF'$ to be the fiber of $\varphi$). We claim that the above space of pairs $(\sL, \varphi)$ is homotopy equivalent to the space of pairs $(\sL \in \Pic(T), x^*\sF \xrightarrow{\rho} \sL)$ for which $\rho$ is surjective on $\pi_0$, where $\rho$ is the map induced by $\varphi$ from the adjunction 
	$$\Map_{\QCoh(X \times_S T)}(\sF, x_* \sL) \simeq \Map_{\QCoh(T)}(x^* \sF, \sL).$$	
In other words, the space of liftings $T \to \foZ$ of $\widetilde{\eta}$ along $\pi_-$ is canonical homotopy equivalent to the space $\PP_{\foM \times_S X}(\sU)(\widetilde{\eta})$ over $\widetilde{\eta}$, where $\PP_{\foM \times_S X}(\sU)$ is the derived projectivization (\S \ref{sec:def:dproj:qcoh}), which will prove the desired assertion. 

Let $\rho \colon x^*\sF \to \sL$ be a morphism which corresponds to $\varphi \colon \sF \to x_* \sL$ through the above adjunction. Then to prove the above claim, it suffices to show the following:
	\begin{enumerate}[label=$(*)$, ref=$*$]
	\item \label{prop:pi_-:foZ:proj:assertion*} $\rho$ is surjective on $\pi_0$ if and only if $\varphi$ is surjective on $\pi_0$.
	\end{enumerate}
We now prove \eqref{prop:pi_-:foZ:proj:assertion*}. First,  if $\varphi$ is surjective on $\pi_0$, then the counit map $x^* x_* \sL \to \sL$ is surjective on $\pi_0$ (since $x$ is an affine morphism), and $x^*\varphi$ is surjective on $\pi_0$ (since the property ``surjection on $\pi_0$" is stable under pullback). Therefore, we obtain that the composition $\rho \colon x^* \sF \xrightarrow{x^* \varphi} x^* x_* \sL \to \sL$ is surjective on $\pi_0$. On the other hand, if $\rho$ is surjective on $\pi_0$, then $x_* \rho$ is surjective on $\pi_0$ (since $x$ is affine). To show $\varphi \colon \sF \to x_*x^* \sF \xrightarrow{x_* \rho} x_* \sL$ is surjective on $\pi_0$, it suffices to show that the unit map $\sF \to x_* x^* \sF$ is surjective on $\pi_0$ for all closed immersions $x$ and connective quasi-coherent sheaves $\sF$. As this assertion is local, it suffices to show that if $A \to B$ is map of simplicial commutative rings which induces a surjective map of discrete commutative rings $\pi_0(A) \twoheadrightarrow\pi_0(B)$, and $M$ is a connective $A$-module, then the natural map $M \to M \otimes_A B$ induces an epimorphism $\pi_0(M) \twoheadrightarrow \pi_0(M \otimes_A B)$. But this is evident by virtue of  $\pi_0(M \otimes_A B) \simeq \Tor_0^{\pi_0(A)} (\pi_0(M),  \pi_0(B))$ (see \cite[Corollary 7.2.1.23]{HA}).
\end{proof}

Next, we restrict our attention to the ``one-step flag of one-point modifications" of quasi-coherent sheaves with a specified range of perfect-amplitude.

\begin{notation}\label{notation:foZ:perf}
Suppose that  the map of prestacks $X \to S$ is smooth, separated and a relative derived scheme of constant relative dimension $d \ge 1$. Then the smooth morphism $X \to S$ is equipped with a relative canonical line bundle given by the formula $\sK_{X/S} = \bigwedge^d L_{X/S}$, where $L_{X/S}$ is the relative cotangent complex (Definition \ref{def:relativecotangent}). In this situation:
	\begin{enumerate}[leftmargin=*]
		\item  For any prestack $Y$, we let $\Perf(Y)_{[0,d-1]}$ denote the full subcategory of $\QCoh(Y)$ spanned by objects of perfect-amplitude contained in $[0,d-1]$ (that is, spanned by connective perfect objects of Tor-amplitude $\le d-1$). 
		\item We let $\foM^{\perf}_{[0,d-1]}$ denote the subfunctor of $\foM$ such that for each map $\eta \colon T = \Spec R \to S$, $\foM^\perf_{[0,d-1]}(\eta) = \Perf(X \times_S T)_{[0,d-1]}^{\simeq}$ is the full subcategory of $\foM(\eta)= \QCoh(X \times_S T)^{\cn,\,\simeq}$ spanned by objects of perfect-amplitude contained in $[0,d-1]$. By abuse of notation, we also let $\sU \in \Perf(X \times_S \foM^{\perf}_{[0,d-1]})_{[0,d-1]}$ denote the universal perfect object.	
		\item Similarly, we let $\foZ^\perf_{[0,d-1]}$ denote the subfunctor of $\foZ$ such that for each map $\eta \colon T = \Spec R \to S$, $\foZ^\perf_{[0,d-1]}(\eta)$ is the full subcategory of $\foZ(\eta)$ spanned by those
		$$\zeta = ((\sF' \xrightarrow{\phi} \sF); \, x \colon T  \to X \times _ST ; \, \sL \in \Pic(T)) \in \foZ(\eta)$$
for which $\sF, \sF' \in \foM^\perf_{[0,d-1]}(T)$ (that is, $\sF$ and $\sF'$ are of perfect-amplitude in $[0,d-1]$). By abuse of notation, we use the same notation as Construction \ref{construction:foZ},
	$$((\sU' \xrightarrow{\phi_{\rm univ}} \sU); \, x_{\rm univ};  \, \sL_{\rm univ}),$$
to denote the universal element of $\foZ^\perf_{[0,d-1]}$, where $\sU', \sU \in \Perf(\foZ^\perf_{[0,d-1]} \times_S X)_{[0,d-1]}$.	
\end{enumerate}
\end{notation}

The following result shows that when restricted to perfect objects with a specified range of Tor-amplitudes, the projection $\pi_+$ also fits in the picture of derived projectivizations. 

\begin{proposition}[Compare with {\cite[Proposition 2.8]{NegShuffle}}] \label{prop:pi_pm:foZ:proj} 
In the situation of Notation \ref{notation:foZ:perf} and Construction \ref{construction:foZ}, the following is true:
	\begin{enumerate}
		\item  \label{prop:pi_pm:foZ:proj-1} The restriction $\pi_-|_{\foZ^{\perf}_{[0,d-1]}}  \colon \foZ^{\perf}_{[0,d-1]}  \to \foM^\perf_{[0,d-1]} \times_S X$ of the projection map $\pi_-$ identifies $\foZ^{\perf}_{[0,d-1]}$ with the derived projectivization $\PP_{\foM^\perf_{[0,d-1]} \times_S X} (\sU)$, and for which there is a canonical identification of line bundles $\sO(1) \simeq \sL_{\rm univ}$.
		\item \label{prop:pi_pm:foZ:proj-2} The restriction $\pi_+|_{\foZ^{\perf}_{[0,d-1]}} \colon \foZ^\perf_{[0,d-1]} \to \foM^\perf_{[0,d-1]} \times_S X$ of the projection map $\pi_+$ identifies $\foZ^\perf_{[0,d-1]} $ with the derived projectivization 
	$\PP_{\foM^{\perf}_{[0,d-1]} \times_S X} (\Sigma^{d-1} \, \sU^\vee \otimes \sK_{X/S})$, and for which there is a canonical identification of line bundles $\sO(1) \simeq \sL_{\rm univ}^\vee$.
	\end{enumerate}
\end{proposition}

\noindent {\em Notation.} In the above assertion \eqref{prop:pi_pm:foZ:proj-2} and for the rest of this section, by abuse of notations, we shall use the same notation $\sK_{X/S}$ to denote the line bundle $f^* \sK_{X/S}$ for any given map of prestacks $f \colon Y \to X$. For example, in the expression of \eqref{prop:pi_pm:foZ:proj-2}, $\sU^\vee \otimes \sK_{X/S}$ denotes the object $\sU^\vee \otimes \pr_2^* \sK_{X/S}$, where $\pr_2 \colon \foM^{\perf}_{[0,d-1]} \times_S X \to X$ is the natural projection.

\begin{proof} For any map $\eta \colon T  = \Spec R \to S$, $R \in \CAlgDelta$, since $X \to S$ is a smooth, separated relative derived scheme, any section $x \colon T \to X \times_S T$ of the projection $X \times_S T \to T$ is a quasi-smooth closed immersion. In particular, $x_* \colon \QCoh(T) \to \QCoh(X \times T)$ carries perfect complexes to perfect complexes (Theorem \ref{thm:Neeman-Lipman-Lurie} \eqref{thm:Neeman-Lipman-Lurie-3}). Let $x_!$ denote the left adjoint of the pullback functor $x^*$, then there is a canonical equivalence $x_!(\sE) \simeq \Sigma^{-d} x_*(\sE \otimes \sK_{X/S}^\vee)$ for all $\sE \in \QCoh(X \times_S T)$ (Theorem \ref{thm:Neeman-Lipman-Lurie} \eqref{thm:Neeman-Lipman-Lurie-2vii}). In particular, for any $\sL \in \Pic(T)$, we have canonical equivalences $(x_* \sL)^\vee \simeq x_!(\sL^\vee) \simeq \Sigma^{-d} x_*(\sL^\vee \otimes \sK_{X/S}^\vee)$ (Theorem \ref{thm:Neeman-Lipman-Lurie} \eqref{thm:Neeman-Lipman-Lurie-2iv}), so that $x_* \sL$ and $\Sigma^d \, (x_* \sL)^\vee$ both have perfect-amplitude contained in $[0,d]$. 

We first prove assertion \eqref{prop:pi_pm:foZ:proj-1}. By virtue of Proposition \ref{prop:pi_-:foZ:proj}, we only need to show for any $\sF \in \Perf(X \times_S T)_{[0, d-1]}$, $\sL \in \Pic(T)$, and a map $\varphi \colon \sF \to x_* \sL$ which is surjective on $\pi_0$, the fiber of $\varphi$ is of perfect-amplitude in $[0, d-1]$. Let $\sF'  = \fib(\varphi)$, then $\sF'$ is clearly a connective perfect object. On the other hand, from the fiber sequence $\Sigma^{d-1} \sF^\vee \to \Sigma^{d-1} \sF'^\vee \to \Sigma^d (x_* \sL)^\vee$ we obtain that $\Sigma^{d-1} \sF'^\vee$ is a connective perfect object. Therefore $\sF'$ is of perfect-amplitude contained in $[0, d-1]$. This proves assertion \eqref{prop:pi_pm:foZ:proj-1}.

We next prove assertion \eqref{prop:pi_pm:foZ:proj-2}. For any given map $\widetilde{\eta} = (f' \times_S g) \colon T \to \foM^{\perf}_{[0,d-1]} \times_S X$, we set $\sF' = (\id_X \times_S f')^* (\sU) \in \Perf(X \times_S T)^\cn_{[0,d-1]}$ and $x = (g \times_S \id_T) \colon T \to X \times_S T$, so that $\widetilde{\eta}^* \,\sU \simeq x^* \sF'$. Then the space of liftings of $\widetilde{\eta}$ along $\pi_+|_{\foZ^{\perf}_{[0,d-1]}}$ is canonically identified with the space of pairs
	$$((\sF' \xrightarrow{\phi} \sF) \in (\Perf(X \times_S T)_{[0,d-1]})_{\sF'/}, \sL \in \Pic(T))$$
for which there is a fiber sequence $\sF' \xrightarrow{\phi} \sF \to x_* \sL$. This space is homotopy equivalent to the space of pairs $(\sL \in \Pic(T), x_* \sL \xrightarrow{\psi} \Sigma \sF')$ for which ${\fib}(\psi)$ is of perfect-amplitude in $[0, d-1]$ (and a homotopy inverse of the above map is given by setting $\sF = {\rm fib}(\psi)$). There are canonical homotopy equivalences of spaces
	\begin{align*}
	&\Map_{\QCoh(X \times_S T)}(x_*\sL, \Sigma \sF') \\
	&\simeq \Map_{\QCoh(X \times_S T)}(\Sigma^{-1} \, \sF'^\vee, (x_*\sL)^\vee) \\
	&\simeq \Map_{\QCoh(X \times_S T)}(\Sigma^{-1} \, \sF'^\vee, x_! (\sL^\vee))  \\		&\simeq \Map_{\QCoh(X \times_S T)}(\Sigma^{-1} \, \sF'^\vee, \Sigma^{-d} x_* (\sL^\vee \otimes \sK_{X/S}^\vee))  \\
	&\simeq  \Map_{\QCoh(T)}(x^*(\Sigma^{d-1} \, \sF'^\vee \otimes \sK_{X/S}), \sL^\vee).
	\end{align*}
Let $\rho \colon x^*(\Sigma^{d-1} \, \sF'^\vee \otimes \sK_{X/S}) \to \sL^\vee$ denote the map which corresponds to $\psi \colon x_*\sL \to \Sigma \, \sF'$ via the above equivalences. To prove the assertion \eqref{prop:pi_pm:foZ:proj-2}, it remains to show 
	\begin{enumerate}[label=$(*')$, ref=$*'$]
	\item \label{proof:prop:pi_pm:foZ:proj:*'}  $\rho$ is surjective on $\pi_0$ if and only if $\fib(\psi)$ is of perfect-amplitude in $[0, d-1]$.
	\end{enumerate}
 On the one hand, the assertion \eqref{prop:pi_-:foZ:proj:assertion*} of Proposition \ref{prop:pi_-:foZ:proj} implies that
	\begin{enumerate}[label=$(**')$, ref=$**'$]
	\item \label{proof:prop:pi_pm:foZ:proj:**'} $\rho$ is surjective on $\pi_0$ if and only if $\Sigma^d \psi^\vee \colon \Sigma^{d-1} \sF'^\vee \to \Sigma^d(x_*\sL)^\vee$ is surjective on $\pi_0$.
	\end{enumerate}
We now deduce assertion \eqref{proof:prop:pi_pm:foZ:proj:*'} from \eqref{proof:prop:pi_pm:foZ:proj:**'} as follows. If we assume $\sF : =\fib(\psi)$ is of perfect-amplitude in $[0,d-1]$, then $\fib(\Sigma^d \psi^\vee) \simeq \Sigma^{d-1} \sF^\vee$ is connective, which implies that $\rho$ is surjective on $\pi_0$, by virtue of assertion \eqref{proof:prop:pi_pm:foZ:proj:**'}. On the other hand, if $\rho$ is surjective on $\pi_0$, then $\sM: = \fib(\Sigma^d \psi^\vee)$ is connective by virtue of assertion \eqref{proof:prop:pi_pm:foZ:proj:**'}. Let $\sF:=\fib(\psi)$, then from the fiber sequence $\sF' \to \sF \to x_* \sL$ we obtain that $\sF$ is connective and perfect. By virtue of the equivalence $\sM \simeq \Sigma^{d-1} \sF^\vee$ again and the connectivity of $\sM$, we obtain $\sF$ is of perfect-amplitude contained in $[0,d-1]$. This proves assertion \eqref{proof:prop:pi_pm:foZ:proj:*'}.
\end{proof}

\subsection{Hecke correspondences: the case of general flags} \label{sec:Hecke:general} 
Let $X \to S$ be a map of prestacks. In this section, we generalize the constructions and results of the preceding \S\ref{sec:Hecke:one-step} to the case of general ``flags of one-point modifications" of quasi-coherent sheaves on $X$ over $S$. 

\begin{construction}[Compare with {\cite[Definition 2.13]{NegHecke}}]\label{construction:foZ_lambda} Let $X$ be a prestack over $S$. A {\em set partition $\lambda$ of size $n$} is an equivalence relation on the set $\{1, 2, \ldots, n\}$. Let $\lambda$ a set partition of size $n \ge 1$, and we let $\foZ_\lambda \colon \CAlgDelta \to \shS$ denote the functor which carries each map $\eta \colon T = \Spec R \to S$ to the full subspace of 
	$$\Fun(\Delta^n, \QCoh(X \times_S T)^\cn)^{\simeq} \times \Maps_{/T}(T, X \times_S T)^{\times n} \times (\Pic(T)^{\simeq})^{\times n}$$
spanned by those elements 
	$$\zeta = ((\sF_0 \xrightarrow{\phi_1} \sF_1 \xrightarrow{\phi_2} \cdots \xrightarrow{\phi_n} \sF_n), \{x_i \colon T  \to X \times_S T\}_{1 \le i \le n}, \{\sL_i\}_{1 \le i \le n})$$
 for which the following conditions are satisfied:
	\begin{itemize}
			\item The sections satisfy $x_i = x_j$ if $i \sim j$ in $\lambda$; 
			\item There are equivalences ${\rm cofib}(\sF_{i-1} \xrightarrow{\phi_i} \sF_{i}) \simeq (x_i)_* \sL_i$.
	\end{itemize}
We denote the universal element of $\foZ_\lambda$ by 
	$$\zeta_{\rm univ} =((\sU_0 \xrightarrow{\phi_{1, \rm univ}} \sU_1 \xrightarrow{\phi_{2, \rm univ}}  \cdots \xrightarrow{\phi_{n, \rm univ}} \sU_n), \{x_{i, \rm univ}\}_{1 \le i \le n}, \{\sL_{i, \rm univ}\}_{1 \le i \le n}),$$
where $\sU_i \in \QCoh(X \times_S \foZ_\lambda)^\cn$, $x_{i, \rm univ} \in \Map_{/\foZ_\lambda}(\foZ_\lambda, X \times_S \foZ_\lambda)$, $\sL_{i, \rm univ} \in \Pic(\foZ_\lambda)$. 
\end{construction}

Notice that if $n=1$, then $\foZ_\lambda = \foZ$. By convention, we set $\foZ_\lambda = \foM$ if $n=0$.

\begin{variant} In the situation of Construction \ref{construction:foZ_lambda}, we let $\widetilde{\foZ}_\lambda$ be the functor which carries each map $\eta \colon T = \Spec R \to S$ to the space of elements
	$$\widetilde{\zeta} = (\zeta,  \{\varphi_i\}_{1 \le i \le n})$$
where
	$$\zeta = ((\sF_0 \xrightarrow{\phi_1} \sF_1 \xrightarrow{\phi_2} \cdots \to \sF_n), \{x_i \colon T  \to X \times_S T\}_{1 \le i \le n}, \{\sL_i\}_{1 \le i \le n}) \in \foZ_\lambda(\eta),$$
	$$(\varphi_i \colon \sF_i \to (x_i)_* \sL_i) \in \Fun(\Delta^1, \QCoh(X \times_S T))^{\simeq},$$
and for which the following conditions are satisfied:
	\begin{itemize}
		\item The sections satisfy $x_i = x_j$ if $i \sim j$ in $\lambda$;
		\item The sequences $\sF_{i-1}\xrightarrow{\phi_i} \sF_i \xrightarrow{\varphi_i} (x_i)_* \sL_i$ are fiber sequences for all $1 \le i \le n$. 
	\end{itemize}
Then by virtue of \cite[Remark 1.1.1.7]{HA}, the natural forgetful map $\widetilde{\foZ}_\lambda(\eta) \to \foZ_\lambda(\eta)$ which carries $\widetilde{\zeta}=(\zeta, \{\varphi_i\}_{1 \le i \le n})$ to $\zeta$ is a trivial Kan fibration. Therefore, we obtain that the forgetful map $\widetilde{\foZ}_\lambda \to \foZ_\lambda$ is a natural isomorphism of functors. From now on, we will fix an inverse of the isomorphism $\widetilde{\foZ}_\lambda \to \foZ_\lambda$. Unwinding the definitions, this means that we will fix a choice of fiber sequences of the universal quasi-coherent sheaves
	$$\sU_{i-1} \xrightarrow{\phi_{i, \rm univ}} \sU_i \xrightarrow{\varphi_{i, \rm univ}} (x_{i, \rm univ})_* ( \sL_{i, \rm univ})$$
in $\QCoh(X \times_S \foZ_\lambda)$ for all $1 \le i \le n$.
\end{variant}

We also consider the natural forgetful maps between the functors of the form $\foZ_\lambda$:

\begin{notation}[{Compare with \cite[\S 2.14]{NegHecke}}] \label{notation:foZ_lambda:pi}
In the situation of Construction \ref{construction:foZ_lambda}:
	\begin{enumerate}[leftmargin=*]
	\item Let $\lambda$ be a set partition of size $n \ge 1$. We let $\lambda|$ denote the set partition of $\{1, \ldots, n-1\}$ which is obtained by dropping the last element of $\lambda$, and let $|\lambda$ denote the set partition of $\{2, \ldots, n\}$ which is obtained by dropping the first element of $\lambda$. 
	\item  We consider the following maps:
		$$
		\begin{tikzcd}
				& \foZ_\lambda \ar{ld}[swap]{p_+} \ar{d}{p_{X,i}} \ar{rd}{p_-} & \\
				\foZ_{\lambda|} & X & \foZ_{|\lambda},
		\end{tikzcd}
		$$		
		where $p_+$, $p_-$, and $p_X$ are the forgetful maps which carry each element
		$$\zeta = ((\sF_0 \xrightarrow{\phi_1} \sF_1 \xrightarrow{\phi_2} \cdots \xrightarrow{\phi_n} \sF_n), \{x_i \colon T  \to X \times_S T\}_{1 \le i \le n}, \{\sL_i\}_{1 \le i \le n}) \in \foZ_\lambda(\eta),$$
	 where $\eta \colon T = \Spec R \to S$ is any given map, to the elements
		\begin{align*}
		p_+(\zeta) &= ((\sF_0 \xrightarrow{\phi_1} \sF_1 \to  \cdots \xrightarrow{\phi_{n-1}} \sF_{n-1}), \{x_i \colon T  \to X \times_S T\}_{1 \le i \le n-1}, \{\sL_i\}_{1 \le i \le n-1}) \in \foZ_{\lambda|}(\eta), \\
		p_-(\zeta) &= ((\sF_1 \xrightarrow{\phi_2} \sF_2 \to \cdots \xrightarrow{\phi_n}  \sF_{n}), \{x_i \colon T  \to X \times_S T\}_{2 \le i \le n}, \{\sL_i\}_{2 \le i \le n}) \in \foZ_{|\lambda}(\eta),
		\end{align*}
and the composition $p_{X,i}(\zeta) = (T \xrightarrow{x_i} T \times_S X \xrightarrow{\pr_X} X) \in X(\eta)$, respectively. 

	\item We define a forgetful map $\pi_-$ with source $\foZ_\lambda$ by the following rules: we let
	\begin{itemize}[leftmargin=*]
		 \item $\pi_- = (p_- \times_S p_{X,1}) \colon \foZ_\lambda \to \foZ_{|\lambda} \times_S X$, if $1$ is not equivalent to any $i \in \{2, \ldots, n\}$ in $\lambda$.
		 \item $\pi_-  = p_- \colon \foZ_\lambda \to \foZ_{|\lambda}$, if $1$ is equivalent to some $i  \in \{2, \ldots, n\}$ in $\lambda$.
	\end{itemize}
Similarly, we define a forgetful map $\pi_+$ with source $\foZ_\lambda$ by the following rules:	we let
	\begin{itemize}[leftmargin=*]
		 \item $\pi_+ = (p_+ \times_S p_{X,n}) \colon \foZ_\lambda \to \foZ_{\lambda|} \times_S X$, if $n$ is not equivalent to any $i \in \{1, \ldots, n-1\}$ in $\lambda$.
		 \item $\pi_+  = p_+ \colon \foZ_\lambda \to \foZ_{\lambda|}$, if $1$ is equivalent to some $i  \in \{1, \ldots, n-1\}$ in $\lambda$.
	\end{itemize}
\end{enumerate}
\end{notation}

The following result is a direct generalization of Proposition \ref{prop:pi_-:foZ:proj}:

\begin{proposition}[Compare with {\cite[Proposition 2.16]{NegHecke})}] \label{prop:pi_-:foZ_lambda:proj} 
In the situation of Construction \ref{construction:foZ_lambda} and Notation \ref{notation:foZ_lambda:pi}, and suppose that the map of prestacks $X \to S$ is separated and a relative derived scheme.
	\begin{enumerate}
			\item \label{prop:pi_-:foZ_lambda:proj-1} 
			If $1$ is not equivalent to any $i \in \{2, \ldots, n\}$ in $\lambda$, then the map $\pi_- \colon \foZ_\lambda \to \foZ_{|\lambda} \times_S X$ identifies $\foZ_\lambda$ with the derived projectivization $\PP_{\foZ_{|\lambda} \times_S X}(\sU_1)$, and for which there is a canonical identification $\sO(1) \simeq \sL_{1, \rm univ}$.
			\item \label{prop:pi_-:foZ_lambda:proj-2} 
			If $1$ is equivalent to some $i  \in \{2, \ldots, n\}$ in $\lambda$, then the map $\pi_-  \colon \foZ_\lambda \to \foZ_{|\lambda}$ identifies $\foZ_\lambda$ with the derived projectivization $\PP_{\foZ_{|\lambda}}(x_{i, \mathrm{univ}}^*(\sU_1))$, and for which there is a canonical identification $\sO(1) \simeq \sL_{1, \rm univ}$.
		\end{enumerate}
\end{proposition}
\begin{proof} The assertion \eqref{prop:pi_-:foZ_lambda:proj-1} can be proved in the same way as Proposition \ref{prop:pi_-:foZ:proj}. To prove assertion  \eqref{prop:pi_-:foZ_lambda:proj-2}, assume that $1 \sim i$ in $\lambda$, where $2 \le i \le n$, so that $x_1 = x_i$, then for any $\eta \colon T = \Spec R \to S$, the fiber of $\pi_-$ over a point
	$$\zeta = ((\sF_1 \xrightarrow{\phi_2} \cdots \xrightarrow{\phi_n} \sF_n), \{x_2 \colon T  \to X \times_S T\}_{2 \le i \le n}, \{\sL_2\}_{1 \le 2 \le n}) \in \foZ_{|\lambda}(\eta)$$
is canonically identified with the space of pairs $(\sF_1 \to x_{1*}\,\sL_1, \sL_1 \in \Pic(T))$ for which the map $\sF_1 \to x_{1*}\, \sL_1$ is surjective on $\pi_0$. By virtue of the adjunction
	$$\Mapsp_{\QCoh(X \times_S T)}(\sF_1, x_{1\,*} \sL_1) \simeq \Mapsp_{\QCoh(T)}(x_i^* \sF_1, \sL_1)$$
and the same argument of the proof of Proposition \ref{prop:pi_-:foZ:proj}, the fiber of $\pi_-$ over $\zeta$ is canonically homotopy equivalent to $\PP_{\foZ_{|\lambda}}(x_i^* \sU_1)(\eta)$, which proves the desired assertion.
\end{proof}

Next, we focus our attention to general ``flags of one-point modifications" of quasi-coherent sheaves with a specified range of perfect-amplitude. 

\begin{notation} \label{notation:foZ_lambda:perf} In the situation of Construction \ref{construction:foZ_lambda} and Notation \ref{notation:foZ:perf}, we let $(\foZ_\lambda)^\perf_{[0,d-1]}$ denote the subfunctor of $\foZ_\lambda$ such that, for each $\eta \colon T = \Spec R \to S$, $(\foZ_\lambda)^\perf_{[0,d-1]}(\eta)$ is the full subcategory of $\foZ_\lambda(\eta)$ spanned by those elements
	$$\zeta=((\sF_0 \xrightarrow{\phi_1} \sF_1 \xrightarrow{\phi_2} \cdots \xrightarrow{\phi_n} \sF_n), \{x_i \colon T  \to X \times_S T\}_{1 \le i \le n}, \{\sL_i\}_{1 \le i \le n}) \in \foZ_\lambda(\eta)$$
for which $\sF_i \in \foM^\perf_{[0,d-1]}(\eta) = \Perf(X \times_S T)_{[0, d-1]}^{\simeq}$, $0 \le i \le n$. By abuse of notation, we denote the universal element of $(\foZ_\lambda)^\perf_{[0,d-1]}$ by the same notation as Construction \ref{construction:foZ_lambda}:
	$$\zeta_{\rm univ} =((\sU_0 \xrightarrow{\phi_{1, \rm univ}} \sU_1 \xrightarrow{\phi_{2, \rm univ}}  \cdots \xrightarrow{\phi_{n, \rm univ}} \sU_n), \{x_{i, \rm univ}\}_{1 \le i \le n}, \{\sL_{i, \rm univ}\}_{1 \le i \le n}),$$
where $\sU_i \in \Perf(X \times_S (\foZ_\lambda)^\perf_{[0,d-1]})_{[0,d-1]}$, $x_{i, \rm univ}$ are sections of the projection $X \times_S (\foZ_\lambda)^\perf_{[0,d-1]} \to (\foZ_\lambda)^\perf_{[0,d-1]}$, $\sL_{i, \rm univ} \in \Pic((\foZ_\lambda)^\perf_{[0,d-1]})$. Then there are fiber sequences on $(\foZ_\lambda)^\perf_{[0,d-1]} \times_S X$:
	$$\sU_{i-1} \xrightarrow{\phi_{i, \rm univ}} \sU_i \xrightarrow{\varphi_{i, \rm univ}} (x_{i, \rm univ})_* \sL_{i, \rm univ}.$$
\end{notation}

The following is a generalization of Proposition \ref{prop:pi_pm:foZ:proj} and \cite[Proposition 2.16]{NegHecke}:

\begin{proposition}\label{prop:pi_pm:foZ_lambda:proj} 
Suppose that $X \to S$ is map of prestacks which is smooth, separated, and a relative derived scheme of constant relative dimension $d \ge 1$. We let $\sK_{X/S}$ denote the relative canonical line bundle of $X \to S$ and let $\lambda$ be a set partition of size $n \ge 2$. Consider the moduli functors and maps defined in Notations \ref{notation:foZ_lambda:pi} and \ref{notation:foZ_lambda:perf}. Then:

\begin{enumerate}[leftmargin=*]
	\item The restriction map $\pi_-|_{(\foZ_\lambda)^{\perf}_{[0,d-1]}}$ is a derived projectivization of a quasi-coherent sheaf of perfect-amplitude contained in $[0,d-1]$. More concretely: 
	\begin{enumerate}[leftmargin=*]
			\item \label{prop:pi_pm:foZ_lambda:proj-1i} If $1$ is not equivalent to any element $i \in \{2, \ldots, n\}$ in $\lambda$, then the restriction 
				$$\pi_-|_{(\foZ_\lambda)^{\perf}_{[0,d-1]}} \colon (\foZ_\lambda)^{\perf}_{[0,d-1]} \to (\foZ_{|\lambda})^{\perf}_{[0,d-1]} \times_S X$$
				 identifies $(\foZ_\lambda)^{\perf}_{[0,d-1]}$ with the derived projectivization 
				 	$\PP_{(\foZ_{|\lambda})^{\perf}_{[0,d-1]} \times_S X}(\sU_1)$, and for which there is a canonical identification $\sO(1) \simeq \sL_{1, \rm univ}$.
			\item \label{prop:pi_pm:foZ_lambda:proj-1ii} If $1$ is equivalent to some $i  \in \{2, \ldots, n\}$ in $\lambda$, then the restriction 
				$$\pi_-|_{(\foZ_\lambda)^{\perf}_{[0,d-1]}} \colon (\foZ_\lambda)^{\perf}_{[0,d-1]} \to (\foZ_{|\lambda})^{\perf}_{[0,d-1]}$$
				 identifies $(\foZ_\lambda)^{\perf}_{[0,d-1]}$ with the derived projectivization 
				 	$\PP_{(\foZ_{|\lambda})^{\perf}_{[0,d-1]}}(x_{i, \mathrm{univ}}^*(\sU_1))$, and for which there is a canonical identification $\sO(1) \simeq \sL_{1, \rm univ}$.
		\end{enumerate}	
	\item The restriction map $\pi_+|_{\foZ^{\perf}_{[0,d-1]}}$ is a derived projectivization of a quasi-coherent sheaf of perfect-amplitude contained in $[0,d-1]$. More concretely: 
	\begin{enumerate}[leftmargin=*]
			\item \label{prop:pi_pm:foZ_lambda:proj-2i} If $1$ is not equivalent to any element $i \in \{2, \ldots, n\}$ in $\lambda$, then the restriction  
				$$\pi_+|_{(\foZ_\lambda)^{\perf}_{[0,d-1]}} \colon (\foZ_\lambda)^{\perf}_{[0,d-1]} \to (\foZ_{\lambda|})^{\perf}_{[0,d-1]} \times_S X$$
				 identifies $(\foZ_\lambda)^{\perf}_{[0,d-1]}$ with the derived projectivization 
				 	$\PP_{(\foZ_{\lambda|})^{\perf}_{[0,d-1]} \times_S X}(\Sigma^{d-1} \,\sU_1^\vee \otimes \sK_{X/S})$,
					 and for which there is a canonical identification $\sO(1) \simeq \sL_{n, \rm univ}^\vee$;
			\item \label{prop:pi_pm:foZ_lambda:proj-2ii} If $1$ is equivalent to some $i  \in \{2, \ldots, n\}$ in $\lambda$, then the restriction 
				$$\pi_+|_{(\foZ_\lambda)^{\perf}_{[0,d-1]}} \colon (\foZ_\lambda)^{\perf}_{[0,d-1]} \to (\foZ_{\lambda|})^{\perf}_{[0,d-1]}$$
				 identifies $(\foZ_\lambda)^{\perf}_{[0,d-1]}$ with the derived projectivization 
				 	$\PP_{(\foZ_{\lambda|})^{\perf}_{[0,d-1]}}(\Sigma^{d-1} x_{i, \mathrm{univ}}^*(\sU_1^\vee) \otimes \sK_{X/S})$,
					and for which there is a canonical identification $\sO(1) \simeq \sL_{n, \rm univ}^\vee$.
		\end{enumerate}
\end{enumerate}
\end{proposition}

\begin{proof}
The assertions \eqref{prop:pi_pm:foZ_lambda:proj-1i} and \eqref{prop:pi_pm:foZ_lambda:proj-2i} follow from the same argument of Proposition \ref{prop:pi_pm:foZ:proj}; The assertions \eqref{prop:pi_pm:foZ_lambda:proj-1ii} and \eqref{prop:pi_pm:foZ_lambda:proj-2ii} are proved similarly, in view of the proof of Proposition \ref{prop:pi_-:foZ_lambda:proj} \eqref{prop:pi_-:foZ_lambda:proj-2}.
\end{proof}

\subsection{Hecke correspondences for surfaces} \label{sec:Hecke:surface}
In the situation of Proposition \ref{prop:pi_pm:foZ_lambda:proj}, if we set $d =2$, that is, if $X \to S$ is a family of smooth, separated derived surfaces, then both the projections $\pi_\pm|_{(\foZ_\lambda)^{\perf}_{[0,1]}}$ are derived projectivizations of sheaves of perfect-amplitude contained in $[0,1]$. In particular, we could apply the generalized Serre's Theorem \ref{thm:Serre:O(d)}, the results on Beilinson's relations (Proposition \ref{prop:PG:dualexc}, Corollary \ref{cor:PVdot:relations}), and the projectivization formula (Theorem \ref{thm:structural}) to these situations. For simplicity of exposition, we only include here part of all these results for the ``one-step flags" $\foZ^{\perf}_{[0,1]}$; The generalizations of these results to the ``general flags" $(\foZ_\lambda)^{\perf}_{[0,1]}$ are straightforward, by virtue of Proposition \ref{prop:pi_pm:foZ_lambda:proj}.

\begin{notation} \label{construction:foZ:perf:surface} 
Let $X \to S$ be map of prestacks which is smooth, separated, and a relative derived surface, with relative canonical line bundle $\sK_{X/S} = \bigwedge^2 L_{X/S}$, where $L_{X/S}$ is the relative cotangent complex (Definition \ref{def:relativecotangent}). Let the functors $\foZ$, $\foZ^\perf_{[0,1]}$, $\foM$, $\foM^\perf_{[0,1]}$, $\foZ_\lambda$, $\foZ_\lambda)^\perf_{[0,1]}$ be defined as in Constructions \ref{construction:foZ} and \ref{construction:foZ_lambda}, and Notations \ref{notation:foZ:perf} and \ref{notation:foZ_lambda:perf}.
	\begin{enumerate}[leftmargin=*]
		\item We let $\foZ_2^\bullet$ (resp. $(\foZ_2^\bullet)^{\perf}_{[0,1]}$) denote the prestack $\foZ_\lambda$ (resp. $(\foZ_\lambda)^\perf_{[0,1]}$) in the case where $\lambda$ is the set partition of $\{1,2\}$ for which $1 \sim 2$. We let
		$$((\sU_0 \xrightarrow{\phi_{1, \rm univ}} \sU_1 \xrightarrow{\phi_{2, \rm univ}}  \sU_2); \, x_{\rm univ};  \, (\sL_{1, \rm univ}, \sL_{2, \rm univ}))$$
denote the universal element of $(\foZ_2^\bullet)^\perf_{[0,1]}$, where $\sU_i \in \Perf(X \times_S (\foZ_2^\bullet)^\perf_{[0,1]})_{[0,1]}$, $x_{\rm univ} \colon (\foZ_2^\bullet)^\perf_{[0,1]} \to X \times_S (\foZ_2^\bullet)^\perf_{[0,1]}$ is a section of the natural projection $X \times_S (\foZ_2^\bullet)^\perf_{[0,1]} \to (\foZ_2^\bullet)^\perf_{[0,1]} $, and $\sL_{i, \rm univ} \in \Pic((\foZ_2^\bullet)^\perf_{[0,1]})$ are line bundles, and for which there are universal fiber sequences of quasi-coherent sheaves on $(\foZ_2^\bullet)^\perf_{[0,1]} \times_S X$:
	$$\sU_{0} \xrightarrow{\phi_{1, \rm univ}} \sU_1 \to (x_{\rm univ})_* \sL_{1, \rm univ}, \qquad \sU_{1} \xrightarrow{\phi_{2, \rm univ}} \sU_2 \to (x_{\rm univ})_* \sL_{2, \rm univ}.$$
These fiber sequences induce canonical maps 
	$$\rho_{1, \rm univ} \colon x_{\rm univ}^* (\sU_1) \to \sL_{1, \rm univ}, \qquad \rho_{2, \rm univ} \colon \Sigma \, x_{\rm univ}^* (\sU_1^\vee) \to \sL_{2, \rm univ}^\vee \otimes \sK_{X/S}^\vee.$$
		\item We let $(\widehat{\foZ}_2^\bullet)^{\perf}_{[0,1]}$ denote the universal incidence locus (Notation \ref{notation:hatZ}) for the derived projectivizations $\pi_-|\foZ^\perf_{[0,1]}$ and $\pi_+|\foZ^\perf_{[0,1]}$ of Proposition \ref{prop:pi_pm:foZ_lambda:proj} over $\foM^\perf_{[0,1]} \times_S X$. Unwinding the definitions, $(\widehat{\foZ}_2^\bullet)^{\perf}_{[0,1]}$ is the closed subfunctor of $(\foZ_2^\bullet)^\perf_{[0,1]}$ defined by the pullback square
	\begin{equation*} 
	\begin{tikzcd}
	(\widehat{\foZ}_2^\bullet)^{\perf}_{[0,1]} \ar{d}[swap]{\iota} \ar{r}{\iota} & (\foZ_2^\bullet)^\perf_{[0,1]} \ar{d}{i_{\mathbf{1}}} \\
		(\foZ_2^\bullet)^\perf_{[0,1]} \ar{r}{i_{\mathbf{0}}} & \VV_{(\foZ_2^\bullet)^\perf_{[0,1]}}(\Sigma \,(\sL_{1, \rm univ}^\vee \boxtimes \sL_{2, \rm univ}) \otimes \sK_{X/S} ).
	\end{tikzcd}
	\end{equation*}
Here, $i_{\mathbf{0}}$ is the inclusion of the zero section, and $i_{\mathbf{1}}$ is the section of the natural projection $\VV_{(\foZ_2^\bullet)^\perf_{[0,1]}}(\Sigma (\sL_{1, \rm univ}^\vee \boxtimes \sL_{2, \rm univ}) \otimes \sK_{X/S} ) \to (\foZ_2^\bullet)^\perf_{[0,1]}$ that classifies the composite map
	$$\Sigma \,\sL_{1, \rm univ}^\vee \boxtimes (\sL_{2, \rm univ} \otimes \sK_{X/S}) \xrightarrow{\Sigma \, \rho_{1, \rm univ}^\vee \boxtimes \rho_{2,\rm univ}^\vee} x_{\rm univ}^*(\sU_1) \otimes x_{\rm univ}^* ( \sU_1^\vee) \xrightarrow{\rm ev} \sO_{(\foZ_2^\bullet)^{\perf}_{[0,1]}}.$$
	\item Let $\delta$ be an integer. For a prestack $Y$, we let $\Perf(Y)_{[0,1], \rk = \delta}$ denote the full subcategory of $\Perf(Y)$ spanned by perfect objects of Tor-amplitude in $[0,1]$ and rank $\delta$. Similarly, we let $\foM^{\perf}_{[0,1], \rk=\delta}$, $\foZ^{\perf}_{[0,1], \rk=\delta}$, $(\foZ_2^\bullet)^\perf_{[0,1], \rk=\delta}$ and $(\widehat{\foZ}_2^\bullet)^{\perf}_{[0,1], \rk=\delta}$ denote the  subfunctors of $\foM^{\perf}_{[0,1]}$, $\foZ^{\perf}_{[0,1]}$, $(\foZ_2^\bullet)^\perf_{[0,1]}$ and $(\widehat{\foZ}_2^\bullet)^{\perf}_{[0,1]}$ such that their spaces of $T$-point over a map $\eta \colon T \to S$ are respectively spanned by those elements for which the components $\sF_i$ all have rank $\delta$, that is, $\sF_i \in \Perf(X \times_S T)_{[0,1], \rk=\delta}$. (These notations are justified by the following fact: if $x$ is a quasi-smooth closed immersion and $\sL$ is a line bundle, then $x_* \sL$ is perfect of rank $0$. Therefore, if a pair of perfect quasi-coherent sheaves $(\sF_{i-1}, \sF_i)$ fits into a fiber sequence $\sF_{i-1} \to \sF_i \to x_* \sL$, then $\rk \sF_{i-1} = \rk \sF_i$.)
	\item In particular, for each $\delta \ge 0$, we have a pullback square and a commutative square:
	$$
	\begin{tikzcd}
		(\foZ_2^\bullet)^\perf_{[0,1], \rk=\delta} \ar{d}[swap]{r_+'} \ar{r}{r_-'}& \foZ^\perf_{[0,1], \rk=\delta}  \ar{d}{\pr_-} \\
		\foZ^\perf_{[0,1], \rk=\delta} \ar{r}{\pr_+} & \foM^\perf_{[0,1], \rk=\delta} \times_S X
	\end{tikzcd} \quad \text{and} \quad
	\begin{tikzcd}
		(\widehat{\foZ}_2^\bullet)^\perf_{[0,1], \rk=\delta} \ar{d}[swap]{r_+} \ar{r}{r_-}& \foZ^\perf_{[0,1], \rk=\delta}  \ar{d}{\pr_-} \\
		\foZ^\perf_{[0,1], \rk=\delta} \ar{r}{\pr_+} & \foM^\perf_{[0,1], \rk=\delta} \times_S X,
	\end{tikzcd}
	$$
respectively, where $\pr_+ = \pi_-|\foZ^\perf_{[0,1], \rk=\delta}$ and $\pr_- = \pi_+|\foZ^\perf_{[0,1], \rk=\delta}$ are the forgetful maps (notice the change of signs due to the different conventions) and the maps $r_\pm$ are given by the compositions $r_{\pm}' \circ \iota$. 
	\end{enumerate}
\end{notation} 
	
\begin{corollary} \label{cor:Hecke:surface} Let $X \to S$ be map of prestacks which is smooth, separated, and a relative derived surface, with relative canonical line bundle $\sK_{X/S}$. Consider the moduli functors and the maps between them defined in Notation \ref{construction:foZ:perf:surface}. For simplicity of expressions, we introduce the following notations:	
	$$\foX := \foM^\perf_{[0,1], \rk=\delta} \times_S X, \qquad \sW := \Sigma\,\sU^\vee \otimes \sK_{X/S} \in \Perf(\foX).$$	
 In particular, $\sL_{\rm univ}$ denotes the universal line bundle on $\foZ^\perf_{[0,1], \rk=\delta}$, $\sU \in \Perf(\foX)$ denotes the universal perfect object over $\foX=\foM^\perf_{[0,1], \rk=\delta} \times_S X$, and the natural projections
	$$\pr_\pm = \pi_{\mp} |_{\foZ^\perf_{[0,1], \rk=\delta}} \colon  \foZ^\perf_{[0,1], \rk=\delta} \to \foX$$
are the maps which, for any given $\eta \colon T = \Spec R \to S$, $R \in \CAlgDelta$, carry the element
		$$\zeta = ((\sF' \xrightarrow{\phi} \sF); \, x \colon T  \to X \times_S T ; \, \sL \in \Pic(T)) \in \foZ^\perf_{[0,1], \rk=\delta}(\eta)$$
to the respective elements
	\begin{align*}
	\pr_+(\zeta) & = (\sF, (T \xrightarrow{x} X \times_S T \xrightarrow{\pr_X} X)) \in \foX(\eta), \\
	\pr_-(\zeta)  &=  (\sF', (T \xrightarrow{x} X \times_S T \xrightarrow{\pr_X} X)) \in \foX(\eta).
	\end{align*}
		\begin{enumerate}[leftmargin=*]
		\item \label{cor:Hecke:surface-1}
		 If the rank $\delta \ge 1$. Then:
		\begin{enumerate}[label=(\theenumi\roman*), ref=\theenumi\emph{\roman*}]
			\item \label{cor:Hecke:surface-1i}
			For any $d \in \ZZ$, there are canonical equivalences
			$$
			 (\pr_{+})_* (\sL_{\rm univ}^{\otimes d})  \simeq 
	 			\begin{cases}
				\Sym_{\foX}^d (\sU)  &  \text{if~}  d \ge 0; \\
				0 &   \text{if ~} -\delta+1 \le d \le -1; \\
				\Sigma^{1-\delta} \big(\Sym_{\foX}^{-d-\delta} (\sU) \otimes_{\sO_{\foX}}  \det \sU\big)^{\vee}	&  \text{if~} d \le -\delta
				\end{cases}
			$$
	and canonical fiber sequences
			\begin{align*}
			\Sym_{\foX}^d  (\sW) \to 
			(\pr_{-})_{*}(\sL_{\rm univ}^{\otimes -d})  \to  
			  \Sigma^{1+\delta} \big(\Sym_{\foX}^{\delta-d} (\sW) \otimes \det \sW \big)^\vee. 
			\end{align*}
			\item \label{cor:Hecke:surface-1ii}
			For any perfect stack $T$ (see Definition \ref{def:perfstacks}; for example, for any quasi-compact, quasi-separated derived scheme $T$) and any morphism $T \to \foX$, we let $\pr_{+,T} \colon \foZ_{T} \to T$ denote the base-change of $\pr_+ \colon \foZ^\perf_{[0,1], \rk=\delta} \to \foX$ along $T \to \foX$, let $\sL_T$ denote the base-change of $\sL_{\rm univ}$, and let $\sF_T$, $x_T$ be similarly defined. Then the sequence of line bundles $\sO_{\foZ_T} =\sL_T^0, \sL_{T}, \cdots, \sL_{T}^{\otimes \delta-1}$ is a relative exceptional sequence on $\foZ_T$ over $T$, whose left dual relative exceptional sequence is given by 
			$\Sigma^{\delta-1} \bigwedge^{\delta-1}(L_{\pr_{+,T}} \otimes \sL_{T}), \cdots, \Sigma (L_{\pr_{+,T}} \otimes \sL_{T}), \sO_{\foZ_T}$. (Notice that, by virtue of Euler's fiber sequence Theorem \ref{thm:proj:Euler}, $L_{\pr_{+,T}} \otimes \sL_{T}$ is canonically equivalent to $\fib(x_{T}^*(\sF_T) \to \sL_{T})$.) Moreover, these two relative exceptional sequences satisfy Beilinson's relations Corollary \ref{cor:PVdot:relations}.
			\item \label{cor:Hecke:surface-1iii}
			The functors 
				$$\Phi = r_{+\,*} \, r_-^* \colon \QCoh(\foZ^\perf_{[0,1], \rk=\delta}) \to \QCoh(\foZ^\perf_{[0,1], \rk=\delta})$$
			and the functors
				$$\Psi_i = \pr_{+}^*(\blank) \otimes \sL_{\rm univ}^{\otimes i}  \colon \QCoh(\foX) \to \QCoh(\foZ^\perf_{[0,1], \rk=\delta}),$$
			 for all $i \in \ZZ$, and their restrictions to the subcategories of perfect objects, are all fully faithful. Furthermore, there are an induced semiorthogonal decompositions
			\begin{align*}
			\QCoh(\foZ^\perf_{[0,1], \rk=\delta})  &= \big\langle \Phi( \QCoh(\foZ^\perf_{[0,1], \rk=\delta})), 
		~ \Psi_1( \QCoh(\foX) ), \cdots, \Psi_\delta (\QCoh(\foX))  \big\rangle.\\	
			\Perf(\foZ^\perf_{[0,1], \rk=\delta})  &= \big\langle \Phi( \Perf(\foZ^\perf_{[0,1], \rk=\delta})),  
			~ \Psi_1( \Perf(\foX) ), \cdots, \Psi_\delta (\Perf(\foX))  \big\rangle.	
			\end{align*}
		\end{enumerate} 
		
		\item \label{cor:Hecke:surface-2}
		If the rank $\delta =0$, then for each $d \in \ZZ$, there are canonical equivalences
			\begin{align*}
			(\pr_{+})_{*} (\sL_{\rm univ}^{\otimes d})  &\simeq \cofib\Big( \big( \Sym_{\foX}^{-d} (\sU) \otimes_{\sO_{\foX}}  \det \sU\big)^{\vee} \to 	\Sym_{\foX}^d (\sU) 
			\Big). \\
			(\pr_{-})_{*} (\sL_{\rm univ}^{\otimes -d}) & \simeq \cofib\Big( \big(\Sym_{\foX}^{-d} (\sW) \otimes_{\sO_{\foX}} \det \sW)^\vee 
		 	\to \Sym_{\foX}^d (\sW)
			\Big).		
			\end{align*}
		Moreover, the functors 
				$$r_{+\,*} \, r_-^* \colon \QCoh(\foZ^\perf_{[0,1], \rk=0}) \to \QCoh(\foZ^\perf_{[0,1], \rk=0}) \qquad r_{+\,*} \, r_-^* \colon \Perf(\foZ^\perf_{[0,1], \rk=0}) \to \Perf(\foZ^\perf_{[0,1], \rk=0})$$
				$$r_{-\,*} \, r_+^* \colon \QCoh(\foZ^\perf_{[0,1], \rk=0}) \to \QCoh(\foZ^\perf_{[0,1], \rk=0}) \qquad r_{-\,*} \, r_+^* \colon \Perf(\foZ^\perf_{[0,1], \rk=0}) \to \Perf(\foZ^\perf_{[0,1], \rk=0})$$
			are exact equivalences of stable $\infty$-categories.
		\item \label{cor:Hecke:surface-3}
		If the rank $\delta \le -1$. Then:
		\begin{enumerate}[label=(\theenumi\roman*), ref=\theenumi\emph{\roman*}]
			\item \label{cor:Hecke:surface-3i}
			For all $d \in \ZZ$, there are canonical equivalences
			$$
			 (\pr_{-})_* (\sL_{\rm univ}^{\otimes -d})  \simeq 
	 			\begin{cases}
				\Sym_{\foX}^d (\sW)  &  \text{if~}  d \ge 0; \\
				0 &   \text{if ~} \delta+1 \le d \le -1; \\
				\Sigma^{1+\delta} \big(\Sym_{\foX}^{-d+\delta} (\sW) \otimes_{\sO_{\foX}}  \det \sW\big)^{\vee}	&  \text{if~} d \le \delta,
				\end{cases}
			$$
	and canonical fiber sequences
			\begin{align*}
			\Sym_{\foX}^d  (\sU) \to 
			(\pr_{+})_{*}(\sL_{\rm univ}^{\otimes d})  \to  
			  \Sigma^{1-\delta} \big(\Sym_{\foX}^{-\delta-d} (\sU) \otimes_{\sO_{\foX}} \det \sU\big)^\vee. 
			\end{align*}
			\item \label{cor:Hecke:surface-3ii}
			For any perfect stack $T$ (for example, for any quasi-compact, quasi-separated derived scheme $T$) and any morphism $T \to \foX$, we let $\pr_{-,T} \colon \foZ_{T} \to T$ denote the base-change of $\pr_- \colon \foZ^\perf_{[0,1], \rk=\delta} \to \foX$ along $T \to \foX$, let $\sL_T$ denote the base-change of $\sL_{\rm univ}$, and let $\sF'_T$, $x_T$ be similarly defined. 
			Then the sequence of line bundles $\sO_{\foZ_T} =\sL_T^0, \sL_{T}^\vee, \cdots, \sL_{T}^{\otimes -\delta-1}$ is a relative exceptional sequence on $\foZ_T$ over $T$, whose left dual relative exceptional sequence is given by 
			$\Sigma^{\delta-1} \bigwedge^{\delta-1}(L_{\pr_{-,T}} \otimes \sL_{T}^\vee), \cdots, \Sigma (L_{\pr_{-,T}} \otimes \sL_{T}^\vee), \sO_{\foZ_T}$. 
			(Notice that, by virtue of Euler fiber sequence Theorem \ref{thm:proj:Euler}, $L_{\pr_{-,T}} \otimes \sL_{T}^\vee$ is canonically equivalent to $\fib(x_{T}^* (\Sigma \, \sF_T'^\vee \otimes \sK_{X_T/T}) \to \sL_{T}^\vee)$.) Moreover, these two relative exceptional sequences satisfy Beilinson's relations Corollary \ref{cor:PVdot:relations}.		

			\item \label{cor:Hecke:surface-3iii}
			The functors 
				$$\Phi' = r_{-\,*} \, r_+^* \colon \QCoh(\foZ^\perf_{[0,1], \rk=\delta}) \to \QCoh(\foZ^\perf_{[0,1], \rk=\delta})$$
			and the functor
				$$\Psi_i' = \pr_{-}^*(\blank) \otimes \sL_{\rm univ}^{\otimes -i}  \colon \QCoh(\foX) \to \QCoh(\foZ^\perf_{[0,1], \rk=\delta}),$$
			 for all $i \in \ZZ$, and their restrictions to the subcategories of perfect objects, are all fully faithful. Furthermore, there are an induced semiorthogonal decompositions
			\begin{align*}
			\QCoh(\foZ^\perf_{[0,1], \rk=\delta})  &= \big\langle \Phi'( \QCoh(\foZ^\perf_{[0,1], \rk=\delta})),  
			~\Psi_1'( \QCoh(\foX) ), \cdots, \Psi_{-\delta}' (\QCoh(\foX))  \big\rangle.	 \\
			\Perf(\foZ^\perf_{[0,1], \rk=\delta})  &= \big\langle \Phi'( \Perf(\foZ^\perf_{[0,1], \rk=\delta})),  ~ \Psi_1'( \Perf(\foX) ), \cdots, \Psi_{-\delta}' (\Perf(\foX))  \big\rangle.	
			\end{align*}
		\end{enumerate} 
	\end{enumerate}
\end{corollary}	
	
\begin{proof} In view of Proposition \ref{prop:pi_pm:foZ:proj}, the desired results follow from the Generalized Serre's Theorem \ref{thm:Serre:O(d)}, the results on Beilinson's relations (Proposition \ref{prop:PG:dualexc}, Corollary \ref{cor:PVdot:relations}), and the projectivization formula (Theorem \ref{thm:structural}).
\end{proof}	

\begin{remark} In the situation of assertion \ref{cor:Hecke:surface-2}, the same argument of \cite[Theorem 4.1]{JL18} shows that the equivalence functors are non-trivial, and the composition of equivalences $(r_{-,*} \, r_+^*) \circ (r_{+,*} \, r_-^*)$ is a spherical twist in the sense of \cite{ST, AL}.
\end{remark}	

\begin{remark} As a consequence, we could ``decategorify" all the formulae of  Corollary \ref{cor:Hecke:surface} and obtain the corresponding formulae on the level of any additive invariant $\EE$ in the sense of \cite[Definition 6.1]{BGT}; see Remark \ref{rmk:additive}. For example, we could take $\EE$ to be algebraic K-theory, topological Hochschild homology, or topological cyclic homology.
\end{remark}

\appendix
\section{Simplicial Models for Symmetric and Koszul Algebras}
	In this appendix, we provide the concrete simplicial models for symmetric algebras and Koszul algebras, especially in the case of two-term complexes of locally free modules. These constructions are useful for computations in concrete examples. (However, they do not play a role in the proofs of this paper, so we might delete them in the final version.)
	
\subsection{Dold--Puppe and Quillen's construction} \label{sec:Dold-Puppe}
In this subsection, we show that Lurie's framework of non-abelian derived functors recovers the construction of Dold--Puppe \cite{DP} and Quillen \cite{Quillen} using simplicial resolutions.
Let $R$ be an ordinary commutative ring, and suppose that $F \colon \Mod_R^\cn \to \Mod_R^\cn$ is a functor which preserves sifted colimits and projective $R$-modules (for example, if $F$ is one of the functors $\Sym^n_R$, $\bigwedge^n_R$, or $\Gamma_R^n$ of Construction \ref{constr:sym.wedge.gamma}). We let $f = F|\Mod_R^\proj$ be the restriction of $F$ to projective $R$-modules. From the canonical equivalence $\Mod_R^\cn \simeq \Nhc(\Fun(\bDelta, \Mod_R^\heartsuit)^\circ)$, we obtain that each $R$-module $M \in \Mod_R^\cn$ corresponds to a simplicial $R$-module $P_\bullet$ for which each $P_i$ is a projective $R$-module. We could regard $P_\bullet$ as a simplicial object of $\Mod^\cn_R$ via the map $\bDelta^\op \to \Mod_R^{\proj} \subseteq \Mod^\cn_R$. By virtue of \cite[Theorem 4.2.4.1]{HTT} and \cite[Proposition 5.5.9.14]{HTT}, we can identify $M$ with the geometric realization $|P_\bullet|$ of $P_\bullet$ (in other words, $P_\bullet$ is a simplicial resolution of $M$). Since $F$ preserves sifted colimits, we obtain canonical equivalences 
	$$F(M) \simeq F(|P_\bullet|) \simeq |f(P_\bullet)|,$$
where $f(P_\bullet)$ is the simplicial object obtained by applying $f$ level-wisely to $P_\bullet$, that is,
	$$f(P_\bullet) \colon \bDelta^\op \xrightarrow{P_\bullet} \Mod_R^{\proj} \xrightarrow{f = F| \Mod_R^{\proj} } \Mod_R^{\proj} \subseteq \Mod^\cn_R.$$
Similarly, the symmetric algebra functor $\Sym_R^* \colon \Mod_R^{\cn} \to \CAlgDelta_R$ could also be computed by choosing a simplicial resolution $P_\bullet$, where $P_n$ are projective $R$-modules.  We could summarize the above discussion by saying that Lurie's construction of  non-abelian derived functors extends the construction of Dold--Puppe \cite{DP} and Quillen \cite{Quillen} using simplicial resolutions.

\subsection{Simplicial symmetric algebras} 
By virtue of the preceding subsection, in the case where $R$ is a discrete commutative ring and $N \in \Mod_R^\cn$ is a connective $R$-module which is represented by a complex $P_*$ of discrete projective $R$-modules, we could compute the derived symmetric algebra $\Sym_R^*(N)$ by the Dold--Kan construction ${\rm DK}_\bullet(P_*)$ (\cite[Construction 1.2.3.5]{HA}). In this subsection, we spell out this construction in the case where $N$ is represented by a cofiber of a map $\rho \colon M' \to M$ between finite locally free $R$-modules.

\begin{construction} Let $R$ be a discrete commutative ring, and let $\rho \colon M' \to M$ be a morphism of finitely generated locally free modules. The {\em simplicial symmetric algebras of $\rho \colon M' \to M$ over $R$} is a simplicial object of discrete commutative rings
	$$\Sym_{\Delta}^*(R, \rho \colon M' \to M) \in \Fun(\bDelta^\op, \CAlg_R^\heartsuit)$$
defined as follows: The terms of $\Sym_{\Delta}^*(R, \rho \colon M' \to M)$ are given by for all $n \ge 0$,
	$$(\Sym_{\Delta}^*(R, \rho \colon M' \to M))_n = S^*_R(M) \otimes S_R^*(M')^{\otimes n};$$
The face maps $d_i \colon (\Sym_{\Delta}^*(R, \rho \colon M' \to M))_n \to (\Sym_{\Delta}^*(R, \rho \colon M' \to M))_{n-1}$, $0 \le i \le n$, and the degeneracy maps $s_i \colon (\Sym_{\Delta}^*(R, \rho \colon M' \to M))_n \to (\Sym_{\Delta}^*(R, \rho \colon M' \to M))_{n+1}$, $0 \le i \le n$, are respectively given by the formulae: for $f \in S_R^*(M)$ and $g_i \in S_R^*(M')$,
	\begin{align*}
		& d_{i}(f \otimes g_1 \otimes \cdots \otimes g_n) = 
		\begin{cases}
		(f \cdot (S_R^*(\rho)(g_1))) \otimes g_2 \otimes \cdots \otimes g_n & \text{if~} i=0; \\
		f \otimes g_1 \otimes \cdots \otimes g_{i-1} \otimes (g_i \cdot g_{i+1}) \otimes g_{i+1} \otimes \cdots \otimes g_n & \text{if~} 1 \le i \le n-1; \\
		(f \cdot (S_R^*(0)(g_n))) \otimes g_1 \otimes \cdots \otimes g_{n-1}& \text{if~} i = n.
		\end{cases}  \\
		&s_{i}(f \otimes g_1 \otimes \cdots \otimes g_n) = f \otimes g_1 \otimes \cdots g_i \otimes 1 \otimes g_{i+1} \otimes \cdots  \otimes g_n.	
	\end{align*}
Here $S_R^*(\rho), S_R^*(0) \colon S_R^*(M') \to S_R^*(M)$ denote the canonical morphisms of algebras induced by the morphisms $\rho \colon M' \to M$ and $0 \colon M' \to M$, respectively. Notice $S_R^*(0)$ by definition factorizes as a composition $S_R^*(M') \twoheadrightarrow R \hookrightarrow S_R^*(M)$. There is a natural grading of $\Sym_{\Delta}^*(R, \rho \colon M' \to M)$,
	$$\Sym_{\Delta}^*(R, \rho \colon M' \to M)  = \bigoplus_{d \ge 0} \Sym_{\Delta}^d(R, \rho \colon M' \to M)$$
where for each $d \ge 0$, $\Sym_{\Delta}^d(R, \rho \colon M' \to M)$ is the simplicial module that has terms 
	$$(\Sym_{\Delta}^d(R, \rho \colon M' \to M))_n = \bigoplus_{d_0 + d_1 + \cdots + d_n = d, d_i \ge 0} S_R^{d_0} M \otimes_R S_R^{d_1} M' \otimes_R \cdots \otimes_R S_R^{d_n} M'.$$
The face maps and degeneracy maps are given by the same formulae.
\end{construction}

\subsection{Simplicial Koszul algebras}
The concrete models of the preceding subsection provides simplicial models for derived zero loci of a cosection. Let $R$ be a discrete commutative ring, and let $\rho \colon M \to R$ be a cosection of a connective $R$-module $M$. By virtue of Remark \ref{rmk:dzero:algebra} and \S \ref{sec:Dold-Puppe}, we could compute the simplicial commutative ring $R \otimes_{{\rm ev}_0, \Sym_R^*(M), {\rm ev}_\rho} R$ (whose spectrum is the derived zero locus of $\rho$) using either the simplicial model $\Sym_\Delta^*(R, \id_M) \otimes_{\Sym_\Delta^*(M), {\rm ev}_\rho} R$ or the simplicial model $\Sym_\Delta^*(R, \rho \colon M \to R) \otimes_{R[T], T \mapsto 1} R$.

In this subsection, we spell out the above simplicial model for $R \otimes_{{\rm ev}_0, \Sym_R^*(M), {\rm ev}_\rho} R$ in the case where $M$ is a finite free $R$-module. Let $R$ be a discrete commutative ring, and let $\mathbf{r} = (r_1, \cdots, r_m)$ be a sequence of elements $r_i \in R$, which we may regard as a section $\mathbf{r} \colon R \to M^\vee$, where $M = R^m$, or dually, a cosection map $\mathbf{r}^\vee \colon M = R^{m} \to R$. Therefore, by adjunction, the cosection maps $\mathbf{r}^\vee$ and $\mathbf{0}^\vee$ induce maps of discrete commutative algebras $\mathrm{ev}_{\mathbf{r}} \colon S_R^*(M) \to R$ and $\mathrm{ev}_{\mathbf{0}} \colon S_R^*(M) \to R$, respectively.

\begin{construction} \label{constr:Koszul_simplicial} Let $R$ be a discrete commutative ring, and let $\mathbf{r} = (r_1, \cdots, r_m)$ be a sequence of elements $r_i \in R$. Then {\em simplicial Koszul algebras of the sequence $\mathbf{r}$}, which we shall denote by $\mathrm{Kos}_{\Delta}^*(R, \mathbf{r})$, is the following simplicial object of the ordinary category of discrete commutative rings
	$$\mathrm{Kos}_{\Delta}^*(R, \mathbf{r}) = \Sym_{\Delta}^*(R,  \id_M \colon M  \to M) \otimes_{S_R^*(M)} R \in \Fun(\bDelta^\op, \CAlg_R^\heartsuit).$$
Here, the $S_R^*(M)$-algebra structure of $R$ is given by the map $\mathrm{ev}_{\mathbf{r}} \colon S_R^*(M) \to R$. Unwinding the definition, the simplicial Koszul algebra takes the form:
	$$\mathrm{Kos}_{\Delta}^*(R, \mathbf{r})_n = S^*_R(M)^{\otimes n} = S_R^*(M) \otimes_R S_R^*(M) \otimes_R \cdots \otimes_R S_R^*(M).$$
The face maps 
	$$d_i \colon S^*_R(M)^{\otimes n}  \to S^*_R(M)^{\otimes (n-1)}$$
 and the degeneracy maps 
	$$s_i \colon S^*_R(M)^{\otimes n}  \to S^*_R(M)^{\otimes (n+1)}$$
are respectively given by the formulae: for all $g_i \in S_R^*(M)$,
	\begin{align*}
		& d_{i}(g_1 \otimes \cdots \otimes g_n) = 
		\begin{cases}
		\mathrm{ev}_{\mathbf{r}}(g_1) \cdot (g_2 \otimes \cdots \otimes g_n) & \text{if~} i=0; \\
		 g_1 \otimes \cdots \otimes g_{i-1} \otimes (g_i \cdot g_{i+1}) \otimes g_{i+1} \otimes \cdots \otimes g_n & \text{if~} 1 \le i \le n-1; \\
		\mathrm{ev}_{\mathbf{0}} (g_n) \cdot (g_1 \otimes \cdots \otimes g_{n-1}) & \text{if~} i = n.
		\end{cases}  \\
		&s_{i}(g_1 \otimes \cdots \otimes g_n) =g_1 \otimes \cdots g_i \otimes 1 \otimes g_{i+1} \otimes \cdots  \otimes g_n.	
	\end{align*}	
In particular, the face maps 
	$d_0, d_1 \colon \mathrm{Kos}_{\Delta}^*(R, \mathbf{r})_1 = S_R^*(M) \to \mathrm{Kos}_{\Delta}^*(R, \mathbf{r})_{0} = R$
 are respectively given by
	$d_0(g) = \mathrm{ev}_{\mathbf{r}}(g)$ and $d_1(g) = \mathrm{ev}_{\mathbf{0}}(g).$
Hence, there is an augmentation map $\mathrm{Kos}_{\Delta}^*(R, \mathbf{r})_{0} \to R/(\mathbf{r} - \mathbf{0}) R = R/ \mathbf{r} R$.
\end{construction}

\begin{remark}[Bar convention] Since $S_R^*(M) = R[X_1, \ldots, X_m]$, therefore we can write
	$$\mathrm{Kos}_{\Delta}^*(R, \mathbf{r})_n = R[X_1, \ldots, X_m]^{\otimes n}.$$
The morphisms $\mathrm{ev}_{\mathbf{r}} \colon R[X_1, \ldots, X_m] \to R$ and $\mathrm{ev}_{\mathbf{0}} \colon R[X_1, \ldots, X_m] \to R$ are respectively given by the formulae: $\mathrm{ev}_{\mathbf{r}} \colon X_i \mapsto r_i$, and $\mathrm{ev}_{\mathbf{0}} \colon X_i \mapsto 0$. If we use the ``bar convention" and write elements of $R[X_1, \ldots, X_m]^{\otimes n}$ in the form $f [g_1 | g_2 | \cdots | g_n]$, where $f \in R$ and $g_i \in R[X_1, \ldots, X_m]$, then we can rewrite the   formulae of face and degeneracy maps as
\begin{align*}
		& d_{i}[g_1 | \cdots | g_n] = 
		\begin{cases}
		\mathrm{ev}_{\mathbf{r}}(g_1)  [g_2 | \cdots | g_n] & \text{if~} i=0; \\
		 [g_1 | \cdots | g_{i-1} | g_i \cdot g_{i+1} | g_{i+2} | \cdots | g_n] & \text{if~} 1 \le i \le n-1; \\
		\mathrm{ev}_{\mathbf{0}} (g_n) [g_1 | \cdots | g_{n-1}] & \text{if~} i = n.
		\end{cases}  \\
		&s_{i}[g_1 | \cdots | g_n] =[g_1 | \cdots | g_i | \,1\, |  g_{i+1}|  \cdots  | g_n].
	\end{align*}	
\end{remark}
	
In the situation of Construction \ref{constr:Koszul_simplicial}, by virtue of Remark \ref{rmk:dzero:algebra} and \S \ref{sec:Dold-Puppe}, we obtain that the geometric realization of the simplicial Koszul algebra $\mathrm{Kos}_{\Delta}^*(R, \mathbf{r})$ (when regarded as a simplicial diagram in $\CAlgDelta_R$) is canonically equivalent to the simplicial commutative ring $R \otimes_{{\rm ev}_0, \Sym_R^*(M), {\rm ev}_{\mathbf{r}^\vee}} R$. Next, we compare it with another common description:
	
\begin{lemma} Let $R$ be a discrete commutative ring, and $\mathbf{r} = (r_1, \ldots, r_m)$ be a sequence of elements of $R$. Then $\mathbf{r}$ determines a canonical ring homomorphism $\ZZ[x_1, \ldots, x_m] \to R$ that carries $x_i$ to $r_i$. If we let $\psi$ denote the canonical equivalence 
	$$\psi \colon \mathrm{N}_\bullet^{\rm hc}(\Fun(\bDelta^\op, \CAlg_R^\heartsuit)^{\circ}) \simeq \CAlgDelta_R,$$ 
then there is a canonical equivalence of simplicial commutative $R$-algebras:
	$$\psi(\mathrm{Kos}_{\Delta}^*(R, \mathbf{r})) \simeq R \otimes_{\ZZ[x_1, \ldots, x_m]} \frac{\ZZ[x_1, \ldots, x_m]}{(x_1, \ldots, x_m)}.$$
\end{lemma}
 
 \begin{proof} Since the formation of $\mathrm{Kos}_{\Delta}^*(R, \mathbf{r})$ commutes with base change of discrete commutative rings $R' \to R$, we may reduce to the case where $R = \ZZ[x_1, \ldots, x_m]$, $r_i = x_i$. It follows from Illusie's result Corollary \ref{cor:Illusie:S^n} that $\mathrm{Kos}_{\Delta}^*(R, \mathbf{r})$ is quasi-isomorphism to the Koszul complex $\mathrm{Kos}^*(R, \mathbf{x})$; On the other hand, we know the latter complex resolves the module $\ZZ[x_1, \ldots, x_m]/(x_1, \ldots, x_m)$. Hence the augmentation map $\mathrm{Kos}_{\Delta}^*(R, \mathbf{x}) \to \ZZ[\mathbf{x}]/(\mathbf{x})$ is a homotopy equivalence of simplicial commutative rings.
 \end{proof}


\addtocontents{toc}{\vspace{\normalbaselineskip}}


\begin{thebibliography}{99}

\bibitem[ABW]{ABW}
Akin, Kaan, David A. Buchsbaum, and Jerzy Weyman. 
{\em Schur functors and Schur complexes.} Advances in Mathematics 44, no. 3 (1982): 207-278.

\bibitem[AT19]{AT19}
Allahverdi, Muberra, and Alexandre Tchernev. 
{\em Acyclicity of Schur complexes and torsion freeness of Schur modules.} Journal of Algebra 535 (2019): 133--158.

\bibitem[AL]{AL}
Anno, Rina, and Timothy Logvinenko, 
{\em Spherical DG-functors.} 
J. Eur. Math. Soc. 19, no. 9 (2017): 2577-2656.

\bibitem[AE]{AE}
Antieau, Benjamin, and Elden Elmanto. 
{\em Descent for semiorthogonal decompositions.} Advances in Mathematics 380 (2021): 107600.


\bibitem[ACGH]{ACGH}
Arbarello, Enrico, Maurizio Cornalba, Phillip Griffiths, and Joseph Harris. {\em Geometry of algebraic curves Vol. I}. Grundlehren der mathematischen Wissenschaften {\bf 267}. Springer--Verlag, New York (1985).

\bibitem[AB12]{AB12}
Arcara, Daniele, and Aaron Bertram. 
{\em Bridgeland-stable moduli spaces for $ K $-trivial surfaces.} Journal of the European Mathematical Society 15.1 (2012): 1--38.


\bibitem[ACH]{ACH}
Arinkin, Dima and C{\u{a}}ld{\u{a}}raru, Andrei and Hablicsek, M{\'a}rton.
{\em Formality of derived intersections and the orbifold HKR isomorphism.} Journal of Algebra 540 (2019): 100--120.


\bibitem[BLM+]{BLM+}
Bayer, Arend and Lahoz, Mart{\'\i} and Macr{\`\i}, Emanuele and Nuer, Howard and Perry, Alexander and Stellari, Paolo
{\em Stability conditions in families.}  arXiv:1902.08184 (2019).

\bibitem[BM11]{BM11} 
Bayer, Arend and Emanuele Macr\'{i}. 
{\em The space of stability conditions on the local projective
plane.} Duke Math. J., 160(2): 263--322, 2011.

\bibitem[BM14]{BM14}
Bayer, Arend, and Emanuele Macrì. 
{\em MMP for moduli of sheaves on $K3$s via wall-crossing: nef and movable cones, Lagrangian fibrations.}
 Inventiones mathematicae 198.3 (2014): 505--590.

\bibitem[BF]{BF}
Behrend, Kai, and Barbara Fantechi. 
{\em The intrinsic normal cone.} Inventiones mathematicae 128, no. 1 (1997): 45--88.

\bibitem[Be]{Be}
Beilinson, Alexander A. {\em Coherent sheaves on $Pn$ and problems of linear algebra.} Functional Analysis and Its Applications 12.3 (1978): 214--216.

\bibitem[BK19]{BK19}
Belmans, Pieter, and Andreas Krug,
{\em Derived categories of (nested) Hilbert schemes.} arXiv:1909.04321 (2019).

\bibitem[BFN]{BFN10}
Ben-Zvi, David, John Francis, and David Nadler. 
{\em Integral transforms and Drinfeld centers in derived algebraic geometry.} Journal of the American Mathematical Society 23, no. 4 (2010): 909--966.

\bibitem[BS]{BS}
Bergh, Daniel and Schn{\"u}rer, Olaf M.
 {\em Conservative descent for semi-orthogonal decompositions.} Adv. Math. 360 (2020): 106882--39.
 
 
\bibitem[BL22]{BL22}
Bhatt, Bhargav, and Jacob Lurie. 
{\em Absolute prismatic cohomology.} arXiv:2201.06120 (2022).
 
 
\bibitem[BGT]{BGT}
Blumberg, Andrew J., Gepner, David, and Tabuada, Gon{\c{c}}alo.
{\em A universal characterization of higher algebraic K-theory.} 
Geometry \& Topology 17, no. 2 (2013): 733--838.


\bibitem[Bo]{Bo}
Bondal, Alexei, 
{\em Representations of associative algebras and coherent sheaves,}
(Russian) Izv. Akad. Nauk SSSR Ser. Mat. 53 (1989), no. 1, 25-44; translation in Math. USSR-Izv. 34 (1990), no. 1, 23--42;  translation in Math. USSR-Izv. 34 (1990), no. 1, 23--42.

\bibitem[BK]{BK}
Bondal Alexei, and Kapranov, Mikhail M.,
{\em Representable functors, Serre functors, and reconstructions,}
(Russian) Izv. Akad. Nauk SSSR Ser. Mat. 53 (1989), no. 6, 1183--1205, 1337; translation in Math. USSR-Izv. 35 (1990), no. 3, 519--541.

\bibitem[Bou]{Bou}
Bourbaki, Nicolas. 
{\em Algebra I: Chapters 1--3.} Springer, 1989.

\bibitem[Bri07]{Bri07}
Bridgeland, Tom.
{\em Stability conditions on triangulated categories.}
Annals of Mathematics (2007): 317-345.

\bibitem[Bri08]{Bri08} 
Bridgeland, Tom.
{\em Stability conditions on $K3$ surfaces.} Duke Math. J., 141(2):241-291,
2008.

\bibitem[Bur]{Bur}
Burban, I. G. O. R. 
{\em Derived categories of coherent sheaves on rational singular curves.}Representations of finite dimensional algebras and related topics, in Lie Theory and geometry. Proceedings from the ICRA X, Fields Inst. Commun 40 (2004): 173--188.


\bibitem[CS19]{CS}
Cesnavicius, Kestutis, and Peter Scholze. 
{\em Purity for flat cohomology.} arXiv:1912.10932 (2019).

\bibitem[CK]{CK}
 Ciocan-Fontanine, Ionu{\c{t}, and Mikhail Kapranov. 
 {\em Derived quot schemes.}
 Annales scientifiques de l'{\'E}cole normale sup{\'e}rieure}. Vol. 34. No. 3. 2001.
 
\bibitem[DP]{DP}
Dold, Albrecht, and Dieter Puppe. 
{\em Non--additive functors, their derived functors, and the suspension homomorphism.} Proceedings of the National Academy of Sciences of the United States of America 44, no. 10 (1958): 1065.

\bibitem[Ei]{Ei}
Eisenbud, David. 
{\em Commutative Algebra: with a view toward algebraic geometry.} Vol. 150. Springer Science \& Business Media, 2013.

\bibitem[ES]{ES}
Ellingsrud, Geir, and Stein Str{\o}mme. 
{\em An intersection number for the punctual Hilbert scheme of a surface.} Transactions of the American Mathematical Society 350.6 (1998): 2547--2552.

\bibitem[Ful]{Ful}
Fulton, William. 
{\em Intersection theory.} Springer Science \& Business Media, 2013.

\bibitem[GR]{GRI}
Gaitsgory, Dennis, and Nick Rozenblyum. 
 {\em A Study in Derived Algebraic Geometry: Volume I: Correspondences and
  Duality}, Vol. 221. American Mathematical Society, 2019.
  
\bibitem[Go]{Go}
Gorodentsev, A. L. 
{\em Transformations of exceptional bundles on Pn.} Math. USSR Izvestiya 32, no. 1 (1989):1--13.

\bibitem[EGA]{EGA} Grothendieck, Alexandre, with Dieudonn{\'e}, J. {\em {\'E}l{\'e}ments de G{\'e}om{\'e}trie Alg{\'e}brique}.Publ. Math. IH{\'E}S {\bf 4} (Chapter 0, 1-7, and I, 1-10), {\bf 8} (II, 118), {\bf 11} (Chapter 0, 8-13, and III, 1-5), {\bf 17} (III, 6-7), {\bf 20} (Chapter 0, 14- 23, and IV, 1), {\bf 24} (IV, 2-7), {\bf 28} (IV, 8-15), and {\bf 32} (IV, 16-21),1960-1967.

\bibitem[EGAI]{EGAI}
Grothendieck, Alexandre and JA Dieudonn{\'e},
{\em El{\'e}ments de g{\'e}om{\'e}trie alg{\'e}brique},
Grundlehren der Mathematischen Wissenschaften. Springer-Verlag, Berlin, vol. 166, 1971.


\bibitem[He21]{He21}
Hekking, Jeroen. 
{\em Graded algebras, projective spectra and blow-ups in derived algebraic geometry.} arXiv:2106.01270 (2021).

\bibitem[Ill]{Ill}
Illusie, Luc. 
{\em Complexe cotangent et d{\'e}formations I \& II.}
Vol. 239. Springer, 2006.

\bibitem[J19]{J19}
Jiang, Qingyuan.
{\em On the Chow theory of projectivizations}. J. Inst. Math. Jussieu., (2021), \href{https://doi.org/10.1017/S1474748021000451}{https://doi.org/10.1017/S1474748021000451}.

\bibitem[J20]{J20}
Jiang, Qingyuan. 
{\em On the Chow theory of Quot schemes of locally free quotients.}\href{https://arxiv.org/abs/2010.10734}{arXiv:2010.10734} (2020).

\bibitem[J21]{J21}
Jiang, Qingyuan. 
{\em Derived categories of Quot schemes of locally free quotients, I.} \href{https://arxiv.org/abs/2107.09193}{arXiv:2107.09193} (2021).

\bibitem[J22b]{J22a}
Jiang, Qingyuan. 
{\em Derived Grassmannians and derived Schur functors.} \href{https://arxiv.org/abs/2212.10488}{arXiv:2212.10488} (2022).

\bibitem[J23a]{J23}
Jiang, Qingyuan. 
{\em Derived categories of derived Grassmannians.} \href{https://arxiv.org/abs/2307.02456}{arXiv: 2307.02456} (2023).

\bibitem[J23b]{J22b}
Jiang, Qingyuan. 
{\em Abel maps for integral curves: a derived perspective.} In preparation (2023).

\bibitem[JL18]{JL18}
Jiang, Qingyuan, and Naichung Conan Leung.
{\em Derived categories of projectivizations and flops}. Adv. Math. (2021),  \href{https://doi.org/10.1016/j.aim.2021.108169}{https://doi.org/10.1016/j.aim.2021.108169}.

\bibitem[JLX17]{JLX17}
Jiang, Qingyuan, Naichung Conan Leung, and Ying Xie. {\em Categorical Pl\" ucker Formula and Homological Projective Duality.} J. Eur. Math. Soc. (JEMS) 23 (2021), no. 6, 1859--1898.

\bibitem[JT17]{JT17}
Jiang, Yunfeng, and Richard Thomas. 
{\em Virtual signed Euler characteristics.} Journal of Algebraic Geometry 26, no. 2 (2017): 379--397.


\bibitem[KPS21]{KPS}
Kalck, Martin, Nebojsa Pavic, and Evgeny Shinder. 
{\em Obstructions to Semiorthogonal Decompositions for Singular Threefolds I: K-Theory.} Moscow Mathematical Journal 21, no. 3 (2021): 567--592.

\bibitem[KV19]{KV19}
Kapranov, Mikhail, and Eric Vasserot. 
{\em The cohomological Hall algebra of a surface and factorization cohomology.}  arXiv:1901.07641 (2019).

\bibitem[Ka19]{Kaw19}
Kawamata, Yujiro. 
{\em Semi-orthogonal decomposition of a derived category of a 3-fold with an ordinary double point.} arXiv:1903.00801 (2019).

\bibitem[Kh20]{Kh20}
Khan, Adeel A. 
{\em Algebraic K-theory of quasi-smooth blow-ups and cdh descent.} Annales Henri Lebesgue 3 (2020): 1091--1116.

\bibitem[Kh21]{Kh21}
Khan, Adeel A. 
{\em Virtual excess intersection theory.} Ann. K-Theory 6 (2021): 559--570.

\bibitem[KR18]{KR18}
Khan, Adeel A., and David Rydh. 
{\em Virtual Cartier divisors and blow-ups.} arXiv:1802.05702 (2018).

\bibitem[Ku07]{Kuz07} 
Kuznetsov, Alexander. 
{\em Homological projective duality},
Publ. Math. de L'IH{\`E}S, {\bf 105}, n. 1 (2007), 157--220.

\bibitem[Ku11]{Kuz11} 
Kuznetsov, Alexander. {\em  Base change for semiorthogonal decompositions.} 
Comp.\ Math.\ 147 (2011), 852--876.

\bibitem[Ku21]{Kuz21}
Kuznetsov, Alexander.
{\em Semiorthogonal decompositions in families.} arXiv:2111.00527 (2021
 
\bibitem[KS22a]{KS22a}
Kuznetsov, Alexander, and Evgeny Shinder. 
{\em Categorical absorptions of singularities and degenerations.} arXiv:2207.06477 (2022).
 
\bibitem[KS22b]{KS22b}
 Kuznetsov, Alexander, and Evgeny Shinder. 
 {\em Homologically finite-dimensional objects in triangulated categories.}arXiv:2211.09418 (2022).
 
 
\bibitem[Laz]{Laz04}
Lazarsfeld, Robert.
{\em Positivity in algebraic geometry I \& II}. 
Vol. 49. Springer Science  \&  Business, 2004.

\bibitem[LT]{LT}
Li, Jun, and Gang Tian. 
{\em Virtual moduli cycles and Gromov-Witten invariants of algebraic varieties.} Journal of the American Mathematical Society 11, no. 1 (1998): 119--174.

\bibitem[L21]{Lin}
Lin, Xun. 
{\em On nonexistence of semi-orthogonal decompositions in algebraic geometry.} arXiv:2107.09564 (2021).

\bibitem[Lip]{Lip} 
Lipman, Joseph. {\em Notes on derived functors and Grothendieck duality.} In Foundations of Grothendieck duality for diagrams of schemes, volume 1960 of Lecture Notes in Math., pages 1--259. Springer, Berlin, 2009.
 
\bibitem[LN]{LN}
Lipman, Joseph and Neeman, Amnon, {\em Quasi-perfect scheme-maps and boundedness of the twisted inverse image functor}, Illinois J. Math. 51 (2007), no. 1, 209--236.


\bibitem[DAG]{DAG}
Lurie, Jacob. 
{\em Derived algebraic geometry.} Doctoral dissertation, Massachusetts Institute of Technology, 2004.

\bibitem[HTT]{HTT}
Lurie, Jacob.
{\em Higher Topos Theory.} Annals of Mathematics Studies, vol. 170, Princeton University Press,
Princeton, NJ, 2009, MR 2522659.

\bibitem[DAG-V]{DAGV}
Lurie, Jacob.
 {\em Derived Algebraic Geometry V: Structured Spaces}, available at 
  \href{https://www.math.ias.edu/~lurie/papers/DAG-V.pdf}{https://www.math.ias.edu/~lurie/papers/DAG-V.pdf}, version dated February 2011.
  
\bibitem[DAG-VII]{DAGVII}
Lurie, Jacob.
 {\em Derived Algebraic Geometry VII: Spectral Schemes}, available at 
  \href{https://www.math.ias.edu/~lurie/papers/DAG-VII.pdf
}{https://www.math.ias.edu/~lurie/papers/DAG-VII.pdf
}, version dated February 2011.
  
\bibitem[DAG-XI]{DAGXI}
Lurie, Jacob.
{\em Derived Algebraic Geometry XI: Descent Theorems}, 
available at \href{https://www.math.ias.edu/~lurie/papers/DAG-XI.pdf}{https://www.math.ias.edu/~lurie/papers/DAG-XI.pdf}, version dated 28 September 2011.

\bibitem[HA]{HA}
Lurie, Jacob.
{\em Higher Algebra}, available at \href{https://www.math.ias.edu/~lurie/papers/HA.pdf}{https://www.math.ias.edu/~lurie/papers/HA.pdf}, version dated September 2017.

\bibitem[SAG]{SAG}
Lurie, Jacob.
 {\em Spectral Algebraic Geometry}, available at
 \href{https://www.math.ias.edu/~lurie/papers/SAG-rootfile.pdf}{https://www.math.ias.edu/~lurie/papers/SAG-rootfile.pdf}, version dated February 2018.

\bibitem[Lu-Ke]{kerodon} 
Lurie, Jacob.
{\em Kerodon}. \url{https://kerodon.net}.  2021. 


\bibitem[May]{May}
May, J. Peter. {\em A concise course in algebraic topology}. University of Chicago press, 1999.

\bibitem[Nee96]{Nee}
Neeman, Amnon. {\em The Grothendieck duality theorem via Bousfield’s techniques and Brown representability.}  J. Amer. Math. Soc., 9(1):205--236, 1996.

\bibitem[Nee10]{Nee10}
Neeman, Amnon. 
{\em Derived categories and Grothendieck duality.} In Triangulated categories, volume 375 of London Math. Soc. Lecture Note Ser., pages 290--350. Cambridge Univ. Press, Cambridge (2010).

\bibitem[Ne17]{NegW}
Negu{\c{t}}, Andrei.
{\em W-algebras associated to surfaces.} arXiv:1710.03217 (2017).

\bibitem[Ne18]{NegHecke}
Negu{\c{t}}, Andrei 
{\em Hecke correspondences for smooth moduli spaces of sheaves.} arXiv:1804.03645 (2018).

\bibitem[Ne19]{NegShuffle}
Negu{\c{t}}, Andrei.
{\em Shuffle algebras associated to surfaces}. Selecta Mathematica, 25(3), pp.1-57 (2019).

\bibitem[No]{No}
Northcott, Douglas Geoffrey. 
{\em Finite free resolutions.} No. 71. Cambridge University Press, 2004.

\bibitem[O92]{Orlov92}
Orlov, Dmitri O. 
{\em Projective bundles, monoidal transformations, and derived categories of coherent sheaves.} Izvestiya Rossiiskoi Akademii Nauk. Seriya Matematicheskaya, 56(4), 852--862.

\bibitem[O97]{Orlov97}
Orlov, Dmitri O. 
{\em Equivalences of derived categories and K3 surfaces.} Journal of Mathematical Sciences 84.5 (1997): 1361--1381.

\bibitem[PTVV]{PTVV}
Pantev, Tony and To{\"e}n, Bertrand and Vaqui{\'e}, Michel and Vezzosi, Gabriele.
{\em Shifted symplectic structures.} Publications math{\'e}matiques de l'IH{\'E}S. 117 (2013): 271--328.

\bibitem[PS21]{PS21}
Pavic, Nebojsa, and Evgeny Shinder. 
{\em Derived categories of nodal del Pezzo threefolds.} arXiv:2108.04499 (2021).

\bibitem[PZ20]{PZ20}
Peng, Junyao, and Yu Zhao. 
{\em A Serre Relation in the $ K $-theoretic Hall algebra of surfaces.} arXiv:2012.07897 (2020).

\bibitem[P19]{Pe19}
Perry, Alexander. 
{\em Noncommutative homological projective duality.}
Advances in Mathematics 350 (2019): 877--972.

\bibitem[PS19]{PS19}
Porta, Mauro, and Francesco Sala. 
{\em Two-dimensional categorified Hall algebras.} arXiv:1903.07253 (2019).

\bibitem[Qu]{Quillen}
Quillen, Daniel. {\em Homotopical algebra.} Vol. 43. Springer, 2006.

\bibitem[Ro]{Ro}
Roby, Norbert. 
{\em Lois polynomes et lois formelles en th{\'e}orie des modules.}
 In Annales scientifiques de l'{\'E}cole Normale Sup{\'e}rieure, vol. 80, no. 3, pp. 213--348. 1963.
 
\bibitem[ST]{ST}
Seidel, Paul, and Richard Thomas.
{\em Braid group actions on derived categories of coherent sheaves.} 
Duke Mathematical Journal 108.1 (2001): 37--108.

\bibitem[Sta]{stacks-project}
{The Stacks Project Authors}.
 \textit{Stacks Project}.
 \url{https://stacks.math.columbia.edu}, 2022.

\bibitem[Tod18]{Tod1}
Toda, Yukinobu.
{\em Birational geometry for d-critical loci and wall-crossing in Calabi-Yau 3-folds.} arXiv preprint arXiv:1805.00182 (2018).

\bibitem[Tod21]{Tod2}
Toda, Yukinobu.
{\em Semiorthogonal decompositions of stable pair moduli spaces via d-critical flips.} Journal of the European Mathematical Society 23, no. 5 (2021): 1675-1725.

\bibitem[Tod19]{Tod3}
Toda, Yukinobu.
{\em On categorical Donaldson-Thomas theory for local surfaces. } arXiv:1907.09076 (2019).

\bibitem[Tod21a]{Tod4}
Toda, Yukinobu. 
{\em Semiorthogonal decompositions for categorical Donaldson--Thomas theory via $\Theta $-stratifications.} arXiv:2106.05496 (2021).

\bibitem[Tod21b]{Tod5}
Toda, Yukinobu. 
{\em Categorical wall-crossing formula for Donaldson-Thomas theory on the resolved conifold.} arXiv: 2109.07064 (2021).

\bibitem[Tod21c]{Tod6}
Toda, Yukinobu. 
{\em Derived categories of Quot schemes of locally free quotients via categorified Hall products.} arXiv:2110.02469 (2021).

\bibitem[To{\"e}14]{ToenDAG}
To{\"e}n, Bertrand. 
{\em Derived algebraic geometry.} EMS Surveys in Mathematical Sciences 1, no. 2 (2014): 153--240.

 \bibitem[TVa07]{TVa07}
To\"en, Bertrand and Michel Vaqui\'e, 
{\em Moduli of objects in dg-categories},
  Ann. Sci. \'Ecole Norm. Sup. (4) \textbf{40} (2007), no.~3, 387--444.
  
 \bibitem[TVe05]{HAGI}
To{\"e}n, Bertrand, and Gabriele Vezzosi. 
 {\em Homotopical algebraic geometry I: Topos theory}. Advances in mathematics 193, no. 2 (2005): 257-372.
  
\bibitem[TVe08]{HAGII}
To{\"e}n, Bertrand, and Gabriele Vezzosi. 
{\em Homotopical Algebraic Geometry II: Geometric Stacks and Applications}. 
Memoirs of the American Mathematical Society 193, no. 902 (2008).


\bibitem[T18]{RT}
Thomas, Richard P.
{\em Notes on homological projective duality.} Algebraic geometry: Salt Lake City 2015, Proc. Sympos. Pure Math., vol. 97, Amer. Math. Soc., Providence, RI, 2018, pp. 585--609.

\bibitem[X21]{Xie21}
Xie, Fei. 
{\em Nodal quintic del Pezzo threefolds and their derived categories.} arXiv:2108.03186 (2021).


\bibitem[Z20]{Z20}
Zhao, Yu. 
{\em A Categorical Quantum Toroidal Action on Hilbert Schemes.} arXiv:2009.11267 (2020).

\bibitem[Z21]{Z21}
Zhao, Yu. 
{\em Moduli Space of Sheaves and Categorified Commutator of Functors.} arXiv:2112.12434 (2021).


    
 \end{thebibliography}
\end{document}